\documentclass[10pt]{amsart}
\usepackage[T1]{fontenc}
\usepackage{lmodern}
\usepackage{amsmath,amssymb,amsthm,mathtools}
\usepackage{mathrsfs}
\usepackage{microtype}
\usepackage{needspace}
\usepackage{float}
\usepackage[hidelinks]{hyperref}
\usepackage[margin=1in]{geometry}
\hypersetup{pdftitle={Mirror Symmetry of Spencer-Hodge Decompositions in Constrained Geometric Systems: Selected Collisions and Weyl Splittings},pdfauthor={Dongzhe Zheng}}

\newcommand{\C}{\mathbb C}
\newcommand{\R}{\mathbb R}
\newcommand{\Z}{\mathbb Z}
\newcommand{\dd}{\mathrm d}
\newcommand{\ii}{\mathrm i}
\newcommand{\SU}{\mathrm{SU}}
\newcommand{\U}{\mathrm U}
\newcommand{\tr}{\operatorname{tr}}
\newcommand{\rank}{\operatorname{rank}}
\newcommand{\Hom}{\operatorname{Hom}}
\newcommand{\End}{\operatorname{End}}
\newcommand{\ind}{\operatorname{ind}}
\newcommand{\Vol}{\operatorname{Vol}}
\newcommand{\pr}{\operatorname{pr}}
\newcommand{\id}{\operatorname{id}}
\newcommand{\cK}{\mathcal K}
\newcommand{\cD}{\mathcal D}
\newcommand{\cA}{\mathcal A}

\newcommand{\Bock}{\operatorname{Bock}}

\newtheorem{theorem}{Theorem}[section]
\newtheorem{proposition}[theorem]{Proposition}
\newtheorem{corollary}[theorem]{Corollary}
\newtheorem{lemma}[theorem]{Lemma}
\theoremstyle{definition}
\newtheorem{definition}[theorem]{Definition}
\newtheorem{example}[theorem]{Example}
\theoremstyle{remark}
\newtheorem{remark}[theorem]{Remark}

\title[Mirror transport and collisions]{Mirror Symmetry of Spencer-Hodge Decompositions in Constrained Geometric Systems:\\Selected Collisions and Weyl Splittings}
\author{Dongzhe Zheng}
\address{Princeton University}
\email{dz5992@princeton.edu}
\date{September 2026}

\begin{document}
\begin{abstract}
We formulate mirror symmetry for compatible principal-bundle
systems with a genuinely graded polynomial Spencer complex.
Exact parity identifies its Hodge spectra, while gauge
implementation is governed by eigenbundle transport. For
parameterized families, a spectral coefficient extension realizes
the eigenline incidence as its projectivized harmonic zero modes.
A relative symplectic suspension separates scalar curvature from
flag energy. On a fixed spatial bundle, mixed-Chern
orthogonality constrains parameter-sphere charges; a nonzero
two-band charge precludes contactness over a closed base.

This incidence viewpoint organizes three kinds of spectral
degeneration. Continuous diagonalization across an arbitrary
closed nowhere-dense collision set is controlled by flag
extension, permutation monodromy, and eigenline Chern classes.
Smooth flat gaps can conceal flag behavior from every jet;
polynomial gaps yield finite-jet semialgebraic blow-ups. Local
Dirichlet minimality then selects a rigid cone. We classify
the minimizing radial cones into every $\C P^N$: the nonconstant
links are isometric projective-line hedgehogs. This gives
projective singular-measure stability and, together with a
Grassmannian degree bound, power convergence at all
three-dimensional minimizing singularities in these targets.

For complexified four-band pencils the same incidence is read
through geometric-rank minors, giving complex isolation and a
degree-twenty translation cover in a crystalline model. For every
spin, jointly polynomial Hermitian unfoldings attain both
universal Weyl-point count bounds; constant spin-$3/2$
translations have two sixteen-point chambers with different gap
distributions. Throughout, eigenflags and geometric rank retain
information absent from the scalar characteristic polynomial.
\end{abstract}
\maketitle

\section{Introduction}

Let $P\to M$ be a principal bundle with compact structure group
$G$ and invariant inner product $B$ on $\operatorname{Lie}(G)$.
Let $A$ be a connection and $u$ a nonzero parallel adjoint field.
The hyperplane $\ker B(u,A)$ is unchanged by
$u\mapsto-u$, yet its Levi form, contactness, and gauge realization
carry distinct information. A polynomial Spencer differential
defined on the correct bigraded coefficient bundle has an exact
degree-parity mirror, giving unitary equivalence of the full Hodge
spectra (Theorem~\ref{thm:spencer-parity}). A Weyl or sign action on
$u$ is implemented by a gauge only when the relevant eigenbundles
admit a transporter (Theorem~\ref{thm:weyl-transport}). These are
the two foundations developed first in
Section~\ref{sec:mirror-foundations}; further algebraic and
gauge refinements are collected in Appendix~\ref{sec:gauge-refinements}.
Throughout, $M$ denotes a spatial base and $X$ an external
parameter space. Gauge transformations cover the identity;
characteristic classes are integral, with
$c_1(E)=[\ii\tr(F_{\nabla^E})/(2\pi)]$.

The connection with collisions becomes precise for a family
$H(z)$ of Hermitian spectral coefficients, where $z$ is an
external parameter. We tensor the Spencer complex with the
two-term eigenvalue complex of $H(z)-\lambda I$. Its
projectivized harmonic zero-mode incidence is
\begin{equation}\label{eq:intro-incidence}
 \mathfrak I_H
 =\{(z,\lambda,[v]):(H(z)-\lambda I)v=0\}.
\end{equation}
Over a simple-spectrum parameter it records the eigenline
cover; at a collision it retains the entire projective
eigenspace. The same family has a compatible-pair suspension on
an exact symplectic base, where the paired curvature depends on
eigenvalue gaps but the gauge class depends on eigenlines
(Proposition~\ref{prop:relative-spectral-suspension}).
The mixed-Chern identity
\eqref{eq:mixed-chern-full} makes the interplay nonformal: a
parameter-sphere charge on a fixed spatial bundle restricts its
spatial Chern classes, and in the closed-base two-band case it
prevents contactness (Theorem~\ref{thm:mixed-chern-orthogonality}).

Three operations on~\eqref{eq:intro-incidence} organize the paper.
Real flag limits govern continuation through collisions;
quantitative gaps expose those limits to finite jets, and energy
minimality selects the tangents that form. Over complex parameters,
geometric-rank minors detect repeated eigenspaces missed by the
scalar discriminant. Each operation has its own stated hypotheses.

\subsection*{Flag limits and global realization}
Linking flags control gauge-orbit closure and simple-spectrum
approximation at an isolated collision. Across any closed
nowhere-dense collision set, diagonalization is governed by flag
extension, permutation monodromy, and eigenline Chern classes.
An isolated flag cluster can be any nonempty compact connected
set; two smooth three-band fields can even share their spectrum
and infinite jet while only one is diagonalizable. Thus the
direct $N\times N$ extension of Cyr and collaborators' derivative
condition~\cite[Theorem~2.6]{CyrEtAl} fails at $N=3$.
Theorem~\ref{thm:interval-flag-limit} gives the exact criterion
for paths with finitely many collisions.

For global assembly on a two-complex $X$, let
$\rho:\pi_1(X)\to S_N$ be the permutation monodromy,
$c\in H^2(X;\Z^N_\rho)$ the twisted Chern class,
$\Sigma$ the coordinate-sum map, and $\Bock$ the integral
Bockstein. The joint obstruction is
$\Sigma_*c+\Bock(\operatorname{sgn}\rho)$
(Theorem~\ref{thm:joint-spectral-layers}). Spectral-cover direct
images extend fixed-polynomial realization to higher dimensions;
on a three-complex the comparison classes form a transfer kernel,
extending Friedman--Park's split-root setting
\cite{FriedmanPark2016,FriedmanPark2024}.

\subsection*{Formation and selection}
For a local Hermitian field $A(x)$, write $g_A(x)$ for the
minimum separation between adjacent prescribed spectral blocks;
for a complete flag the blocks are its ordered eigenlines.
The local geometry changes when a
gap or an energy condition is imposed:
\begin{equation}\label{eq:intro-collision-hierarchy}
\begin{aligned}
 \text{smooth flat gaps}
   &\ \Longrightarrow\ \text{infinite jets may miss the flag limit},\\
 g_A(x)\ge c|x|^q
   &\ \Longrightarrow\ \text{a finite-jet semialgebraic flag graph},\\
 \text{polynomial gap and local flag-energy minimality}
   &\ \Longrightarrow\ \text{selected minimizing tangents}.
\end{aligned}
\end{equation}
The selected normal form gives power convergence and an exact
near-identity gauge (Theorem~\ref{thm:selected-collision-normal-form}).
The two-band coefficient vector $(x,y,z^3)$ has a polynomial gap
and a horn fiber, showing the extra role of minimality. For
ordered partial flags, the first nonzero density is $4\pi$,
realized precisely by pair-root hedgehogs
(Theorem~\ref{thm:primitive-flag-density}).

The central classification identifies every smooth
$g:S^2\to\C P^N$ whose radial extension minimizes Dirichlet
energy: $g$ is constant or a holomorphic or antiholomorphic
projective-line map of absolute degree one, with
$g^*g_{\rm FS}=\tfrac14g_{S^2}$
(Theorem~\ref{thm:projective-minimizing-cones}). We use the
Fubini--Study metric of holomorphic sectional curvature four;
the line cone has energy $2\pi$ (or $8\pi$ in Fujioka's
curvature-one normalization). This answers Fujioka's
classification question and his conjecture that these lines
exhaust the nonconstant minimizing tangent maps
\cite[pp.~127--128]{Fujioka1993}.
The proof combines two-centre surgery, a unitary-averaged filling
bound, and Fujioka's higher-degree exclusion and line minimality.
Mixed-Chern orthogonality then excludes contact for this unit
link charge in a fixed-bundle family over a closed spatial base
(Corollary~\ref{cor:selected-contact-obstruction}); the open-base
suspension realizes the link with contact away from its collision.

Grassmannian averaging and Dong's uniton bound~\cite{Dong2002}
restrict minimizing links into $\operatorname{Gr}(2,4)$ to
constants or holomorphic or antiholomorphic spheres of absolute
Pl\"ucker degree at most three
(Theorem~\ref{thm:grassmannian-rank-two-holomorphic-selection}).
Their Jacobi fields are integrable, giving positive-power tangent
convergence at every three-dimensional minimizing singularity
by Adams--Simon~\cite{AdamsSimon1988}
(Theorem~\ref{thm:unconditional-selected-power}); the $\C P^2$
case was studied by Lemaire--Wood~\cite{LemaireWood2002}.
Nonzero determinant-line charge excludes a block-preserving
unitary orientation-reversing mirror
(Corollary~\ref{cor:selected-mirror-obstruction}).
For $\C P^N$, the classification also yields entire-map rigidity
and uniform singularity separation, extending
Almgren--Lieb~\cite{AlmgrenLieb1988}, and extends the
singular-measure stability theorem of
Mazowiecka--Mi\'skiewicz--Schikorra~\cite[Introduction]{MMS2024}
in every domain dimension $m\ge3$
(Theorem~\ref{thm:projective-boundary-stability}).

\subsection*{Geometric rank and real splitting}
For complexified four-band pencils, a repeated root of the
characteristic polynomial need not give a repeated eigenspace.
The geometric rank-two ideal proves complex isolation and
length twenty for every real-isolated parameter of the
crystalline pencil of Alpin and collaborators
(Theorem~\ref{thm:four-band-complex-isolation}), resolving the
specific conjecture of Frank, Pint\'er, Varjas, and
P\'alyi~\cite[Section~4.6]{FrankPinterVarjasPalyi}.
Its finite-flat translation cover also supports exact real-point
certificates and reveals regular-nilpotent discriminant tacnodes.

For spin $s$, put $n=2s+1$. The universal local Weyl-point
counts
\[
 L_n=\frac{n(n^2-1)}6,\qquad U_n=\frac{n^2(n^2-1)}{12}
\]
are attained by jointly polynomial Hermitian unfoldings for
every spin (Corollary~\ref{cor:spin-both-bounds-sharp}). The
lower endpoint has precedents in biaxial-spin models
\cite{KececiogluGarg,Bruno,GhassemiTabrizi}; the upper
construction realifies the entire complex intersection.
Constant translations form a narrower class. For spin $3/2$,
two open sixteen-point chambers have gap counts $(3,6,7)$ and
$(3,4,9)$ (Theorem~\ref{thm:spin-constant-sixteen} and
Proposition~\ref{prop:spin-nine-outer}); the sharp constant
maximum remains between $16$ and $20$.

All manifolds and connections are smooth unless a finite CW
complex is specified.

\section{Compatible systems and exact Spencer--Hodge mirrors}
\label{sec:mirror-foundations}

There are two mirror questions for a compatible system. A sign
change of the coefficient differential can act by an exact
parity on its Hodge complex. Realizing a symmetry of a parallel
adjoint field by a gauge transformation also requires a bundle
transporter. We establish both statements with their respective
hypotheses before studying parameter-dependent spectral
collisions.

\subsection{Compatible hyperplanes and their curvature}
\label{sec:geometry}

Let $G$ be compact and connected with an invariant positive inner
product $B$ on its Lie algebra $\mathfrak g$.  Let $M$ be connected,
let $A$ be a connection on a right principal $G$-bundle $P\to M$,
and let
$u\in\Gamma(\operatorname{ad}P)$ be nonzero and parallel:
$\nabla_Au=0$.  In equivariant-map notation, put
\begin{equation}\label{eq:alpha-beta}
 \alpha=B(u,A),\qquad \cD=\ker\alpha,\qquad
 \pi^*\beta=B(u,F_A).
\end{equation}
Parallelism and the Bianchi identity make $\beta$ a closed two-form
on $M$.  The norm $B(u,u)$ is a positive constant.  At $p\in P$, if
$V_p=\ker\dd\pi_p$ and $u(p)=\mu$, then
\begin{equation}\label{eq:vertical-dimension}
 \cD_p\cap V_p=(\ker B(\mu,\cdot))_p^{\#},
 \qquad \dim(\cD_p\cap V_p)=\dim G-1.
\end{equation}
The horizontal bundle $\ker A$ is a separate splitting of $TP$.

\begin{proposition}[Levi blocks]\label{prop:levi}
The horizontal--vertical mixed block of $\dd\alpha|_{\cD}$ vanishes.
For horizontal lifts $X^H,Y^H$ and vertical fundamental vectors
$a^\#,b^\#$ with $a,b\perp\mu$,
\begin{equation}\label{eq:levi}
 \dd\alpha(X^H,Y^H)=\beta(X,Y),\qquad
 \dd\alpha(a^\#,b^\#)=B(\mu,[a,b]).
\end{equation}
Writing $\mathfrak g_\mu=\{a\in\mathfrak g:[a,\mu]=0\}$, the
Levi radical is exactly
\begin{equation}\label{eq:levi-radical}
 \ker(\dd\alpha|_{\cD_p})
 = (\ker\beta_{\pi(p)})^H
   \oplus(\mathfrak g_\mu\cap\mu^\perp)^\#.
\end{equation}
If $\dim M$ is even, $\alpha$ is contact exactly when $\beta$ is
symplectic and $\mathfrak g_\mu=\R\mu$.  In particular, this is
the criterion for $G=\SU(2)$; for a group of rank at least two,
the vertical Levi radical is nonzero for every nonzero $\mu$.
\end{proposition}

\begin{proof}
Since $\nabla_Au=0$, differentiation of $B(u,A)$ gives the stated
horizontal curvature term.  In the fiber identification
$p g\leftrightarrow g$, equivariance gives
$\alpha=B(\mu,\dd g\,g^{-1})$; the right Maurer--Cartan equation
gives the positive vertical sign in~\eqref{eq:levi}.
Equivariance and parallelism cancel the
mixed terms.  By invariance of $B$,
$B(\mu,[a,b])=B([\mu,a],b)$, so the kernel of the vertical
block on $\mu^\perp$ is
$\mathfrak g_\mu\cap\mu^\perp$.  The horizontal kernel is
$(\ker\beta)^H$, proving~\eqref{eq:levi-radical}.  A nonzero
$\mu$ lies in $\mathfrak g_\mu$, hence this vertical kernel
vanishes exactly when $\mathfrak g_\mu=\R\mu$.  Every element
belongs to the Lie algebra of some maximal torus, so
$\dim\mathfrak g_\mu\ge\operatorname{rank}G$.  In
$\mathfrak{su}(2)$ the centralizer of a nonzero element is
precisely its span.  This is the contact-coupling criterion in
the present geometry~\cite{Lerman}.
\end{proof}

For a block unitary connection on a Hermitian bundle
$E=\bigoplus E_j$, a parallel skew-Hermitian field acts by a constant
scalar on each $E_j$.  Its paired curvature depends on the trace
curvatures of the blocks.  Their higher Chern classes need not be
recoverable from that two-form. This distinction reappears in
the block-swap index of Appendix~\ref{sec:index}.

\subsection{A genuine polynomial Spencer--Hodge complex}
\label{sec:spencer}

Let $(W,\nabla^W)$ be a flat orthogonal real vector bundle of rank
$r$ over a closed Riemannian manifold $(M,g)$, and let $m\ge1$.
The fiberwise truncated polynomial-form bundle is
\[
 \cK_m^b(W)=\bigoplus_{j=0}^m
       \operatorname{Sym}^j(W^*)\otimes\Lambda^b W^*,
 \qquad 0\le b\le r.
\]
In a parallel orthonormal frame, with polynomial coordinates
$x_1,\ldots,x_r$ and exterior generators $\epsilon^1,\ldots,
\epsilon^r$, define
\begin{equation}\label{eq:spencer-delta}
 \delta(f\otimes\eta)
   =\sum_{a=1}^r \frac{\partial f}{\partial x_a}
       \otimes(\epsilon^a\wedge\eta).
\end{equation}
The expression is invariant under orthogonal frame changes.  It has
bidegree $(j,b)\mapsto(j-1,b+1)$ and satisfies $\delta^2=0$, since
polynomial derivatives commute and exterior products anticommute.
This is the classical polynomial Spencer differential; it is also the
fiberwise polynomial de Rham differential~\cite{Seiler}.
Here $t\mapsto-t$ changes only the coefficient differential.  It is
not identified with $u\mapsto-u$ for an arbitrary adjoint field.

Let $\varpi>0$ be smooth and use the same weighted inner product
$\int_M\varpi\langle\cdot,\cdot\rangle\,\dd V_g$ for every parameter.
Set
\[
 \cA^k=\bigoplus_{a+b=k}\Omega^a(M;\cK_m^b(W)),
 \qquad
 D_t\big|_{\Omega^a(M;\cK_m^b)}
       =\dd_{\nabla^W}+(-1)^a t\delta,
 \quad t\in\R.
\]
Here $\dd_{\nabla^W}$ is the covariant exterior differential for the
induced flat connection on each tensor bundle.  The first summand
raises the base-form degree; the second raises the internal exterior
degree.  Flatness and parallelness of $\delta$ imply $D_t^2=0$.

\begin{theorem}[Parity and the complete Hodge spectrum]\label{thm:spencer-parity}
Let $J$ act by $(-1)^j$ on symmetric polynomial degree $j$.  Then
$J$ is unitary for the weighted metric and, in every total degree,
\begin{equation}\label{eq:parity}
 JD_tJ^{-1}=D_{-t},\qquad
 J\Delta_tJ^{-1}=\Delta_{-t},
 \qquad \Delta_t=D_tD_t^*+D_t^*D_t.
\end{equation}
The operators $\Delta_t$ are elliptic and have compact resolvent.
Their entire spectra agree with multiplicities, and $J$ identifies
their harmonic spaces at $t$ and $-t$.  Moreover the product metric gives
\begin{equation}\label{eq:product-spectrum}
 \Delta_t=\Delta_{M,\varpi,\nabla^W}\otimes1
       +t^2 1\otimes(\delta\delta^*+\delta^*\delta),
\end{equation}
with the coefficient-bundle interpretation of the tensor notation.
The equality of spectra is an exact conjugacy and requires no compact
perturbation argument.
\end{theorem}

\begin{proof}
The base differential preserves symmetric degree and commutes with
$J$.  Formula~\eqref{eq:spencer-delta} lowers that degree by one, so
$J\delta J^{-1}=-\delta$.  This proves the first identity in
\eqref{eq:parity}.  Unitarity and the common weight imply the same
identity for formal adjoints and hence for Hodge Laplacians.
The principal symbol of $D_t+D_t^*$ is Clifford multiplication by
a nonzero cotangent vector on the base, so $\Delta_t$ has scalar
principal symbol $|\xi|^2$ and is elliptic on the finite-rank
coefficient bundle.  Compactness of $M$ gives compact resolvent.
The horizontal and vertical differentials act on different graded
factors and anticommute.  Their formal adjoints also anticommute in
the mixed terms, because $\delta$ is parallel and the weight is a
scalar function on the base.  Expanding the Laplacian leaves exactly
the two terms in~\eqref{eq:product-spectrum}.
\end{proof}

\begin{example}[A zero-mode wall]\label{ex:zero-wall}
Take $M$ to be a point, $r=m=1$, and make
$1,x,\epsilon,x\epsilon$ orthonormal.  The only nonzero value of
$\delta$ is $\delta x=\epsilon$.  For $t\ne0$ the two harmonic
vectors are $1$ and $x\epsilon$; at $t=0$ all four vectors are
harmonic.  The nonzero eigenvalue is $t^2$ in degrees zero and one.
Thus the full spectra at $t$ and $-t$ agree while the zero-mode
multiplicity jumps at the degenerate parameter.  Tensoring with the
de Rham complex of any closed $M$ gives the same phenomenon above a
positive-dimensional base.
\end{example}

\begin{remark}[Curvature boundary]
If $W$ is supplied with a nonflat connection, the same formula defines
an operator but $D_t^2$ contains the induced curvature action
$(\dd_{\nabla^W})^2$.  The resulting object is a curved differential
unless that action vanishes on the chosen coefficient bundle.  A
parallel adjoint field alone does not make the full adjoint bundle
flat; its span is a flat line.  Theorem~\ref{thm:spencer-parity}
applies to the stated flat coefficient geometry.
\end{remark}

\subsection{Weyl transporters and magnetic mirror gaps}
\label{sec:transport}

Let $M$ now be closed and connected.  Suppose a determinant-trivialized
Hermitian bundle has an orthogonal splitting
\[
 E=L_1\oplus\cdots\oplus L_n
\]
into lines, with diagonal unitary connection
$\nabla=\bigoplus_i\nabla_i$.  A regular parallel field has distinct
constant eigenvalues on these lines.  For $w\in S_n$ define the
\emph{Weyl transporter lines}
\begin{equation}\label{eq:Kwi}
 K_{w,i}=\Hom(L_i,L_{w(i)}),\qquad
 \nabla_{w,i}=\nabla_{w(i)}\otimes\nabla_i^*,\qquad
 \mu_{w,i}=\lambda_0(\nabla_{w,i}^*\nabla_{w,i}).
\end{equation}
The bottom magnetic eigenvalue is nonnegative and exists by elliptic
compactness.  A $w$-gauge means a determinant-one unitary gauge map
carrying each $L_i$ to $L_{w(i)}$; it need not fix the numerical
eigenvalue assigned to that line.

\begin{theorem}[Topological and spectral transporter tests]
\label{thm:weyl-transport}
A $w$-gauge exists if and only if
\begin{equation}\label{eq:chern-test}
 c_1(L_i)=c_1(L_{w(i)})\quad\text{for every }i.
\end{equation}
An $\nabla$-preserving $w$-gauge exists if and only if
$\mu_{w,i}=0$ for every $i$.  If $F$ is any $w$-gauge and
$s_i:L_i\to L_{w(i)}$ are its unitary components, then
\begin{equation}\label{eq:weyl-energy}
 \int_M|\nabla^{\End E}F|_{\mathrm{HS}}^2
 =\sum_i\int_M|\nabla_{w,i}s_i|^2
 \ge\Vol(M)\sum_i\mu_{w,i}.
\end{equation}
Thus the vector $(\mu_{w,1},\ldots,\mu_{w,n})$ resolves the
connection-preserving part of a global Weyl mirror after the Chern
test.
\end{theorem}

\begin{proof}
The restriction of a $w$-gauge to $L_i$ is a unitary nowhere-zero
section of $K_{w,i}$.  Such a section exists exactly when the line
bundle is trivial, which is~\eqref{eq:chern-test}.  Given these
sections, their direct sum is unitary.  Its determinant is a circle-
valued function; multiplying one component by the inverse
determinant produces a special-unitary gauge.  It preserves
$\nabla$ precisely when all the line components are parallel.
For a unitary connection on a line over a connected closed base,
$\mu_{w,i}=0$ exactly when a nonzero parallel section exists: a
zero-eigenvector $s$ obeys
$0=\langle\nabla^*\nabla s,s\rangle_{L^2}=\|\nabla s\|_2^2$,
and a nonzero parallel section has constant positive norm.  If all
components are parallel, their determinant is constant, so the
determinant adjustment can be made by a constant phase.
The blockwise Hilbert--Schmidt norm gives the identity in
\eqref{eq:weyl-energy}; the inequality follows from the Rayleigh
principle and $|s_i|=1$.
\end{proof}

For singular eigenblocks, the transporter is a space of
\emph{isomorphisms}, not a vector space.  A linear Hom-bundle
eigenvalue can detect a parallel map of deficient rank.  The next
definition retains invertibility.

\begin{definition}[Unitary transporter energy]\label{def:nonlinear}
For equal-rank Hermitian bundles with unitary connections
$(E,\nabla^E)$ and $(F,\nabla^F)$ over $M$, set
\begin{equation}\label{eq:nonlinear-energy}
 \mathfrak m(E,F)=
 \inf_{\phi:E\stackrel{\simeq}{\longrightarrow}F\ \mathrm{unitary}}
 \frac1{\Vol(M)}\int_M
       |\nabla^{\Hom(E,F)}\phi|_{\mathrm{HS}}^2,
\end{equation}
with value $+\infty$ when no unitary bundle isomorphism exists.
\end{definition}

\begin{theorem}[Zero energy and the nonlinear transporter]
\label{thm:nonlinear}
On a closed connected $M$, $\mathfrak m(E,F)=0$ if and only if there
is a parallel unitary isomorphism $E\to F$.  If the common rank is
$r$ and $\lambda_0^{\Hom}$ is the bottom eigenvalue of the Hom-bundle
connection Laplacian, then
\begin{equation}\label{eq:nonlinear-lower}
 \mathfrak m(E,F)\ge r\lambda_0^{\Hom}.
\end{equation}
The implication $\lambda_0^{\Hom}=0\Rightarrow\mathfrak m=0$ can
fail when $r>1$.
\end{theorem}

\begin{proof}
Only the zero-energy implication requires proof.  Let $\phi_j$ be
unitary isomorphisms with energies tending to zero.  Their pointwise
Hilbert--Schmidt norm is $\sqrt r$, so they have bounded $W^{1,2}$
norm in a finite system of trivializations.  Rellich compactness
gives a subsequence converging strongly in $L^2$ and almost
everywhere to a section $\phi$.  The unitary matrices form a closed
compact subset of each Hom fiber, hence $\phi$ is unitary almost
everywhere.  Its weak covariant derivative vanishes.  Parallel
transport, or local regularity for the first-order equation
$\nabla\phi=0$, makes $\phi$ a smooth parallel unitary
isomorphism.  The converse is immediate.  Every unitary $\phi$ has
$\int|\phi|^2=r\Vol(M)$, and the Rayleigh principle proves
\eqref{eq:nonlinear-lower}.

For the final assertion take $M=S^1$ and trivial rank-two bundles
whose flat holonomies are respectively
$\operatorname{diag}(1,e^{\ii\theta})$ and
$\operatorname{diag}(1,e^{\ii\varphi})$, where the two nontrivial
phases are distinct.  Projection between the two fixed first
eigenlines is a nonzero parallel Hom section, so
$\lambda_0^{\Hom}=0$.  A parallel isomorphism would intertwine the
full holonomy matrices, which have different eigenvalue multisets.
Thus $\mathfrak m>0$ although the linear bottom eigenvalue is zero.
\end{proof}

\begin{theorem}[Transporter assignment and flat-holonomy formula]
\label{thm:assignment}
Suppose the bundles in Definition~\ref{def:nonlinear} split as
$E=\bigoplus_{i=1}^rL_i$ and $F=\bigoplus_{j=1}^rK_j$, with split
unitary connections.  Set
\[
 c_{ji}=\lambda_0\!\left(
 \nabla_{\Hom(L_i,K_j)}^*\nabla_{\Hom(L_i,K_j)}\right).
\]
Then
\begin{equation}\label{eq:assignment-bound}
 \mathfrak m(E,F)\ge\min_{\sigma\in S_r}\sum_{i=1}^r c_{\sigma(i),i}.
\end{equation}
If $M=\R^d/\Lambda$ is a flat torus and both connections are flat,
write their holonomy-character multisets as
$\{[a_1],\ldots,[a_r]\}$ and $\{[b_1],\ldots,[b_r]\}$ in the
dual flat torus $\R^d/(2\pi\Lambda^*)$.  With $d_*$ the quotient
distance induced by the Euclidean metric on the covector space,
equality holds and
\begin{equation}\label{eq:transport-wasserstein}
 \mathfrak m(E,F)
 =\min_{\sigma\in S_r}\sum_i d_*([a_i],[b_{\sigma(i)}])^2
 =rW_2^2\!\left(\frac1r\sum_i\delta_{[a_i]},
                      \frac1r\sum_j\delta_{[b_j]}\right).
\end{equation}
In particular the nonlinear energy detects the complete multiset
of flat holonomy characters, while the bottom linear Hom eigenvalue
detects only one matching pair.
\end{theorem}

\begin{proof}
For a unitary isomorphism $\phi:E\to F$, write $\phi_{ji}$ for
its linewise entries.  Orthogonality of the split Hom connection
and the Rayleigh principle give
\[
 \frac1{\Vol(M)}\int_M|\nabla\phi|^2
 \ge\sum_{j,i}c_{ji}p_{ji},\qquad
 p_{ji}=\frac1{\Vol(M)}\int_M|\phi_{ji}|^2.
\]
Pointwise unitarity makes every row and column sum of $(p_{ji})$
equal to one.  The minimum of a linear functional on the doubly
stochastic matrices is attained at a permutation matrix, proving
\eqref{eq:assignment-bound}.

On the flat torus every flat unitary representation splits into
characters.  In flat frames the Hom line from $[a_i]$ to $[b_j]$
has Fourier eigenvalues $|b_j-a_i+2\pi\xi|^2$, $\xi\in\Lambda^*$;
its bottom is $d_*([a_i],[b_j])^2$.  For a minimizing permutation,
put a minimizing unit-modulus Fourier mode in each selected matrix
entry and zero elsewhere.  This is a global unitary isomorphism
attaining the lower bound.  The last equality is the finite
 equally weighted transport-assignment identity.
\end{proof}

\begin{remark}
Formula~\eqref{eq:transport-wasserstein} concerns unitary block
isomorphisms.  A prescribed determinant, as in a special-unitary
ambient gauge, can impose an additional phase constraint and raise
the minimum.  The assignment bound~\eqref{eq:assignment-bound}
continues to hold under that constraint.
\end{remark}

For a principal reduction with compact stabilizer $H$, the
isomorphism torsor between a reduction and its Weyl twist is the
principal-bundle counterpart of~\eqref{eq:nonlinear-energy}.
Its sections are actual mirror lifts; parallel sections are
connection-preserving lifts.  Formula~\eqref{eq:nonlinear-energy}
provides the corresponding concrete analytic model for unitary
eigenblocks.

\begin{corollary}[The rank-one sign mirror]\label{cor:su2-transport}
Let $E=L\oplus L^{-1}$ with its determinant trivialization and a
diagonal $\SU(2)$ connection.  The transporter line for swapping the
two summands is $K=L^{-2}$.  A sign gauge exists exactly when
$2c_1(L)=0$.  Every such gauge has the form
\[
 F_s=\begin{pmatrix}0&-s^{-1}\\s&0\end{pmatrix},
 \qquad s:L\to L^{-1}\ \text{unitary},
\]
and satisfies $F_s^2=-I$.  It preserves the connection exactly when
$s$ is parallel, equivalently when
$\lambda_0(\nabla_K^*\nabla_K)=0$.
\end{corollary}

\begin{proof}
The topological and parallel-section claims are
Theorem~\ref{thm:weyl-transport}.  The determinant-one off-diagonal
matrix is forced by unitarity once $s$ is chosen, and direct
multiplication gives $F_s^2=-I$.
\end{proof}

\begin{remark}[Spectral equality without a gauge mirror]
Complex conjugation is an antiunitary map from $L$ to the distinct
bundle $L^{-1}$ and intertwines their dual unitary connections.
For every common positive scalar weight $\varpi$, it intertwines the
covariant exterior derivatives, their weighted adjoints, and their
covariant Hodge Laplacians in each form degree.  The two full spectra
therefore agree even when $2c_1(L)\ne0$ and no sign gauge exists.
This is the standard charge-conjugation symmetry of magnetic
operators; spectra alone do not supply the missing bundle
isomorphism.  If the connection is flat, the
covariant exterior derivative is a genuine complex and the same map
gives an anti-linear twisted-cohomology isomorphism.  For nonzero
curvature its square is curvature wedge multiplication, so no
twisted cohomology is asserted.
\end{remark}

\section{Spectral incidence in a parameterized mirror family}
\label{sec:spectral-incidence}

A field parallel on each spatial slice has constant spectrum
there, while its eigenvalues and eigenflags may vary with an
external parameter. This distinction connects the exact
Spencer--Hodge mirror of Section~\ref{sec:mirror-foundations}
with collision geometry.

\subsection{A relative compatible-pair suspension}

Let $(Y^{2m},\omega_Y=\dd\vartheta)$ be an exact symplectic
manifold, let $X$ be a smooth parameter manifold, and let
$H:X\to\operatorname{Herm}_N$ be smooth. Set
\[
 \widehat H=H-\frac{\tr H}{N}I_N,\qquad
 q=\ii\widehat H\in\mathfrak{su}(N),\qquad
 B(a,b)=-\frac12\tr(ab),\qquad f=B(q,q).
\]
Removing the scalar trace does not change eigenspaces.
For $N=2$, write $\sigma=(\sigma_1,\sigma_2,\sigma_3)$
for the Pauli matrices, normalized by
$(v\cdot\sigma)^2=|v|^2I_2$ for $v\in\R^3$.

\begin{proposition}[Relative compatible-pair suspension]
\label{prop:relative-spectral-suspension}
Fix $z\in X$ and let $g\in\SU(N)$ denote the fiber
coordinate on $P=Y\times\SU(N)$. Define
\[
 A_z=\operatorname{Ad}_{g^{-1}}(q(z)\vartheta)+g^{-1}\dd_Yg,
 \qquad u_z=\operatorname{Ad}_{g^{-1}}q(z).
\]
Then $\nabla^{Y}_{A_z}u_z=0$. Writing
$\lambda_1(z),\ldots,\lambda_N(z)$ for the eigenvalues of
$H(z)$ counted with multiplicity, one has
\begin{equation}\label{eq:relative-suspension-curvature}
 F_{A_z}=\operatorname{Ad}_{g^{-1}}q(z)\,\omega_Y,
 \qquad
 \beta_z:=B(u_z,F_{A_z})=f(z)\omega_Y,\qquad
 f(z)=\frac12\tr\widehat H(z)^2
     =\frac1{2N}\sum_{i<j}(\lambda_i(z)-\lambda_j(z))^2.
\end{equation}
For $N=2$, the one-form
$\alpha_z=B(u_z,A_z)$ is contact exactly when $q(z)\ne0$.
At $q(z)=0$ it vanishes; this is a boundary point of the
nonzero-field compatible-pair class. For $N\ge3$ the vertical
Levi radical has dimension at least $N-2$, even at simple
spectrum.

If $N=2$ and $H=aI+h\cdot\sigma$ for real-valued $a$ and
$h:X\to\R^3$, choose a Riemannian
metric on $X$ and put
$n=h/|h|$, $P_+=(I+n\cdot\sigma)/2$ where $h\ne0$, and
\[
 e_F:=\frac12\bigl(|\dd_XP_+|_{\mathrm{HS}}^2+
                    |\dd_X(I-P_+)|_{\mathrm{HS}}^2\bigr)
      =\frac12|\dd_Xn|^2.
\]
With this metric, the same field satisfies
\begin{equation}\label{eq:curvature-flag-energy-splitting}
 f=|h|^2=\frac14(\lambda_+-\lambda_-)^2,
 \qquad
 |\dd_Xh|^2=\frac{|\dd_Xf|^2}{4f}+2f e_F.
\end{equation}
The $\C P^1$ Dirichlet density in the Fubini--Study
normalization used below is $e_{\rm FS}=e_F/2$.
Thus the ordered two-band flag energy has twice the
Fubini--Study projective energy: their unit-hedgehog
densities are $4\pi$ and $2\pi$, respectively.
\end{proposition}

\begin{proof}
For a fixed $z$, $q(z)$ has no $Y$-derivative. In the global
section $g=I$, the connection is $q(z)\vartheta$ and its parallel
field is $q(z)$; their bracket vanishes. The curvature is
$q(z)\dd\vartheta$. Gauge equivariance yields the first two
formulas in~\eqref{eq:relative-suspension-curvature}. The last
identity follows from
$\sum_{i<j}(\lambda_i-\lambda_j)^2
 =N\sum_i(\lambda_i-\tr H/N)^2$.
Proposition~\ref{prop:levi} gives contactness and the vertical
radical. In two bands, $B(q,q)=|h|^2$ and
$|\dd P_+|_{\mathrm{HS}}^2=|\dd n|^2/2$.
Writing $h=\sqrt f\,n$ and using
$\langle n,\dd n\rangle=0$ proves
\eqref{eq:curvature-flag-energy-splitting}.
\end{proof}

The field $u_z$ is parallel along $Y$, generally not along $X$.
Positive-dimensional exact symplectic $Y$ is nonclosed; $Y=\R^2$
gives a global model. The energy identity is kinematic, with
local minimality imposed separately below.

\subsection{A spatial--parameter compatibility law}

The open-base suspension allows arbitrary parameter links.
Over a closed base, the total family bundle constrains their
coexistence with spatial curvature through the integral flux
$\Phi(\mathcal E):=\pi_!c(\mathcal E)$ for
$\pi:M\times S^2\to M$, where $\pi_!$ denotes slant product
with the oriented parameter sphere. This retains torsion.

\begin{theorem}[Mixed-Chern orthogonality]
\label{thm:mixed-chern-orthogonality}
Let $M$ be connected and let $\mathcal E\to M\times S^2$ be
a Hermitian bundle of rank $N$. Suppose a continuous,
simple-spectrum Hermitian family on $\mathcal E$ has ordered
eigenlines $L_1,\ldots,L_N$. With $\eta$ the positive generator
of $H^2(S^2;\Z)$, write uniquely
\[
 c_1(L_i)=\pr_M^*a_i+k_i\pr_{S^2}^*\eta,
 \qquad a_i\in H^2(M;\Z),\quad k_i\in\Z.
\]
Then in the integral cohomology ring of $M$,
\begin{equation}\label{eq:mixed-chern-full}
 \Phi(\mathcal E)=
 \sum_{i=1}^N k_i\prod_{j\ne i}(1+a_j).
\end{equation}
If $\mathcal E\simeq\pr_M^*E$ for a fixed spatial bundle
$E\to M$, then $\Phi(\mathcal E)=0$ and, in particular,
\begin{equation}\label{eq:mixed-chern-first}
 \sum_i k_i=0,\qquad \sum_i k_i a_i=0.
\end{equation}
For $N=2$ and $\det\mathcal E$ trivial, if
$a=c_1(L_1|_{M\times\{s\}})$
and $k=c_1(L_1|_{\{x\}\times S^2})[S^2]$, this becomes
$\pi_!c_2(\mathcal E)=-2ka$. Thus $2ka=0$ when the
mixed second Chern class vanishes, in particular for a fixed
spatial bundle.

Assume additionally that $M$ is closed of positive even dimension,
$\mathcal E$ has structure group $\SU(2)$, and every member of the family
is the Hermitian form of a nonzero parallel adjoint field $u_s$
for a connection $A_s$ on its spatial slice. If $k\ne0$ and
$\pi_!c_2(\mathcal E)=0$, the
paired curvature $\beta_s=B(u_s,F_{A_s})$ is exact for every
$s\in S^2$. Consequently no form
$\alpha_s=B(u_s,A_s)$ is contact.
\end{theorem}

\begin{proof}
The eigenline decomposition gives
$\mathcal E=\bigoplus_iL_i$. Since $\eta^2=0$, the
coefficient of $\eta$ in the Whitney product for
$c(\mathcal E)$ is the right side of
\eqref{eq:mixed-chern-full}. Slant product extracts exactly
that coefficient. If $\mathcal E$ is a pullback, naturality
makes it zero. Its degree-zero part gives $\sum k_i=0$;
its degree-two part is
$(\sum k_i)(\sum a_i)-\sum k_i a_i$, proving
\eqref{eq:mixed-chern-first}. For a determinant-trivialized
rank-two bundle, $a_2=-a_1$ and $k_2=-k_1$, so
$\pi_!c_2(\mathcal E)=-2ka$. These are the Whitney
sum and naturality identities in integral Chern theory
\cite{HatcherVB}.

In the last setting, the two eigenlines are parallel for
$A_s$. Label $L_1$ so that the skew-Hermitian $u_s$ acts
on it by $+\ii b_s$ and on $L_2$ by $-\ii b_s$, where
$b_s>0$ is constant along $M$.
With the Chern--Weil convention fixed in the introduction,
\[
 [\beta_s]=-2\pi b_s\,a_{\R}
 \quad\text{in }H^2_{\rm dR}(M).
\]
Since $2ka=0$ and $k\ne0$, $a_{\R}=0$; hence $\beta_s$
is exact. An exact two-form cannot be symplectic on a closed
positive-dimensional manifold. Proposition~\ref{prop:levi}
therefore excludes contactness of $\alpha_s$.
\end{proof}

For a fixed spatial bundle, a nonzero two-band link charge
therefore excludes contactness over a closed base. In higher
rank, the full flux identity restricts eigenline Chern classes;
the vertical Levi radical already excludes contact for
$\SU(N)$, $N\ge3$.

\begin{corollary}[Charge--weight cancellation]
\label{cor:charge-weight-cancellation}
Suppose $\mathcal E=\pr_M^*E$ has structure group $\SU(N)$,
and $u_s$ is parallel on each spatial slice with fixed
eigenvalues $\ii\mu_i$. If the parameter-sphere charges
satisfy $k_i=c\mu_i$ for one nonzero real constant $c$,
then the paired curvature $B(u_s,F_{A_s})$ is exact on $M$
for every $s$.
\end{corollary}

\begin{proof}
Since the eigenlines are parallel,
$[B(u_s,F_{A_s})]=-\pi\sum_i\mu_i(a_i)_{\R}$.
Equation~\eqref{eq:mixed-chern-first} and $c\ne0$
make this class zero.
\end{proof}

For the canonical spin-$s$ sphere
$n\mapsto n\cdot(J_x,J_y,J_z)$, where
$(J_x,J_y,J_z)$ is the irreducible spin-$s$ Hermitian
representation and $m=-s,-s+1,\ldots,s$, the weight-$m$ eigenline
has parameter charge $-2m$ in our Chern convention.
Consequently its charge vector aligns with the weight vector,
so Corollary~\ref{cor:charge-weight-cancellation} applies to
any fixed-bundle spatial family with this eigenflag pattern.
The real Weyl resolutions of Section~\ref{sec:spin-counts}
are local parameter-space constructions and require no such
spatial realization.

The vanishing-flux condition is sharp. For a possibly twisted
$\SU(2)$ family bundle
$\mathcal E=L\oplus L^{-1}\to M\times S^2$ with
$c_1(L)=a+k\eta$, its mixed second Chern class is
\begin{equation}\label{eq:mixed-chern-compensation}
 \pi_!c_2(\mathcal E)=-2ka\in H^2(M;\Z).
\end{equation}
This mixed class compensates the fixed-bundle obstruction. On
$M\times S^2=\C P^1\times\C P^1$ take
$L=\mathcal O(1,1)$. Each spatial slice of
$L\oplus L^{-1}$ is isomorphic to
$\mathcal O(1)\oplus\mathcal O(-1)$, and a block connection
with positive area curvature on $\mathcal O(1)$ makes
$B(u_s,F_{A_s})$ symplectic. The induced one-form on its
principal $\SU(2)$ bundle is contact by
Proposition~\ref{prop:levi}. Nevertheless
$\pi_!c_2(\mathcal E)=-2c_1(\mathcal O(1))\ne0$.
Isomorphic spatial slices need not form a pullback family.

The positive eigenline $L_+=\operatorname{im}P_+$ over
$X\setminus h^{-1}(0)$ carries the information omitted by $f$.
A continuous family of $\SU(2)$ gauges that moves $q(z)$
to the fixed positive Cartan ray exists exactly when $L_+$
is trivial, equivalently when $c_1(L_+)=0$. A choice of unit
positive eigenvector supplies such a gauge and its orthogonal
complement; conversely a gauge supplies the eigenvector.
If gauges are allowed to depend on $y\in Y$, restriction to one
$y$ gives the same necessary condition.

\begin{example}[Scalar data with arbitrary link charge]
\label{ex:relative-flat-charge}
On $X=B^3$, take $H_k(r\omega)=\varphi(r)n_k(\omega)\cdot\sigma$,
with $H_k(0)=0$, $\varphi(r)=e^{-1/r^2}$ for $r>0$ and
$\deg n_k=k$. These smooth fields share the polynomial
$\lambda^2-\varphi(r)^2$, curvature scalar
$\beta_z=\varphi(r)^2\omega_Y$, contact-failure set, and
infinite jet at the origin. Yet $c_1(L_+)[S^2]=-k$, so
different charges forbid a parameter-continuous gauge on
the punctured ball. The locally minimizing nonconstant
radial projectors among them are precisely the isometric
degree-one hedgehogs
(Theorem~\ref{thm:projective-minimizing-cones}).
\end{example}

\subsection{The spectral Spencer--Hodge extension}

The corrected Spencer complex of Section~\ref{sec:spencer}
admits the external spectral coefficient $H(z)$, with the
closed-base hypotheses of compact-resolvent Hodge theory.
An isolated spectral cluster of a norm-resolvent-smooth
self-adjoint family yields such a coefficient bundle locally
via Riesz projections~\cite{Kato}. For $\lambda\in\R$ let
$C_{z,\lambda}$ be the two-term Hilbert complex
\[
 C^0_{z,\lambda}=\C^N
 \xrightarrow{\,H(z)-\lambda I\,}
 C^1_{z,\lambda}=\C^N .
\]
Complexify $(\cA^\bullet,D_t)$, take the total complex
\[
 \mathbb S^k_{t,z,\lambda}
   =\bigoplus_{a+b=k}\cA^a\otimes C^b_{z,\lambda},
 \qquad
 \mathbb D=D_t\otimes1+(-1)^a1\otimes(H(z)-\lambda I)
 \quad\text{on the summand of degree }a,
\]
where the second term acts only from $b=0$ to $b=1$.
This two-term resolution of the eigenvalue equation tensors
with the Spencer complex because $H(z)$ is constant along
its spatial base.

\begin{proposition}[Harmonic incidence of the spectral extension]
\label{prop:spectral-spencer-incidence}
Let the Spencer base be closed and connected, its coefficient
bundle flat, and $t\ne0$. Then $\mathbb D^2=0$ and the total
Hodge Laplacian is
\begin{equation}\label{eq:spectral-spencer-laplacian}
 \mathbb\Delta_{t,z,\lambda}
 =\Delta_t\otimes1+
 1\otimes
 \begin{pmatrix}
  (H(z)-\lambda I)^*(H(z)-\lambda I)&0\\
 0&(H(z)-\lambda I)(H(z)-\lambda I)^*
 \end{pmatrix}.
\end{equation}
Its degree-zero harmonic space has the canonical description
\begin{equation}\label{eq:spectral-spencer-zero-mode}
 \ker\mathbb\Delta^0_{t,z,\lambda}
   =\C\,1\otimes\ker(H(z)-\lambda I).
\end{equation}
Consequently the projectivized degree-zero harmonic incidence is
\begin{equation}\label{eq:spectral-incidence-definition}
 \mathfrak I_H
 =\{(z,\lambda,[v]):(H(z)-\lambda I)v=0\}.
\end{equation}
At a real eigenvalue of multiplicity $m$, its fiber is
$\C P^{m-1}$. Over the simple-spectrum locus, the fibers form
the eigenline spectral cover, and ordered choices form the
eigenflag. For a polynomial or real-analytic $H$ with a fixed
complexification $H_{\C}$, let $I_{N-1}(K)$ denote the ideal
of $(N-1)$-minors of a matrix $K$. The complex incidence
fiber has positive dimension precisely on the determinantal locus
\begin{equation}\label{eq:incidence-rank-ideal}
 \rank(H_{\C}(z)-\lambda I)\le N-2,
 \quad\text{equivalently}\quad
 I_{N-1}(H_{\C}(z)-\lambda I)=0.
\end{equation}
Repeated roots of the characteristic polynomial need not lie
in~\eqref{eq:incidence-rank-ideal}.
\end{proposition}

\begin{proof}
The two summands of $\mathbb D$ anticommute by the total-degree
sign, and each squares to zero. Their adjoints obey the same
product rule, so mixed Laplacian terms cancel and give
\eqref{eq:spectral-spencer-laplacian}. For $t\ne0$, a degree-zero
Spencer harmonic section is parallel along the base and is
killed by the polynomial derivative $\delta$. The latter has
only constant polynomials in its degree-zero kernel; connectedness
makes the section a constant multiple of $1$. This proves
\eqref{eq:spectral-spencer-zero-mode}. The remaining assertions
are linear algebra: the incidence fiber is
$\mathbb P\ker(H-\lambda I)$, and its projective dimension is
positive exactly when the matrix has rank at most $N-2$.
The minors express that rank condition. At a complex parameter,
the algebraic differential $H_{\C}-\lambda I$ remains the
relevant object; its Hodge adjoint is not holomorphic and need
not identify the right and left eigenvectors.
\end{proof}

The incidence fiber at a collision contains all lines in the
eigenspace, while an approaching family may select only some of
them. Write $\mathcal F_N=\U(N)/\U(1)^N$ for ordered
orthogonal line flags and $\mathcal D_N=\mathcal F_N/S_N$
for unordered flags. To retain the approach history, on a
dense real simple-spectrum locus $X^\circ$ let
$q_H:X^\circ\to\mathcal D_N$ be the
unordered eigenline flag and form its closed graph
\begin{equation}\label{eq:incidence-flag-graph}
 \Gamma_H=\overline{\{(z,q_H(z)):z\in X^\circ\}}
       \ \subset X\times\mathcal D_N.
\end{equation}
The fiber of $\Gamma_H$ at a collision is its actual flag
cluster set. Its ordered covering records the additional
continuation data needed for a gauge. For a globally simple
Hermitian field, real eigenvalues can be ordered by size, so
their root cover has no permutation monodromy. The permutation
obstruction in Section~\ref{sec:flag-applications} concerns
continuation of an unordered flag through collisions. The
finite complex root covers of
Section~\ref{subsec:spectral-cover-transfer} concern normal
complex-spectrum realizations; both are readings of the
incidence geometry with different hypotheses.

Replacing $z$ by a Spencer-base coordinate introduces $\dd H$
into $\mathbb D^2$. In the flat model of
Theorem~\ref{thm:spencer-parity}, $t\mapsto-t$ alone leaves
eigenflags stationary; it agrees with $u\mapsto-u$ only when
a model identifies these actions.

\subsection{Flag energy and mirror charge}

On a simple-spectrum region, write
$H=\sum_{a=1}^{\ell}\lambda_aP_a$ with distinct eigenvalue
functions and constant-rank orthogonal projections. Let $\nabla$
be a Hermitian connection on the ambient bundle and put
$e_{\rm flag}=\frac12\sum_a|\nabla P_a|_{\rm HS}^2$.
Differentiating $HP_a=\lambda_aP_a$ gives the exact identity
\begin{equation}\label{eq:incidence-gap-energy}
 e_{\rm flag}
 =2\sum_{a<b}
 \frac{|P_b(\nabla H)P_a|_{\rm HS}^2}
      {(\lambda_a-\lambda_b)^2},
 \qquad
 |\nabla H|_{\rm HS}^2
 =\sum_a(\rank P_a)|\dd\lambda_a|^2+
 2\sum_{a<b}|P_b(\nabla H)P_a|_{\rm HS}^2.
\end{equation}
The first formula retains the directional information that
blows up as a spectral gap closes. The second separates scalar
eigenvalue motion from eigenflag rotation. Both formulas are
pointwise and make no minimizing assumption.

\begin{proposition}[Mirror action on linking charges]
\label{prop:mirror-link-charge}
Let $S:S^2\to S^2$ have degree $d_S=\deg S\in
\{+1,-1\}$. Let $E_a$ be the ordered spectral blocks of a
Hermitian field on a linking sphere. Suppose a continuous
complex-linear, respectively complex-antilinear, bundle
isomorphism over $S$ sends $E_a(\omega)$ to
$E_{w(a)}(S\omega)$ for a permutation $w$. Write
$q_a=\langle c_1(\det E_a),[S^2]\rangle$ and set
$\varepsilon=+1$, respectively $-1$. Then
\begin{equation}\label{eq:mirror-charge-law}
              d_S q_{w(a)}=\varepsilon q_a.
\end{equation}
In particular, a complex-linear orientation-reversing mirror
fixing a charged block has no continuous gauge lift on the
linking sphere. An antilinear mirror or a permutation of
charged blocks may satisfy the same symmetry.
\end{proposition}

\begin{proof}
The isomorphism identifies $E_a$ with $S^*E_{w(a)}$ in the
linear case, and $\overline{E_a}$ with $S^*E_{w(a)}$ in the
antilinear case. Taking determinant first Chern classes gives
$q_a=d_S q_{w(a)}$ or $-q_a=d_S q_{w(a)}$,
respectively.
\end{proof}

For a concrete joint mirror, put $H(t,\omega)=t\,\omega\cdot\sigma$
on $(-1,1)\times S^2$ and use the same $t$ in $D_t$.
Theorem~\ref{thm:spencer-parity} makes the Spencer Hodge
spectra at $t$ and $-t$ identical, and the relative suspension
has $f=t^2$. For $t>0$ the positive Hermitian eigenline is
$P_\omega$; for $t<0$ it is $I-P_\omega$, so their charges
are opposite. A pointwise Weyl interchange exists, while a
continuous family of unitary gauges on the angular $S^2$
cannot identify these charged eigenlines. This is an exact
model of spectral mirror symmetry with a global gauge
obstruction. It does not identify the coefficient parity with
every adjoint-field sign transformation.

The same incidence space has an algebraic face. Put
$s(z)=z_1^2+z_2^2+z_3^2$ and, for $m\ge0$, consider
\begin{equation}\label{eq:quadratic-factor-family}
 H_m(x)=|x|^{2m}x\cdot\sigma.
\end{equation}
All $H_m$ have the same real eigenflag $x\mapsto
(I+(x/|x|)\cdot\sigma)/2$, hence the same minimizing
projective link and the same flag energy. Their complex
characteristic polynomials are
$\lambda^2-s(z)^{2m+1}$, whose discriminant supports all
equal the isotropic quadric $s=0$. For $m=0$, the complex
geometric double-eigenspace locus is only $z=0$. For $m\ge1$,
$H_{m,\C}$ vanishes on $s=0$, so this locus contains the whole
nonzero isotropic quadric. Thus real variational selection,
scalar discriminant support, and complex geometric
eigenspace degeneracy are three distinct readings of the
same spectral incidence. Section~\ref{sec:four-band-isolation}
studies the higher-rank determinantal reading.

\section{Collision links and gauge limits}
\label{sec:collision}

Here and in the local formation sections below, $A(x)$
denotes a Hermitian matrix field of the type written $H(z)$
in Section~\ref{sec:spectral-incidence}.
At a spectral collision, projectors on a linking sphere retain data
absent from the characteristic polynomial.  The model below uses
these data to distinguish exact gauge orbits, orbit closures, and
simple-spectrum perturbations.  Relevant precedents include defect
charges~\cite{TeoKane}, twofold Weyl points
\cite{FrankPinterVarjasPalyi}, multiband Hopf systems
\cite{LapierreNeupertTrifunovic}, and the variable-multiplicity
question of Friedman and Park~\cite[\S8]{FriedmanPark2016}.

Use the Pauli matrices $\sigma$ fixed in
Section~\ref{sec:spectral-incidence}. Let $B^d\subset\R^d$ be
the closed unit ball, $d\ge2$, and fix
\begin{equation}\label{eq:collision-flat-function}
 \varphi(r)=\begin{cases}e^{-1/r^2},&r>0,\\0,&r=0.\end{cases}
\end{equation}
For a smooth map $n:S^{d-1}\to S^2$, define a traceless Hermitian
matrix field by
\begin{equation}\label{eq:collision-field}
 A_n(0)=0,\qquad A_n(r\theta)=\varphi(r)n(\theta)\cdot\sigma
 \quad (r>0).
\end{equation}
It is smooth and flat at the origin.  Its characteristic polynomial
is $\lambda^2-\varphi(r)^2$, independent of $n$, and the origin is its
only multiple-spectrum point.  We call the homotopy class of $n$
the \emph{collision-link class}.  The associated positive spectral
projector on each linking sphere is
$P_n(\theta)=\frac12(I_2+n(\theta)\cdot\sigma)$.

For two such fields, set
$e_U^{n,m}=UA_nU^{-1}-A_m$ for a smooth gauge
$U:B^d\to\SU(2)$.  Norms below are norms of matrix fields in the
fixed trivialization of $B^d\times\C^2$.

\begin{theorem}[Exact and approximate gauges at a collision]
\label{thm:collision-orbit}
Let $n,m:S^{d-1}\to S^2$ be smooth.
\begin{enumerate}
\item A continuous gauge $U:B^d\to\SU(2)$ satisfying
      $UA_nU^{-1}=A_m$ exists exactly when there is one constant
      rotation $R\in\mathrm{SO}(3)$ with $m=R\circ n$.  In that event the
      gauge may be taken constant and smooth.
\item $A_m$ belongs to the closure of the smooth gauge orbit of
      $A_n$ in the $C^\infty$ topology exactly when $n$ and $m$
      are homotopic as maps $S^{d-1}\to S^2$.
\item If $[n]\ne[m]$, then every continuous gauge satisfies
      \begin{equation}\label{eq:collision-max-gap}
       \|e_U^{n,m}\|_{C^0(B^d,\mathrm{op})}=2\varphi(1).
      \end{equation}
\end{enumerate}
\end{theorem}

\begin{proof}
For $r>0$, equality of the fields is equivalent to
$\operatorname{Ad}_{U(r,\theta)}n(\theta)=m(\theta)$.
As $r\downarrow0$, continuity of $U$ makes $U(r,\theta)$ tend
to the single value $U(0)$, uniformly in $\theta$.  Its adjoint
rotation therefore carries $n$ to $m$.  The double covering
$\SU(2)\to\mathrm{SO}(3)$ supplies a constant gauge in the converse
direction.

Suppose next that $n_t$ is a smooth homotopy from $n$ to $m$.
The action fibration $\SU(2)\to S^2$ lifts this family starting
at the identity.  Concretely, the angular velocity
$n_t\times\partial_tn_t$ defines a smooth rotation ODE whose
solution $R_t(\theta)$ satisfies $R_tn=n_t$ and $R_0=I$; lift
this path of rotations to $\SU(2)$.  The endpoint is a smooth
transporter $T:S^{d-1}\to\SU(2)$ with
$\operatorname{Ad}_Tn=m$.  The path of lifts also shows that $T$
is nullhomotopic.  Choose a smooth filling
$G:B^d\to\SU(2)$ of $T$ which is radially constant on a boundary
collar.  For $\varepsilon>0$, set $U_\varepsilon=T(\theta)$ on
$r\ge\varepsilon$ and
$U_\varepsilon(x)=G(x/\varepsilon)$ on $r\le\varepsilon$;
the collar makes this a smooth gauge.  Its defect vanishes outside
$B_\varepsilon$.  Every derivative of
$\varphi(r)n(\theta)$ and $\varphi(r)m(\theta)$ is flat at $0$,
while derivatives of $G(x/\varepsilon)$ grow by at most powers of
$\varepsilon^{-1}$.  For every pair $k,N$ there is a constant
$C_{k,N}$ such that
$\|D^ke_{U_\varepsilon}^{n,m}\|_\infty
 \le C_{k,N}\varepsilon^N$ after increasing the flatness order.
Thus the defects tend to zero in all $C^k$ seminorms.

Conversely, if $e_{U_j}^{n,m}\to0$ uniformly, then on the unit
sphere the maps
$h_j=\operatorname{Ad}_{U_j|_{S^{d-1}}}n$ tend uniformly to $m$.
The restriction $U_j|_{S^{d-1}}$ extends over the ball and is
nullhomotopic, so $h_j$ is homotopic to $n$.  For large $j$,
$\|h_j-m\|_\infty<2$; normalized straight-line interpolation
then makes $h_j$ homotopic to $m$.  This proves the necessity in
the second assertion.

Finally, for every $r>0$ the restriction of $U$ to $S_r^{d-1}$
extends across the smaller ball.  Hence
$h_r=\operatorname{Ad}_{U(r,\cdot)}n$ is homotopic to $n$.
If $[n]\ne[m]$, the maps $h_r$ and $m$ must be antipodal at
some angular point: otherwise normalized straight-line
interpolation would be a homotopy.  Their chordal distance has
supremum two on every sphere, so the matrix defect has supremum
$2\varphi(r)$ there.  The function $\varphi$ increases with $r$,
giving~\eqref{eq:collision-max-gap}.
\end{proof}

For a constant $n$ and a nonconstant $m$ with image in an open
hemisphere, the fields lie in the same smooth orbit closure but
admit no exact gauge.

\begin{theorem}[Distance to simple spectrum]
\label{thm:collision-perturbation}
Let $\mathcal S_2$ be the set of continuous complex $2\times2$
matrix fields on $B^d$ whose eigenvalues are distinct at every
point.  If $[n]\ne0$, then
\begin{equation}\label{eq:collision-perturbation-distance}
 \inf_{B\in\mathcal S_2}
   \|B-A_n\|_{C^0(B^d,\mathrm{op})}=\varphi(1).
\end{equation}
The same infimum holds if $B$ is required to be smooth, Hermitian,
and traceless.  If $[n]=0$, smooth Hermitian simple-spectrum fields
converge to $A_n$ in $C^\infty$.
\end{theorem}

\begin{proof}
Suppose a simple-spectrum field $B$ satisfies
$\delta:=\|B-A_n\|_\infty<\varphi(1)$.  On $S^{d-1}$ the two
eigenvalues of $A_n$ are $\pm\varphi(1)$.  Choose a circle about
$+\varphi(1)$ of radius $s$ with $\delta<s<\varphi(1)$.
The normal resolvent estimate for $A_n$ on that circle is
$\|(z-A_n)^{-1}\|\le s^{-1}$.  Neumann's lemma shows that this
circle remains in the resolvent set of
$A_n+t(B-A_n)$ for every $t\in[0,1]$.  Its rank-one Riesz
projection consequently gives a homotopy on the boundary from
$P_n$ to the eigenline projector of $B$ near $+\varphi(1)$.
Since the ball is simply connected and $B$ has distinct
eigenvalues, its two eigenvalue branches can be ordered globally;
the corresponding eigenline projector extends over $B^d$.
The boundary map $P_n:S^{d-1}\to\C P^1\cong S^2$ would then be
nullhomotopic, a contradiction.  The argument allows nonnormal
$B$, since Riesz projectors still have continuous one-dimensional
ranges.

For the reverse inequality use the constant smooth Hermitian
field $B_\eta=\eta\sigma_3$, $\eta>0$.  It has simple spectrum
and $\|B_\eta-A_n\|_\infty\le\varphi(1)+\eta$.

If $[n]=0$, extend $n$ smoothly across $B^d$, arranging an
extension $N$ which is radially equal to $n$ on a boundary collar
and constant near the center.  On a shrinking ball $B_\varepsilon$
replace the radial projector of $A_n$ by the rescaled extension
$N(x/\varepsilon)$ and replace $\varphi$ by a smooth strictly
positive radial amplitude which matches $\varphi$ on a collar of
$r=\varepsilon$ and is $O(\varphi(\varepsilon))$ inside.
The resulting Hermitian field has simple spectrum everywhere.
Flatness of $\varphi$ absorbs every power of $\varepsilon^{-1}$
introduced by rescaling, so it converges to $A_n$ in $C^\infty$.
\end{proof}

Taking $d=3$, $n(\theta)=\theta$, and $\varphi(r)=r$ in the
lower-bound argument recovers the Weyl-point obstruction.  Thus
the unrestricted perturbability sentence in
\cite[\S8]{FriedmanPark2016} fails in uniform matrix norm even
on a ball.  The flat gap additionally permits the smooth
orbit-closure statement of Theorem~\ref{thm:collision-orbit}.

The perturbation criterion has an unstable extension to larger
simultaneous collisions.  Fix pairwise distinct real numbers
$a_3,\ldots,a_N\in(-1,1)$ and put $a_1=1$, $a_2=-1$.
For $N\ge2$ set
\begin{equation}\label{eq:collision-stabilized-field}
 A_{n,N}(0)=0,\qquad
 A_{n,N}(r\theta)=\varphi(r)
 \bigl(n(\theta)\cdot\sigma\oplus
       \operatorname{diag}(a_3,\ldots,a_N)\bigr).
\end{equation}
Its ordered eigenlines define a map
$\Phi_{n,N}:S^{d-1}\to\mathcal F_N$ to the complete unitary
flag manifold $\mathcal F_N=\U(N)/\U(1)^N\cong\SU(N)/T^{N-1}$.
When $\sum_{j=3}^N a_j=0$, the fields in
\eqref{eq:collision-stabilized-field} are traceless, and their
multiples by $\ii$ are $\mathfrak{su}(N)$-valued.

\begin{theorem}[Finite-band stabilization of a collision]
\label{thm:collision-stabilization}
The field $A_{n,N}$ belongs to the $C^\infty$ closure of smooth
Hermitian simple-spectrum $N\times N$ fields exactly when
$[\Phi_{n,N}]=0$ in $\pi_{d-1}(\mathcal F_N)$.
If $[\Phi_{n,N}]\ne0$, its distance in uniform operator norm from
the set of all continuous simple-spectrum complex $N\times N$
matrix fields is at least
\begin{equation}\label{eq:collision-N-gap}
 \frac{\varphi(1)}2\min_{i\ne j}|a_i-a_j|>0.
\end{equation}
For $d\ge4$, the collision-link class maps to
$[\Phi_{n,N}]$ by the homomorphism
\begin{equation}\label{eq:collision-bott-map}
 \pi_{d-1}(S^2)\cong\pi_{d-1}(\SU(2))
 \longrightarrow\pi_{d-1}(\SU(N))
 \cong\pi_{d-1}(\mathcal F_N)
\end{equation}
induced by the upper-block inclusion $\SU(2)\hookrightarrow\SU(N)$.
\end{theorem}

\begin{proof}
If $\Phi_{n,N}$ is nullhomotopic, fill it by a smooth complete
flag constant near the center and equal to the given flag on a
boundary collar.  Rescale the filling into $B_\varepsilon$,
interpolating the eigenvalues to distinct
$O(\varphi(\varepsilon))$ constants.  As in the final paragraph
of the proof of Theorem~\ref{thm:collision-perturbation},
flatness makes the simple-spectrum fields converge in $C^\infty$.

Conversely, a perturbation smaller than half the boundary
eigenvalue separation preserves disjoint Riesz contours along
the straight interpolation, as in
Theorem~\ref{thm:collision-perturbation}.  Its globally labeled
eigenlines extend over the ball, so the boundary flag is null.
For nonnormal fields, Gram--Schmidt identifies
$\mathrm{GL}_N(\C)/(\C^\times)^N$ with $\mathcal F_N$ up to
homotopy.  This proves~\eqref{eq:collision-N-gap}.

For $d\ge4$, the homotopy exact sequences of
$S^1\to\SU(2)\to S^2$ and
$T^{N-1}\to\SU(N)\to\mathcal F_N$ identify the two outer
homotopy groups in~\eqref{eq:collision-bott-map} with the
indicated groups of special unitary matrices.  Naturality of
the two fibrations identifies the middle arrow with the
upper-block inclusion.
\end{proof}

\begin{corollary}[A torsion crossing removed by a third band]
\label{cor:collision-torsion}
Let $d=5$ and let $n:S^4\to S^2$ represent the nonzero class
of $\pi_4(S^2)\cong\Z/2$.  The two-band field $A_n$ has
distance exactly $\varphi(1)$ from the simple-spectrum locus.
For $A_{n,3}=\operatorname{diag}(A_n,0)$, smooth Hermitian
simple-spectrum fields converge in $C^\infty$.
In contrast, a nonzero Hopf class in $\pi_3(S^2)\cong\Z$
remains nonzero in the resolved-link family
\eqref{eq:collision-stabilized-field} for every finite $N$.
\end{corollary}

\begin{proof}
The first assertion is Theorem~\ref{thm:collision-perturbation}.
Bott periodicity gives $\pi_4(\SU(3))=0$, so
Theorem~\ref{thm:collision-stabilization} applies to the
three-band field.  In degree three,
$\pi_3(\SU(2))\to\pi_3(\SU(N))$ is an isomorphism for
$N\ge2$, proving the last assertion~\cite{Bott}.
\end{proof}

The $S^4$ torsion class has zero first Chern and stable complex
$K$-classes.  Its removal requires a third band to join the
collision: a band at fixed spectral distance retains a separated
Riesz projector and leaves the two-band obstruction intact.
For $d=3$ the linking $S^2$ records degree by $c_1$; for
$d\ge4$, $H^2(S^{d-1})=0$ while the higher homotopy charge
in~\eqref{eq:collision-bott-map} can survive.

\subsection{Partial flags and spectator blocks}

The complete flag used above fixes a simple spectrum on every
linking sphere.  A repeated eigenvalue within a spectator block
leaves room to choose its internal spectral flag.  The resulting
criterion applies to every fixed multiplicity partition.

\begin{theorem}[Partial-flag resolution criterion]
\label{thm:partial-flag-resolution}
Let $H:S^{d-1}\to\operatorname{Herm}_N(\C)$ be smooth with
$t\ge2$ constant eigenvalues $a_1<\cdots<a_t$ of fixed multiplicities
$m_1,\ldots,m_t$, and let
$q:S^{d-1}\to\mathcal F_{\mathbf m}
=\U(N)/(\U(m_1)\times\cdots\times\U(m_t))$
be its ordered spectral partial flag.  Define
$A_H(0)=0$ and $A_H(r\theta)=\varphi(r)H(\theta)$ for $r>0$.
Then $A_H$ belongs to the $C^\infty$ closure of the smooth
Hermitian simple-spectrum fields on $B^d$ if and only if $q$ is
nullhomotopic.  If $q$ is not nullhomotopic, its uniform
operator-norm distance from every continuous complex
simple-spectrum field, including nonnormal fields, is at least
\begin{equation}\label{eq:partial-flag-gap}
 \frac{\varphi(1)}2\min_{1\le i<t}(a_{i+1}-a_i)>0.
\end{equation}
\end{theorem}

\begin{proof}
Put $\Delta=\min_i(a_{i+1}-a_i)$.  As in the proof of
Theorem~\ref{thm:collision-stabilization}, a perturbation of
norm less than $\varphi(1)\Delta/2$ preserves disjoint Riesz
contours around the boundary clusters throughout the straight
interpolation.  Their grouped ranges give a partial flag
homotopic to $q$.  The simple eigenlines of $B$ can be labeled
globally on the ball and grouped to extend this flag, forcing
$[q]=0$.  Orthogonalization handles nonnormal Riesz ranges.
This proves necessity and~\eqref{eq:partial-flag-gap}.

Conversely, suppose $q$ is nullhomotopic.  The complete-flag
bundle $\mathcal F_N\to\mathcal F_{\mathbf m}$ has the
homotopy lifting property.  Lifting a contraction of $q$ from
a fixed complete flag over its constant endpoint gives a smooth
complete-flag refinement $\Phi$ of $q$ that is itself
nullhomotopic.  Split the eigenvalues within each block by
distinct constants of size at most $\delta$, using $\Phi$ as
eigenlines.  This gives a smooth simple-spectrum link
$H_\delta$ with $H_\delta\to H$ in $C^\infty$ as
$\delta\downarrow0$.  The field $\varphi(r)H_\delta(\theta)$
is simple for $r>0$.  Fill its nullhomotopic complete flag
inside $B_\varepsilon$, and interpolate its ordered eigenvalues
there to distinct constants of size $O(\varphi(\varepsilon))$.
The resulting field is smooth and simple even at the center.
All derivatives of the inner modification tend to zero as
$\varepsilon\downarrow0$ because $\varphi$ is flat.  Choose
$\varepsilon$ after $\delta$, then let $\delta\downarrow0$.
This yields $C^\infty$ convergence to $A_H$.
\end{proof}

The class $[q]\in[S^{d-1},\mathcal F_{\mathbf m}]$ is the
\emph{spectral-partition charge}.  Merging clusters projects the
charge to a coarser partial flag; resolving a cluster requires
a lift.  The next corollary computes when a repeated middle
block absorbs an opposite link.

\begin{corollary}[Cancellation by a degenerate spectator pair]
\label{cor:spectator-cancellation}
Let $d\ge4$, $n:S^{d-1}\to S^2$, and
$A_n^{(s)}=A_n\oplus0_s$ on $B^d$, where $s\ge1$.
Choose a Hopf lift $g:S^{d-1}\to\SU(2)$ of $n$, and put
$k=d-1$.  The field $A_n^{(s)}$ admits arbitrarily close
smooth Hermitian simple-spectrum approximations in $C^\infty$
if and only if
\begin{equation}\label{eq:relative-collision-class}
 i_{2,s+2*}[g]\in
 \operatorname{im}\bigl(
 \pi_k(\SU(s))\longrightarrow\pi_k(\SU(s+2))\bigr),
\end{equation}
where the image is zero when $s=1$.  In particular every link
class is removable for $s\ge2$.  For $s=1$, the obstruction is
the ordinary inclusion image
$\pi_k(\SU(2))\to\pi_k(\SU(3))$.
Thus a nonzero Hopf link on $B^4$ persists after one zero band
is adjoined and disappears after two zero bands are adjoined.
\end{corollary}

\begin{proof}
The Hopf lift exists because $H^2(S^{d-1};\Z)=0$.
The boundary partial flag has multiplicities $(1,s,1)$ and is
the image of $\operatorname{diag}(g,I_s)$ under
$\SU(s+2)\to\mathcal F_{(1,s,1)}$.
In the homotopy sequence of this fibration, the image of this
class vanishes precisely when it lies in the image of the
block subgroup $\U(1)\times\U(s)\times\U(1)$.
Since $k\ge3$, the circle factors have no $\pi_k$, giving
\eqref{eq:relative-collision-class}.  If $s\ge2$, an inverse
representative of $[g]$ can be placed in an $\SU(2)$ subgroup
of the zero block.  The two block inclusions into $\SU(s+2)$
are conjugate and induce the same map on $\pi_k$; their classes
cancel.  Theorem~\ref{thm:partial-flag-resolution} then gives
the approximation.  For $s=1$ there is no such block; in degree
three the inclusion $\pi_3(\SU(2))\to\pi_3(\SU(3))$ is an
isomorphism.
\end{proof}

Cancellation uses an unresolved spectator block.  Resolving it
as in~\eqref{eq:collision-stabilized-field} fixes the complete
flag and restores its inclusion obstruction.  For $d=3$, a
nonzero degree persists after any number of zero bands because
the outer eigenline has nonzero $c_1$ on the linking $S^2$.

\begin{corollary}[First band-resolution thresholds]
\label{cor:first-band-thresholds}
For a nonzero two-band link in dimensions $d=4,5,6,7$, the
classes killed by adjoining one unresolved zero band are,
respectively,
\[
 0\subset\pi_3(\SU(2))=\Z,\qquad
 \Z/2=\pi_4(\SU(2)),\qquad
 \Z/2=\pi_5(\SU(2)),\qquad
 \{0,6\}\subset\pi_6(\SU(2))=\Z/12.
\]
Thus in dimension seven the unique order-two class is removed by
a third band, while every other nonzero class requires a fourth
band.  Two unresolved zero bands remove every link class in all
dimensions $d\ge4$.  A nonzero four-dimensional Hopf class
survives every resolved-band extension
\eqref{eq:collision-stabilized-field}.
\end{corollary}

\begin{proof}
The criterion for one zero band is the inclusion
$\pi_{d-1}(\SU(2))\to\pi_{d-1}(\SU(3))$ in
Corollary~\ref{cor:spectator-cancellation}.  In degree three it
is an isomorphism, and in degree four the target is zero.
In degree five the source is $\Z/2$ and the target is $\Z$,
so the map vanishes.  In degree six the source is $\Z/12$
and the target is $\Z/6$~\cite{MimuraToda,Toda}.
The homotopy sequence of
$\SU(2)\to\SU(3)\to S^5$ contains
\[
 \pi_6(\SU(3))\longrightarrow\pi_6(S^5)
 \longrightarrow\pi_5(\SU(2))\longrightarrow\pi_5(\SU(3)).
\]
The last map is zero from $\Z/2$ to $\Z$, making the middle
map an isomorphism of $\Z/2$ groups.  Exactness then makes
$\pi_6(\SU(2))\to\pi_6(\SU(3))$ surjective, with kernel
the unique order-two subgroup $\{0,6\}$ of $\Z/12$.
The two-band cancellation and resolved Hopf assertions follow
from Corollary~\ref{cor:spectator-cancellation} and
Theorem~\ref{thm:collision-stabilization}.
\end{proof}

\begin{theorem}[A Bott--Chern hierarchy with spectator gap reduction]
\label{thm:stable-chern-collision}
For every integer $m\ge1$ there is a smooth Hermitian field
$A_m:B^{2m+1}\to\operatorname{Herm}_{2m}(\C)$ whose spectrum
on $r>0$ consists of $m$ copies each of $+\varphi(r)$ and
$-\varphi(r)$.  Its positive eigenbundle on the linking sphere
has nonzero top Chern number.  Write
\begin{equation}\label{eq:stable-chern-distance}
 D_m(s)=\inf_{B\in\mathcal S_{2m+s}}
  \|B-(A_m\oplus0_s)\|_{C^0(B^{2m+1},\mathrm{op})},
\end{equation}
where $\mathcal S_{2m+s}$ contains all continuous complex
simple-spectrum fields of that size.  Then
\[
 D_m(0)=\varphi(1),\qquad
 \frac{\varphi(1)}2\le D_m(s)\le\varphi(1)\quad(s\ge1),\qquad
 D_m(s)=\frac{\varphi(1)}2\quad(s\ge2m).
\]
The same values and bounds hold when the infimum is restricted to
smooth Hermitian simple-spectrum fields.  In particular the
collision obstruction persists under arbitrarily many zero
spectator bands, although enough such bands halve its optimal
uniform gap.
\end{theorem}

\begin{proof}
Choose a generator $g:S^{2m-1}\to\U(m)$ of
$\pi_{2m-1}(\U(m))\cong\Z$, and let $E\to S^{2m}$ be its
clutching bundle.  Bott integrality gives
$|\langle c_m(E),[S^{2m}]\rangle|=(m-1)!$~\cite{Bott,HatcherVB}.
Let $F$ be clutched by $g^{-1}$.  The block map
$\operatorname{diag}(g,g^{-1}):S^{2m-1}\to\U(2m)$ is
nullhomotopic: the two block inclusions induce the same map
on $\pi_{2m-1}$, so their classes cancel.  Hence $E\oplus F$
is trivial.  In a unitary trivialization put
$H=+I_E\oplus(-I_F)$ on $S^{2m}$ and
$A_m(r\theta)=\varphi(r)H(\theta)$, with $A_m(0)=0$.

Suppose a simple-spectrum $B$ is at distance strictly less than $\varphi(1)/2$ from
$A_m\oplus0_s$ with $s\ge1$.  On the boundary the Riesz circle
centered at $+\varphi(1)$ with radius $\varphi(1)/2$ separates its
rank-$m$ cluster from the zero and negative clusters throughout
the straight interpolation.  It gives a rank-$m$ spectral
subbundle homotopic to $E$.  Since $B$ is simple on the ball,
its eigenvalues can be globally labeled and this rank-$m$
subbundle extends over $B^{2m+1}$.  Its top Chern class must
vanish, a contradiction.  Riesz ranges make the same argument
valid for nonnormal $B$.  When $s=0$, the positive and negative
clusters are separated by $2\varphi(1)$, so the identical argument
works for every perturbation of norm strictly below $\varphi(1)$.

Conversely, a constant Hermitian matrix $B_\eta$ with
$2m+s$ distinct eigenvalues of absolute value at most $\eta$
has simple spectrum and
$\|B_\eta-(A_m\oplus0_s)\|_\infty\le\varphi(1)+\eta$.
Letting $\eta\downarrow0$ gives the upper bound and the exact
value for $s=0$.

For $s\ge2m$, use the first $2m$ spectator coordinates as a
second trivial bundle $E'\oplus F'$ with $E'\cong E$ and
$F'\cong F$; leave the remaining $s-2m$ coordinates fixed.
On the linking sphere define $K$ to have eigenvalue $+1/2$
on $E\oplus F'$, eigenvalue $-1/2$ on $F\oplus E'$, and
eigenvalue $0$ on the remaining trivial summand.  Both the
positive and negative eigenbundles of $K$ are trivial, since
$E\oplus F$ is trivial.  They therefore admit smooth global
orthonormal frames.  Along with a constant frame of the zero
summand these give a map $S^{2m}\to\U(2m+s)$, which is
nullhomotopic because $\pi_{2m}(\U(2m+s))=0$~\cite{Bott}.
The partial spectral flag of $K$ is consequently nullhomotopic.
Theorem~\ref{thm:partial-flag-resolution} supplies smooth
Hermitian simple-spectrum fields converging in $C^\infty$ to
$\varphi(r)K(\theta)$ on the ball.  Since $K$ commutes with
$H\oplus0_s$ and
$\|K-(H\oplus0_s)\|_{\mathrm{op}}=1/2$ pointwise, these
fields approach $A_m\oplus0_s$ within $\varphi(1)/2$ in the
uniform norm.  This matches the lower bound.
\end{proof}

For $m=2$, adjoining a trivial line puts the two rank-two
blocks on $S^4$ under the block-swap index of
Appendix~\ref{sec:index}. Their top-Chern charge retains
a positive distance after any number of spectators, although
the optimal gap can halve; the torsion charges above can vanish.

\section{Flag extension and local collision spaces}
\label{sec:flag-applications}

\subsection{Spectral descent from a dense simple stratum}

The finite-collision hypothesis can be replaced by a topological
condition on the unordered eigenline flag.  Let $X$ be a connected
finite CW complex with a compatible metric and let
$A:X\to\operatorname{Herm}_N(\C)$ be continuous.  Assume that its
simple-spectrum locus $O_A$ is dense; the collision set may be any
closed nowhere-dense subset.  Put
\[
 \mathcal F_N=\U(N)/\U(1)^N,\qquad
 \mathcal D_N=\mathcal F_N/S_N.
\]
The unordered orthogonal eigenline decomposition is a canonical
map $q_A:O_A\to\mathcal D_N$.  For unordered families of rank-one
projections define
\[
 d_{\mathcal D}(\{P_j\},\{Q_j\})
  =\min_{\pi\in S_N}\max_j\|P_j-Q_{\pi(j)}\|_{\mathrm{op}},
\]
and let
\[
 \omega_A(\delta)=\sup\{d_{\mathcal D}(q_A(x),q_A(y)):
          x,y\in O_A,\ d_X(x,y)\le\delta\}.
\]
The covering $\mathcal F_N\to\mathcal D_N$ has deck group
$S_N$.  Since $\mathcal F_N$ is simply connected, an extension
$\bar q_A:X\to\mathcal D_N$ has a monodromy homomorphism
$\mu_A:\pi_1(X)\to S_N$, well-defined up to conjugacy.
These covering and line-bundle obstructions have antecedents in
the diagonalization theory of Grove and Pedersen~\cite{GrovePedersen}.

\begin{theorem}[Three-layer spectral descent]
\label{thm:spectral-descent}
Under the preceding assumptions, a continuous
$U:X\to\U(N)$ with $U^*AU$ diagonal exists if and only if:
\begin{enumerate}
\item $\omega_A(\delta)\to0$ as $\delta\downarrow0$;
\item the monodromy $\mu_A$ of the resulting extension is trivial;
\item for an ordered lift $\widetilde q_A:X\to\mathcal F_N$,
each induced eigenline bundle $L_j$ has
$c_1(L_j)=0$ in $H^2(X;\Z)$.
\end{enumerate}
On a finite graph the third condition is automatic.  On an
interval only the first condition remains; this includes
infinitely many and accumulating collisions.
\end{theorem}

\begin{proof}
The target $\mathcal D_N$ is compact and complete.  Since
$O_A$ is dense in compact $X$, the first condition is exactly
uniform continuity of $q_A$, and gives a unique continuous
extension.  The map $\pi_1(\U(1)^N)\to\pi_1(\U(N))$ is
surjective, so $\mathcal F_N$ is simply connected and its
quotient by the free permutation action is a regular covering.
The covering-space lifting criterion identifies condition 2
with the existence of $\widetilde q_A$.  Two such lifts differ
by a constant permutation.  On a finite CW complex, condition 3
then trivializes every induced complex line bundle.  Unit
sections of the mutually orthogonal $L_j$ assemble into a
continuous unitary $U$.  Every off-diagonal entry of $U^*AU$
vanishes on dense $O_A$ and hence on $X$.

Conversely, a continuous diagonalizer defines an ordered flag
on all of $X$.  Its unordered image extends $q_A$ and is
uniformly continuous; its columns trivialize every $L_j$.
The graph and interval conclusions follow from
$H^2(\text{graph};\Z)=0$ and $\pi_1(\text{interval})=0$.
\end{proof}

The permutation layer is independently realizable.  For any
finite connected graph $X$ and any homomorphism
$\rho:\pi_1(X)\to S_N$, choose a spanning tree and assign
the corresponding permutations to its complementary oriented
edges.  Let $D$ be a diagonal matrix with distinct real entries.
On an edge $e\cong[0,1]$, choose a smooth unitary path $V_e$
from $I$ to the assigned permutation matrix and set
\[
 A_e(s)=\chi(s)V_e(s)DV_e(s)^*,\qquad
 \chi(s)=\begin{cases}
 e^{-1/[s(1-s)]},&0<s<1,\\0,&s=0,1.
 \end{cases}
\]
The edge fields glue at scalar zero vertices with all edgewise
derivatives flat.  Their unordered flags extend by the fixed
coordinate flag at each vertex, while continuation of an
ordered flag along a graph loop gives precisely $\rho$.  Thus
every permutation representation is realized by a Hermitian
field with simple spectrum on open edges and an extending
unordered flag; it is continuously diagonalizable exactly when
$\rho$ is trivial.

Flatness is unnecessary even for nontrivial monodromy.  On
$S^1=\R/\Z$, the real-analytic two-level field
\begin{equation}\label{eq:analytic-permutation-loop}
 A(t)=\tfrac12\sin(2\pi t)\sigma_3+
       \tfrac12(1-\cos(2\pi t))\sigma_1
\end{equation}
has one collision, at $t=0$.  For $0<t<1$ it equals
$\sin(\pi t)(\cos(\pi t)\sigma_3+
\sin(\pi t)\sigma_1)$.  Its unordered flag has matching
limits at the collision, but an ordered lift changes the two
lines after one circuit.  Hence $\mu_A$ is a transposition and
no continuous periodic unitary diagonalizer exists.

\subsection{A finite-collision diagonalization criterion}

For a continuous Hermitian path with finitely many crossings,
the relevant flag is a limit at each crossing.  This distinguishes
continuous eigenvalues from continuous eigenvectors, a distinction
already present in Rellich's smooth rotating example
\cite{Rellich,MalvettiDirrVomEndeSchulteHerbrueggen}.

\begin{theorem}[One-dimensional flag-limit criterion]
\label{thm:interval-flag-limit}
Let $A:[a,b]\to\operatorname{Herm}_N(\C)$ be continuous and
have simple spectrum except at finitely many points.  At a
collision $t_0$, let $E_\lambda$ be the eigenspace of each distinct
eigenvalue $\lambda$ of $A(t_0)$.  Near $t_0$, use the grouped
Riesz projection and a unitary trivialization continuous through
$t_0$ and equal to the identity there to identify its moving
spectral subspace with $E_\lambda$.
For $t\ne t_0$, let $F_\lambda(t)$ be the resulting unordered
decomposition of $E_\lambda$ into orthogonal eigenlines, regarded
as a point of
$\mathcal F_{\dim E_\lambda}/S_{\dim E_\lambda}$.
There is a continuous $U:[a,b]\to\U(N)$ with $U^*AU$
diagonal if and only if, for every collision and every cluster,
the two one-sided limits of $F_\lambda(t)$ exist and agree.
At an endpoint collision only the corresponding one-sided limit
is required.
\end{theorem}

\begin{proof}
The columns of a continuous diagonalizer give the required
limits.  Conversely, the quotient of the complete flag manifold
by the finite permutation group is a covering space.  Near a
limiting unordered flag, its preimage separates into disjoint
ordered neighborhoods.  A continuous labeling of the simple
eigenlines on either punctured side eventually lies in one of
these neighborhoods; thus each line has a limit.  Equality of
the unordered limits permits a fixed permutation matching the
limits across the crossing.  On an interval, each continuous
complex line bundle has a continuous unit section.  Choosing
these sections and adjusting their phases at the finitely many
crossings gives a continuous orthonormal eigenframe.
\end{proof}

The criterion is independent of the chosen continuous
trivialization: the transition between two choices has a single
limit at $t_0$ and acts on both one-sided flags.  It gives a useful first-order sufficient
condition.

\begin{corollary}[Simple derivatives within collision blocks]
\label{cor:compressed-derivative}
Suppose $A\in C^1([a,b],\operatorname{Herm}_N(\C))$.  At each collision
$t_0$ and each eigenspace $E_\lambda$ of $A(t_0)$, let
$P_\lambda$ be the orthogonal projection onto $E_\lambda$.
If every compressed derivative
\[
 C_\lambda=P_\lambda A'(t_0)P_\lambda|_{E_\lambda}
\]
has simple spectrum, then $A$ admits a continuous unitary
diagonalization.  For real-symmetric $A$, it admits a continuous
orthogonal diagonalization.
\end{corollary}

\begin{proof}
The grouped Riesz projections near $t_0$ are $C^1$ by standard
spectral perturbation theory~\cite{Kato}.  The polar
part of the sum of their products with the fixed projections
gives a $C^1$ unitary $W(t)$, with $W(t_0)=I$, identifying every
moving cluster with $E_\lambda$.  On that fixed space put
$H_\lambda(t)=P_\lambda W(t)^*A(t)W(t)P_\lambda$.
Differentiation at $t_0$ gives
\[
 H_\lambda(t)=\lambda I+(t-t_0)C_\lambda+o(|t-t_0|);
\]
the terms from $W'(t_0)$ cancel because $A(t_0)$ is scalar
on $E_\lambda$.  Therefore
$(H_\lambda(t)-\lambda I)/(t-t_0)$ tends to the simple
Hermitian matrix $C_\lambda$ on both sides.  Hence the eigenvalues
within this cluster are distinct for every sufficiently small
$0<|t-t_0|$, and their unordered eigenline flag tends to the
flag of $C_\lambda$.  Thus each collision is isolated.
The collision set is closed in the compact interval $[a,b]$, so
it is finite, and Theorem~\ref{thm:interval-flag-limit} applies.
Real polar
decomposition and real eigenline sections prove the real case.
\end{proof}

Cyr, Ekstrand, Meyers, Peoples, and Peters proved the $N=2$
version with the condition that the whole $A'(t_0)$ has simple
spectrum and asked for an $N\times N$ analogue
\cite[Theorem~2.6 and Introduction]{CyrEtAl}.  At a two-level
collision $E_\lambda=\C^2$, so their condition agrees with
Corollary~\ref{cor:compressed-derivative}.  For $N=3$, the
literal full-derivative condition fails, even when the full
infinite jet and pointwise spectrum agree with those of a smoothly
diagonalizable path.  Proposition~\ref{prop:isospectral-jet-opposite}
gives a minimal three-band pair and its explicit limiting flags.

The compressed-derivative hypothesis is sufficient.  For instance,
the path $A(t)=t^2\operatorname{diag}(1,-1)$ is
diagonalized by a constant frame although its derivative
vanishes at the collision.  Theorem~\ref{thm:interval-flag-limit}
also covers crossings whose first nonzero effective splitting
occurs at a higher order.

\subsection{A collision-supported atlas of spectral pathologies}
\label{subsec:spectral-pathology-atlas}

The extension, permutation, and Chern layers of
Theorem~\ref{thm:spectral-descent} can be prescribed before a
matrix field is constructed; their covering and eigenline
ingredients follow~\cite{GrovePedersen,FriedmanPark2016}.
Let $X$ be a connected finite CW complex,
let $Z\subset X$ be closed and nowhere dense, and put
$O=X\setminus Z$.  Fix $N\ge2$ and write
$p:\mathcal F_N\to\mathcal D_N$ for the ordered-to-unordered
flag covering.  For each component $C$ of $O$, choose a path
from a base point of $X$ to a point of $C$.  Let $K_Z$ be the
normal subgroup of $\pi_1(X)$ generated by the corresponding
conjugates of the images of $\pi_1(C)$, and set
\begin{equation}\label{eq:collision-monodromy-quotient}
 G_{\mathrm{coll}}(X,Z)=\pi_1(X)/K_Z.
\end{equation}
Changing the connecting paths leaves $K_Z$ unchanged.

\begin{theorem}[Realization and confinement of spectral flags]
\label{thm:collision-flag-realization}
With the preceding notation:
\begin{enumerate}
\item A continuous map $q:O\to\mathcal D_N$ is the unordered
spectral flag of some continuous Hermitian field that is scalar
on $Z$ and simple on $O$ if and only if $q|_C$ lifts through
$p$ for every component $C$ of $O$.
\item If such a $q$ extends continuously to $\bar q:X\to
\mathcal D_N$, its monodromy factors through
$G_{\mathrm{coll}}(X,Z)$.
\item If $\dim X\le2$, every homomorphism
$\rho:G_{\mathrm{coll}}(X,Z)\to S_N$ occurs, up to conjugacy,
as the monodromy of an extending spectral flag of such a field.
All fields in this realization can be chosen to have the same
prescribed characteristic polynomial
$\prod_{j=1}^N(\lambda-h(x)a_j)$, where
$a_1<\cdots<a_N$ are fixed real numbers and $h:X\to[0,\infty)$
is any continuous function with zero set $Z$.
If $X$ is a compact smooth surface, every such representation
also has a smooth Hermitian realization which is flat on $Z$.
\end{enumerate}
\end{theorem}

\begin{proof}
If $A$ is simple on $O$, ordering its real eigenvalues gives a
continuous ordered eigenflag on each component of $O$, proving
necessity in the first assertion.  Conversely, choose an ordered
lift $(P_1,\ldots,P_N)$ of $q$ on each component and fix
$a_1<\cdots<a_N$.  Define
\begin{equation}\label{eq:flag-realization-field}
 A(x)=h(x)\sum_{j=1}^N a_jP_j(x)\quad(x\in O),
 \qquad A(x)=0\quad(x\in Z).
\end{equation}
Because $\|A(x)\|\le h(x)\max_j|a_j|$, this field is
continuous at every point of $Z$.  On $O$ it has the prescribed
simple eigenvalues and unordered flag $q$.  This proves the
first assertion and the characteristic-polynomial claim.

If $q$ extends, the restriction of its monodromy to the
fundamental group of each component of $O$ is trivial, since
$q|_C$ lifts.  Its kernel therefore contains $K_Z$, proving
the second assertion.

Suppose now that $X$ has dimension at most two.  The full flag
manifold $\mathcal F_N$ is simply connected, so
$\pi_1(\mathcal D_N)\cong S_N$.  Let $\widehat\rho$ be the
composite of the quotient map $\pi_1(X)\to
G_{\mathrm{coll}}(X,Z)$ with $\rho$.  Map the vertices of a
CW structure to a fixed point of $\mathcal D_N$ and the
one-cells to loops representing $\widehat\rho$ on generators.
Every attaching loop of a two-cell maps to the identity in
$\pi_1(\mathcal D_N)$, hence its image bounds a disk in
$\mathcal D_N$.  These fillings give a continuous
$\bar q:X\to\mathcal D_N$ with monodromy $\widehat\rho$.
The restriction to each component of $O$ has trivial
monodromy, and therefore lifts through $p$.  Applying
\eqref{eq:flag-realization-field} completes the continuous
construction.

For a smooth surface, smooth approximation replaces $\bar q$
by a smooth map in the same homotopy class.  Choose a smooth
function $h$ which is positive on $O$ and flat on $Z$; such a
function is obtained by a rapidly convergent sum of positive
bumps supported in $O$.  On each component of $O$ the ordered
lift of the smooth $\bar q$ is smooth.  All derivatives of its
rank-one projections are uniformly bounded near $Z$: the
covering $p$ has finitely many smooth local inverse branches,
and $\bar q$ is smooth on compact $X$.  Every derivative of
$h\sum_j a_jP_j$ therefore tends to zero at $Z$ by the
Leibniz rule.  Extension by zero is smooth and flat on $Z$,
with the same monodromy and characteristic polynomial.
\end{proof}

If a component of $O$ surjects onto $\pi_1(X)$, then
$G_{\mathrm{coll}}(X,Z)=1$; this holds when $Z$ is a smooth
submanifold of codimension at least two, since generic loops
avoid it.  On a finite graph with $Z$ containing every
vertex, $G_{\mathrm{coll}}(X,Z)=\pi_1(X)$.  These permutations
concern eigenlines: the constructed fields still have globally
ordered real eigenvalue functions $ha_1,\ldots,ha_N$.

The Chern layer is equally flexible on two-complexes.

\begin{proposition}[Complete Chern realizability in dimension two]
\label{prop:chern-realizability-two-complex}
Let $X$ be a finite connected CW complex of dimension at most
two.  For any classes $c_1,\ldots,c_N\in H^2(X;\Z)$ with
$\sum_jc_j=0$, there is a simple-spectrum continuous Hermitian
field on $X$ with constant distinct eigenvalues whose ordered
eigenlines $L_j$ satisfy $c_1(L_j)=c_j$.
\end{proposition}

\begin{proof}
Choose line bundles $L_j$ with the specified first Chern
classes.  Their direct sum $E$ has $c_1(E)=0$.
Complex rank-$N$ bundles over a two-complex are classified by
their first Chern class, so $E$ is trivial.  A unitary
trivialization realizes the $L_j$ as mutually orthogonal line
subbundles of $X\times\C^N$.  For fixed distinct real $a_j$,
the field $A=\sum_ja_jP_{L_j}$ has the stated properties.
\end{proof}

Friedman--Park's line-bundle examples~\cite{FriedmanPark2016}
precede this two-complex realization.  The obstruction persists
without a collision: on $S^2$, $A(n)=n\cdot\sigma$ has a
Hopf eigenline and is at uniform operator-norm distance exactly
$1$ from continuously diagonalizable Hermitian fields:
perturbations of norm less than $1$ preserve the spectral gap
and the Chern class, while the constant zero field is
diagonalizable and has distance $1$.  In contrast,
on a finite graph the continuously diagonalizable Hermitian
fields are dense in the uniform norm.  Indeed, the repeated
eigenvalue locus in $\operatorname{Herm}_N(\C)$ has a
stratification of real codimension at least three.  An
edgewise smooth approximation can be perturbed relative to
generic simple vertex values so that every edge avoids that
locus.  The resulting simple eigenlines are trivial because
$H^2(X;\Z)=0$.  Thus flag-extension and permutation failures
on a graph can have zero distance from diagonalizable fields,
while a Chern charge can impose a positive distance.

The last distinction is invisible to even complete Taylor data
at a smooth collision.  We give a three-band pair with the same
pointwise spectrum and the same infinite jet at its unique
collision.  Write
\[
 w(t)=\frac{(t,0,1)^T}{\sqrt{1+t^2}},\quad
 u(t)=\frac{(1,0,-t)^T}{\sqrt{1+t^2}},\quad
 v=(0,1,0)^T,
\]
and $\eta(t)=e^{-1/t^2}$ for $t\ne0$, $\eta(0)=0$.
The vectors $u(t),v,w(t)$ form an orthonormal frame.

\begin{proposition}[Isospectral and jet-equivalent opposite outcomes]
\label{prop:isospectral-jet-opposite}
On $[-1,1]$ put
\begin{align*}
 A_{\mathrm{good}}(t)
  &=w(t)w(t)^T+\eta(t)
       \bigl(u(t)u(t)^T-vv^T\bigr),\\
 A_{\mathrm{bad}}(t)
  &=w(t)w(t)^T+\eta(t)
       \bigl[\cos(1/t)(u(t)u(t)^T-vv^T)
          +\sin(1/t)(u(t)v^T+vu(t)^T)\bigr]
       \quad(t\ne0),
\end{align*}
and set $A_{\mathrm{bad}}(0)=A_{\mathrm{good}}(0)=e_3e_3^T$.
Both fields are smooth and real symmetric.  For every $t$ they
have the same characteristic polynomial
$(\lambda-1)(\lambda^2-\eta(t)^2)$, and $t=0$ is their only
eigenvalue collision.  Their Taylor jets at $0$ agree to all
orders; in particular their common first derivative is
$e_1e_3^T+e_3e_1^T$, with simple spectrum $\{-1,0,1\}$.
The field $A_{\mathrm{good}}$ has a smooth orthogonal
diagonalizer, whereas $A_{\mathrm{bad}}$ has no continuous
unitary diagonalizer.
\end{proposition}

\begin{proof}
In the frame $(u,v,w)$, the matrices are respectively
$\operatorname{diag}(\eta,-\eta,1)$ and
\[
 \begin{pmatrix}
 \eta\cos(1/t)&\eta\sin(1/t)&0\\
 \eta\sin(1/t)&-\eta\cos(1/t)&0\\
 0&0&1
 \end{pmatrix}.
\]
The claimed spectrum follows.  Since
$0<\eta(t)\le e^{-1}<1$ for $t\ne0$, it is simple there.
Every derivative of $\eta(t)e^{\mathrm i/t}$ is a finite sum
bounded by $C_k|t|^{-M_k}e^{-1/t^2}$ near zero.  It therefore
extends smoothly with all derivatives zero at $0$.  The same
holds for the real and imaginary parts and for their products
with the smooth vectors $u,v$.  Both fields consequently have
the same infinite jet as $w(t)w(t)^T$ at $0$, whose first
derivative is the stated matrix.

The smooth orthogonal frame $(u,v,w)$ diagonalizes
$A_{\mathrm{good}}$.  For $A_{\mathrm{bad}}$ the positive
$\eta$ eigenline on $t>0$ is
\[
 \C\bigl(\cos(1/(2t))u(t)+\sin(1/(2t))v\bigr).
\]
Along $t_n=(2\pi n)^{-1}$ this line tends to $\C e_1$;
along $s_n=((2n+1)\pi)^{-1}$ it tends to $\C e_2$.
Its spectral projection has no limit at zero.  A continuous
diagonalizer would give a continuous eigenline for this simple
eigenvalue on the connected half-interval $t>0$, a
contradiction.  The unordered flag also fails to converge:
the phases $0$ and $\pi/2$ yield limiting decompositions
$\{\C e_1,\C e_2,\C e_3\}$ and
$\{\C(e_1+e_2),\C(e_1-e_2),\C e_3\}$.
\end{proof}

Rellich's analytic eigenframe theorem~\cite{Rellich} excludes
this flat oscillation in the analytic one-parameter setting.
The two-band pair
$\eta(t)\sigma_3$ and
$\eta(t)(\cos(1/t)\sigma_3+\sin(1/t)\sigma_1)$
with both values set to zero at $t=0$
already has identical pointwise spectra and infinite collision
jets with opposite diagonalization outcomes.  Its derivative
vanishes at the collision.  The three-band construction above
is the smallest such pair with a simple full derivative.

The pair also separates orbit closure from exact conjugacy.  Choose
an even smooth cutoff $\chi_\varepsilon$ which is zero for
$|t|\le\varepsilon/2$ and one for $|t|\ge\varepsilon$.
In the $(u,v,w)$ frame, rotate the $(u,v)$ plane by
$\chi_\varepsilon(t)/(2t)$ and fix $w$; the rotation is set
to the identity near zero.  This gives a smooth
$U_\varepsilon:[-1,1]\to\mathrm{SO}(3)$ satisfying
$U_\varepsilon A_{\mathrm{good}}U_\varepsilon^T
=A_{\mathrm{bad}}$ for $|t|\ge\varepsilon$.  Inside the
remaining interval, every derivative of the defect is bounded
by a power of $\varepsilon^{-1}$ times a flat bound for $\eta$.
Hence $U_\varepsilon A_{\mathrm{good}}U_\varepsilon^T
\to A_{\mathrm{bad}}$ in $C^\infty$, although no continuous
exact gauge exists, even with identical spectra and jets.

\subsection{The complete fixed-polynomial flag space at a scalar collision}
\label{subsec:flag-cluster-space}

Fix a compact metrizable $X$, a closed nowhere-dense $Z\subset X$,
and $O=X\setminus Z$.  Let continuous real functions
$\lambda_1\le\cdots\le\lambda_N$ be strictly ordered on $O$
and have the common value $\lambda_0(z)$ at each $z\in Z$.
Write $\mathcal F_N=\U(N)/\U(1)^N$ for ordered orthogonal
eigenline flags and $p:\mathcal F_N\to\mathcal D_N$ for the
finite permutation quotient.

\begin{theorem}[Fixed-polynomial fields and flag-cluster fibers]
\label{thm:fixed-polynomial-flag-space}
The assignment
\begin{equation}\label{eq:fixed-polynomial-flag-field}
 A_f(x)=\sum_{j=1}^N\lambda_j(x)P_j(f(x))\quad(x\in O),
 \qquad A_f(z)=\lambda_0(z)I\quad(z\in Z)
\end{equation}
is a bijection from $C(O,\mathcal F_N)$ onto the continuous
Hermitian fields with these prescribed eigenvalue functions.
Their uniform operator-norm distance is exactly
\begin{equation}\label{eq:gap-weighted-flag-metric}
 d_\lambda(f,g)=\sup_{x\in O}
 \left\|\sum_j\lambda_j(x)
       \bigl(P_j(f(x))-P_j(g(x))\bigr)\right\|_{\mathrm{op}}.
\end{equation}
For $q=p\circ f$ and $z\in Z$, set
\[
 K_z(q)=\bigcap_{V\ni z}
       \overline{q(V\cap O)}\ \subset\mathcal D_N.
\]
The unordered flag $q$ extends continuously to $X$ if and
only if every $K_z(q)$ is a singleton.

If $X=B^d$, $Z=\{0\}$, and $d\ge2$, then $K_0(q)$ is a
nonempty compact connected set.  Conversely, every nonempty
compact connected $K\subset\mathcal D_N$ occurs as $K_0(q)$
for a continuous field~\eqref{eq:fixed-polynomial-flag-field}
with the same fixed polynomial
$\prod_j(\lambda-|x|a_j)$, where
$a_1<\cdots<a_N$ are any fixed real numbers.  This last
realization has a nullhomotopic linking flag.
\end{theorem}

\begin{proof}
On $O$ the ordered spectral projections recover $f$.
At $z\in Z$,
\[
 \|A_f(x)-\lambda_0(z)I\|_{\mathrm{op}}
    =\max_j|\lambda_j(x)-\lambda_0(z)|\longrightarrow0,
\]
independently of $f$.  This proves the bijection;
subtraction gives~\eqref{eq:gap-weighted-flag-metric}, which
is a metric by simplicity on $O$.

The graph closure of $q$ has fiber $K_z(q)$ over $z$.
Singleton fibers make its projection to $X$ a continuous
bijection from compact to Hausdorff, hence a homeomorphism;
the inverse extends $q$.  The converse is immediate.

For $d\ge2$, each punctured small ball is connected.
The closures of its images under $q$ are nested nonempty
compact connected sets; their intersection is $K_0(q)$.
Conversely, connectedness of a compact set $K$ in the manifold
$\mathcal D_N$ implies that every sufficiently small metric
neighborhood of $K$ is path connected.  Choose successively finer finite
nets in $K$, join their points by paths in neighborhoods
shrinking to $K$, and concatenate the paths at times tending
to zero.  This gives a continuous
$\gamma:(0,1]\to\mathcal D_N$ whose cluster set at zero
is exactly $K$.  Lift $\gamma$ through the covering $p$ to
$\widetilde\gamma:(0,1]\to\mathcal F_N$ and put
$f(x)=\widetilde\gamma(|x|)$.  Formula
\eqref{eq:fixed-polynomial-flag-field}, with
$\lambda_j(x)=|x|a_j$, gives the claimed field.  The flag is
constant on every linking sphere, so its link class is zero.
\end{proof}

A nonzero link class forces $K_0(q)$ outside every
contractible normal ball of $\mathcal D_N$: otherwise small
linking spheres would have zero class.  An uncharged collision,
however, can have any prescribed continuum of limiting flags.

The ordered cluster set supports a sharper charge statement.
For $f:B^d\setminus\{0\}\to\mathcal F_N$ put
\[
 K(f)=\bigcap_{\varepsilon>0}
 \overline{f(B_\varepsilon\setminus\{0\})},
 \qquad
 \ell(f)=[f|_{S^{d-1}_r}]\in[S^{d-1},\mathcal F_N].
\]
The class $\ell(f)$ is independent of $r>0$.  For a compact
$K\subset\mathcal F_N$, let $\mathcal R_d(K)$ be the set of
link classes of maps with $K(f)=K$.

\begin{theorem}[Cluster-link image principle]
\label{thm:cluster-link-image}
For any compact $K\subset\mathcal F_N$ and $d\ge2$,
\begin{equation}\label{eq:cluster-link-upper}
 \mathcal R_d(K)\subset
 \bigcap_{U\supset K}
  \operatorname{im}\bigl([S^{d-1},U]
       \longrightarrow[S^{d-1},\mathcal F_N]\bigr),
\end{equation}
where $U$ ranges over open neighborhoods of $K$.  If $K$ is
path connected, then
\begin{equation}\label{eq:cluster-link-lower}
 \operatorname{im}\bigl([S^{d-1},K]
       \longrightarrow[S^{d-1},\mathcal F_N]\bigr)
       \subset\mathcal R_d(K).
\end{equation}
Consequently, for a compact connected smoothly embedded
submanifold $K$, these sets coincide.  Every class in this
common image is realized by a smooth Hermitian field on $B^d$
with one fixed characteristic polynomial, scalar collision
only at the origin, and zero infinite jet there.
\end{theorem}

\begin{proof}
If $K(f)=K$, compactness implies that for every open
$U\supset K$, the image of a sufficiently small punctured
ball lies in $U$.  A small linking sphere then factors
through $U$, proving~\eqref{eq:cluster-link-upper}.

For~\eqref{eq:cluster-link-lower}, choose a representative
$h:S^{d-1}\to K$ and a small closed disk
$D\subset S^{d-1}$ on which $h$ is constant, with value $p$;
precomposition by a map homotopic to the identity gives this
without changing its class.  Choose a continuous path
$\gamma:[0,\infty)\to K$ with $\gamma(0)=p$ and whose
tail-cluster set is all of $K$.  Such a path is obtained by
joining successively finer finite nets inside the path-connected
compact set $K$.  Let $\phi:S^{d-1}\to[0,1]$ be supported in
the interior of $D$ and equal to one at $\theta_0\in D$.
For $r>0$, define
\[
 f(r\theta)=
 \begin{cases}
   \gamma\bigl(\phi(\theta)/r\bigr),&\theta\in D,\\
   h(\theta),&\theta\notin D.
 \end{cases}
\]
The definitions agree near $\partial D$.  On each linking
sphere, replacing $\phi/r$ by $s\phi/r$, $0\le s\le1$,
homotopes $f$ to $h$.  All values lie in $K$, and the path
$f(r\theta_0)=\gamma(1/r)$ has cluster set $K$; hence
$K(f)=K$.

If $K$ is a compact smooth submanifold, a tubular neighborhood
retracts to $K$, so the upper bound equals the lower one.
Choose $h,\phi,\gamma$ smooth.  The net path can be assembled
from bounded-length smooth geodesic segments using a fixed
endpoint-flat reparametrization; its derivatives of every fixed
order are bounded uniformly in time.  For fixed distinct
$a_1<\cdots<a_N$, put $\varphi(r)=e^{-1/r^2}$ for $r>0$,
$\varphi(0)=0$, and set
\[
 H(r\theta)=\varphi(r)\sum_j a_jP_j(f(r\theta)),
 \qquad H(0)=0.
\]
Every derivative of the projection sum has at most polynomial
growth in $1/r$; the flat factor $\varphi$ therefore makes $H$
smooth with all jets zero at the origin.  Its characteristic
polynomial is $\prod_j(\lambda-\varphi(r)a_j)$, independent of
the chosen link class and cluster realization.
\end{proof}

For a general continuum, its neighborhood homotopy system
retains information invisible to maps into the set itself
\cite{MardesicSegal}.  The next criterion realizes precisely
its spherical classes in fixed-polynomial collisions.

\begin{theorem}[Coherent-neighborhood realization]
\label{thm:cluster-shape-realization}
Let $K\subset\mathcal F_N$ be nonempty, compact, and connected,
and let $d\ge2$.  Choose a Riemannian metric on $\mathcal F_N$
and nested neighborhoods
$U_m=\{y:\operatorname{dist}(y,K)<2^{-m}\}$.  Then
\begin{equation}\label{eq:cluster-shape-realization}
 \mathcal R_d(K)=
 \operatorname{im}\left(
    \varprojlim_m[S^{d-1},U_m]
      \longrightarrow [S^{d-1},\mathcal F_N]\right).
\end{equation}
Every class in this set is realized by a continuous Hermitian
field with the single prescribed polynomial
$\prod_j(\lambda-|x|a_j)$, for any fixed distinct
$a_1<\cdots<a_N$.  There is also one positive smooth radial
function $\varphi$, flat at zero, such that every class in
$\mathcal R_d(K)$ has a smooth realization with cluster $K$,
zero infinite jet, and the same polynomial
$\prod_j(\lambda-\varphi(|x|)a_j)$.
\end{theorem}

\begin{proof}
The neighborhoods $U_m$ are path connected: each is the union
of metric balls centered on the connected set $K$.
If $K(f)=K$, choose $r_m\downarrow0$ so that
$f(B_{r_m}\setminus\{0\})\subset U_m$.  The restrictions
$h_m=f|_{S^{d-1}_{r_m}}$ map to $U_m$, and the annuli between
successive radii give homotopies in $U_m$ from $h_m$ to
$h_{m+1}$.  They determine an element of the inverse limit
whose image is $\ell(f)$.

Conversely, represent an inverse-limit element by maps
$h_m:S^{d-1}\to U_m$ and homotopies in $U_m$ between
$h_m$ and $h_{m+1}$.  Let $E_m\subset K$ be a finite
$2^{-m}$-net.  Before the $m$th homotopy, insert a loop of
sphere maps based at $h_m$ that visits every point of $E_m$.
To construct it, deform $h_m$ by a degree-one pinch so that
it is constant, with value $p$, on a small angular disk $D$.
For each $e\in E_m$, choose a path
$\gamma_e:[0,1]\to U_m$ from $p$ to $e$ and a bump
$\phi:D\to[0,1]$ vanishing near $\partial D$ and equal to
one at its center.  On $D$ use
$\theta\mapsto\gamma_e(t\phi(\theta))$ for $t$ running from
zero to one and back, retaining the pinched map outside $D$.
Undo the pinch after all excursions.  This loop begins and
ends at $h_m$, remains in $U_m$, and hits every $e\in E_m$.
Place it and the chosen homotopy in the annulus
$2^{-(m+1)}\le|x|\le2^{-m}$.  The resulting punctured-ball
map has cluster set contained in
$\bigcap_m\overline{U_m}=K$.  Since it hits each $E_m$ at a
radius tending to zero, its cluster set contains $K$.
Multiplication by $|x|$ in
\eqref{eq:fixed-polynomial-flag-field} gives the asserted
continuous field.

For smooth realization, approximate each punctured-ball flag
map by a smooth one with pointwise error tending to zero at
the origin.  A sufficiently close approximation is homotopic
on each link and has the same cluster set.  The set
$[S^{d-1},\mathcal F_N]$ is countable, so enumerate the
realizable classes and choose one smooth flag map for each.
Take a nonnegative smooth radial partition of unity
$\{\psi_m\}$ on the punctured ball, with $\psi_m$ supported
in a dyadic annulus.  Choose positive coefficients $c_m$
successively small enough that, on the support of $\psi_m$,
all derivatives through order $m$ of
$c_m\psi_m$ and of
$c_m\psi_m\sum_j a_jP_j(f_i)$, for $i\le m$, are bounded by
$2^{-m^2}$.  Then $\varphi=\sum_m c_m\psi_m$ is positive,
smooth, radial, and flat at zero.  For every fixed $i$,
$\varphi\sum_j a_jP_j(f_i)$ extends smoothly with all jets zero.
The same $\varphi$ works for the entire enumeration, proving the
common-polynomial assertion.
\end{proof}

For $d=3$ and a compact connected smooth ordered cluster $K$,
Theorem~\ref{thm:cluster-link-image} identifies the realizable
charge lattice with the spherical homotopy image
\[
 Q_K=\operatorname{im}\bigl(\pi_2(K)\to\pi_2(\mathcal F_N)\bigr)
 \subset\{(d_1,\ldots,d_N)\in\Z^N:\sum_jd_j=0\}.
\]
The based description agrees with the free-homotopy image
because $\mathcal F_N$ is simply connected.
This can be smaller than the information in $H^2(K;\Z)$.
For example, let $C\subset\C P^2$ be a smooth plane cubic,
and regard its tautological line as $L_1$ in the trivial
rank-three bundle.  Its orthogonal rank-two complement has a
nowhere-zero section over the real two-dimensional $C$;
splitting off that section yields an embedding
$K\cong C\hookrightarrow\mathcal F_3$.  The cubic has genus
one, so $\pi_2(K)=0$, while
$\langle c_1(L_1|_K),[K]\rangle=-3$.
Thus $Q_K=0$ despite nonzero restricted Chern cohomology.
By contrast, the Schubert sphere obtained by varying $L_1$
inside a fixed $\C^2\subset\C^3$ and taking its orthogonal
line as $L_2$ has
$Q_K=\Z(-1,1,0)$.  Both cluster geometries occur with the
same fixed polynomial and zero infinite jet, while the second
supports every integer multiple of its spherical charge.

Local flag limits carry permutation and Chern data. Their
realization on an entire base is a separate assembly problem:
the spectral cover records the permutation action, and its line
bundles record the residual charge.

\section{Global realization of eigenflags and charges}
\label{sec:spectral-realization}

\subsection{Permutation--Chern coupling on a two-complex}
\label{subsec:spectral-joint-layer}

The permutation and line-bundle layers can interact even when
the collision set supports every prescribed permutation.  The
eigenvalue-permutation local system and comparison bundles of
Friedman and Park~\cite{FriedmanPark2016} motivate the natural
coefficient system.  Explicitly, put
\[
 \mathcal T=\U(1)^N\subset\U(N),\qquad
 \mathcal N=N_{\U(N)}(\mathcal T)
              =\mathcal T\rtimes S_N,
 \qquad \mathcal D_N=\U(N)/\mathcal N,
\]
where the splitting uses permutation matrices.  Fix a marked
permutation cover $P_\rho\to X$ with monodromy
$\rho:\pi_1(X)\to S_N$.  The permutation lattice
$M_\rho=\Z^N_\rho$ is a local system, with sum map
$\Sigma:M_\rho\to\Z$.  A principal $\mathcal N$-bundle $Q$
whose quotient by $\mathcal T$ is the marked $P_\rho$ has a
twisting class
\[
 c(Q;P_\rho)\in H^2(X;M_\rho).
\]
It is normalized to be zero for the bundle obtained from
$P_\rho$ by the permutation-matrix splitting.  The exponential
sequence for the twisted torus identifies the remaining classes
of such marked reductions with $H^2(X;M_\rho)$.  Let
$\Bock:H^1(X;\Z/2)\to H^2(X;\Z)$ be the integral Bockstein.

\begin{theorem}[Joint realization of permutation and twisted charge]
\label{thm:joint-spectral-layers}
Let $X$ be a connected finite CW complex of dimension at most
two.  Fix $\rho:\pi_1(X)\to S_N$ and
$c\in H^2(X;M_\rho)$.  There is a map
$q:X\to\mathcal D_N$ with marked monodromy $\rho$ and
normalizer twisting class $c$ if and only if
\begin{equation}\label{eq:joint-permutation-chern}
 \Sigma_*c+\Bock(\operatorname{sgn}\circ\rho)=0
                \quad\text{in }H^2(X;\Z).
\end{equation}
If moreover $Z\subset X$ is closed and nowhere dense,
$O=X\setminus Z$, and $\rho$ is trivial on the fundamental
group of every component of $O$, then whenever this condition
holds, for any continuous
$h:X\to[0,\infty)$ with zero set $Z$ and any fixed
$a_1<\cdots<a_N$, the map $q$ is the unique extension of the
unordered eigenflag of a continuous Hermitian field scalar on
$Z$, simple on $O$, and having characteristic polynomial
$\prod_j(\lambda-h(x)a_j)$.  If $X$ is a compact smooth
surface, the same marked pair $(\rho,c)$ also has a smooth
realization flat on $Z$.
\end{theorem}

\begin{proof}
For a marked $\mathcal N$-bundle $Q$, let
$E_Q=Q\times_{\mathcal N}\C^N$.  The determinant of a monomial
unitary matrix is the product of its diagonal entries times
the sign of its permutation.  The determinant line of the
zero-twist bundle is the complexification of the real sign
line of $\rho$, whose first Chern class is
$\Bock(\operatorname{sgn}\circ\rho)$.  A torus twist of class
$c$ adds $\Sigma_*c$.  Thus
\begin{equation}\label{eq:joint-determinant}
 c_1(E_Q)=\Sigma_*c+
                 \Bock(\operatorname{sgn}\circ\rho).
\end{equation}
If $q$ exists, $Q=q^*\U(N)$ and $E_Q$ is trivial, proving
necessity.  Conversely, the twisted-torus classification gives
the marked $Q$ with the specified class.  A rank-$N$ complex
bundle over a two-complex is trivial exactly when $c_1=0$.
Equation~\eqref{eq:joint-determinant} therefore trivializes
$E_Q$.  Under a unitary trivialization, $Q$ is a reduction of
the trivial frame bundle and determines $q:X\to\mathcal D_N$.

The restriction of $\rho$ to every component of $O$ is
trivial.  Hence $q|_O$ has an ordered lift on each component.
Use its projections in~\eqref{eq:flag-realization-field} to
construct the claimed field.  Since $O$ is dense, $q$ is its
unique continuous flag extension.  Smooth approximation of
$q$ and a positive function flat on $Z$ give the surface
assertion, with the same marked pair $(\rho,c)$, by the argument of
Theorem~\ref{thm:collision-flag-realization}.
\end{proof}

For trivial $\rho$, condition~\eqref{eq:joint-permutation-chern}
is the familiar equation $\sum_jc_1(L_j)=0$.  A nontrivial
permutation can force a nonzero twisted class even though the
ambient bundle is trivial.  The smallest such model has two
bands and a projective-plane base.

The based and free homotopy classes of self-maps of
$\R P^2$ of odd lifted degree were classified by Fenille and
Gon\c{c}alves~\cite[Section~3]{FenilleGoncalves}.  We use that
classical projective-plane topology in the following spectral
construction.

\begin{proposition}[Odd twisted charges from one fixed spectrum]
\label{prop:rp2-odd-spectral-charge}
Set $X=\R P^2$, $Z=\R P^1$, and $N=2$.  Let $\rho$ exchange
the two flags along the generator of $\pi_1(X)=\Z/2$.
Then $H^2(X;M_\rho)\cong\Z$, and a twisting class occurs in
Theorem~\ref{thm:joint-spectral-layers} exactly when its integer
coordinate is odd.  For every positive odd $m$ there is a
smooth Hermitian field $A_m$ with this charge, scalar exactly
on $Z$ and simple on $O$, such that all $A_m$ have the same
characteristic polynomial $\lambda^2-h^2$ for one fixed smooth
function $h$.  Fields with different positive odd charges are
not conjugate by a continuous global unitary gauge.
\end{proposition}

\begin{proof}
The ordered two-flag space is the Bloch sphere
$\mathcal F_2\cong S^2$, and exchanging orthogonal lines acts
antipodally.  Thus $\mathcal D_2\cong\R P^2$.  The complement
$O=\R P^2\setminus\R P^1$ is an open disk, so the
representation $\rho$ factors through $G_{\mathrm{coll}}(X,Z)$.
The generator of $\pi_1(X)$ swaps the two coordinates of
$M_\rho=\Z^2$.  Cellular cochains give
\[
 H^2(X;M_\rho)=\Z^2/\langle(1,1)\rangle\cong\Z,
 \qquad [(a,b)]\longmapsto a-b.
\]
Since $H^2(X;\Z)=\Z/2$, the sum map is reduction modulo two.
The Bockstein of the nontrivial sign class generates this
$\Z/2$.  Equation~\eqref{eq:joint-permutation-chern} is
therefore equivalent to oddness.

Write $S^2=\{(w,s)\in\C\times\R:|w|^2+s^2=1\}$.  For odd
$m>0$, the map
\[
 f_m(w,s)=\frac{(w^m,s)}{\sqrt{|w|^{2m}+s^2}}
\]
is antipodal-equivariant and has degree $m$.  It descends to
$q_m:\R P^2\to\mathcal D_2\cong\R P^2$ with monodromy
$\rho$.  The pulled-back positive eigenline on the double
cover $S^2$ has first Chern number $m$ up to one fixed sign
convention; $M_\rho$ is the direct image of the integer
system under this double cover, and the covering-space
identification
$H^2(X;M_\rho)\cong H^2(S^2;\Z)$ gives its twisting
coordinate $m$ with that convention.

Choose the same smooth function $h>0$ on $O$, flat and zero
on $Z$, for every $m$.  An ordered lift of $q_m|_O$ defines
rank-one projections $P_+,P_-$; put
$A_m=h(P_+-P_-)$ on $O$ and $A_m=0$ on $Z$.
The flatness of $h$ makes each field smooth across $Z$.
The characteristic polynomial is $\lambda^2-h^2$.
Every map $\R P^2\to\U(2)$ is nullhomotopic: $H^1(\R P^2;\Z)=0$
and $\pi_2\U(2)=0$.  A global unitary conjugacy between two
fields would therefore make their uniquely extended maps
$q_m,q_n:\R P^2\to\mathcal D_2$ homotopic.  The sphere lifts
would then have degrees equal up to sign.  Distinct positive
odd $m$ give inequivalent fields.
\end{proof}

The coupling~\eqref{eq:joint-permutation-chern} uses only the
determinant on a two-complex.  In higher dimensions the higher
Chern classes and unstable bundle data of $E_Q$ may obstruct a
normalizer reduction with the same $(\rho,c)$.

\subsection{A line on the spectral cover and its transfer obstruction}
\label{subsec:spectral-cover-transfer}

The permutation local system and comparison bundles of
Friedman--Park~\cite{FriedmanPark2016,FriedmanPark2024}
combine into one line bundle on a finite cover, extending the
determinant test beyond two-complexes.

\begin{theorem}[Spectral-cover realization]
\label{thm:spectral-cover-transfer}
Let $X$ be a connected finite CW complex and $N\ge2$.
Isomorphism classes of principal
$\mathcal N=\U(1)^N\rtimes S_N$ bundles over $X$ correspond
to isomorphism classes of pairs
\[
       (\pi:Y\longrightarrow X,\ L\longrightarrow Y),
\]
where $\pi$ is an $N$-sheeted cover and $L$ is a Hermitian line
bundle.  The associated rank-$N$ bundle is the finite direct
image
\[
             E=\pi_*L,\qquad
             E_x=\bigoplus_{y\in\pi^{-1}(x)}L_y.
\]
The pair is realized by a map $q:X\to\mathcal D_N$ if and only
if $E$ is trivial.  Moreover, two such maps are related by a
continuous $\U(N)$ gauge transformation if and only if their
spectral-cover pairs are isomorphic.

For a marked permutation cover with monodromy $\rho$, the
covering-space identification
$H^2(X;M_\rho)\cong H^2(Y;\Z)$ identifies its twisting
class with $c_1(L)$.  The first transfer obstruction is
\begin{equation}\label{eq:spectral-cover-determinant}
 c_1(E)=\operatorname{tr}_{Y/X}c_1(L)
                    +\Bock(\operatorname{sgn}\circ\rho).
\end{equation}
Consequently, for $\dim X\le3$ the vanishing of
\eqref{eq:spectral-cover-determinant} is necessary and
sufficient for realization.  For $\dim X\le4$, the necessary
and sufficient conditions are $c_1(E)=c_2(E)=0$.
For $\dim X\le2N$, realization is equivalent to
\begin{equation}\label{eq:spectral-cover-K-transfer}
              \pi_![L]-N=0\qquad\text{in }\widetilde K^0(X),
\end{equation}
where $\pi_!$ is finite-cover direct-image transfer.
In arbitrary dimension, triviality of the actual rank-$N$
bundle $E$ remains the exact criterion.
\end{theorem}

\begin{proof}
Given a principal $\mathcal N$ bundle, quotient by
$\mathcal T=\U(1)^N$ to obtain its $S_N$ permutation cover;
the associated cover of $N$ labels is $Y$.  The coordinate
line at each label descends to $L$.  Conversely, the frames of
$\pi_*L$ formed by choosing one unit vector in every fiber
line, in arbitrary order, constitute a principal $\mathcal N$
bundle.  The constructions invert each other on objects and
bundle isomorphisms.  A unitary trivialization of $E$ gives
a reduction of the trivial $\U(N)$ frame bundle, hence $q$.
Two trivializations differ by a gauge.  Conversely, an
isomorphism of pairs induces an isomorphism of their direct
images and therefore a gauge between the corresponding maps.

The local system $M_\rho$ is $\pi_*\Z_Y$, so covering-space
cochains identify its second cohomology with $H^2(Y;\Z)$.
The determinant calculation in~\eqref{eq:joint-determinant},
under $H^2(X;M_\rho)\cong H^2(Y;\Z)$, gives
\eqref{eq:spectral-cover-determinant}.  Rank-$N$ complex
bundles on a three-complex are classified by $c_1$, since
$\pi_2(B\U(N))=\Z$ and $\pi_3(B\U(N))=0$.  On a four-complex
the next obstruction is $c_2$, and $(c_1,c_2)$ classifies
rank-$N$ bundles for $N\ge2$.  Finally, the stabilization
$B\U(N)\to B\U$ is injective on bundle classes over complexes
of dimension at most $2N$: the successive stabilization
fibers start with $\U(N+1)/\U(N)=S^{2N+1}$.
This proves~\eqref{eq:spectral-cover-K-transfer} in that range.
\end{proof}

For a cover that comes from a polynomial, the trivial line has a
distinguished direct-image framing.  This elementary observation
also supplies the reference matrix in the next theorem.

\begin{lemma}[Vandermonde realization of a simple polynomial]
\label{lem:vandermonde-normal-realization}
Let $X$ be a compact space and let $\mu\in C(X)[t]$ be monic of
degree $N$ with distinct roots at every point.  Its root cover
$\pi:Y_\mu\to X$ satisfies
$\pi_*\underline{\C}_{Y_\mu}\cong\underline{\C}^{N}$.
Consequently, $\mu$ is the characteristic polynomial of a
continuous normal $N\times N$ matrix on $X$.
\end{lemma}

\begin{proof}
Let $\lambda:Y_\mu\to\C$ be the root coordinate.  The
functions $1,\lambda,\ldots,\lambda^{N-1}$ are global sections
of $\pi_*\underline{\C}_{Y_\mu}$.  At each $x$ their values
on the $N$ roots form a Vandermonde matrix, which is invertible
because the roots are distinct.  Thus they give a global frame.
Gram--Schmidt with the direct-image metric makes the frame
unitary.  Multiplication by $\lambda$ on the direct-image
bundle is normal in that metric and has characteristic
polynomial $\mu$; express it in the unitary frame.
\end{proof}

In particular, every Chern class of the permutation direct image
$\pi_*\underline{\C}_{Y_\mu}$ vanishes, including its determinant
class $\Bock(\operatorname{sgn}\circ\rho)$.  This restricts which
abstract permutation covers arise from simple polynomials.

The same direct-image construction gives a bundle-theoretic
criterion in all dimensions and an explicit low-dimensional
answer to the realizability question for comparison obstructions posed by
Friedman--Park~\cite[\S8]{FriedmanPark2016}.  Their later
low-dimensional calculation assumes globally ordered roots
\cite[Theorem~4.1]{FriedmanPark2024}; the following statement also
allows root monodromy.  A normal matrix is essential here: for a
Hermitian matrix the distinct real roots have a global order.

\begin{theorem}[Realizable comparison classes for a fixed polynomial]
\label{thm:fixed-polynomial-obstruction-image}
Let $X$ be a connected finite CW complex, and let
$\mu\in C(X)[t]$ be monic of degree $N\ge2$ with distinct roots
at every point.  Let $A:X\to M_N(\C)$ be any continuous normal
matrix with characteristic polynomial $\mu$; one exists by
Lemma~\ref{lem:vandermonde-normal-realization}.
Let $\pi:Y_\mu\to X$ be its root cover and let $L_A\to Y_\mu$
be the eigenline of $A$.  For $h\in H^2(Y_\mu;\Z)$, denote by
$M_h$ the line bundle with $c_1(M_h)=h$ and put
\[
             E_h=\pi_*(L_A\otimes M_h).
\]
Under the canonical covering-space identification
$H^2(X;\pi_*\Z_{Y_\mu})\cong H^2(Y_\mu;\Z)$, the image of the
Friedman--Park comparison obstruction for matrices with
characteristic polynomial $\mu$, relative to $A$, is exactly
\begin{equation}\label{eq:fixed-polynomial-obstruction-image}
 \mathcal R_A(\mu)
   =\{h\in H^2(Y_\mu;\Z):E_h\cong\underline{\C}^{N}\}.
\end{equation}
The comparison class gives a bijection between
$\mathcal R_A(\mu)$ and unitary-gauge equivalence classes of
normal matrices with characteristic polynomial $\mu$.
For the Vandermonde reference matrix $A_\mu$, whose spectral
line is trivial (the root summand is generated locally by its
coordinate unit vector), this becomes the absolute description
\begin{equation}\label{eq:absolute-polynomial-realization}
 \mathcal R_\mu
   =\{h\in H^2(Y_\mu;\Z):
                  \pi_*M_h\cong\underline{\C}^{N}\}.
\end{equation}

If $\dim X\le3$, this is the explicit subgroup
\begin{equation}\label{eq:three-dimensional-realization-kernel}
 \mathcal R_A(\mu)
     =\ker\bigl(\operatorname{tr}_{Y_\mu/X}:
             H^2(Y_\mu;\Z)\longrightarrow H^2(X;\Z)\bigr).
\end{equation}
Thus the gauge classes form a torsor for this kernel, including
when the roots have nontrivial permutation monodromy.  If
$\dim X\le4$, the exact conditions are
$\operatorname{tr}_{Y_\mu/X}h=0$ and $c_2(E_h)=0$.  For
$\dim X\le2N$, they are equivalently the equation
\[
 \pi_!\bigl([L_A][M_h]\bigr)=N\quad\text{in }K^0(X).
\]
\end{theorem}

\begin{proof}
For another normal matrix $B$ with characteristic polynomial
$\mu$, let $L_B$ be its eigenline on $Y_\mu$.  On an open set
where the roots are labeled, the comparison transporter has one
circle coordinate for each root.  Its coordinate line is
$\operatorname{Hom}(L_A,L_B)$ at that root, by the usual
line-bundle description of the comparison obstruction
\cite[Proposition~7.1]{FriedmanPark2016}.  Relabeling permutes
the coordinates.  The covering-space identification therefore
gives
\[
 \theta(A,B)=c_1\bigl(\operatorname{Hom}(L_A,L_B)\bigr)
       \in H^2(Y_\mu;\Z).
\]
Consequently $\theta(A,B)=h$ exactly when
$L_B\cong L_A\otimes M_h$.  Since $\pi_*L_B$ is the trivial
ambient bundle of $B$, this proves necessity in
\eqref{eq:fixed-polynomial-obstruction-image}.

Conversely, suppose $E_h$ is trivial.  Equip it with the direct
image Hermitian metric and choose a unitary trivialization.
On the summand $(L_A\otimes M_h)_y\subset(E_h)_{\pi(y)}$,
let an endomorphism act by multiplication by the root coordinate
$\lambda(y)$.  The resulting endomorphism is continuous and
normal, has characteristic polynomial $\mu$, and after
trivialization is a matrix $B$ with $\theta(A,B)=h$.  If two
such matrices have the same comparison class, their eigenlines
on the fixed cover are isomorphic; pushing forward that
isomorphism and taking its unitary part gives a unitary gauge
between the matrices.  This proves the bijection.

The determinant identity~\eqref{eq:spectral-cover-determinant}
applied to $L_A\otimes M_h$ and $L_A$ gives
$c_1(E_h)=\operatorname{tr}_{Y_\mu/X}h$, because
$E_0=\pi_*L_A$ is trivial and the permutation Bockstein
term cancels.  Rank-$N$ bundles on a three-complex are
classified by $c_1$, yielding~\eqref{eq:three-dimensional-realization-kernel}.
The four-dimensional and stable-range assertions follow from
Theorem~\ref{thm:spectral-cover-transfer}.
\end{proof}

\begin{example}[A monodromic obstruction group with one gauge class]
\label{ex:monodromic-transfer-zero}
Let $X=S^1\times S^2$, write $z\in S^1\subset\C$, and set
\[
       A(z,p)=\begin{pmatrix}0&z\\1&0\end{pmatrix}.
\]
This matrix is unitary, and its polynomial $t^2-z$ has a
nontrivial two-sheeted root cover
$Y_\mu\cong S^1\times S^2$, doubling the circle coordinate.
The comparison obstruction group is
$H^2(Y_\mu;\Z)\cong\Z$, while transfer on $H^2$ is
multiplication by $2$.  Hence~\eqref{eq:three-dimensional-realization-kernel}
contains only zero: every normal matrix on $X$ with this fixed
polynomial is unitarily gauge-equivalent to $A$.
\end{example}

At the opposite extreme, let $X$ be the mapping torus of the
orthocomplement involution on $\C P^1$ used in
\cite[Example~7.3]{FriedmanPark2016}, and let
$p:X\to S^1$ be its projection.  For the fixed polynomial
$t^2-p(x)$, the root cover is $S^1\times S^2$ and
$H^2(X;\Z)=0$, while $H^2(Y_\mu;\Z)\cong\Z$.
Formula~\eqref{eq:three-dimensional-realization-kernel} shows that
every integer comparison class occurs.  Thus the two examples
have the same root-exchange monodromy and opposite realization
behavior, determined by cohomological transfer.

The kernel in~\eqref{eq:three-dimensional-realization-kernel}
need not describe a subgroup of realizable classes in dimension
four.  For example, take $X=S^2\times S^2$, the split polynomial
$t(t-1)$, and $A=\operatorname{diag}(0,1)$.  Write
$\alpha,\beta\in H^2(X;\Z)$ for the two standard generators.
Under $H^2(Y_\mu;\Z)=H^2(X;\Z)^2$, the classes
$(\alpha,-\alpha)$ and $(\beta,-\beta)$ are realized:
the corresponding sum of two lines has $c_1=c_2=0$.
Their sum $(\alpha+\beta,-\alpha-\beta)$ is not realized,
since its direct image has
$c_2=-(\alpha+\beta)^2=-2\alpha\beta\ne0$.
This is the two-line Chern-class phenomenon already implicit
in~\cite[Example~2.5]{FriedmanPark2024}; it shows why the
three-dimensional torsor statement does not extend unchanged
to four-complexes.

For a scalar-collision field with simple-spectrum complement
$O$, the cover must split over each component of $O$, permitting
strict eigenvalue orderings.  The pair $(Y,L)$ classifies the
extended unordered flag's gauge orbit; orderings and scalar
functions remain field data.  Embedded-flag homotopy also has
unstable information invisible to the pair: the Hopf classes in
$\pi_3(\mathcal F_2)=\pi_3(S^2)=\Z$ are distinct as flag
homotopy classes although their eigenline bundles are all
trivial and their flags are gauge-equivalent.

\begin{example}[The first higher transfer obstruction]
\label{ex:cyclic-three-transfer}
Let $X=(B\Z/3)^{(4)}$ be the four-skeleton of the standard
classifying CW complex and let $\rho:\Z/3\to S_3$ act by a
three-cycle.  Take the split normalizer bundle, so that its
twisted line class is zero.  If
$t\in H^2(X;\Z)=\Z/3$ is the generator and $\ell$ is the
associated character line with $c_1(\ell)=t$, then the
finite direct image of the trivial line on the three-sheeted
cover is the complex regular-representation bundle
\[
                 E=\underline{\C}\oplus\ell
                      \oplus\ell^{\otimes2}.
\]
The three-cycle has positive sign, and
\[
 c_1(E)=3t=0,
 \qquad c_2(E)=2t^2\ne0
                   \quad\text{in }H^4(X;\Z)=\Z/3.
\]
The inclusion of the four-skeleton in $B\Z/3$ preserves
$H^4(-;\Z)$, so $t^2$ remains nonzero on $X$.
Thus the determinant/Bockstein condition is satisfied, but
no map $q:X\to\mathcal D_3$ has this marked cover and line.
There is no corresponding higher obstruction on a
three-complex.  The example isolates a torsion $c_2$
contribution created solely by permutation pushforward.
Its marked permutation cover cannot be a root cover of a
simple polynomial, since Lemma~\ref{lem:vandermonde-normal-realization}
would give a trivial direct image of the trivial line.
\end{example}

\subsection{The topological envelope of linear Weyl-point charges}
\label{subsec:odd-hamiltonian-topology}

Frank, Pint\'er, Varjas, and P\'alyi ask which Chern-number
patterns can occur for linear multifold Hamiltonians
\cite[Remarks~4.4.6 and~4.5.10]{FrankPinterVarjasPalyi}.
Linearity implies an antipodal relation on the surrounding
sphere.  The normalizer criterion gives an exact classification
of the Chern patterns permitted by that relation alone.  Fix
the orientation of $S^2$ and order bands by increasing energy.

\begin{theorem}[Antipodal parity and complete topological realization]
\label{thm:odd-hamiltonian-chern-patterns}
Let $H:S^2\to\operatorname{Herm}_N(\C)$ be continuous, have
simple spectrum everywhere, and satisfy $H(-n)=-H(n)$.
Let $L_j$ be its ordered eigenlines and
$d_j=\langle c_1(L_j),[S^2]\rangle$.  Put
$m=\lfloor N/2\rfloor$.  Then
\begin{equation}\label{eq:odd-hamiltonian-chern-symmetry}
 d_{N+1-j}=-d_j,
 \qquad d_{m+1}=0\quad\text{if }N=2m+1.
\end{equation}
If $N=2m$, there is also the parity law
\begin{equation}\label{eq:odd-hamiltonian-even-parity}
 \sum_{j=1}^{m}d_j\equiv m\pmod2.
\end{equation}
Conversely, every integer vector satisfying
\eqref{eq:odd-hamiltonian-chern-symmetry} and, in even rank,
\eqref{eq:odd-hamiltonian-even-parity} is realized by an odd
continuous simple-spectrum $H$ with one fixed collection of
distinct symmetric eigenvalues
$a_1<\cdots<a_N$, $a_{N+1-j}=-a_j$.
Every odd continuous simple-spectrum field can be deformed
through such fields to one with these fixed eigenvalues while
keeping its ordered eigenprojections unchanged.
\end{theorem}

\begin{proof}
The unordered eigenflag is the same at $n$ and $-n$, so it
descends to a map $q:\R P^2\to\mathcal D_N$.  Its permutation
monodromy on the generator of $\pi_1(\R P^2)$ is the reversal
$w_0:j\mapsto N+1-j$.  Indeed, the increasing eigenvalue
order of $-H(n)$ reverses that of $H(n)$.  Pulling ordered lines
back to $S^2$ gives~\eqref{eq:odd-hamiltonian-chern-symmetry};
in odd rank the middle line descends to $\R P^2$ and hence
has zero first Chern number after pullback.

For $N=2m$, the permutation lattice $M_{w_0}$ is the sum of
$m$ swapped copies of $\Z^2$.  Cellular cochains for
$\R P^2$ give
\[
 H^2(\R P^2;M_{w_0})\cong\Z^m,
\]
where the $j$th integer is the Chern number $d_j$ on the
ordered double cover, up to the common orientation convention.
The sum map to $H^2(\R P^2;\Z)=\Z/2$ sends this vector to
$\sum_jd_j\bmod2$.  Since $w_0$ is a product of $m$
transpositions, its sign Bockstein is $m\bmod2$.  The joint
criterion~\eqref{eq:joint-permutation-chern} proves
\eqref{eq:odd-hamiltonian-even-parity}.

For $N=2m+1$, the middle fixed coordinate contributes an
additional $\Z/2$ summand to
$H^2(\R P^2;M_{w_0})$.  Its class $\varepsilon$ is the
first Chern class of the descended middle line.  The joint
condition becomes
$\sum_{j=1}^m d_j+\varepsilon+m=0\pmod2$;
for any chosen $d_1,\ldots,d_m$ it determines
$\varepsilon$ and imposes no further condition visible on
$S^2$.

Conversely, choose the twisted class with the specified pair
coordinates and, in odd rank, the required middle torsion
coordinate.  Theorem~\ref{thm:joint-spectral-layers}, in its
flag-map assertion, gives $q:\R P^2\to\mathcal D_N$.
The pullback to $S^2$ has an ordered lift
$(P_1,\ldots,P_N)$, marked so that
$P_j(-n)=P_{N+1-j}(n)$.  Set
\[
 H(n)=\sum_{j=1}^N a_jP_j(n).
\]
The spectrum is simple and constant, and the symmetry of the
$a_j$ gives $H(-n)=-H(n)$.  The prescribed Chern numbers are
those of its eigenlines.

Finally, for an arbitrary $H$ as in the theorem, let
$\lambda_1(n)<\cdots<\lambda_N(n)$ be its ordered continuous
eigenvalues.  The path
\[
 H_s(n)=\sum_{j=1}^N
       \bigl((1-s)\lambda_j(n)+s a_j\bigr)P_j(n),
       \qquad 0\le s\le1,
\]
has strictly ordered eigenvalues for every $s$ and keeps the
projections fixed.  The identities
$\lambda_j(-n)=-\lambda_{N+1-j}(n)$ and
$a_j=-a_{N+1-j}$ make each $H_s$ odd.
\end{proof}

The topological envelope is larger than the linear locus.  A
real-linear two-band pencil has the form
$H(k)=a(k)I_2+(Bk)\cdot\sigma$; simplicity on $S^2$ forces
$B\in\mathrm{GL}_3(\R)$.  Its Bloch map
$n\mapsto Bn/|Bn|$ has degree $\operatorname{sgn}\det B$,
so its ordered Chern vector is $(-1,1)$ or $(1,-1)$, with
both possibilities realized.  It is also complex-isolated in
the geometric-multiplicity sense: a two-by-two matrix has a
two-dimensional eigenspace only when scalar, and $Bz=0$
forces $z=0$.  Theorem~\ref{thm:odd-hamiltonian-chern-patterns}
realizes every odd charge among continuous odd fields.

\begin{proposition}[Embedded noncentral spectral spheres]
\label{prop:noncentral-spectral-embedding}
Let $A:\R^3\to\operatorname{Herm}_N(\C)$ be real-linear and
suppose $A(k)$ has simple spectrum for every $k\ne0$.
For every ordered band $j$ with $2j\ne N+1$, its eigenline map
\[
       f_j:S^2\longrightarrow\C P^{N-1}
\]
is a smooth embedding.  In particular the full ordered
spectral-flag map is an embedding.
\end{proposition}

\begin{proof}
If a line $\ell$ is an eigenline at independent directions
$k,k'$, linearity makes it an eigenline of every $A(v)$ on
$E=\operatorname{span}_{\R}\{k,k'\}$.  Its ordered index is
constant on the connected circle $E\cap S^2$, whereas
$A(-v)=-A(v)$ reverses that index.  A noncentral eigenline
therefore cannot recur at independent directions or antipodes,
so $f_j$ is injective.

If $df_j(k)(v)=0$ for nonzero $v\in T_kS^2$, differentiating
the eigenvector equation and projecting to its orthogonal
complement gives $A(v)u(k)\in\C u(k)$.  The same line would
then occur at the independent directions $k,v$.  Hence $f_j$
is an immersion and, by compactness, an embedding.  The full
flag contains $f_1$, so it is embedded too.
\end{proof}

For $N=3$, the image of the lowest band is an embedded sphere
in $\C P^2$ whose homology degree is the negative of its line
Chern number.  The genus bound of Kronheimer--Mrowka
\cite[Theorem~1]{KronheimerMrowka1994} therefore gives
$|d_1|\le2$.  The antipodal relation fixes the Chern vector
as $(d_1,0,-d_1)$.  Hu, Hou, Zhang, and Zhang had already
established this three-band range and constructed all five
values $d_1=0,\pm1,\pm2$
\cite[main text and Supplementary Material]{HuHouZhangZhang2018}.
The embedding argument gives a geometric explanation of the
bound.  Its noncentral hypothesis is sharp: in the gapped
three-band pencil $(k\cdot\sigma)\oplus0$, the middle eigenline
is a constant spectator.

\subsection{Global balance of four-dimensional Hopf defects}

Local spectral links on a closed base satisfy a global constraint.
The relation between Hopf defects and instanton number on $S^4$
was studied by Jahn and Bruckmann in Abelian gauges
\cite{Jahn,Bruckmann}; the following bundle formulation also
records the self-intersection of the eigenline.

\begin{theorem}[Twisted Hopf-charge balance and realization]
\label{thm:hopf-charge-balance}
Let $X$ be a connected closed oriented smooth four-manifold,
$E\to X$ the complex rank-two bundle of a principal $\SU(2)$
bundle, and $Z=\{p_1,\ldots,p_k\}$ a nonempty finite set.
Suppose a smooth unit traceless Hermitian endomorphism of
$E|_{X\setminus Z}$ has positive eigenline $L_Y$.  Its first
Chern class extends uniquely to $c\in H^2(X;\Z)$.  Let $h_i$
be the Hopf invariant of the positive eigenline on a small
linking $S_i^3$, oriented as the boundary of a ball about $p_i$,
with the standard Hopf fibration assigned $+1$.  Then
\begin{equation}\label{eq:hopf-charge-balance}
       \sum_{i=1}^{k}h_i
       =\langle c^2+c_2(E),[X]\rangle.
\end{equation}
Conversely, for any $c\in H^2(X;\Z)$ and any integers
$h_1,\ldots,h_k$ satisfying~\eqref{eq:hopf-charge-balance},
there is a smooth traceless Hermitian endomorphism of $E$
which vanishes to infinite order precisely at the prescribed
points, has simple spectrum elsewhere, and has those eigenline
class and local charges.
\end{theorem}

\begin{proof}
Removing finitely many points from a four-manifold leaves $H^2$
unchanged.  Extend $L_Y$, as an abstract line bundle, to
$L\to X$ with $c_1(L)=c$.  Its inclusion in $E$ over
$X\setminus Z$ is a nowhere-zero section of
$V=\Hom(L,E)=L^{-1}\otimes E$.  Extend this section over the
small balls, allowing isolated zeros.  In ball trivializations
its normalized boundary value is a map $S^3\to S^3$.  Its
degree is both the local zero index and the Hopf invariant of
its projectivization $S^3\to\C P^1$.  Changing a ball
trivialization has nullhomotopic boundary restriction, so the
integer is intrinsic.  The sum of zero indices is the Euler
number $\langle c_2(V),[X]\rangle$.  Since $c_1(E)=0$, the
splitting principle yields $c_2(V)=c_2(E)+c_1(L)^2$, proving
\eqref{eq:hopf-charge-balance}.

For realization, choose $L$ with $c_1(L)=c$.  An embedded tree
joining the prescribed points has a closed regular neighborhood
$B\cong B^4$ containing them.  The complement
$X\setminus\operatorname{int}B$ has the homotopy type of a
three-dimensional CW complex; the sphere bundle of the complex
rank-two bundle $V$ therefore has a section there.  In a
trivialization of $V|_B$, its normalized restriction to
$\partial B$, oriented as the boundary of $B$, has degree
$\langle c_2(V),[X]\rangle$.  Remove disjoint small balls
about the $p_i$ inside $B$.  A map from their complement to
$S^3$ with the given outer boundary value and prescribed
inner boundary degrees exists exactly when the inner degrees
sum to the outer degree: the only extension obstruction lies
in the relative fourth cohomology, and evaluates to that
difference.  Smooth this section and make its projectivization
on each punctured ball radially constant, equal to its linking
flag $n_i(\theta)$.  The resulting line defines a unit Hermitian
involution away from $Z$.  Multiplying it by a positive radial
function flat at each $p_i$ produces
the required smooth endomorphism and preserves its charges.
\end{proof}

For trivial $E$, the balance is $\sum_i h_i=\langle
c^2,[X]\rangle$; in particular, defects on $S^4$ have total
charge zero.  A pair of charges $+1,-1$ is realizable, though
each defect retains a local norm obstruction: with boundary
eigenvalues $\pm a$, an everywhere simple-spectrum field at
uniform distance less than $a$ from the given field would
extend its nonzero Hopf eigenline across the enclosed ball.
A defect-free reduction has $E\cong L\oplus L^{-1}$, hence
$c_2(E)=-c^2$; compare Appendix~\ref{sec:gauge-refinements}.

The preceding criteria describe which flag and charge data
can be assembled globally. Quantitative spectral gaps and
minimizing energy further restrict the limits that can form
at a collision. We turn to these conditions in succession.

\section{Finite-jet formation and energy selection}
\label{sec:selected-formation}

\subsection{Finite-jet extraction and quantitative flag blow-up}
\label{subsec:finite-jet-formation}

The preceding arbitrary-cluster constructions use flat gaps.
A polynomial lower gap makes the local flag geometry finite-jet
controlled.
For a Hermitian field $A$ with simple spectrum on
$B^d\setminus\{0\}$, write
\[
 \delta_A(x)=\min_{i<j}|\lambda_i(A(x))-\lambda_j(A(x))|,
 \qquad f_A(x)=(P_1^A(x),\ldots,P_N^A(x))\in\mathcal F_N,
\]
where the eigenvalues are in increasing order.  The graph of
$(\theta,f_A(r\theta))$ retains flags approached through curved
horns, beyond individual ray limits.  Set
\[
 G_r(A)=\{(\theta,f_A(r\theta)):\theta\in S^{d-1}\}
       \subset S^{d-1}\times\mathcal F_N.
\]

\begin{theorem}[Gap-controlled finite-jet rigidity]
\label{thm:finite-jet-rigidity}
Let $A:B^d\to\operatorname{Herm}_N$ be smooth, with
$A(0)=aI$, simple spectrum off the origin, and
\begin{equation}\label{eq:polynomial-gap}
             \delta_A(x)\ge c|x|^q
             \quad(0<|x|<r_0)
\end{equation}
for some $c,q>0$.  Then the following hold.

\begin{enumerate}
\item If $B$ is a continuous Hermitian field with
$\|B(x)-A(x)\|=o(|x|^q)$ uniformly as $x\to0$, then $B$ is
simple on a punctured neighborhood of zero and, for $x\ne0$,
\begin{equation}\label{eq:finite-jet-projector-estimate}
 \max_j\|P_j^B(x)-P_j^A(x)\|
       \le C\frac{\|B(x)-A(x)\|}{\delta_A(x)}=o(1).
\end{equation}
Their ordered flag clusters, their graph blow-up limits, and
their linking classes coincide.  If $A$ and $B$ have the same
pointwise eigenvalues, there is a continuous unitary gauge
$W$ near zero with $W(0)=I$ and $B=WAW^*$.

\item If $B$ is $C^M$, shares the $M$-jet of $A$ at zero, and
$M\ge q$ is an integer, the conclusions of part~(1) apply.
In particular, a sufficiently high finite jet determines the
local continuous gauge orbit within a fixed pointwise-isospectral
class.

\item The compact sets $G_r(A)$ converge in Hausdorff distance
to a semialgebraic set
$G_0(A)\subset S^{d-1}\times\mathcal F_N$.  For some
$C,\eta>0$,
\begin{equation}\label{eq:graph-blowup-rate}
                   d_H(G_r(A),G_0(A))\le Cr^\eta.
\end{equation}
The projection of $G_0(A)$ to $\mathcal F_N$ is the ordered
flag cluster $K(A)$.  Both $G_0(A)$ and $K(A)$ have real
semialgebraic dimension at most $d-1$; in particular $K(A)$
has finite Betti numbers.  A finite Taylor polynomial of $A$
has exactly the same $G_0(A)$ and $K(A)$.
\end{enumerate}
Every real-analytic $A$ with an isolated simple-spectrum
collision satisfies~\eqref{eq:polynomial-gap} for some $q$.
\end{theorem}

\begin{proof}
For $\|B-A\|<\delta_A/4$, the ordered eigenvalues remain
separated.  The standard spectral-projector perturbation
estimate~\cite{DavisKahan,Kato} gives
\eqref{eq:finite-jet-projector-estimate}.  Consequently the
two graph slices have Hausdorff distance tending to zero,
their cluster sets agree, and the pointwise short-geodesic
homotopy of their flags identifies the link classes.  Define
\[
 S(x)=\sum_jP_j^B(x)P_j^A(x),\qquad
 W(x)=S(x)(S(x)^*S(x))^{-1/2}.
\]
The estimate implies $S\to I$, so $W$ extends continuously
with $W(0)=I$.  The identities
$SP_j^A=P_j^B S$ and $[S^*S,P_j^A]=0$ show that
$WP_j^AW^*=P_j^B$.  Equal pointwise eigenvalues then give
$B=WAW^*$.  Taylor's theorem with Peano remainder proves
part~(2).

Choose an integer $M>q$ and let $T_MA$ be the degree-$M$
Taylor polynomial at zero.  Smoothness gives
$\|A-T_MA\|=O(|x|^{M+1})$.  Part~(1) applies with the
quantitative bound
\begin{equation}\label{eq:taylor-flag-rate}
 d_H(G_r(A),G_r(T_MA))=O(r^{M+1-q}).
\end{equation}
The graph of the ordered spectral-flag map of $T_MA$ is
semialgebraic: it is specified by ordered real eigenvalues,
orthogonal rank-one projections, and the equations
$(T_MA)P_j=\lambda_jP_j$.  Its closure over the oriented
real blow-up of $B^d$ is semialgebraic.  The fiber over $r=0$,
denoted $G_0$, is semialgebraic and has dimension at most
$d-1$ by the semialgebraic frontier-dimension theorem.
Semialgebraic curve selection and eventual monotonicity of
the radial coordinate show that $G_r(T_MA)\to G_0$ in
Hausdorff distance.  The Hausdorff distance is a
semialgebraic function of $r$, so its convergence has a
positive power rate by the semialgebraic \L ojasiewicz
inequality~\cite{BierstoneMilman}.  Equation
\eqref{eq:taylor-flag-rate} proves the same assertions for
$A$.  Projection gives the claims about $K(A)$.

Finally, for real-analytic $A$, the discriminant
$D_A=\prod_{i<j}(\lambda_i-\lambda_j)^2$ is a real-analytic
function of the matrix coefficients.  It vanishes only at
zero.  The analytic \L ojasiewicz inequality yields
$D_A(x)\ge c_0|x|^m$.  All gaps are bounded above on a
small closed ball, so $D_A\le C\delta_A^2$ and
\eqref{eq:polynomial-gap} follows with $q=m/2$.
\end{proof}

The comparison also applies to nonisolated analytic collision
sets with non-scalar matrices on the collision locus.

\begin{corollary}[Analytic collision-set jet determinacy]
\label{cor:analytic-collision-set-jet}
Let $Z$ be a compact collision set of a real-analytic
Hermitian field $A$ on a neighborhood of $Z$, and suppose
$A$ has simple spectrum at every nearby point outside $Z$.
Choose a relatively compact open neighborhood $U_0$ of
$Z$ whose closure lies in the analytic domain and whose
collisions are precisely $Z$.  There is an integer $M_0$,
depending on $A$ and $U_0$, with the following property.  If a
$C^{M}$ Hermitian field $B$, $M\ge M_0$, is pointwise
isospectral to $A$ and has the same derivatives through
order $M$ at every point of $Z$, then a continuous unitary
gauge $W$, equal to $I$ on $Z$, satisfies $B=WAW^*$ on a
neighborhood of $Z$.  Moreover, the ordered flags of $A$ and
$B$ approach one another uniformly as $x\to Z$, so their
flag-cluster fibers and linking classes agree.  The set $Z$
may be singular and its matrices may have non-scalar repeated
eigenvalues.
\end{corollary}

\begin{proof}
The analytic discriminant $D_A$ has zero set exactly $Z$
on $U_0$.  The compact \L ojasiewicz inequality, after
shrinking $U_0$ if necessary, gives
$D_A(x)\ge c\operatorname{dist}(x,Z)^m$ there.  Boundedness
of the other gaps gives
$\delta_A(x)\ge c'\operatorname{dist}(x,Z)^{m/2}$
off $Z$.  Take an integer $M_0\ge m/2$.  For
$F=B-A$, all derivatives through order $M$ vanish on $Z$.
Taylor's formula at a nearest point of $Z$, with a uniform
Peano remainder on the compact neighborhood, yields
$\|F(x)\|=o(\operatorname{dist}(x,Z)^M)$ uniformly as
$x\to Z$.  The projector estimate
\eqref{eq:finite-jet-projector-estimate} and the same polar
intertwiner $W$ therefore apply off $Z$ and satisfy
$W(x)\to I$ as $x\to Z$.  Define $W=I$ on $Z$.
The equal jets give $A=B$ on $Z$, so the isospectral gauge
identity holds there as well.  The uniform projector estimate
also identifies all approaching flag clusters, and the
near-identity gauge identifies their link classes.
\end{proof}

The jet threshold remains necessary along positive-dimensional
collision sets.  In $\mathbb R^4$ put
$u=x_1^2+x_2^2-1$, $w=x_3+\ii x_4$, and
$Z=\{u=w=0\}\cong S^1$.  For $k\ge2$, the polynomial fields
\[
 A_k^\pm=u\sigma_3+\operatorname{Re}(w^k)\sigma_1
                 \pm\operatorname{Im}(w^k)\sigma_2
\]
have identical pointwise spectra and identical derivatives
through order $k-1$ along $Z$.  Their gap has optimal order
$k$ in the distance to $Z$, while their eigenlines have
opposite Chern numbers of absolute value $k$ on every small
normal two-sphere.  Thus no continuous isospectral gauge
identifies them near $Z$.

Analytic eigenframes after suitable blow-ups were established
by Kurdyka and Paunescu~\cite{KurdykaPaunescu}.  Here the
graph $G_0(A)$ records horns directly in the original
parameter space.  For example,
\[
               A(x,y)=x\sigma_3+y^2\sigma_1
\]
has $\delta_A\ge c|(x,y)|^2$.  On nonvertical rays its flag
tends to $\pm\sigma_3$ and on vertical rays to $+\sigma_1$.
Above either vertical direction, the horns $x=ty^2$ fill the
semicircle between the two poles in $G_0(A)$, which ray limits
miss.

The threshold in~\eqref{eq:polynomial-gap} is sharp in the
jet order.  For $k\ge2$, put $w=y+\ii z$ and
\begin{equation}\label{eq:sharp-jet-charge-example}
 A_k^\pm(x,y,z)=x\sigma_3+
   \operatorname{Re}(w^k)\sigma_1
   \pm\operatorname{Im}(w^k)\sigma_2.
\end{equation}
These fields have the same pointwise spectrum and the same
$(k-1)$-jet, and their optimal gap order is $q=k$.
Both ordered clusters fill $\mathcal F_2\cong S^2$, yet their
linking Chern numbers are $\pm k$, so no continuous isospectral
gauge exists.  A common $k$-jet would force the gauge of
Theorem~\ref{thm:finite-jet-rigidity}; the flat-gap pair in
Proposition~\ref{prop:isospectral-jet-opposite} retains an
obstruction even with a common infinite jet.

\subsection{Stratified partial flags and weighted collision graphs}
\label{subsec:stratified-flag-blowup}

Near a singular collision set, the stable gap scale is the
distance to the entire set: small spheres may approach other
branches.  Partial flags retain separated eigenspaces through
collisions within a block.

Let $U\subset\R^d$ be open, let
$A:U\to\operatorname{Herm}_N$ be real analytic, and fix
a block partition $N=n_1+\cdots+n_\ell$, with cumulative
indices $m_a=n_1+\cdots+n_a$ for $a<\ell$.  Write
$\lambda_1\le\cdots\le\lambda_N$ for its ordered
eigenvalues, and set
\[
 g_A(x)=\min_{a<\ell}
       \bigl(\lambda_{m_a+1}(x)-\lambda_{m_a}(x)\bigr),
 \qquad Z=\{g_A=0\},\qquad O=U\setminus Z.
\]
On $O$ the ordered block projections
$F_A=(P_1^A,\ldots,P_\ell^A)$ form a map into the partial
flag manifold
$\mathcal F_{n_1,\ldots,n_\ell}
 :=\U(N)/(\U(n_1)\times\cdots\times\U(n_\ell))$.
Internal block collisions are allowed.

For $p\in Z\cap\overline O$ and sufficiently small $r>0$,
define the compact spherical flag graph
\begin{equation}\label{eq:stratified-spherical-graph}
 C_r(A,p)=\overline{\{(\theta,F_A(p+r\theta)):
        \theta\in S^{d-1},\ p+r\theta\in O\}}
 \subset S^{d-1}\times\mathcal F_{n_1,\ldots,n_\ell}.
\end{equation}
The closure keeps flags approached along a sphere toward
its intersections with $Z$.

\begin{theorem}[Weighted partial-gap rigidity and graph extraction]
\label{thm:stratified-partial-flag-extraction}
Let $K\subset Z$ be compact, and work in a sufficiently
small relatively compact neighborhood $V$ of $K$.
There are $c,q>0$ such that
\begin{equation}\label{eq:weighted-partial-gap}
       g_A(x)\ge c\operatorname{dist}(x,Z)^q
                      \qquad(x\in V\cap O).
\end{equation}
Suppose a continuous Hermitian field $B$ equals $A$ on
$Z\cap V$ and
$\|B-A\|=o(\operatorname{dist}(x,Z)^q)$
uniformly as $x\to Z$.  Then $B$ has the same block
separation locus near $K$, and
\[
       \max_a\|P_a^B(x)-P_a^A(x)\|=o(1)
                                \qquad(x\to Z, x\in O).
\]
A continuous unitary $W$, equal to $I$ on $Z$, carries
each block projection of $A$ to the corresponding one
of $B$ on $O$ near $K$.  The same conclusions follow if
$B-A$ is $C^M$ and all its derivatives through an integer
$M\ge q$ vanish along $Z$.

For every $p\in K\cap\overline O$, the graphs
$C_r(A,p)$ converge in Hausdorff distance to a compact
globally subanalytic set $C_0(A,p)$, with
\begin{equation}\label{eq:stratified-graph-rate}
 \dim C_0(A,p)\le d-1,\qquad
 d_H(C_r(A,p),C_0(A,p))\le C_pr^{\eta_p}
\end{equation}
for some $C_p,\eta_p>0$.  Under the preceding comparison,
$C_0(B,p)=C_0(A,p)$ and all partial-flag cluster fibers
over $K$ agree.
\end{theorem}

\begin{proof}
On a compact analytic chart the ordered eigenvalues and
$g_A$ are globally subanalytic.  The compact subanalytic
\L ojasiewicz inequality~\cite{BierstoneMilman} applied
to its zero set gives~\eqref{eq:weighted-partial-gap}.
Weyl's eigenvalue estimate preserves each block boundary
when $\|B-A\|<g_A/4$.  The spectral-projector estimate
\cite{DavisKahan,Kato}, applied to cumulative block
projections, gives
\[
 \max_a\|P_a^B-P_a^A\|
       \le C\frac{\|B-A\|}{g_A}=o(1).
\]
The polar construction in the proof of
Theorem~\ref{thm:finite-jet-rigidity}, with block projections
in place of lines, gives $W\to I$ on $Z$ and intertwines the
blocks.  Since $A=B$ on $Z$, their block-collision loci agree
near $K$.  Vanishing $M$-jets along $Z$ give the stated
$o(\operatorname{dist}(x,Z)^M)$ bound by uniform
Taylor--Peano remainder at a nearest point.

The graph of $F_A$ on $O$ is globally subanalytic: its
block projections are characterized by spectral order,
orthogonality, and the matrix equations for invariant
subspaces.  Its closure over the oriented spherical
blow-up at $p$ remains globally subanalytic.  Its fiber
at $r=0$ is $C_0(A,p)$; the frontier-dimension theorem
gives $\dim C_0\le d-1$.  Compactness gives the upper
Hausdorff limit.  Subanalytic curve selection and eventual
monotonicity of the radial coordinate give the lower limit,
so $d_H(C_r,C_0)\to0$.  This distance is itself a
subanalytic function of $r$, and the subanalytic
\L ojasiewicz inequality supplies the power rate in
\eqref{eq:stratified-graph-rate}.  The uniform projector
comparison makes the graphs of $A$ and $B$ asymptotically
Hausdorff-close, proving equality of their limits and
cluster fibers.
\end{proof}

On a compact patch of a clean stratum $S\subset Z$ whose
tube meets no other stratum, normal coordinates $x=z+n$
give $g_A(z+n)\ge c|n|^q$.  The same proof gives a
subanalytic joint normal graph in $(z,n/|n|,F_A)$, a power
Hausdorff rate, and finite-normal-jet determination; its
limit has dimension at most $d-1$.  Tangential drift matters:
\[
       A(t,s)=s^2\sigma_3+ts\sigma_1,\qquad Z=\{s=0\},
\]
has only the $+\sigma_3$ limit in the positive normal
slice at $t=0$, while the horns $t=cs$, $s\downarrow0$,
fill the upper Bloch semicircle in the joint graph.

The unitary in Theorem~\ref{thm:stratified-partial-flag-extraction}
intertwines blocks; their internal geometry can still obstruct
conjugacy.  With $u=(a,b,c)\in\R^3$,
$w=b+\ii c$, and $k\ge2$, let
\[
 H_k^\pm(u)=a\sigma_3+\operatorname{Re}(w^k)\sigma_1
                       \pm\operatorname{Im}(w^k)\sigma_2,
\quad
 A^\pm(t,u)=(-t^2)\oplus
       (t^2I_2+t^3H_k^\pm(u))\oplus5.
\]
For the block partition $(1,2,1)$ in a small box, the
interblock gap has order $t^2$ and $Z=\{t=0\}$.
The fields $A^\pm$ are polynomial and pointwise
isospectral; they have equal normal $2$-jets along $Z$
and identical constant partial flags.  Yet no continuous
unitary conjugates them on any neighborhood of $Z$:
on a small sphere in a fixed slice $t>0$, their internal
positive eigenlines have opposite Chern numbers of
absolute value $k$.  Thus the partial-flag graph and
the internal eigenline class occupy different levels of
the collision invariant.

\subsection{Variational selection inside the collision space}
\label{subsec:variational-selection}

A variational condition selects a smaller collision class.
For a two-band projector
$P=(I+n\cdot\sigma)/2$, $n:B^3\to S^2$, put
\begin{equation}\label{eq:projector-energy}
 E(P;U)=\int_U|\dd P|_{\mathrm{HS}}^2
          =\frac12\int_U|\dd n|^2.
\end{equation}
This ordered two-band flag energy is twice the
Fubini--Study projective energy used in
Section~\ref{sec:projective-surgery}; the unit hedgehog has
densities $4\pi$ and $2\pi$, respectively.
Schoen--Uhlenbeck established minimizing-map blow-up
compactness and partial regularity~\cite{SchoenUhlenbeck};
Brezis--Coron--Lieb classified nonconstant minimizing cones
into $S^2$ in dimension three~\cite{BrezisCoronLieb}.
Together they select a small class inside the cluster space
of Theorem~\ref{thm:cluster-shape-realization}.

\begin{theorem}[Selected two-band collision and its first jet]
\label{thm:variational-jet-selection}
Let $h:B^3\to\R^3$ be smooth, with $h^{-1}(0)=\{0\}$, and set
$H=h\cdot\sigma$ and $n=h/|h|$ off zero.  Suppose
$n\in W^{1,2}(B^3,S^2)$ is locally Dirichlet-minimizing
across zero among all $S^2$-valued competitors with the same
boundary trace, and that $n$ does not extend continuously
over zero.  Then every blow-up tangent at zero is a hedgehog
$n_0(x)=O(x/|x|)$ for some $O\in O(3)$.  The local eigenline
charge has absolute value one, the ordered flag cluster is
$\mathcal F_2\cong S^2$, and
\begin{equation}\label{eq:hedgehog-energy-density}
           \lim_{r\downarrow0}\frac{E(P;B_r)}r=4\pi.
\end{equation}

Suppose further that the first nonzero Taylor term of $h$ is
the homogeneous polynomial map $B_k$ of degree $k$.
Then there is a unique tangent rotation $O$ and a nonzero
homogeneous polynomial $q_{k-1}$ such that
\begin{equation}\label{eq:minimizing-jet-factor}
          B_k(x)=q_{k-1}(x)Ox,
          \qquad q_{k-1}(x)\ge0\quad(x\in\R^3).
\end{equation}
In particular $k$ is odd.  Conversely, every nonzero
nonnegative homogeneous polynomial $q$ of even degree occurs
as the factor of the first nonzero jet of such a polynomial
field.
\end{theorem}

\begin{proof}
Minimizing-map compactness gives a homogeneous minimizing
tangent along every sequence of radii after passage to a
subsequence; on compact annuli away from its singular point,
the convergence of minimizing maps is smooth by the
regularity and $\varepsilon$-regularity theory
of~\cite{SchoenUhlenbeck,BrezisCoronLieb}.  The same
$\varepsilon$-regularity theorem excludes a constant
tangent at a nonremovable singularity.  The classification of
Brezis--Coron--Lieb then gives $O(x/|x|)$, with $O\in O(3)$.
Its sphere map has degree $\det O=\pm1$, and its energy in
$B_r$ under~\eqref{eq:projector-energy} is $4\pi r$.
Monotonicity yields~\eqref{eq:hedgehog-energy-density}.
Since a hedgehog tangent maps the linking sphere onto $S^2$,
the flag cluster contains all of $\mathcal F_2$.

Put $U=\{\theta\in S^2:B_k(\theta)\ne0\}$, a nonempty open
set.  Taylor's theorem gives
$n(r\theta)\to B_k(\theta)/|B_k(\theta)|$ locally uniformly
on $U$.  Comparison with a minimizing tangent shows that
$B_k(\theta)$ is a positive multiple of $O\theta$ there.
Thus the polynomial identities
\[
 x_i(O^{-1}B_k)_j-x_j(O^{-1}B_k)_i=0
                    \qquad(1\le i,j\le3)
\]
hold on an open cone and hence everywhere.  Coprimeness of
the coordinate polynomials $x_i$ shows that
$O^{-1}B_k=q_{k-1}(x)x$ for a homogeneous polynomial
$q_{k-1}$.  Positivity on $U$ and continuity force
$q_{k-1}\ge0$ everywhere.  A nonzero nonnegative homogeneous
polynomial has even degree, so $k-1$ is even.  Any other
hedgehog tangent agrees with $B_k/|B_k|$ on $U$ and therefore
has the same orthogonal map $O$.  Compactness now gives the
unique tangent without selecting a subsequence.

For the converse, choose $L$ with $2L>\deg q$ and set
\[
                  h(x)=\bigl(q(x)+|x|^{2L}\bigr)Ox.
\]
It has the stated first jet and no zero off the origin, while
$h/|h|=O(x/|x|)$.  The hedgehog is locally minimizing by
the theorem of Brezis--Coron--Lieb.
\end{proof}

The finite jet makes the tangent hedgehog unique.  Without it,
all minimizing tangents still lie in the orthogonal hedgehog
family.  Mere stationarity or finite projector energy is weaker:
a radial cone over a degree-$m$ sphere map, $|m|>1$, has
finite energy and yields a smooth Hermitian field after
multiplication by a flat gap, but is not minimizing
\cite{BrezisCoronLieb}.

\subsection{Selected collisions and quantitative local models}
\label{subsec:selected-collisions}

Fix a partition $\mathbf n=(n_1,\ldots,n_\ell)$ of $N$, $\ell\ge2$,
and equip
$\mathcal F_{\mathbf n}$ with the invariant metric induced by its
ordered orthogonal projections.  For $F=(P_1,\ldots,P_\ell)$ put
\begin{equation}\label{eq:selected-flag-energy}
 \mathscr E(F;V)=\frac12\sum_{a=1}^{\ell}
                    \int_V|\dd P_a|_{\mathrm{HS}}^2.
\end{equation}
For the complete two-band flag this agrees with
\eqref{eq:projector-energy}.  Variational selection means that
the flag locally minimizes this energy across the collision
among all $W^{1,2}$ flag-valued maps with the same boundary
trace.  The selected collision class studied here is its
intersection with the polynomial-gap class.

\begin{theorem}[Finite-jet selection and quantitative flag model]
\label{thm:selected-collision-normal-form}
Let $A:B^3\to\operatorname{Herm}_N$ be smooth, with $A(0)=aI$.
Suppose its ordered spectral blocks of ranks $\mathbf n$ are
separated on $B^3\setminus\{0\}$, and their minimum boundary gap
satisfies
\begin{equation}\label{eq:selected-gap}
             g_A(x)\ge c|x|^q\qquad(0<|x|<r_0)
\end{equation}
for some $c,q>0$.  Let $F_A=(P_1^A,\ldots,P_\ell^A)$ be the
associated partial flag.  Assume $F_A\in W^{1,2}(B^3,\mathcal
F_{\mathbf n})$ is locally minimizing for~\eqref{eq:selected-flag-energy}
across zero and has no continuous extension there.  Then there is a
unique smooth, nonconstant, semialgebraic map
$\varphi:S^2\to\mathcal F_{\mathbf n}$ with these properties.

\begin{enumerate}
\item The cone $F_0(x)=\varphi(x/|x|)$ is degree-zero homogeneous
and locally minimizing.  The entire rescaled family converges to
it.  For every compact annulus $K\Subset\R^3\setminus\{0\}$ and
every integer $j\ge0$, some $C_{K,j},\alpha_{K,j}>0$ satisfy
\begin{equation}\label{eq:selected-power-rate}
 \|F_A(r\,\cdot)-F_0\|_{C^j(K)}
                  \le C_{K,j}r^{\alpha_{K,j}}.
\end{equation}

\item The directional graph limit, taken in Hausdorff distance,
is single-valued:
\begin{equation}\label{eq:selected-graph}
 \lim_{r\downarrow0}
 \{(\theta,F_A(r\theta)):\theta\in S^2\}
          =\operatorname{graph}\varphi.
\end{equation}
It is determined, as a semialgebraic set, by every sufficiently
high finite Taylor polynomial of $A$ at zero.

\item The density obeys
\begin{equation}\label{eq:selected-density}
 \Theta_A(0):=\lim_{r\downarrow0}r^{-1}
                  \mathscr E(F_A;B_r)
       =\mathscr E(F_0;B_1)
       \ge\varepsilon_{\mathbf n}>0.
\end{equation}
For some $C,\beta>0$,
$|r^{-1}\mathscr E(F_A;B_r)-\Theta_A(0)|\le Cr^\beta$.

\item There is a continuous unitary $W:B_{r_1}\to\U(N)$,
$W(0)=I$, satisfying
\begin{equation}\label{eq:selected-partial-gauge}
 P_a^A(x)=W(x)\varphi_a(x/|x|)W(x)^*
                   \qquad(x\ne0,\ 1\le a\le\ell).
\end{equation}
On every fixed annulus, the rescaled $W$ approaches $I$ with a
positive power rate in each $C^j$ norm.  If every block has rank
one, with ordered eigenvalues $\lambda_a(x)$, then also
\begin{equation}\label{eq:selected-full-gauge}
 A(x)=W(x)\left(\sum_{a=1}^N
     \lambda_a(x)\varphi_a(x/|x|)\right)W(x)^*.
\end{equation}
\end{enumerate}
The constants and exponents depend on the field and, in
\eqref{eq:selected-power-rate}, on $K$ and $j$.
In particular, the density and the first Chern numbers of all
linking block bundles are determined by the same finite jet.
\end{theorem}

\begin{proof}
Minimizing-map blow-up compactness and monotonicity
\cite{SchoenUhlenbeck} give, from every sequence $r_i\downarrow0$,
a subsequence converging strongly in $W^{1,2}_{\mathrm{loc}}$ to a
degree-zero homogeneous minimizing tangent.  In three dimensions
such a tangent is smooth off the origin: any nonzero singular point
would generate a singular ray, contrary to the zero-dimensional
singular-set bound.  Epsilon regularity upgrades the convergence
on compact annuli to $C^\infty$ convergence.  It excludes a
constant tangent at a nonremovable point and supplies the positive
lower bound in~\eqref{eq:selected-density}.

Choose an integer $M>q$ and let $T_MA$ be the degree-$M$ Taylor
polynomial of $A$ at zero.  Weyl's eigenvalue inequality and
Taylor's theorem preserve the ordered block separation of $T_MA$
on a smaller punctured ball.  The block-projector estimate gives,
uniformly for $\theta\in S^2$,
\[
 \|F_A(r\theta)-F_{T_MA}(r\theta)\|=O(r^{M+1-q}).
\]
The flag map of a polynomial Hermitian field is semialgebraic.
For each $\theta$, the bounded semialgebraic curve
$r\mapsto F_{T_MA}(r\theta)$ has a limit as $r\downarrow0$.
Hence $F_A(r\theta)$ has the same pointwise limit
$\varphi(\theta)$; its graph is semialgebraic by quantifier
elimination.  Every smooth minimizing tangent agrees with these
limits on the unit sphere.  All tangents therefore coincide,
$\varphi$ is smooth and nonconstant, and compactness gives
$C^\infty$ convergence of the whole rescaled family on annuli.
In particular,~\eqref{eq:selected-graph} holds.  Replacing $M$
by any larger integer leaves its limit unchanged.

The scalar function
\[
 d_M(r)=\max_{\theta\in S^2}
       \|F_{T_MA}(r\theta)-\varphi(\theta)\|
\]
is semialgebraic.  It tends to zero because the full family
$F_A(r\,\cdot)$ converges uniformly and its distance from
$F_{T_MA}(r\,\cdot)$ is $O(r^{M+1-q})$.
The semialgebraic growth dichotomy therefore gives
$d_M(r)\le Cr^\eta$ for some $\eta>0$.  The same projector
comparison gives a $C^0$ power rate for $F_A(r\,\cdot)$ on
every fixed annulus.  On a slightly larger annulus, smooth
compactness of the entire minimizing blow-up family supplies
uniform $C^m$ bounds for each fixed $m$.  Interpolation between
these bounds and the $C^0$ rate proves
\eqref{eq:selected-power-rate}, with a possibly smaller positive
exponent for each $j$.

Strong convergence of blow-ups gives the density identity in
\eqref{eq:selected-density}.  The $C^1$ estimate on
$\{1/2\le|x|\le1\}$ bounds the difference of the energies of
$F_A$ and $F_0$ on a shell of radius $s$ by $Cs^{1+\alpha}$.
Summing over dyadic shells inside $B_r$ gives the density rate.

Finally put $Q_a(x)=\varphi_a(x/|x|)$ and
\[
 S(x)=\sum_aP_a^A(x)Q_a(x),\qquad
 W(x)=S(x)(S(x)^*S(x))^{-1/2}.
\]
The $C^0$ estimate makes $S$ invertible near zero and gives
$W\to I$.  The identities $SQ_a=P_a^AS$ and
$[S^*S,Q_a]=0$ prove~\eqref{eq:selected-partial-gauge}.
Smoothness of the polar map near $I$ transfers the annular
power rates to $W$.  For rank-one blocks the spectral
decomposition of $A$ gives~\eqref{eq:selected-full-gauge}.
The gauge identifies the block bundles on a linking sphere
with those of $\varphi$; the density is likewise the energy
of $\varphi$.  Both are therefore determined by its finite-jet
graph.
\end{proof}

Earlier tangent uniqueness for analytic targets is due to
Simon~\cite{Simon1983}; Gulliver--White proved power convergence
in dimension three for two-dimensional targets~\cite{GulliverWhite}.
Here the spectral gap gives finite-jet determination and
partial-flag power rates.

\begin{corollary}[Two-band selected collision normal form]
\label{cor:selected-two-band-normal-form}
Let $H(x)=a(x)I+h(x)\cdot\sigma$ on $B^3$, where
$a:B^3\to\R$ and $h:B^3\to\R^3$ are smooth,
$h^{-1}(0)=\{0\}$, and $|h(x)|\ge c|x|^q$ near zero.  If
$n=h/|h|$ is locally Dirichlet-minimizing across zero and does
not extend continuously there, then there are a unique
$O\in O(3)$ and a continuous $W:B_{r_1}\to\U(2)$,
$W(0)=I$, such that
\begin{equation}\label{eq:selected-two-band-exact-form}
 H(x)=a(x)I+|h(x)|\,W(x)
       \bigl(O(x/|x|)\cdot\sigma\bigr)W(x)^*.
\end{equation}
On fixed annuli, the rescaled $n$ and $W$ approach
$O(x/|x|)$ and $I$ at positive power rates in every $C^j$
norm.  The Bloch-map degree is $\det O=\pm1$; in the
convention of this paper the positive eigenline has Chern
number $-\det O$.  The graph limit is
$\operatorname{graph}(O|_{S^2})$, and the normalized
projector-energy density is $4\pi$.  The scalar factor
$|h|$ and $W$ are continuous at the collision; no smooth
extension is asserted for them.
\end{corollary}

\begin{proof}
Apply Theorem~\ref{thm:selected-collision-normal-form} to the
complete two-band flag.  The cone classification of
Brezis--Coron--Lieb~\cite{BrezisCoronLieb} gives
$\varphi(\theta)=O\theta$.  The remaining claims follow from
\eqref{eq:projector-energy} and the theorem.
\end{proof}

\begin{corollary}[Finite-jet exclusion and selected gauge orbits]
\label{cor:selected-exclusion-orbits}
Under the gap assumption of
Theorem~\ref{thm:selected-collision-normal-form}, compute the
oriented partial-flag graph from any Taylor polynomial of degree
$M>q$.  If its limit has a nonsingleton fiber above some
$\theta\in S^2$, the flag cannot be locally Dirichlet-minimizing
across a nonremovable collision.  The same exclusion holds if the
limit is a smooth graph whose cone is not harmonic.  These are
necessary tests.

For complete flags, let $A$ and $B$ both satisfy the theorem and
have the same pointwise eigenvalues.  They are related near zero
by a continuous unitary gauge if and only if a single constant
$U_0\in\U(N)$ satisfies
\begin{equation}\label{eq:selected-orbit-criterion}
       \varphi_a^B(\theta)
          =U_0\varphi_a^A(\theta)U_0^*
       \quad\text{for all $a$ and $\theta\in S^2$}.
\end{equation}
For two bands, this condition is precisely equality of the
two link-charge signs.  Thus a fixed pointwise spectrum that
admits a selected two-band germ has exactly two continuous
gauge classes of selected germs, distinguished by chirality.
\end{corollary}

\begin{proof}
The first assertions follow from the single-valued smooth
minimizing graph in Theorem~\ref{thm:selected-collision-normal-form}.
If $B=VAV^*$ with $V$ continuous at zero, take the limit along
$r\theta$ to obtain~\eqref{eq:selected-orbit-criterion} with
$U_0=V(0)$.  Conversely, apply~\eqref{eq:selected-full-gauge}
to $A$ and $B$; their near-identity gauges and $U_0$ compose to
give a continuous gauge between the fields.  In two bands the
constant $\U(2)$ action on $\C P^1\cong S^2$ is the standard
$\operatorname{SO}(3)$ action.  Two orthogonal hedgehogs differ by such an
action exactly when their determinants agree.  Complex
conjugation of a selected field reverses that determinant while
preserving its spectrum and minimizing property, so both classes
occur.
\end{proof}

The exclusion test detects a unit-charge pathology missed by
linking degree.  For
\[
                    A(x,y,z)=(x,y,z^3)\cdot\sigma,
\]
the collision is isolated, the gap has a polynomial lower
bound, and the link degree is $+1$.  Its flag lies in $W^{1,2}$:
the squared gradient is bounded by a constant times
$(1+z^4)/(x^2+y^2+z^6)$, whose transverse integral grows only
logarithmically as $z\to0$.  On the sphere of radius $r$, the
points with $x=ar^3$, $y=0$, and
$z=r\sqrt{1-a^2r^4}$ all approach the north direction, while
their flags tend to $(a,0,1)/\sqrt{a^2+1}$.  The directional
graph therefore has a nonsingleton fiber and the projector
cannot minimize~\eqref{eq:selected-flag-energy}.  A single-valued
hedgehog graph alone does not imply minimization: harmonicity
must still be checked.

For a complete rank-one flag, the selected energy is the
gap-normalized spectral quantity in
\eqref{eq:incidence-gap-energy}. Ordinary matrix
Dirichlet energy weights the same rotations by the squared
gaps and can lose them at a collision.  For example,
$A(x)=x\cdot\sigma$ has projector-energy density $4\pi$,
while $r^{-1}\int_{B_r}|\dd A|_{\mathrm{HS}}^2\to0$.

The two hypotheses of Theorem~\ref{thm:selected-collision-normal-form}
are independent.  The smooth field
$h(x)=e^{-1/|x|^2}x$, extended by zero, has an exactly
minimizing hedgehog flag and a gap flat to every order.
Conversely, the polynomial higher-charge fields in
\eqref{eq:sharp-jet-charge-example} have a polynomial gap
and cannot minimize the two-band flag energy.

\subsection{The primitive multiband density}
\label{subsec:primitive-flag-cones}

The variational condition has a sharp first energy level for every
ordered partial flag.  A \emph{pair-root sphere} is obtained by
choosing two distinct blocks $a,b$, one complex line in each, and
varying these two lines as orthogonal complements in their fixed
two-dimensional sum.  All remaining directions stay fixed.  With
the metric in~\eqref{eq:selected-flag-energy}, this sphere is a
unit round $S^2$.  A \emph{pair-root hedgehog} is its composition
with $x\mapsto O(x/|x|)$ for $O\in O(3)$.

\begin{theorem}[Sharp primitive density for partial flags]
\label{thm:primitive-flag-density}
Let $\mathbf n=(n_1,\ldots,n_\ell)$, $\ell\ge2$, and let
$F:B^3\to\mathcal F_{\mathbf n}$ be a nonconstant degree-zero
homogeneous locally minimizing map for~\eqref{eq:selected-flag-energy}.
Then
\begin{equation}\label{eq:primitive-density}
       \Theta(F):=\mathscr E(F;B_1)\ge4\pi.
\end{equation}
Equality holds precisely for the pair-root hedgehogs, up to a
constant $\U(N)$ action.  Every pair of distinct ordered blocks
occurs.  In particular, when $\ell\ge3$, a nonadjacent pair is a
distinct fixed-label gauge orbit from an adjacent pair.
The set of densities of these cones is locally finite in
$(0,\infty)$.  In particular, there is a constant
$\delta_{\mathbf n}>0$ such that every density belongs to
\begin{equation}\label{eq:primitive-next-gap}
       \{4\pi\}\cup[4\pi+\delta_{\mathbf n},\infty).
\end{equation}

For any smooth flag link $\psi:S^2\to\mathcal F_{\mathbf n}$,
put $c_a=\langle c_1(\operatorname{im}P_a|_{S^2}),[S^2]\rangle$.
Its radial extension, whether minimizing or not, satisfies
\begin{equation}\label{eq:flag-charge-density}
       \mathscr E(\psi(x/|x|);B_1)
                    \ge2\pi\sum_{a=1}^{\ell}|c_a|.
\end{equation}
Thus the primitive models have charge vectors
$\pm(e_a-e_b)$, and a link with $\sum_a|c_a|\ge4$ has
energy at least $8\pi$.
\end{theorem}

\begin{proof}
Write $\mathcal F_{\mathbf n}=\U(N)/H$, where
$H=\prod_a\U(n_a)$, and let $\mathfrak u(N)=\mathfrak h\oplus
\mathfrak m$ be the orthogonal decomposition for
$Q(X,Y)=-2\operatorname{tr}(XY)$.  At the base flag,
\[g_{\mathcal F}(X,X)=\sum_a\|[X,P_a]\|_{\mathrm{HS}}^2
                    =Q(X,X)\quad(X\in\mathfrak m).\]
The two off-block directions supported on one line in each of
two blocks form a Lie triple system.  Their $\U(2)$ orbit is
therefore a totally geodesic unit round pair-root sphere.
For $Q$-orthonormal $X,Y\in\mathfrak m$, the normal-homogeneous
curvature formula is
\[
 K(X,Y)=\|[X,Y]_{\mathfrak h}\|_Q^2
                   +\tfrac14\|[X,Y]_{\mathfrak m}\|_Q^2.
\]
Diagonalizing $X=\ii\operatorname{diag}(\lambda_j)$ gives
$2\sum_j\lambda_j^2=1$ and
\[
 \|[X,Y]\|_Q^2
 \le\max_{i,j}(\lambda_i-\lambda_j)^2\|Y\|_Q^2
 \le1.
\]
Consequently $K\le1$.  Equality requires
$[X,Y]_{\mathfrak m}=0$ and equality in the last two
inequalities.  Then $X,Y,[X,Y]$ form an $\mathfrak{su}(2)$
triple supported on one complex two-plane.  The two simple
eigenlines of the nonzero block-diagonal matrix $[X,Y]$ lie
in different $H$-blocks: if they lay in the same block, $X$
would be block diagonal.  Thus a $K=1$ tangent plane is
exactly a pair-root plane.

Write $F(x)=\varphi(x/|x|)$.  Its link is a smooth nonconstant
harmonic map $S^2\to\mathcal F_{\mathbf n}$: the
three-dimensional singular-set bound for minimizers and
homogeneity rule out a singular ray away from zero.  The Hopf
differential vanishes on $S^2$, so $\varphi$ is a branched
conformal minimal immersion.  At regular points its induced
Gaussian curvature satisfies
$K_{\mathrm{ind}}=K_{\mathcal F}(T\varphi)-|B|^2/2\le1$.
If the total branch order is $b$, branched Gauss--Bonnet gives
\[
 4\pi+2\pi b=\int_{S^2}K_{\mathrm{ind}}\,\dd A
       \le\operatorname{Area}(\varphi)
       =\mathscr E(F;B_1).
\]
At equality, $b=0$, $B=0$, and every tangent plane has
curvature one.  The preceding equality case places the image
in a totally geodesic pair-root sphere; its area is $4\pi$,
so the link is a degree-$\pm1$ conformal parametrization of that
sphere.  Local minimality of the cone in the full flag implies
minimality against competitors inside the root sphere.
Brezis--Coron--Lieb's classification of minimizing radial
$S^2$ maps~\cite{BrezisCoronLieb} then restricts the
parametrization to an orthogonal hedgehog.

Conversely, let $F_*$ be a pair-root hedgehog involving
blocks $a,b$.  For any flag competitor $F$ with the same
boundary trace,
\begin{equation}\label{eq:primitive-plucker-bound}
\mathscr E(F;B_1)\ge
 2\,\mathscr E(P_a;B_1),\qquad
 \mathscr E(P_a;B_1):=\tfrac12\int_{B_1}|\dd P_a|^2_{\rm HS}.
\end{equation}
Indeed, if $X_{bc}$ denotes an infinitesimal off-block motion,
the respective pointwise densities are
$2\sum_{b<c}|X_{bc}|^2$ and
$\sum_{b\ne a}|X_{ab}|^2$.
The Pl\"ucker line of $\operatorname{im}P_a$ has precisely
the same projector metric as $P_a$.  Its boundary trace for
$F_*$ lies on a totally geodesic projective line, and its
radial extension minimizes energy among all projective-space
maps by Fujioka's theorem~\cite[Theorem~2]{Fujioka1993}.
The constant change from Fujioka's Fubini--Study normalization
to the projector normalization does not change minimizers.
Hence $\mathscr E(P_a;B_1)\ge2\pi$, while $F_*$ has energy
$4\pi$.  This proves local minimality and sharpness of
\eqref{eq:primitive-density}.

Finally, each block projection has a Grassmannian K\"ahler
calibration, so
$\mathscr E(P_a;B_1)\ge2\pi|c_a|$ for a homogeneous
flag link.  Summing over $a$ proves~\eqref{eq:flag-charge-density}.

For local finiteness, suppose locally minimizing homogeneous
maps $F_j$ have distinct densities $\Theta_j$ converging to
a finite number.  Strong local
$W^{1,2}$ compactness for minimizing maps
\cite{SchoenUhlenbeck,Luckhaus1988} gives a homogeneous
minimizing limit $F_\infty$ with the limiting density.  On a fixed
annulus the links of $F_j$ converge strongly in $W^{1,2}(S^2)$
to its smooth link; each link is harmonic.  The target flag is
closed and real analytic, so Feehan's relative harmonic-map
energy-gap theorem~\cite{Feehan2018} makes their link
energies equal to that of $F_\infty$ for all large $j$, a
contradiction.  Local finiteness and the equality classification
give~\eqref{eq:primitive-next-gap}.
\end{proof}

The qualifier ``pair-root'' is essential for an ordered
spectrum.  In a three-band basis, the polynomial field
\[
 A(x,y,z)=
 \begin{pmatrix}
 z&0&x-\ii y\\ 0&0&0\\ x+\ii y&0&-z
 \end{pmatrix}
\]
has eigenvalues $-|x|,0,|x|$ and a minimizing full-flag
cone supported on blocks $1$ and $3$.  Its Chern vector is
$\pm(e_1-e_3)$, which a constant unitary conjugation with
fixed eigenvalue labels cannot turn into an adjacent-block
vector.  Relabeling by the Weyl group identifies these
primitive models when labels themselves may be permuted.

Fujioka proved minimality of projective-line cones and posed
the classification of all minimizing tangent maps into
complex projective spaces~\cite{Fujioka1993}.  Theorem~
\ref{thm:primitive-flag-density} classifies the \emph{lowest}
nonzero density for every ordered partial flag; it leaves the
higher-density minimizing-cone classification open.

\begin{corollary}[Primitive selected collisions]
\label{cor:primitive-selected-collisions}
Under the hypotheses of Theorem~\ref{thm:selected-collision-normal-form},
$\Theta_A(0)\ge4\pi$.  If $\Theta_A(0)=4\pi$, its finite-jet
tangent is a pair-root hedgehog.  The exact continuous gauge
in~\eqref{eq:selected-partial-gauge} then identifies the
spectral partial flag with that primitive model on a punctured
ball, with a near-identity limit at the collision.  The same
conclusion follows from the weaker numerical test
$\Theta_A(0)<4\pi+\delta_{\mathbf n}$.
\end{corollary}

\section{Projective cone selection by surgery and averaged filling}
\label{sec:projective-surgery}

Theorem~\ref{thm:projective-minimizing-cones} classifies minimizing
radial extensions into every $\C P^N$: a nonconstant smooth link
must be a round holomorphic or antiholomorphic projective line.
The proof uses two-centre surgery for degree-two holomorphic
links and an averaged filling bound for nonholomorphic harmonic
links. Fujioka's results cover the remaining holomorphic degrees.

\subsection{Calibrated charge displacement and singularity splitting}
\label{subsec:calibrated-charge-surgery}

The polar K\"ahler calibration expresses energy change through
displaced charges and analytic excess.

Give $\C P^m$ the metric
$g_{\rm FS}=\frac12\tr(\dd P)^2$ and K\"ahler form
$\omega_{\rm FS}$, normalized by
$\int_{\C P^1}\omega_{\rm FS}=\pi$.  For a map
$F:B^3\to\C P^m$ write
$E(F)=\int_{B^3}|\dd F|_{g_{\rm FS}}^2$.
If $F_r(n)=F(rn)$ is smooth, let
\[
 d_F(r)=\frac1\pi\int_{S^2}F_r^*\omega_{\rm FS}\in\Z.
\]
For a positively oriented local orthonormal frame $(e_1,e_2)$
on the round sphere, define the \emph{calibrated excess}
\begin{equation}\label{eq:calibrated-excess}
 \mathcal X(F)=\int_0^1\!\int_{S^2}
 \left(r^2|\partial_rF|^2+
       |\dd F_r(e_2)-J\dd F_r(e_1)|^2\right)
       \dd\Omega\,\dd r.
\end{equation}
The integrand is independent of the oriented choice of frame.

\begin{theorem}[Radial charge and excess identity]
\label{thm:radial-charge-excess}
Let $g:S^2\to\C P^m$ be smooth of degree $d$.
Let $F\in W^{1,2}(B^3,\C P^m)$ have trace $g$, be smooth
up to the boundary away from a finite set of isolated interior
points $S$, and let
\[
 q_p=\frac1\pi\int_{\partial B_\varepsilon(p)}
                    F^*\omega_{\rm FS}\in\Z
 \qquad(p\in S)
\]
be its local charges, with each small sphere oriented out of
$B_\varepsilon(p)$.  For the radial cone
$u_g(x)=g(x/|x|)$ one has
\begin{equation}\label{eq:radial-charge-ledger}
 E(F)-E(u_g)
 =\mathcal X(F)-\mathcal X(u_g)
          -2\pi\sum_{p\in S\setminus\{0\}}
                         q_p|p|.
\end{equation}
If $g$ is holomorphic of degree $d>0$, then
$\mathcal X(u_g)=0$ and $E(u_g)=2\pi d$.
In this case a strict energy decrease requires positive radial
charge displacement and enough displacement to exceed the
calibrated excess.  Among maps continuous on the punctured closed
ball, with the stated smoothness and trace, $u_g$ is the unique
minimizer.
\end{theorem}

\begin{proof}
For almost every $r$ whose sphere avoids $S$, polar coordinates
give
\[
 E(F)=\int_0^1\!\int_{S^2}
        \bigl(r^2|\partial_rF|^2+|\dd F_r(e_1)|^2
                                  +|\dd F_r(e_2)|^2\bigr)
                                  \dd\Omega\,\dd r.
\]
The pointwise Wirtinger identity, with
$\omega_{\rm FS}(v,w)=g_{\rm FS}(Jv,w)$, is
\[
 |v|^2+|w|^2
       =|w-Jv|^2+2\omega_{\rm FS}(v,w).
\]
Integrating it on the sphere yields
\begin{equation}\label{eq:radial-charge-profile}
 E(F)=\mathcal X(F)+2\pi\int_0^1d_F(r)\,\dd r.
\end{equation}
The closedness of $\omega_{\rm FS}$ and Stokes' theorem show
that $d_F(r)$ jumps by $q_p$ as a concentric sphere passes a
singularity $p$.  Since the outer degree is $d$,
\[
 d-d_F(r)=\sum_{p\in S:\,|p|>r}q_p
 \quad\text{for almost every }r\in(0,1).
\]
Integrating in $r$ gives
$\int_0^1(d-d_F(r))\,\dd r
 =\sum_{p\ne0}q_p|p|$.
For a general smooth link the same calculation applied to $u_g$
gives $E(u_g)=\mathcal X(u_g)+2\pi d$.  Subtracting proves
\eqref{eq:radial-charge-ledger}.

If $g$ is holomorphic, then $\mathcal X(u_g)=0$.  If also $F$
is continuous on the punctured closed ball, it has no
off-center charge and $d_F(r)=d$ throughout.  Thus
$E(F)-E(u_g)=\mathcal X(F)\ge0$.  Equality forces
$\partial_rF=0$ almost everywhere; its trace then makes
$F=u_g$ almost everywhere.
\end{proof}

The same bookkeeping applies to the ordered partial flags of
\eqref{eq:selected-flag-energy}.  Their normal homogeneous metric
need not be K\"ahler; its Grassmannian components still carry
independent calibrations.  Let $\omega_a$ and $J_a$ be
the K\"ahler form and complex structure on
$\operatorname{Gr}(n_a,N)$ for the projector metric
$g_a=\frac12\tr(\dd P_a)^2$.  Our sign convention is
$[\omega_a]/\pi=-c_1(\operatorname{im}P_a)$.

\begin{proposition}[Vector-charge ledger for partial flags]
\label{prop:flag-charge-ledger}
Let $F=(P_1,\ldots,P_\ell):B^3\to\mathcal F_{\mathbf n}$
and its boundary trace $f$ satisfy the regularity and
finite-singularity assumptions of
Theorem~\ref{thm:radial-charge-excess}, and write
$u_f(x)=f(x/|x|)$.
For $s=(s_1,\ldots,s_\ell)\in\{-1,1\}^\ell$ set
\[
 \mathcal X_s(F)=\sum_{a=1}^\ell\int_0^1\!\int_{S^2}
 \left(r^2|\partial_rP_a|_{g_a}^2+
       |\dd(P_a)_r(e_2)-s_aJ_a\dd(P_a)_r(e_1)|_{g_a}^2
 \right)\dd\Omega\,\dd r,
\]
and define the local charge vector by
$q_{a,p}=\pi^{-1}\int_{\partial B_\varepsilon(p)}P_a^*\omega_a$.
Then $\sum_aq_{a,p}=0$ and
\begin{equation}\label{eq:flag-vector-charge-ledger}
 \mathscr E(F;B^3)-\mathscr E(u_f;B^3)
  =\mathcal X_s(F)-\mathcal X_s(u_f)
       -2\pi\sum_{p\in S\setminus\{0\}}|p|
                       \sum_{a=1}^\ell s_aq_{a,p}.
\end{equation}
If $f$ is a pair-root holomorphic or antiholomorphic sphere,
one may choose $s_a$ to make $\mathcal X_s(u_f)=0$.
\end{proposition}

\begin{proof}
Apply the pointwise Wirtinger identity in the proof of
Theorem~\ref{thm:radial-charge-excess} to every $P_a$, with
$J_a$ replaced by $s_aJ_a$, and sum.  Stokes' theorem gives
the radial jump of each component charge.  The sum of the
charges vanishes because
$\bigoplus_a\operatorname{im}P_a$ is the trivial rank-$N$
bundle on each linking sphere.  On a pair-root sphere only
two projectors vary, and they have opposite holomorphic
orientations; choose the corresponding two signs and leave
the other signs arbitrary.
\end{proof}

There is a general way to propose charge-splitting competitors.
The centering condition below removes the first far-field term,
making the entire three-dimensional core cost finite.

\begin{proposition}[Balanced multi-centre surgery]
\label{prop:balanced-multicentre-surgery}
Let $M,Y$ be compact smooth Riemannian manifolds,
$h:S^2\to M$ and $\mu:M^k\to Y$ be $C^3$, with
$\mu$ invariant under permutations of its $k\ge2$ factors.
Choose distinct nonzero $v_1,\ldots,v_k\in\R^3$ with
$\sum_i v_i=0$, and put
\[
 f(n)=\mu(h(n),\ldots,h(n)),\qquad
 U_0(x)=f(x/|x|).
\]
For a smooth radial cutoff $\chi$ equal to $1$ on $[0,1/2]$
and $0$ on $[3/4,1]$, let
\[
 U_a(x)=\mu\left(
 h\left(\frac{x-a\chi(|x|)v_1}
              {|x-a\chi(|x|)v_1|}\right),\ldots,
 h\left(\frac{x-a\chi(|x|)v_k}
              {|x-a\chi(|x|)v_k|}\right)\right).
\]
These maps are in $W^{1,2}(B^3,Y)$ and have the boundary
trace $f$.  If $V(y)$ is the same expression with
$a\chi$ replaced by $1$, then the integral
\[
 C(\mu,h,v)=\int_{\R^3}
      \bigl(|\dd V|_Y^2-|\dd U_0|_Y^2\bigr)\,\dd y
\]
is absolutely convergent and
\begin{equation}\label{eq:multicentre-energy-expansion}
 E_Y(U_a;B^3)-E_Y(U_0;B^3)
       =aC(\mu,h,v)+O(a^2),
 \qquad E_Y(U;\Omega)=\int_\Omega|\dd U|_Y^2.
\end{equation}
In particular, $C(\mu,h,v)<0$ excludes the cone $U_0$
from the minimizing class.  This strict inequality persists
under sufficiently small $C^3$ changes of $h$ and symmetric
$C^3$ changes of $\mu$, and $C^2$ changes of the target metric.
\end{proposition}

\begin{proof}
For $y=rn$ and $r\to\infty$,
\[
 \frac{y-v_i}{|y-v_i|}
   =n-r^{-1}(I-n\otimes n)v_i+O(r^{-2}).
\]
The partial differentials of the symmetric map $\mu$ agree
on its diagonal.  Their first-order contributions cancel
because $\sum_i v_i=0$.  Thus, in smooth local target
coordinates, $V-U_0=O(r^{-2})$ and
$\dd V-\dd U_0=O(r^{-3})$.  Since
$\dd U_0=O(r^{-1})$, the energy-density difference is
$O(r^{-4})$, integrable at infinity in dimension three.
Near a center $v_i$, only its factor can be singular and its
derivative is $O(|y-v_i|^{-1})$, which has integrable square.
The cone's singularity at zero has the same property.

On $B_{1/2}$, scaling $x=ay$ yields
\[
 E_Y(U_a;B_{1/2})-E_Y(U_0;B_{1/2})
  =a\int_{B_{1/(2a)}}
       (|\dd V|_Y^2-|\dd U_0|_Y^2)\,\dd y
  =aC(\mu,h,v)+O(a^2).
\]
On the fixed outer annulus, the first $a$-derivative of
$U_a$ vanishes by the same centering argument, also where
$\dd\chi\ne0$.  There $U_a-U_0=O_{C^1}(a^2)$ and the
remaining energy difference is $O(a^2)$.
The estimates hold uniformly under the stated perturbations,
so dominated convergence makes $C$ continuous in those data.
\end{proof}

For the two-centre competitor of
Proposition~\ref{prop:veronese-not-minimizing}, the charges are
$+1$ at $\pm ae_3$, so \eqref{eq:radial-charge-ledger} makes
the strict comparison equivalent to $\mathcal X(V_a)<4\pi a$.

\subsection{Two-centre surgery for conic links}
\label{subsec:veronese-splitting}

We first construct the two-centre competitor for the balanced
Veronese conic; the subsequent Takagi comparison extends its
strict energy decrease to every smooth conic.

There is a first, purely stationary selection test before the
minimizing comparison.  Equip $\C P^m$ with the Fubini--Study
metric $g_{\rm FS}=\frac12\tr(dP)^2$ and K\"ahler form
$\omega_{\rm FS}$.

\begin{proposition}[First-moment balance of a stationary cone]
\label{prop:conic-first-moment}
Let $g:S^2\to\C P^m$ be smooth, conformal, and harmonic, with
$g^*g_{\rm FS}=\lambda(n)^2g_{S^2}$, and set
$u(x)=g(x/|x|)$.  The radial extension $u$ is stationary under
all compactly supported domain variations if and only if
\begin{equation}\label{eq:conic-first-moment}
             \int_{S^2}n\,\lambda(n)^2\,d\Omega(n)=0.
\end{equation}
In particular, minimization of the projector energy for its
boundary trace implies~\eqref{eq:conic-first-moment}.
For a holomorphic link this is the vanishing of the first
spherical moment of $g^*\omega_{\rm FS}$.
\end{proposition}

\begin{proof}
Multiplying the energy by a positive constant does not change
stationarity, so use the standard Dirichlet normalization
$\frac12\int|du|_{g_{\rm FS}}^2$.  Its stress tensor on the
punctured ball is
\[
 S=\tfrac12|du|^2I-du^*du
   =\frac{\lambda(n)^2}{r^2}n\otimes n,
       \qquad r=|x|,\quad n=x/r.
\]
It is divergence-free away from the origin, since $g$ is harmonic.
Its flux through every centered sphere is the vector on the
left of~\eqref{eq:conic-first-moment}.  Integrating $S:DX$
over $B^3\setminus B_\varepsilon$ for a compactly supported
domain variation $X$ gives, as $\varepsilon\downarrow0$,
minus this flux dotted with $X(0)$.  Since $X(0)$ is arbitrary,
domain stationarity is equivalent to vanishing of the flux;
minimization implies domain stationarity.
For a holomorphic map, $g^*\omega_{\rm FS}=\lambda^2d\Omega$.
\end{proof}

\begin{proposition}[Splitting the balanced Veronese cone]
\label{prop:veronese-not-minimizing}
Let $b:S^2\to\C P^1$ be the Bloch-sphere identification and
let
\[
 \nu_2:\C P^1\longrightarrow\C P^2,\qquad
 [u:v]\longmapsto[u^2:\sqrt2uv:v^2].
\]
For $x\in B^3\setminus\{0\}$ put
$V_0(x)=\nu_2(b(x/|x|))$.  With the projector metric
$E(V;\Omega)=\frac12\int_\Omega|dP_V|_{\rm HS}^2$,
the map $V_0$ has energy $4\pi$ on $B^3$ and is not a
Dirichlet minimizer for its boundary trace.  More precisely,
there are competitors $V_a\in W^{1,2}(B^3,\C P^2)$ with
the same trace and
\begin{equation}\label{eq:veronese-energy-drop}
 E(V_a;B^3)-E(V_0;B^3)
       =aC_{\rm split}+O(a^2),\qquad
 C_{\rm split}<-\frac{\pi}{6}.
\end{equation}
The associated two-block partial-flag cone
$x\mapsto(P_{V_0}(x),I-P_{V_0}(x))$ is likewise not
locally minimizing.
\end{proposition}

\begin{proof}
Write the symmetric-product map as
\[
 \mathfrak m:\C P^1\times\C P^1\longrightarrow\C P^2,\qquad
 \mathfrak m([u:v],[s:t])
       =[us:(ut+vs)/\sqrt2:vt].
\]
It satisfies $\mathfrak m(\ell,\ell)=\nu_2(\ell)$.
Choose a smooth radial cutoff $\chi$ equal to $1$ on
$[0,1/2]$ and $0$ on $[3/4,1]$.  For $0<a<1/4$ set
\[
 V_a(x)=\mathfrak m\left(
 b\left(\frac{x-a\chi(|x|)e_3}
              {|x-a\chi(|x|)e_3|}\right),
 b\left(\frac{x+a\chi(|x|)e_3}
              {|x+a\chi(|x|)e_3|}\right)\right).
\]
The only singularities are at $\pm ae_3$.  Near either one,
one factor is a unit hedgehog and the other is smooth, so
$|dV_a|=O(|x\mp ae_3|^{-1})$; hence $V_a\in W^{1,2}$.
The trace at $|x|=1$ equals that of $V_0$.
Since $\mathfrak m$ is symmetric, $V_a=V_{-a}$ on the fixed
outer annulus $|x|\ge1/2$, where the factors are smooth in
$a$.  Its energy change there is $O(a^2)$.

On $B_{1/2}$ the cutoff is constant and $V_a(ay)=V_1(y)$.
The symmetry in $a$ also gives
$V_1(y)-V_0(y)=O_{C^1(S^2)}(|y|^{-2})$ at infinity;
the radial derivative of the difference is
$O(|y|^{-3})$.  Consequently the energy-density
difference is $O(|y|^{-4})$.  It is absolutely integrable,
and scaling the full inner ball, including both cores, yields
\[
 E(V_a;B^3)-E(V_0;B^3)
   =aC_{\rm split}+O(a^2),\qquad
 C_{\rm split}
   =\int_{\R^3}\bigl(e(V_1)-|y|^{-2}\bigr)\,dy,
\]
where $e(V)=\frac12|dP_V|_{\rm HS}^2$.
The Veronese pullback metric is $\frac12g_{S^2}$, so
$e(V_0)=|y|^{-2}$ and $E(V_0;B^3)=4\pi$.

We evaluate the coefficient.  Use cylindrical coordinates
$(\rho,z,\phi)$ and put
\[
 R_\pm=\sqrt{\rho^2+(z\mp1)^2},\quad
 \theta_\pm=\operatorname{atan2}(\rho,z\mp1),\quad
 m=\frac{\theta_++\theta_-}{2},\quad
 \delta=\frac{\theta_+-\theta_-}{2}.
\]
In a projective chart the symmetric-product vector is
\[
 w=\left(\frac{\cos\delta+\cos m}{2},\
       \frac{e^{i\phi}\sin m}{\sqrt2},\
       \frac{e^{2i\phi}(\cos\delta-\cos m)}{2}\right).
\]
Writing $C=\cos^2\delta$ and $q^2=\cos^2m$, the
projector metric computed from
$\bigl(|dw|^2|w|^2-|w^*dw|^2\bigr)/|w|^4$ is diagonal
in $(m,\delta,\phi)$ with coefficients
\[
 g_{mm}=\frac1{1+C},\qquad
 g_{\delta\delta}=\frac{1-C}{(1+C)^2},\qquad
 g_{\phi\phi}
 =\frac{C(1+C)+q^2(1-3C)}{(1+C)^2}.
\]
Moreover
$|\nabla m|^2=(\rho^2+z^2)/(R_+^2R_-^2)$ and
$|\nabla\delta|^2=1/(R_+^2R_-^2)$.
Introduce prolate coordinates
$\xi=(R_++R_-)/2\ge1$ and
$\eta=(R_--R_+)/2\in[-1,1]$.  Then
\[
 dV=(\xi^2-\eta^2)\,d\xi\,d\eta\,d\phi,\quad
 C=\frac{\xi^2-1}{\xi^2-\eta^2},\quad
 q^2=\eta^2 C.
\]
Direct substitution, followed by partial fractions, gives
\begin{equation}\label{eq:veronese-split-integrand}
 C_{\rm split}=2\pi\int_1^\infty\!\int_{-1}^1
 \left[
 \frac{y}{\xi^2-(1+y)/2}
 +\frac{1-y^2}{4[\xi^2-(1+y)/2]^2}
 -\frac{1-2y}{\xi^2-(1-y)}
 \right]\,d\eta\,d\xi,\qquad y=\eta^2.
\end{equation}
All apparent singularities are integrable by the preceding
absolute-integrability estimate.

For $0\le c<1$, let
\[
 I(c)=\int_1^\infty\frac{d\xi}{\xi^2-c}
      =\frac{\operatorname{artanh}\sqrt c}{\sqrt c}
      =\sum_{k=0}^\infty\frac{c^k}{2k+1}.
\]
Integrating~\eqref{eq:veronese-split-integrand} in $\xi$
and using
$\int_0^1(1-2t^2)I(1-t^2)\,dt=1$
(substitute $t=\sin\theta$ and integrate by parts) yields
\[
 C_{\rm split}=4\pi(J-\tfrac12),\qquad
 J=\int_0^1\frac{5t^2-1}{4}
                  I\left(\frac{1+t^2}{2}\right)dt.
\]
The positive power series gives
$1\le I(c)\le1-\frac12\log(1-c)$.
Split the integral at $b=1/\sqrt5$, where its weight
changes sign.  The lower bound for $I$ on $[0,b]$ and
upper bound on $[b,1]$ imply
\[
 J\le\frac16+\frac18
       \int_b^1(5t^2-1)\log\frac2{1-t^2}\,dt
 <\frac16+\frac7{24}=\frac{11}{24}.
\]
For the last strict inequality, integrate the logarithmic
term by parts; it becomes
\[
 \frac18\left[
  \frac{22}{9}-\frac23\log2+
  \int_0^b(1-5t^2)\log\frac2{1-t^2}\,dt
 \right]
 <\frac7{24},
\]
using $\log2>2/3$, $\log(5/2)<1$, and
$2/(3\sqrt5)<1/3$.  Thus
$C_{\rm split}<4\pi(11/24-1/2)=-\pi/6$.
The flag energy for $(P,I-P)$ is twice the projective
energy, so the same competitor gives a strict decrease.
\end{proof}

\begin{proposition}[Openness of strict cone nonminimality]
\label{prop:strict-conic-nonminimality-open}
Let $Y$ be a closed Riemannian manifold. If a smooth link
$g_0:S^2\to Y$ has a strict same-trace Sobolev competitor
for its radial cone, then every sufficiently $C^1$-close
link has a nonminimizing radial cone.
\end{proposition}

\begin{proof}
Place the strict competitor in $B_{1-\eta}$ for fixed small
$\eta>0$ and join $g_0$ to the new boundary link in the outer
shell by short target geodesics. At $g=g_0$ the energy gap
is $(1-\eta)$ times the original gap. Both the competitor
energy and the cone energy vary continuously with $g$ in
$C^1$, so the strict inequality persists.
\end{proof}

\begin{corollary}[Spin-one first-jet exclusion]
\label{cor:spin-one-first-jet-exclusion}
If $H_B(x)=\sum_jx_j(J_j+B_j)$ with spin-one matrices $J_j$
and sufficiently small Hermitian $B_j$, its top-eigenline
cone is not minimizing in $\C P^2$. The same conclusion
holds for the top-line-versus-complement flag of any smooth
field $A(x)=cI+H_B(x)+O(|x|^2)$ near zero.
\end{corollary}

\begin{proof}
The top spectral gap on $S^2$ is at least
$1-2\sqrt3\max_j\|B_j\|$. Its eigenline converges in $C^1$
to the Veronese link, so the preceding proposition applies.
For the nonlinear field, the gap is bounded below by
$c_0|x|$, its projector has derivative $O(|x|^{-1})$, and
its rescalings converge smoothly on compact annuli to that
cone. Minimizing-map compactness would make this tangent
minimizing, contradicting its strict competitor. The
corresponding two-block flag energy is twice the projective
energy.
\end{proof}

At the diagonal of the symmetric product,
$g_{\delta\delta}=(1-\cos^2\delta)/(1+\cos^2\delta)^2$
vanishes quadratically; separating the two primitive centers
exploits this transverse branch geometry.

\subsection{Excluding all holomorphic degree-two cones}
\label{subsec:all-conic-splitting}

Use the projector energy
$E(U)=\frac12\int|dP_U|_{\rm HS}^2$.

\begin{proposition}[Uniform conic splitting]
\label{prop:all-conic-splitting}
For $\sigma_j>0$ let
\[
g_\sigma([u:v])=
\left[(u^2-v^2)/\sqrt{\sigma_1}:
 i(u^2+v^2)/\sqrt{\sigma_2}:
 2uv/\sqrt{\sigma_3}\right].
\]
The radial cone on $g_\sigma$ is not Dirichlet minimizing.
After choosing a source axis corresponding to $\min_j\sigma_j$,
the two-centre competitors with separation $2\varepsilon$ have
the same boundary trace and satisfy
\[
E(V_\varepsilon;B^3)-4\pi
 =\varepsilon C_\sigma+O_\sigma(\varepsilon^2),
\qquad C_\sigma\leq C_{\rm split}<-\pi/6,
\]
where $C_{\rm split}$ is the balanced coefficient in
Proposition~\ref{prop:veronese-not-minimizing}.
\end{proposition}

\begin{proof}
The three quadratic coordinates form the real spin-one triple;
source rotations and unitary target permutations let us take
$\alpha_3=\max_j\alpha_j$, where $\alpha_j=1/\sigma_j$.
In the monomial basis
$[u^2:\sqrt2uv:v^2]$, the target Hermitian norm is,
up to scale, represented by
\[
H=\begin{pmatrix}a&0&c\\0&b&0\\c&0&a\end{pmatrix},
\quad a=\alpha_1+\alpha_2,\quad
c=\alpha_2-\alpha_1,\quad b=2\alpha_3.
\]
Thus $b\geq a+|c|$. Use the symmetric-product map and
the two centres $\pm e_3$ from the balanced proof. With
$R_\pm=\sqrt{\rho^2+(z\mp1)^2}$, introduce prolate coordinates
$\xi=(R_++R_-)/2\geq1$ and
$\eta=(R_--R_+)/2\in[-1,1]$. Put
$s^2=\xi^2-1$ and $t^2=1-\eta^2$. A homogeneous vector
for the split map is
\[
W=\bigl((1+\eta)s,\,
        \sqrt2\,\xi t e^{i\phi},\,
        (1-\eta)s e^{2i\phi}\bigr).
\]
The Euclidean metric is
$\frac{\xi^2-\eta^2}{s^2}d\xi^2+
 \frac{\xi^2-\eta^2}{t^2}d\eta^2+s^2t^2d\phi^2$.
For fixed $\xi$, $[W]$ is a holomorphic degree-two sphere:
it is the Veronese map followed by
$\operatorname{diag}(1,\xi/s,1)$. Write its conformal factor as
$\lambda_\xi^2$, so that
$\int_{S^2}\lambda_\xi^2\,d\Omega=2\pi$. Define
\[
h=\frac ca,\quad R=\frac ba\frac{\xi^2}{\xi^2-1},\quad
Q=1+\eta^2+(R+h\cos2\phi)(1-\eta^2),\quad
p=\frac{R(1-\eta^2)}Q.
\]
The middle-coordinate projective scaling has squared speed
$p(1-p)$. For $E_\xi=E([W];\{\xi'<\xi\})$,
holomorphicity of the slices and the prolate metric give
\[
\frac{dE_\xi}{d\xi}
=4\pi+\frac{M+\xi^{-2}V}{\xi^2-1},\quad
M=\int_{S^2}(1-\eta^2)\lambda_\xi^2\,d\Omega,\quad
V=\int_{S^2}p(1-p)\,d\Omega.                     \tag{*}
\]
The cone energy inside $\{\xi<L\}$ is $4\pi L+o(1)$,
since its boundary differs from the radius-$L$ sphere
by $O(L^{-1})$. The balanced proof's integrable-tail
estimate persists under the fixed target transformation. Thus
\[
C_\sigma=-4\pi+
 \int_1^\infty\frac{M+\xi^{-2}V}{\xi^2-1}\,d\xi. \tag{**}
\]

We compare the integrand in two steps. Put $y=\eta^2$,
$t_\phi=h\cos2\phi$, and
$A=1+y+t_\phi(1-y)$, so $Q=A+R(1-y)$.
The Fubini--Study area form on a slice gives
$\lambda_\xi^2=\frac12+\frac14\Delta_{S^2}\log Q$.
Integration by parts, with $\tau=\log R$, yields
\[
\partial_\tau M
=-2R\int_{S^2}\frac{y(1-y)}{Q^2}\,d\Omega<0,
\qquad
\partial_\tau V
=\int_{S^2}p(1-p)(1-2p)\,d\Omega.
\]
Moreover
\[
\partial_\tau(M+V)=R\int_{S^2}
\frac{(1-y)^2[(1+t_\phi)A-
R((1+t_\phi)+(3-t_\phi)y)]}{Q^3}\,d\Omega\leq0
\]
when $R\geq1+|h|$. Indeed $1+t_\phi\geq0$,
$R\geq1+t_\phi$, and the second parenthesis is at least
$A$. Since $0<\xi^{-2}<1$, these two derivative
inequalities imply $\partial_\tau(M+\xi^{-2}V)<0$.
Consequently, for fixed $h$, decreasing $b/a$ to
$1+|h|$ can only increase $(**)$. A shift
$\phi\mapsto\phi+\pi/2$ changes $h$ to $-h$;
we henceforth take $h\geq0$.

On the boundary $b/a=1+h$, set
$K=\xi^2/(\xi^2-1)=1+k$,
$U=1-\eta^2$, and $v=(1+\cos2\phi)/2$. Then
$Q=2+U[k(1+h)+2hv]$ and $\xi^{-2}=k/(1+k)$.
Differentiating along this boundary path and integrating
the $M$ term in $\eta$ by parts gives
\[
\frac d{dh}(M+\xi^{-2}V)
=\int_{S^2}\frac{U(N_0+hN_1)}{Q^3}\,d\Omega,
\]
where the following factored polynomials are nonpositive
for $k\geq0$ and $U,v\in[0,1]$:
\[
\begin{aligned}
N_0={}&-8(1-U)^2v-4U(1-U)[k^2+2(k+1)v]
         -2k^2U^2(1-v),\\
N_1={}&-4U(1-U)\bigl[k(k+1)(1-v)^2+
      k(2k+3)v(1-v)+(k^2+2k+2)v^2\bigr]\\
    &-2kU^2\bigl[(k+2)(1-v)^2+kv(1-v)\bigr].
\end{aligned}
\]
For verification, $N_0+hN_1$ expands to
$-2(1-U)(k+2v)Q+
k(2-2Uv)[Q-2(1+h)(1+k)U]$.
The boundary integrand therefore decreases from its
balanced value at $h=0$. Formula $(**)$ gives
$C_\sigma\leq C_{\rm split}<-\pi/6$.
Finally, rescale the full two-centre core and cut off the
separation near $\partial B^3$, exactly as in
Proposition~\ref{prop:veronese-not-minimizing}.
The outer-annulus change is $O_\sigma(\varepsilon^2)$
and proves the asserted trace-preserving expansion.
\end{proof}

\begin{theorem}[Holomorphic minimizing links]
\label{thm:holomorphic-minimizing-links}
For $m\geq1$, a holomorphic link
$g:\C P^1\to\C P^m$ has a Dirichlet-minimizing
radial extension if and only if it is constant or
has degree one with $g^*g_{\rm FS}=\frac14g_{S^2}$,
the standard projective-line hedgehog normalization.
In particular, no holomorphic degree-two radial cone minimizes.
\end{theorem}

\begin{proof}
First consider a smooth conic $C\subset\C P^2$.
Its defining equation is a nondegenerate complex symmetric
quadratic form. Unitary Takagi diagonalization gives
$\sigma_1z_1^2+\sigma_2z_2^2+\sigma_3z_3^2=0$
with $\sigma_j>0$, whose standard parameterization is $g_\sigma$.
Every other holomorphic parameterization is
$g_\sigma\circ\varphi$ for some
$\varphi\in\operatorname{PSL}_2(\C)$.
The two orthogonal source half-turns
of $g_\sigma$ act by unitary target sign changes;
thus $\mu=g_\sigma^*\omega_{\rm FS}$ has zero
first moment. This is the stationarity condition of
Proposition~\ref{prop:conic-first-moment}.
It singles out $g_\sigma$ up to source rotations:
by polar decomposition every Möbius map is a round
rotation, an axial boost $T_{s,v}$, and another rotation.
Writing $q=n\cdot v$, the boost satisfies
\[
(T_{s,v}(n)-n)\cdot v
=\frac{s(1-q^2)}{1+s q}>0
\qquad(0<s<1,\ |q|<1).
\]
Hence a nonzero boost sends any smooth positive
zero-moment measure to one with nonzero moment.
Since $(g_\sigma\circ\varphi)^*\omega_{\rm FS}
=(\varphi^{-1})_*\mu$, a stationary
parameterization differs from $g_\sigma$ only by
a round rotation. Proposition~\ref{prop:all-conic-splitting}
excludes it; nonstationary parameterizations cannot minimize.

Any degree-two holomorphic map into $\C P^m$
is given by quadratic sections spanning a vector space
of dimension at most three. If their rank is three,
its image is a smooth conic in a totally geodesic
$\C P^2$ and the preceding argument applies.
If their rank is two, its image is a projective line
and the map has degree two into that line.
Brezis--Coron--Lieb's sphere-target classification
\cite{BrezisCoronLieb} excludes the latter:
its competitor in $\C P^1$ remains admissible in $\C P^m$.
For degree one, the first-moment test and the same
M\"obius boost argument above force a round
parameterization, equivalently
$g^*g_{\rm FS}=\frac14g_{S^2}$.
Fujioka's Theorem~2 \cite{Fujioka1993} proves
that these links do minimize in $\C P^m$.
Fujioka's Theorem~1 excludes holomorphic degree
greater than two. Constants minimize trivially.
\end{proof}

Complex conjugation of $\C P^m$ is an isometry, so the
classification also holds for antiholomorphic links.

\subsection{Unitary averaged collapse and projective minimizing cones}
\label{subsec:unitary-averaged-collapse}

The projective contraction used by Brezis--Coron--Lieb and
Fujioka~\cite{BrezisCoronLieb,Fujioka1993} admits an averaged
energy identity for arbitrary links.  We use the Fubini--Study
metric of holomorphic sectional curvature four and
$E(g)=\int_{S^2}|\dd g|^2$; a projective-line hedgehog has
energy $2\pi$.  Write $\mathfrak m(g)$ for the infimum of
Dirichlet energies of $W^{1,2}(B^3,\C P^m)$ maps with trace $g$.

\begin{theorem}[An averaged filling bound]
\label{thm:unitary-averaged-filling}
For $m\ge2$ define
\[
 G_m(s)=m\int_0^1
 \frac{q(1-q)^m}{[q+s^2(1-q)]^2}\,\dd q,
 \qquad J_m=\int_0^1G_m(s)\,\dd s,
 \qquad K_m=4\pi J_m.
\]
Every smooth $g:S^2\to\C P^m$ satisfies
\begin{equation}\label{eq:averaged-filling-deficit}
 E(g)-\mathfrak m(g)
 \ge \bigl(\sqrt{E(g)}-\sqrt{K_m}\bigr)_+^2,
 \qquad
 J_m\le\frac m2\left(\frac56+\frac{3\pi^2}{32}\right)<m.
\end{equation}
In particular, a Dirichlet-minimizing radial cone satisfies
$E(g)<4\pi m$.
\end{theorem}

\begin{proof}
For $s>0$ put
$D_s[z_0:\cdots:z_m]=[z_0:sz_1:\cdots:sz_m]$.
If $\dd A$ is Haar probability measure on $\U(m+1)$, then
\begin{equation}\label{eq:unitary-average-metric}
 \int_{\U(m+1)}A^*D_s^*g_{\rm FS}\,\dd A=g_{\rm FS}.
\end{equation}
Indeed, the average is an invariant Hermitian metric, so
isotropy irreducibility makes it a scalar multiple of
$g_{\rm FS}$. Its associated two-form is the average of
$A^*D_s^*\omega_{\rm FS}$. Each projective automorphism
preserves $[\omega_{\rm FS}]$, which fixes that scalar at one.
Consequently the mean angular energy of $D_sAg$ is $E(g)$,
with no conformality assumption on $g$.

For a projective point $[z]$, put $q=|z_0|^2/|z|^2$.
The projective metric gives
\[
 |\partial_sD_s[z]|^2
   =\frac{q(1-q)}{[q+s^2(1-q)]^2}.
\]
For a Haar-distributed $A$, the coordinate $q$ of $A[z]$
has density $m(1-q)^{m-1}$; the mean squared speed is
therefore $G_m(s)$, independently of $[z]$.

Fix $0<\varepsilon<1$ and set
\[
 \eta(r)=\frac{1-\varepsilon/r}{1-\varepsilon},
 \qquad
 F_A(r,n)=A^{-1}D_{\eta(r)}Ag(n)
       \quad(\varepsilon<r\le1).
\]
On $r\le\varepsilon$ set $F_A=A^{-1}[1:0:\cdots:0]$.
This is an admissible filling for almost every $A$.
To check the inner trace, for each fixed $n$ the probability
that $(Ag(n))_0=0$ is zero. Fubini gives a null exceptional
set of $n$ for almost every $A$. As $r\downarrow\varepsilon$,
the shell maps thus converge almost everywhere and in $L^2$
to the stated constant. The finite mean energy computed below
gives finite Sobolev energy for almost every $A$, so this
$L^2$ limit is its inner trace. Sobolev gluing with the
constant core applies. The outer trace is exactly $g$.

Equation~\eqref{eq:unitary-average-metric} and radial integration
give the exact mean energy
\begin{equation}\label{eq:averaged-collapse-energy}
 \int_{\U(m+1)}E(F_A)\,\dd A
  =(1-\varepsilon)E(g)
    +4\pi\int_\varepsilon^1 r^2\eta'(r)^2G_m(\eta(r))\,\dd r
  =(1-\varepsilon)E(g)+\frac{\varepsilon K_m}{1-\varepsilon}.
\end{equation}
The formulas first hold on truncated shells; Tonelli's theorem
and the next estimate give the finite limit.
An admissible $A$ has energy no greater than this mean.
For $E(g)>K_m$, minimizing the last expression at
$1-\varepsilon=\sqrt{K_m/E(g)}$ proves the deficit bound.
For $E(g)\le K_m$, use the radial filling itself.

It remains to estimate the universal constant. For $m\ge2$,
$G_m(s)\le(m/2)G_2(s)$. With $q=(1+t^2)^{-1}$,
\[
 G_2(s)=4\int_0^\infty
       \frac{t^5\,\dd t}{(1+t^2)^3(1+s^2t^2)^2}.
\]
Integrating in $s$ first and then setting $t=\tan\theta$ yields
\[
 J_2=2\int_0^\infty\frac{t^5\,\dd t}{(1+t^2)^4}
       +2\int_0^{\pi/2}\theta\sin^4\theta\,\dd\theta
     =\frac56+\frac{3\pi^2}{32}<2.
\]
This proves~\eqref{eq:averaged-filling-deficit}. A minimizing
cone has zero filling deficit, hence $E(g)\le K_m<4\pi m$.
\end{proof}

\begin{lemma}[The harmonic-sequence energy floor]
\label{lem:full-harmonic-energy-floor}
Let $g:S^2\to\C P^m$ be a full harmonic sphere which is
neither holomorphic nor antiholomorphic. Then $m\ge2$ and
\begin{equation}\label{eq:nonholomorphic-energy-floor}
             E(g)\ge2\pi(3m-2)\ge4\pi m.
\end{equation}
\end{lemma}

\begin{proof}
The harmonic-sequence theorem of Eells--Wood
\cite[Theorem~6.9]{EellsWood1983} writes $g=f_j$, where
$f_0$ is a full holomorphic sphere and $1\le j\le m-1$.
Let $\delta_i$ be the Pl\"ucker degree of its holomorphic
osculating map $S^2\to\operatorname{Gr}(i+1,m+1)$, with
$\delta_{-1}=\delta_m=0$. The Pl\"ucker identities and
harmonic-sequence energy formula are
\[
 2\delta_i-\delta_{i-1}-\delta_{i+1}=2+r_i,
 \quad r_i\ge0,
 \qquad
 E(f_j)=2\pi(\delta_{j-1}+\delta_j).
\]
These formulas include ramification; see
\cite[Section~2, Theorem~2.1 and Lemma~2.4]{Kawabe2013}.
For completeness, if $L_i$ are the successive orthogonal
line bundles and $c_i=\deg L_i$, their nonzero holomorphic
derivative has divisor degree
$r_i=c_{i+1}-c_i-2$. Since
$\delta_i=-\sum_{a=0}^ic_a$, this gives the first identity.
The two adjacent derivative components of $L_j$ have
integrated energies $2\pi\delta_{j-1}$ and $2\pi\delta_j$,
giving the second.
The sequence $b_i=(i+1)(m-i)$ has the same zero endpoint
values and satisfies $2b_i-b_{i-1}-b_{i+1}=2$.
Discrete concavity therefore gives $\delta_i\ge b_i$.
Consequently
\[
 E(f_j)\ge2\pi\bigl[m+2j(m-j)\bigr]
             \ge2\pi(3m-2),
\]
as required.
\end{proof}

\begin{theorem}[Projective minimizing cones]
\label{thm:projective-minimizing-cones}
For every $N\ge1$, a smooth link $g:S^2\to\C P^N$
has a Dirichlet-minimizing radial extension if and only if
it is constant or is a holomorphic or antiholomorphic
projective-line map of absolute degree one satisfying
$g^*g_{\rm FS}=\tfrac14g_{S^2}$.
In particular, every nonholomorphic, nonantiholomorphic
harmonic link admits a filling $F$ with the same trace and
\begin{equation}\label{eq:uniform-nonholomorphic-deficit}
 E(F)\le(1-\gamma)E(g),\qquad
 \gamma=\left(1-\sqrt{\frac5{12}+\frac{3\pi^2}{64}}\right)^2>0.
\end{equation}
\end{theorem}

\begin{proof}
A minimizing cone is harmonic away from the origin, so its
link is harmonic. Take the smallest projective linear
subspace $\C P^m$ containing its image. This subspace is
totally geodesic; its competitors are also competitors in
$\C P^N$. If the link is neither holomorphic nor
antiholomorphic, Lemma~\ref{lem:full-harmonic-energy-floor}
gives $E(g)\ge4\pi m$, while
Theorem~\ref{thm:unitary-averaged-filling} gives
$E(g)\le K_m=4\pi J_m<4\pi m$, a contradiction.
Quantitatively, they give
$K_m/E(g)\le J_2/2$. In
\eqref{eq:averaged-collapse-energy} choose
$1-\varepsilon=\sqrt{J_2/2}$. The resulting mean is at most
$(1-\gamma)E(g)$, and some admissible filling attains this
bound. The remaining holomorphic and antiholomorphic
links are classified by
Theorem~\ref{thm:holomorphic-minimizing-links} and complex
conjugation. Constant maps minimize trivially.
\end{proof}

\subsection{A zero-charge collision removed by smoothing}
\label{subsec:real-projective-smoothing}

The charge-displacement term in
Theorem~\ref{thm:radial-charge-excess} detects one
source of energy reduction.  A different mechanism occurs for
nonholomorphic links with zero K\"ahler charge: their central
singularity can sometimes disappear under a smooth filling.
For a harmonic link $g$, write its angular Jacobi form as
\[
 \mathcal Q_g(V)=\int_{S^2}\left(
 |\nabla V|^2-
 \sum_{i=1}^2\langle R(V,\dd g(e_i))\dd g(e_i),V\rangle
 \right)\dd\Omega.
\]

\begin{proposition}[A mixed-angle Hardy test]
\label{prop:mixed-angle-hardy-test}
Let $g:S^2\to\C P^m$ be harmonic, with
$g^*g_{\rm FS}=\lambda^2g_{S^2}$. Away from branch points define
its K\"ahler angle by
$c=\langle J\dd g(e_1),\dd g(e_2)\rangle/\lambda^2$.
If
\begin{equation}\label{eq:mixed-angle-threshold}
 \int_{S^2}\lambda^4(1-c^2)\,\dd\Omega
  >\frac1{16}\int_{S^2}\lambda^2\,\dd\Omega,
\end{equation}
then the radial cone on $g$ is not Dirichlet minimizing.
More precisely, for each source conformal field
$X=\nabla(a\cdot n)$ and $V=J\dd g(X)$, its angular Jacobi
form satisfies
\begin{equation}\label{eq:mixed-angle-jacobi}
 \mathcal Q_g(V)=-4\int_{S^2}
       \lambda^4(1-c^2)|X|^2\,\dd\Omega.
\end{equation}
\end{proposition}

\begin{proof}
The Hopf differential of a harmonic two-sphere vanishes, so $g$
is weakly conformal. A conformal flow on the source preserves
harmonicity; hence $W=\dd g(X)$ is a Jacobi field and
$\mathcal Q_g(W)=0$. Since $J$ is parallel,
$|\nabla(JW)|=|\nabla W|$. The sectional-curvature numerator of
$\C P^m$ in this normalization is
$|A\wedge B|^2+3\langle JA,B\rangle^2$.
Comparing the curvature terms for $JW$ and $W$, and using
$|\dd g(e_i)|=\lambda$, gives
$\mathcal Q_g(JW)-\mathcal Q_g(W)
 =-4\int\lambda^4(1-c^2)|X|^2$, proving
\eqref{eq:mixed-angle-jacobi}.
For the three coordinate fields $X_j=\nabla n_j$,
$\sum_j|X_j|^2=2$ pointwise. Thus one nonzero $V_j$ has
angular Rayleigh quotient at most
$-4\int\lambda^4(1-c^2)/\int\lambda^2$.
Under~\eqref{eq:mixed-angle-threshold} this is below $-1/4$.
The radial second variation for a separated field
$\eta(r)V_j(n)$ is proportional to
$\int_0^1(r^2\eta'^2+\kappa_j\eta^2)\,\dd r$,
where $\kappa_j<-1/4$ is that quotient. The sharp radial
Hardy constant is $1/4$, so a compactly supported radial
profile makes this variation negative.
\end{proof}

\begin{proposition}[Real-projective smoothing]
\label{prop:real-projective-smoothing}
Let $m\geq2$ and regard $\C P^2$ as a totally geodesic subspace
of $\C P^m$. The map
\[
 g:S^2\longrightarrow\C P^2,\qquad
 g(n)=[n_1:n_2:n_3],
\]
is a full, nonholomorphic harmonic sphere with
$g^*\omega_{\rm FS}=0$. Its radial cone has energy $8\pi$ in
$B^3$ and is not Dirichlet minimizing. In fact, for $t>0$ the
smooth map
\[
 G_t(x)=[x_1:x_2:x_3+i t(1-|x|^2)]
\]
has the same boundary trace, has no singularities, and satisfies
\[
 E(G_t;B^3)=8\pi-4\pi^2t+O(t^2)
 \qquad(t\downarrow0).
\]
\end{proposition}

\begin{proof}
The real locus $\mathbb RP^2\subset\C P^2$ is the fixed set of
complex conjugation, hence is totally geodesic and Lagrangian.
With our Fubini--Study normalization its induced metric has
curvature one. The two-to-one map $S^2\to\mathbb RP^2$ is a local
isometry, so $g$ is harmonic, $|dg|^2=2$, and its K\"ahler
charge vanishes. Here $\lambda=1$ and $c=0$, so
Proposition~\ref{prop:mixed-angle-hardy-test} detects a
negative mode with quotient $-4$. Its radial energy is $8\pi$.

For a real profile $b(r)$ put
$F_b(r,n)=[rn+i b(r)e_3]$ and $Q=r^2+b^2$. The projective metric
of a homogeneous vector $z$ is
\[
 |d[z]|^2=\frac{|dz|^2}{|z|^2}
          -\frac{|\langle z,dz\rangle|^2}{|z|^4}.
\]
Spherical integration of its radial and angular parts gives the
exact identity
\begin{equation}\label{eq:rp2-smoothing-energy}
 E(F_b;B^3)-8\pi=\frac{8\pi}{3}\int_0^1
 \frac{r^2(rb'-b)^2-4r^2b^2-3b^4}{(r^2+b^2)^2}\,dr.
\end{equation}
For the constant core profile $b=t$, this integral is
$-8\pi t\arctan(1/t)=-4\pi^2t+O(t^2)$.
Now choose $b=t(1-r^2)$ to fix the boundary.  Write
$R=r^2+t^2$ and $Q=r^2+t^2(1-r^2)^2$. For $0<t\leq1$,
$Q\geq R/4$ and $|Q-R|\leq2t^2r^2$. The difference of the two
numerators in~\eqref{eq:rp2-smoothing-energy} is
\[
 t^2r^4(10-3r^2)
 +3t^4\bigl[1-(1-r^2)^4\bigr].
\]
These bounds show that the two integrands differ by at most
$Ct^2$ uniformly on $[0,1]$, giving the asserted expansion.
The homogeneous lift of $G_t$ never vanishes on the closed ball:
its real part vanishes only at $x=0$, where its imaginary part
is $t e_3$. Thus $G_t$ is smooth and has the stated boundary trace.
\end{proof}

The criterion~\eqref{eq:mixed-angle-threshold} has a precise
degenerating boundary even among balanced harmonic links.

\begin{proposition}[Balanced mixed-angle degeneration]
\label{prop:balanced-gauss-degeneration}
For $s>0$, let $g_s$ be the first Gauss transform of the
holomorphic conic $[z^2:sz:1]$ on $\C P^1$.
Each $g_s:S^2\to\C P^2$ is full, harmonic, nonholomorphic,
and has zero first spherical energy moment. Nevertheless,
\[
 \frac{\int_{S^2}\lambda_s^4(1-c_s^2)\,\dd\Omega}
      {\int_{S^2}\lambda_s^2\,\dd\Omega}
  =O\bigl(s^2\log(1/s)\bigr)\longrightarrow0
       \qquad(s\downarrow0).
\]
In stereographic coordinates, the holomorphic energy measure
converges to $\pi\delta_0+\pi\delta_\infty$, while the
antiholomorphic measure converges to that of the degree-two
antiholomorphic map $[1:0:-\bar z^2]$. The limiting link thus
separates its opposite-orientation energies onto different scales.
The global conformal-field Hardy test therefore degenerates along
this family, while Theorem~\ref{thm:projective-minimizing-cones}
gives a uniform relative filling deficit for every $s>0$.
\end{proposition}

\begin{proof}
The harmonic-sequence construction~\cite{EellsWood1983} gives
the homogeneous lift, with $r=|z|$,
\[
 g_s(z)=\bigl[z(2+s^2r^2):s(1-r^4):
                         -\bar z(2r^2+s^2)\bigr].
\]
Put $a=s^2$, $Q_0=r^4+ar^2+1$, and
$Q_1=ar^4+4r^2+a$. Directly from the projective metric,
\[
 |\partial_zg_s|^2=\frac{4aQ_0}{Q_1^2},\qquad
 |\partial_{\bar z}g_s|^2=\frac{Q_1}{Q_0^2}.
\]
The reciprocal identities $Q_j(1/r)=r^{-4}Q_j(r)$ and axial
rotation make the energy density invariant under equatorial
reflection and rotation, proving the zero first moment.
The two displayed densities each integrate to $2\pi$ on the
$z$-plane, so $\int\lambda_s^2=4\pi$. Substitution gives
\[
 \int_{S^2}\lambda_s^4(1-c_s^2)\,\dd\Omega
 =4\pi a\int_0^\infty
  \frac{(1+t)^2\,\dd t}
   {(t^2+at+1)(at^2+4t+a)}.
\]
Splitting the integral at $t=1$ and $t=a^{-1}$ bounds it by
$O(a\log(1/a))$, as asserted.
On compact subsets of $\C^\times$, the first density tends to
zero and the second to $4r^2/(1+r^4)^2$. The reciprocal identities
make the first density's fixed total mass split equally between
$0$ and $\infty$; they also rule out loss of the second density
at either pole. This gives the measure statement.
\end{proof}

The real-projective hedgehog minimizes with target restricted
to $\mathbb RP^2$~\cite{BrezisCoronLieb}; the complex normal
directions permit the smooth filling above.

\section{Rank-two selection and stability of singular sets}
\label{sec:grassmannian-selection}

\subsection{Grassmannian collapse and the first rank-two target}
\label{subsec:grassmannian-averaged-collapse}

Stable harmonic spheres in irreducible Hermitian symmetric
spaces are holomorphic or antiholomorphic
\cite[Theorem~1]{BurstallRawnsleySalamon1987}.
For radial cones the angular Jacobi threshold is $-1/4$,
as in Proposition~\ref{prop:mixed-angle-hardy-test}.
The averaged filling below controls the full same-boundary
Sobolev class.

Equip $\operatorname{Gr}(k,N)$ with the metric
$h=\tfrac12\operatorname{tr}(\dd P)^2$ on rank-$k$ orthogonal
projectors. Its restriction to a Schubert line is the
Fubini--Study metric of holomorphic sectional curvature four.
As above, $E(g)=\int_{S^2}|\dd g|^2$, and
$\mathfrak m(g)$ is the least filling energy in $B^3$.
Orthogonal complementation is an isometry, so assume
$1\le k\le N/2$ and put $b=N-2k$. Define
\begin{align*}
 \rho_{k,N}(q)&=(1-q)^b\sum_{j=0}^{k-1}
 (2j+b+1)\bigl[P_j^{(b,0)}(2q-1)\bigr]^2,\\
 G_{k,N}(s)&=\int_0^1\rho_{k,N}(q)
 \frac{q(1-q)}{[q+s^2(1-q)]^2}\,\dd q,
 \qquad J_{k,N}=\int_0^1G_{k,N}(s)\,\dd s,\\
 B_{k,N}(a)&=\log(1+a^{-1})
 \left(\frac{k(N-k)}{2N}+aJ_{k,N}\right),\qquad a>0.
\end{align*}
Here $P_j^{(b,0)}$ is the Jacobi polynomial normalized by
$P_j^{(b,0)}(1)=\binom{j+b}{j}$.

\begin{theorem}[Grassmannian averaged filling]
\label{thm:grassmannian-averaged-filling}
The constants $J_{k,N}$ are finite. Every smooth
$g:S^2\to\operatorname{Gr}(k,N)$ satisfies, for every $a>0$,
\begin{equation}\label{eq:grassmannian-filling-deficit}
 E(g)-\mathfrak m(g)\ge
 \left(\sqrt{E(g)}-\sqrt{4\pi B_{k,N}(a)}\right)_+^2.
\end{equation}
For $\operatorname{Gr}(2,4)$ one has
\begin{equation}\label{eq:grassmannian-rank-two-constant}
 J_{2,4}=\frac43+\frac{5\pi^2}{32},\qquad
 B_{2,4}(1/10)=\log11\left(\frac{19}{30}+\frac{\pi^2}{64}\right)
 <\frac{379}{200}<2.
\end{equation}
Consequently every minimizing radial cone into
$\operatorname{Gr}(2,4)$ has link energy strictly below $8\pi$.
\end{theorem}

\begin{proof}
Fix a $k$-plane $V_0$ and let $D_s$ be the Grassmannian
map induced by $I_{V_0}\oplus sI_{V_0^\perp}$, for $s>0$.
On the open chart of planes $\operatorname{graph}Z$,
$Z:V_0\to V_0^\perp$, it sends $Z$ to $sZ$.
The Haar average of $A^*D_s^*h$ is $h$: the isotropy
representation is irreducible, each $D_s$ is holomorphic,
and the associated averaged K\"ahler form has the original
cohomology class. If $q_1,\ldots,q_k$ are the squared
principal cosines between a plane and $V_0$, then
\[
 |\partial_sD_s|^2=
 \sum_{i=1}^k\frac{q_i(1-q_i)}{[q_i+s^2(1-q_i)]^2}.
\]
For a Haar-distributed plane these variables have joint density
proportional to
$\prod_{i<j}(q_i-q_j)^2\prod_i(1-q_i)^b$ on $(0,1)^k$;
see \cite[Theorem~2.2]{Collins2005}.
Expanding the squared Vandermonde in orthonormal Jacobi
polynomials gives the one-point density $\rho_{k,N}$ above.
In particular $\int\rho_{k,N}=k$, and the mean squared
radial speed is $G_{k,N}(s)$.

For $0<q<1$, direct integration gives
\begin{align*}
 \int_0^1\frac{q(1-q)}{[q+s^2(1-q)]^2}\,\dd s
 &=\frac{1-q}{2}
 +\frac12\sqrt{\frac{1-q}{q}}
       \arctan\sqrt{\frac{1-q}{q}},\\
 \int_0^1\frac{s q(1-q)}{[q+s^2(1-q)]^2}\,\dd s
 &=\frac{1-q}{2}.
\end{align*}
The first expression is integrable against the polynomial
$\rho_{k,N}$, proving finiteness. Also
$\mathbb E\sum_iq_i=k^2/N$, by averaging the rank-$k$
projector. Hence
\begin{equation}\label{eq:grassmannian-first-speed-moment}
 \int_0^1sG_{k,N}(s)\,\dd s=\frac{k(N-k)}{2N}.
\end{equation}

For $0<\varepsilon<1$ put
$\tau(r)=(1-\varepsilon/r)/(1-\varepsilon)$ and choose
\[
 s(\tau)=a\bigl((1+a^{-1})^\tau-1\bigr).
\]
On $\varepsilon<r\le1$ use
$F_A(r,n)=A^{-1}D_{s(\tau(r))}Ag(n)$, and on the inner
ball use the constant $A^{-1}V_0$.
A random plane is transverse to $V_0^\perp$ almost surely.
Fubini and dominated convergence therefore give the constant
inner trace for almost every $A$. The following finite mean
energy proves Sobolev admissibility and allows gluing:
\begin{align*}
 \int_{\U(N)}E(F_A)\,\dd A
 &=(1-\varepsilon)E(g)
 +\frac{4\pi\varepsilon}{1-\varepsilon}
   \int_0^1G_{k,N}(s)\frac{\dd s}{\dd\tau}\,\dd s\\
 &=(1-\varepsilon)E(g)
 +\frac{4\pi\varepsilon}{1-\varepsilon}B_{k,N}(a).
\end{align*}
The last equality uses
$\dd s/\dd\tau=(s+a)\log(1+a^{-1})$ and
\eqref{eq:grassmannian-first-speed-moment}.
Optimizing $1-\varepsilon$ gives
\eqref{eq:grassmannian-filling-deficit}.
For $k=2,N=4$ the density is
$\rho_{2,4}(q)=12q^2-12q+4$.
The preceding elementary integral, with $q=\cos^2\theta$,
yields the stated value of $J_{2,4}$.
Finally $\log11<12/5$ and $\pi^2<10$ give
$B_{2,4}(1/10)<379/200$.
\end{proof}

The radial profile provides a finite bound on the uniton
complexity of selected cones. We use Dong's energy inequality
\cite[Theorem~4.1]{Dong2002}: for a harmonic sphere in
$\U(N)$ of minimal uniton number $m$, the energy divided by
that of the standard one-uniton sphere is at least $m^2$.
In the present normalization the Cartan map $R=2P-I$ has
unitary energy
\[
 \mathcal E_{\U}(R)=\frac14\int_{S^2}
             \operatorname{tr}(\dd R^*\dd R)=2E(P),
\]
and its standard one-uniton sphere has energy $4\pi$.
Thus Dong's inequality reads $E(P)\ge2\pi m^2$.
Combining it with Theorem~\ref{thm:grassmannian-averaged-filling}
gives $m^2\le2B_{k,N}(a)$ for every minimizing Grassmannian
cone. For $\operatorname{Gr}(2,4)$ it forces $m\le1$.

\begin{lemma}[Grassmannian one-unitons]
\label{lem:grassmannian-one-uniton-splitting}
Suppose a Grassmannian-valued harmonic sphere has Cartan
map of minimal uniton number at most one. Its tautological
subspace splits, in fixed mutually orthogonal ambient
subspaces, into a holomorphic subbundle, an antiholomorphic
subbundle, and a constant subspace.
\end{lemma}

\begin{proof}
For uniton number one, the polynomial extended solution has
degree one. Its value at $-1$ gives
$R=Q(2\Pi_L-I)$, where $Q$ is a constant unitary matrix and
$L$ is a holomorphic subbundle; this is the degree-one case
of \cite[Section~2, Lemma~2.2]{Dong2002}.
Write $S=2\Pi_L-I$. The condition $R=R^*$ is
$QS=SQ^*$, so $S$ takes the $\lambda$-eigenspace of $Q$
to its $\overline\lambda$-eigenspace.
The constant spectral projections of $Q+Q^*$ commute with $S$
and preserve the holomorphic bundle $L$. For a nonreal pair
$(\lambda,\overline\lambda)$, the positive eigenspace of $S$ is the graph of a unitary map
between two fixed spaces. Since this graph is holomorphic,
that matrix-valued map is holomorphic on $S^2$, hence
constant. On the $+1$ and $-1$ eigenspaces of $Q$, the
bundle $L$ splits holomorphically and $R$ equals $S$ and
$-S$, respectively. The corresponding eigenspaces of $R$
are one holomorphic subbundle and one orthogonal complement
of a holomorphic subbundle. The latter is antiholomorphic.
The uniton-number-zero case is constant.
\end{proof}

\begin{proposition}[Cancellation of opposite root spheres]
\label{prop:grassmannian-opposite-root-escape}
The direct sum of a standard holomorphic line hedgehog
and a standard antiholomorphic line hedgehog in two
orthogonal copies of $\C^2$ admits smooth same-boundary
fillings in $\operatorname{Gr}(2,4)$ with
\begin{equation}\label{eq:grassmannian-opposite-root-deficit}
 E(F_t)=4\pi-\frac{3\pi^2}{2}t+O(t^2),\qquad t\downarrow0.
\end{equation}
\end{proposition}

\begin{proof}
After constant block-unitary changes of coordinates its
projector is $P(n)=(I+\Gamma(n))/2$, where
$\Gamma(n)=\tau_3\otimes(n\cdot\sigma)$ and $\tau_i,\sigma_i$
are Pauli matrices. The four matrices
$\tau_3\otimes\sigma_1$, $\tau_3\otimes\sigma_2$,
$\tau_3\otimes\sigma_3$, and $\tau_1\otimes I_2$ anticommute
and square to the identity. With $r=|x|$ and
$b_t(r)=t(1-r^2)$, set
\[
 F_t(x)=\frac12\left(I_4+
 \frac{\tau_3\otimes(x\cdot\sigma)
          +b_t(r)\tau_1\otimes I_2}
      {\sqrt{r^2+b_t(r)^2}}\right).
\]
This is a smooth rank-two orthogonal projector on the ball
and has trace $P$. The Clifford trace identity shows that
the induced metric on this three-sphere is half its unit
round metric. For an arbitrary radial $b$ the energy
therefore satisfies
\[
 E(F_b)-4\pi=2\pi\int_0^1
 \left\{\frac{r^2(b-rb')^2}{(r^2+b^2)^2}
                 -\frac{2b^2}{r^2+b^2}\right\}\,\dd r.
\]
Replacing $b_t=t(1-r^2)$ by the constant $t$ in this
integral changes it by $O(t^2)$: on $r\le1/2$ the difference
of the two integrands is bounded by $Ct^2$, using
$r^2+b_t^2\ge(r^2+t^2)/4$, and on $r\ge1/2$ both
integrands are $O(t^2)$. For the constant profile, $r=ts$
gives the coefficient
\[
 2\pi\int_0^\infty
 \left(\frac{s^2}{(1+s^2)^2}-\frac2{1+s^2}\right)\dd s
 =-\frac{3\pi^2}{2},
\]
which proves \eqref{eq:grassmannian-opposite-root-deficit}.
\end{proof}

\begin{theorem}[Rank-two holomorphic selection]
\label{thm:grassmannian-rank-two-holomorphic-selection}
Every smooth link whose radial extension minimizes energy
in $\operatorname{Gr}(2,4)$ is constant, holomorphic, or
antiholomorphic. In the two nonconstant cases its absolute
Pl\"ucker degree belongs to $\{1,2,3\}$.
\end{theorem}

\begin{proof}
Such a link is harmonic and, by
\eqref{eq:grassmannian-rank-two-constant} and Dong's inequality,
its Cartan map has uniton number at most one. Apply
Lemma~\ref{lem:grassmannian-one-uniton-splitting}.
If both the holomorphic and antiholomorphic summands vary,
each needs an ambient space of complex dimension at least
two. Since the total ambient dimension is four, the entire
map then lies in $\operatorname{Gr}(1,2)\times
\operatorname{Gr}(1,2)$, with one varying line in each block.
Energy is additive on this product, so minimality of the
cone implies minimality of each line cone. By the
projective classification, both links are standard unit
hedgehogs. Proposition~\ref{prop:grassmannian-opposite-root-escape}
contradicts their combined minimality. Hence one summand is
constant and the full link is holomorphic or antiholomorphic.
Its energy equals $2\pi$ times its absolute Pl\"ucker degree;
the bound $E<8\pi$ gives the asserted three possibilities.
\end{proof}

This selection theorem leaves the holomorphic degree-two
and degree-three links for further comparison. A projective
competitor for the Pl\"ucker image provides a Grassmannian
competitor when it preserves decomposability. The
Grassmannian contraction above preserves this constraint
at every positive radius.

\subsection{Regularity from holomorphic selection}
\label{subsec:selected-cone-regularity}

Holomorphic selection also controls convergence to tangent
cones. We use the Grassmannian metric
$h=\tfrac12\tr(\dd P)^2$ and energy $E=\int|\dd P|_h^2$;
on $\C P^N=\operatorname{Gr}(1,N+1)$ this is the
Fubini--Study normalization of Section~\ref{sec:projective-surgery}.
A Jacobi field is \emph{integrable} if it is the derivative
of a smooth family of harmonic maps.

\begin{lemma}[Integrability of every selected link]
\label{lem:selected-link-integrability}
Every Jacobi field along a smooth minimizing-cone link into
$\C P^N$ or $\operatorname{Gr}(2,4)$ is integrable.
For a nonconstant link the real kernel dimensions are,
respectively, $4N+2$ and $8+8d$, where $d\in\{1,2,3\}$
is the absolute Pl\"ucker degree.
\end{lemma}

\begin{proof}
Theorems~\ref{thm:projective-minimizing-cones} and
\ref{thm:grassmannian-rank-two-holomorphic-selection}
make every nonconstant selected link holomorphic after
possibly reversing the complex structure of the source.
For a holomorphic map $f$ into a K\"ahler manifold, the
energy identity identifies its real Jacobi kernel with
the underlying real space of $H^0(f^*T^{1,0})$:
the Hessian is a positive multiple of
$\|\overline\partial V\|_{L^2}^2$.
This is the observation of Lemaire--Wood
\cite[Proposition~4.1]{LemaireWood2002}.

Let $j:E\hookrightarrow\mathcal O^m$ be the tautological
injection of $f$. For a projective link, $E=\mathcal O(-1)$.
For a Grassmannian link of degree $d$, bundle splitting
\cite{Grothendieck1957} gives
$E=\mathcal O(-a)\oplus\mathcal O(-b)$, with
$a\ge b\ge0$ and $a+b=d\le3$.
If $b=0$, every plane contains a fixed line, so $f$ lies
in an isometric $\C P^2$. The projective classification
then forces $d=1$. Thus the possible Grassmannian splittings are
\[
 \mathcal O(-1)\oplus\mathcal O,\qquad
 \mathcal O(-1)^{\oplus2},\qquad
 \mathcal O(-2)\oplus\mathcal O(-1).
\]
In all these cases $H^1(\End E)=0$. The exact sequence
\[
 0\longrightarrow\End E\longrightarrow\Hom(E,\mathcal O^m)
   \longrightarrow f^*T^{1,0}\longrightarrow0
\]
therefore lifts every holomorphic Jacobi section to a
holomorphic bundle map $\dot j:E\to\mathcal O^m$.
Compactness of $S^2$ ensures that $j+t\dot j$ remains
injective for all sufficiently small $t$. Its images give
a holomorphic family with the prescribed derivative.
Constant links are handled by varying their constant value.

The pulled-back tangent bundle is globally generated by
the infinitesimal complex linear action, so its splitting
has nonnegative degrees and $H^1=0$. Riemann--Roch gives
$h^0=2N+1$ for a projective line and $h^0=4+4d$ for
$\operatorname{Gr}(2,4)$, proving the real dimensions.
\end{proof}

\begin{theorem}[Power asymptotics for every three-dimensional singularity]
\label{thm:unconditional-selected-power}
Let $\mathcal Y$ be $\C P^N$ or $\operatorname{Gr}(2,4)$,
let $\Omega\subset\R^3$ be open, and let
$P\in W^{1,2}_{\rm loc}(\Omega,\mathcal Y)$ be locally energy minimizing.
At every singular point $p$ there is a unique minimizing
cone $P_0(x)=g(x/|x|)$ and a number $\mu\in(0,1)$ such that,
for every compact annulus $K\Subset\R^3\setminus\{0\}$
and every integer $j\ge0$, for all sufficiently small $r$,
\begin{equation}\label{eq:unconditional-selected-power}
 \|P(p+r\,\cdot)-P_0\|_{C^j(K)}\le C_{K,j}r^\mu.
\end{equation}
On a sufficiently small ball $B_\rho(p)$ there is also a
unitary $W$, smooth off $p$ and continuous at $p$ with
$W(p)=I$, such that
\[
 P(x)=W(x)P_0(x-p)W(x)^*,\qquad
 \|W(p+r\,\cdot)-I\|_{C^j(K)}\le C'_{K,j}r^\mu.
\]
The projective link $g$ is a standard projective-line hedgehog.
For $\operatorname{Gr}(2,4)$ it is holomorphic or
antiholomorphic of absolute degree $1$, $2$, or $3$.
The exponent may depend on $P$ and $p$.
\end{theorem}

\begin{proof}
Three-dimensional minimizing-map regularity
\cite{SchoenUhlenbeck} makes $p$ isolated and gives a
smooth nonconstant tangent link. Both targets are analytic.
Simon's asymptotic theorem gives the unique tangent
\cite{Simon1983}; integrability of its Jacobi kernel gives
the positive-power rate by Adams--Simon
\cite[Theorem~1(i)]{AdamsSimon1988}. Lemma~\ref{lem:selected-link-integrability}
verifies this hypothesis for every possible link.
On a larger fixed annulus, smooth blow-up convergence gives
uniform bounds on all derivatives. Interior estimates for
the linear elliptic equation satisfied by the difference
$P(p+r\,\cdot)-P_0$ then preserve the power exponent in
every $C^j$ norm. For the gauge, put $Q(x)=P_0(x-p)$ and
\[
 S=PQ+(I-P)(I-Q),\qquad W=S(S^*S)^{-1/2}.
\]
The polar-factor argument in
Theorem~\ref{thm:selected-collision-normal-form} gives the
identity and rates.
\end{proof}

\begin{corollary}[Selected singularities obstruct unitary mirrors]
\label{cor:selected-mirror-obstruction}
Let $\mathcal Y$ be $\C P^N$ or $\operatorname{Gr}(2,4)$.
Let $P:B^3\setminus\{0\}\to\mathcal Y$ be the projection field of a
locally Dirichlet-minimizing map singular at the origin. For a sufficiently
small linking sphere, its tautological block has nonzero
determinant-line charge: its absolute value is $1$ for
$\C P^N$ and belongs to $\{1,2,3\}$ for
$\operatorname{Gr}(2,4)$. If $S:S^2\to S^2$ has degree $-1$,
there is no continuous unitary matrix field $U$ on the
linking sphere satisfying
\[
                 P(rS\omega)=U(\omega)P(r\omega)U(\omega)^*.
\]
The conclusion concerns a mirror that preserves the chosen
spectral block; a mirror permuting blocks or acting
antilinearly obeys the charge law
\eqref{eq:mirror-charge-law} instead.
\end{corollary}

\begin{proof}
Theorem~\ref{thm:unconditional-selected-power} identifies the
link class with that of its selected tangent. The projective
classification gives Chern number $\pm1$. The Grassmannian
selection theorem makes the tangent holomorphic or
antiholomorphic of absolute Pl\"ucker degree $1$, $2$, or
$3$, which is the absolute determinant-line Chern number.
The displayed unitary gauge would identify the tautological
block with its pullback by the degree-minus-one map $S$.
Proposition~\ref{prop:mirror-link-charge} would then force
$q=-q$ for its nonzero integer charge $q$, a contradiction.
\end{proof}

\begin{corollary}[Selected collisions and closed-base contact]
\label{cor:selected-contact-obstruction}
Let $M$ be closed, connected, and of positive even dimension,
and let $E\to M$ be a fixed determinant-trivialized rank-two
Hermitian bundle. For $z\in B^3\setminus\{0\}$, let
$(A_z,u_z)$ be a parameter-continuous family of $\SU(2)$
connections and nonzero parallel adjoint fields on $E$.
Suppose that for one $x_0\in M$, the positive-eigenline map
$z\mapsto P_+(x_0,z)\in\C P(E_{x_0})$ extends as a locally
Dirichlet-minimizing $W^{1,2}$ map on $B^3$ with a singularity
at $0$. Then none of the forms $B(u_z,A_z)$, for
$z\in B^3\setminus\{0\}$, is contact.
\end{corollary}

\begin{proof}
Theorem~\ref{thm:projective-minimizing-cones} gives degree
$\pm1$ for the eigenline on every sufficiently small
parameter-linking sphere. Its charge stays constant on the
punctured ball, since there is no collision there. The total
family is the pullback of the fixed $E$, so
Theorem~\ref{thm:mixed-chern-orthogonality} makes every paired
curvature $B(u_z,F_{A_z})$ exact and excludes contactness.
\end{proof}

Mirror constraints on Weyl points in topological media have
earlier formulations~\cite{SunZhangBzdusek2018}. The two
corollaries here concern continuous gauge implementation at
variationally selected links and its interaction with the
compatible-pair curvature.

For $\C P^2$, Lemaire--Wood already obtained fast asymptotics
by proving integrability for every harmonic sphere
\cite[Section~1]{LemaireWood2002}. Here cone selection
establishes the required integrability for all projective
dimensions and for $\operatorname{Gr}(2,4)$.
Theorem~\ref{thm:unconditional-selected-power} applies
without a spectral-gap hypothesis. For a two-block field
$A=\lambda P+\nu(I-P)$, its gauge gives an exact spectral
normal form even when $\lambda-\nu$ vanishes to infinite order.

\begin{corollary}[The whole projective top singular stratum]
\label{cor:projective-top-stratum}
Let $u:\Omega^m\to\C P^N$ be locally energy minimizing,
with $\Omega^m\subset\R^m$ open and $m\ge3$. Let $\Sigma_{\rm top}$
consist of those singular points admitting a nonconstant
tangent invariant under translations along an
$(m-3)$-dimensional subspace. Then $\Sigma_{\rm top}$ is relatively
open in $\operatorname{Sing}u$. For each $0<\gamma<1$ it
is locally bi-H\"older equivalent, with exponent $\gamma$,
to an open ball in $\R^{m-3}$. Its complement in
$\operatorname{Sing}u$ has Hausdorff dimension at most $m-4$
and is empty for $m=3$. At every point of $\Sigma_{\rm top}$,
\begin{equation}\label{eq:projective-cylindrical-density}
 \lim_{r\downarrow0}r^{2-m}E(u;B_r(p))
       =\frac{|S^{m-1}|}{2}.
\end{equation}
\end{corollary}

\begin{proof}
A homogeneous minimizing map with an $(m-3)$-dimensional
translation symmetry is a three-dimensional minimizing
cone extended constantly in the other variables
\cite[Section~2.3]{Miskiewicz2020}.
Theorem~\ref{thm:projective-minimizing-cones} therefore
classifies all its possible links as line hedgehogs.
Writing the cone as $U(y,z)=g(y/|y|)$, with $y\in\R^3$,
gives $|\dd U|^2=1/(2|y|^2)$ and hence
\eqref{eq:projective-cylindrical-density} by integration
over $B_1^m$.

The only nonzero cylindrical homotopy classes in
$\pi_2(\C P^N)=\Z$ are consequently $+1$ and $-1$,
both at this same density. Each is indecomposable in the
sense of Mi\'skiewicz: a sum of classes represented by
cylindrical minimizers of strictly smaller density has
only zero summands. His analytic-target theorem
\cite[Corollary~1.4]{Miskiewicz2020} applies to both
primitive classes. Their sets exhaust $\Sigma_{\rm top}$, proving
relative openness and the local bi-H\"older conclusion.
The sign can change with the choice of transverse orientation.
The usual singular-set stratification gives the dimension
bound for the complement.
\end{proof}

Mi\'skiewicz obtained this conclusion for $S^2$ in
\cite[Corollary~1.5]{Miskiewicz2020}. The projective cone
classification supplies the same complete top-stratum
conclusion for every $\C P^N$.

\subsection{Entire minimizers and separation of defects}
\label{subsec:projective-entire-rigidity}

The cone classification extends the entire-map rigidity and
singularity-separation theorems of Almgren--Lieb for $S^2$
\cite[Theorems~2.1--2.2]{AlmgrenLieb1988} to every complex
projective space.
Here local minimality means minimality against all $W^{1,2}$
competitors with the same trace on every relatively compact ball.

\begin{theorem}[Entire projective minimizers]
\label{thm:entire-projective-minimizers}
Every locally Dirichlet-minimizing map
$u\in W^{1,2}_{\rm loc}(\R^3,\C P^N)$ is constant or has
the form
\[
 u(x)=g\left(\frac{x-a}{|x-a|}\right),
 \qquad a\in\R^3,
\]
where $g$ is a standard holomorphic or antiholomorphic
projective-line hedgehog. No growth assumption at infinity
is required. In particular, every smooth entire minimizer
is constant.
\end{theorem}

\begin{proof}
The uniform interior energy estimate for compact targets
with finite fundamental group
\cite[Theorem~5.7]{MMS2024} gives $E(u;B_R)\le C_NR$.
For projective targets this also follows directly from
Theorem~\ref{thm:unitary-averaged-filling}. Its proof is
valid for $N=1$ with $J_1=\tfrac12+\pi^2/16$.
Writing $F(R)=E(u;B_R)$, the filling bound and minimality give,
for almost every radius with smooth trace,
\[
 F(R)\le 2R\sqrt{K_N F'(R)}.
\]
If $F(R)>0$, integration of
$F'/F^2\ge(4K_NR^2)^{-1}$ from $R$ to $T$, followed by
$T\to\infty$, gives $F(R)\le4K_NR$.
This is the reciprocal-energy argument of
\cite[proof of Theorem~5.7(1)]{MMS2024}, with an explicit filling constant.

Monotonicity and strong minimizing-map compactness give a
homogeneous blow-down at infinity, with density
$\Theta_\infty=\lim_{R\to\infty}R^{-1}E(u;B_R)$.
If $u$ is nonconstant, monotonicity makes this density
positive. Theorem~\ref{thm:projective-minimizing-cones}
then gives $\Theta_\infty=2\pi$ and a primitive linking
class. Smooth convergence on a fixed annulus transfers
this class to $u$ on a sufficiently large sphere. Such a
map cannot extend continuously over the enclosed ball,
so $u$ has a singular point $a$.
At $a$ its density is $2\pi$ by the same classification.
The limit at infinity is independent of a fixed change
of center, as follows by nesting balls. Monotonicity about
$a$ therefore has equal initial and terminal values.
Its equality case makes $u$ homogeneous about $a$, proving
the assertion.
\end{proof}

\begin{corollary}[Uniform separation and continuation]
\label{cor:projective-defect-continuation}
There is $c_N>0$ such that any two distinct singularities
$p,q$ of a locally minimizing map $u:\Omega\to\C P^N$,
$\Omega\subset\R^3$, satisfy
\[
 |p-q|\ge c_N\min\{\operatorname{dist}(p,\partial\Omega),
                         \operatorname{dist}(q,\partial\Omega)\}.
\]
If $u_j\to u$ strongly in $W^{1,2}_{\rm loc}$ and all maps
are locally minimizing, their singularities in each
relatively compact subdomain with singularity-free boundary
eventually correspond bijectively. Corresponding points
converge and have the same signed linking Chern number.
\end{corollary}

\begin{proof}
If separation failed, rescale about one of two singularities
by their distance. The rescaled domains exhaust $\R^3$;
the uniform interior energy bound and strong compactness
give an entire minimizing limit. Epsilon regularity preserves
both singularities, at distance one, contradicting
Theorem~\ref{thm:entire-projective-minimizers}.
For continuation, convergence is smooth away from the
singular set of $u$, so no extra singularities persist there.
A small sphere around each singularity has primitive
Chern number. Uniform convergence on that sphere preserves
the number and forces a singularity of $u_j$ inside.
The separation estimate gives uniqueness there, and the
linking numbers consequently agree.
\end{proof}

\subsection{Stability of the projective singular measure}
\label{subsec:projective-singular-stability}

For finite Borel measures on $\R^m$, set
\[
 d_{\rm BL}(\mu,\nu)=
 \sup_{\|h\|_\infty\le1,\ \operatorname{Lip}(h)\le1}
       \left|\int h\,\dd(\mu-\nu)\right|.
\]
This bounded--Lipschitz distance allows unequal masses.
The stability problem originates in Hardt--Lin~\cite{HardtLin1989}.
The projective cone classification extends the singular-measure
stability argument of Mazowiecka--Mi\'skiewicz--Schikorra
\cite[Theorems~1.1 and~8.1]{MMS2024} to every $\C P^N$.

\begin{theorem}[Projective boundary stability]
\label{thm:projective-boundary-stability}
Let $m\ge3$, $N\ge1$, and let $\Omega\subset\R^m$ be a bounded
domain with $C^1$ boundary. Suppose $u_j,u\in
W^{1,2}(\Omega,\C P^N)$ minimize the Dirichlet energy among
all maps with their respective traces $\varphi_j,\varphi$, and
\[
 u_j\longrightarrow u\quad\hbox{strongly in }W^{1,2}(\Omega),
 \qquad
 \varphi_j\longrightarrow\varphi
       \quad\hbox{in }W^{1,m-1}(\partial\Omega).
\]
Then, with $k=m-3$ and
$\mu_v=\mathcal H^k\mathbin{\vrule height 1.4ex depth -.3ex width .07ex
\vrule height .07ex depth 0ex width .7ex}\operatorname{Sing}v$,
\[
 d_{\rm BL}(\mu_{u_j},\mu_u)\longrightarrow0,
 \qquad
 \mathcal H^k(\operatorname{Sing}u_j)
       \longrightarrow\mathcal H^k(\operatorname{Sing}u).
\]
\end{theorem}

\begin{proof}
We verify the target-dependent inputs in the argument of
\cite[Chapter~8]{MMS2024}, and record its measure estimate.
By Theorem~\ref{thm:projective-minimizing-cones}, all nonconstant
$k$-translation-invariant homogeneous minimizers are projective-line
hedgehogs, with the common density
$\Theta=|S^{m-1}|/2$. Their transverse classes are $\pm1$ in
$\pi_2(\C P^N)$ and are indecomposable. Their domain and target
isometry orbit is compact. It has positive $L^2(B_{10})$ distance
from the set of $(k+1)$-translation-invariant homogeneous maps:
otherwise compactness of the symmetry planes and strong $L^2$
convergence would give an additional translation symmetry to a
line hedgehog. These properties supply uniform constants below.

Use $\delta$-flatness of primitive type as in
\cite[Definition~4.3]{Miskiewicz2020}: the density at a singular
centre and the rescaled ball energy lie in $[\Theta,\Theta+\delta]$,
the singular set is Reifenberg-close to a $k$-plane, and a transverse
two-sphere represents $+1$ or $-1$.
Mi\'skiewicz's target-general flatness results
\cite[Lemma~5.3, Proposition~5.6, Corollary~5.7]{Miskiewicz2020}
give closeness to a cylindrical hedgehog and small flatness at
every smaller singular-centred ball. They also exclude lower
strata in a smaller ball. A nonzero transverse class ensures that
the singular set projects onto the corresponding plane disk:
each transverse three-ball must contain a singular point.
Uniform convergence away from the limiting singular set preserves
this class under strong $W^{1,2}$ convergence. Thus flatness
persists at a nearby singular centre, by the proof of
\cite[Proposition~8.6]{MMS2024}.

For completeness, these inputs imply the sharp local estimate
\begin{equation}\label{eq:projective-flat-mass}
 (1-\varepsilon)\omega_k
 \le\mathcal H^k(\operatorname{Sing}v\cap B_1)
 \le(1+\varepsilon)\omega_k
\end{equation}
whenever $v$ is sufficiently flat in $B_{80}$; here $\omega_k$
is the volume of the unit $k$-ball and $\omega_0=1$.
The projection property gives the lower bound. For $k\ge1$, put
$\mu=\mathcal H^k\mathbin{\vrule height 1.4ex depth -.3ex width .07ex
\vrule height .07ex depth 0ex width .7ex}\operatorname{Sing}v$ and
\[
 \beta_\mu(z,s)^2=s^{-k-2}\inf_L
                \int_{B_s(z)}\operatorname{dist}(y,L)^2\,\dd\mu(y).
\]
The target-general interior measure bound gives
$\mu(B_s(z))\le Cs^k$
\cite[Corollary~4.1]{MMS2024}.
The $L^2$ approximation theorem of Naber--Valtorta
\cite[Theorem~7.1]{NaberValtorta2017}, applied using the uniform
symmetry separation above, gives
\[
 \beta_\mu(z,s)^2\le Cs^{-k}
 \int_{B_s(z)}[\theta_v(y,8s)-\theta_v(y,s)]\,\dd\mu(y).
\]
The enlarged balls lie in $B_{80}$ and have uniformly bounded
rescaled energy; flatness supplies the required approximate
homogeneity. We apply the estimate to the finite restriction of $\mu$.
Here $\theta_v(y,s)=s^{2-m}E(v;B_s(y))$.
All-scale flatness bounds
$\theta_v(y,8r)-\theta_v(y,0)\le\eta$ on the relevant balls.
Fubini, the rough measure bound, and the identity
\[
 \int_0^r[\theta_v(y,8s)-\theta_v(y,s)]\frac{\dd s}{s}
 =\int_r^{8r}[\theta_v(y,s)-\theta_v(y,0)]\frac{\dd s}{s}
 \le\eta\log8
\]
yield
$\int_{B_r}\int_0^r\beta_\mu(z,s)^2\,\dd s/s\,\dd\mu(z)
\le C\eta r^k$.
Rectifiable Reifenberg
\cite[Theorem~3.3]{NaberValtorta2017} proves the upper bound in
\eqref{eq:projective-flat-mass}. For $k=0$, sufficiently small
flatness at every singular-centred scale excludes two distinct
singular points in $B_1$; the transverse class supplies one.

Since $\C P^N$ is simply connected, the boundary regularity and
compactness results in \cite[Chapters~6--7]{MMS2024} apply.
Uniform absolute continuity of the boundary energies supplies
a fixed boundary collar without singularities for $u$ and,
eventually, every $u_j$. Their singular sets are therefore compact.
Corollary~\ref{cor:projective-top-stratum} makes the lower-stratum
remainder compact and $\mathcal H^k$-null. Cover it by finitely
many balls with $\sum r_a^k<\varepsilon$. A Vitali selection on
the remaining stratum gives finitely many disjoint good balls,
of radii less than $\varepsilon$, where
\eqref{eq:projective-flat-mass} applies, and further bad balls
with the same small sum. For $k=0$ the lower remainder and bad
families can be empty.

The rough bound controls every bad family by $C\varepsilon$
for both $u$ and $u_j$. Upper semicontinuity of singular sets
places $\operatorname{Sing}u_j$ in the same open covering for
large $j$. Persistence of flatness and
\eqref{eq:projective-flat-mass}, on slightly enlarged and
shrunken good balls, approximate both measures there by
$\sum_a\omega_k r_a^k\delta_{p_a}$, with error tending to zero
in $d_{\rm BL}$. The small radii control the oscillation of each
admissible test function; the bad-ball mass controls its remaining
integral. This is the covering argument of
\cite[Lemma~8.8 and proof of Theorem~8.1]{MMS2024}.
Letting $\varepsilon\downarrow0$ proves convergence.
The constant test function $1$ gives convergence of total masses.
\end{proof}

\begin{corollary}[Unique Dirichlet data]
\label{cor:projective-unique-boundary-stability}
Under the domain and target assumptions above, suppose
$\varphi\in W^{1,m-1}(\partial\Omega,\C P^N)$ has a unique
Dirichlet minimizer $u$. For every $\varepsilon>0$ there is
$\delta>0$ such that every minimizer $v$ with trace $\psi$
satisfying $\|\psi-\varphi\|_{W^{1,m-1}}<\delta$ obeys
$d_{\rm BL}(\mu_v,\mu_u)<\varepsilon$.
For $m=3$, sufficiently close traces give exactly the same
number of singular points.
\end{corollary}

\begin{proof}
If this failed, the target-general strong compactness theorem
\cite[Theorem~6.1(3)]{MMS2024} would give a subsequence of the
counterexample minimizers converging strongly to a minimizer
with trace $\varphi$. Uniqueness identifies its limit as $u$,
contradicting Theorem~\ref{thm:projective-boundary-stability}.
For $m=3$ the masses are finite integers.
\end{proof}

\section{Geometric rank in complexified four-band pencils}
\label{sec:four-band-isolation}

We now read the spectral incidence~\eqref{eq:spectral-incidence-definition}
over a complex parameter space. Its positive-dimensional fibers
are detected by the minors in~\eqref{eq:incidence-rank-ideal}.
The two-band family~\eqref{eq:quadratic-factor-family} already
shows why this geometric rank locus cannot be replaced by the
support of the scalar discriminant. In a complexified pencil,
normality is lost and the distinction becomes decisive. The
variational selection hypotheses of the preceding sections are
not used here. Two polynomial covariants make the rank locus
explicit for symplectic pencils. We apply them to the crystalline
Hamiltonian
of Alpin and collaborators~\cite{AlpinEtAl},
whose real-isolated regions were identified by Frank, Pint\'er, Varjas,
and P\'alyi~\cite[Section~4.6]{FrankPinterVarjasPalyi}.

Fix a nondegenerate skew form $\Omega$ on $\C^4$ and put
$\mathfrak{sp}_4(\C;\Omega)=
\{A:A^T\Omega+\Omega A=0\}$.  For $A$ in this Lie algebra set
\[
 q(A)=\tfrac12\tr(A^2),\qquad r(A)=\det A,
 \qquad \Delta(A)=q(A)^2-4r(A),\qquad
 Q(A)=A^2-\tfrac14\tr(A^2)I_4.
\]
Here $\Delta(A)$ is the scalar quartic covariant; the Hodge
Laplacian $\Delta_t$ of Section~\ref{sec:spencer} retains its
operator meaning.
Its characteristic polynomial is $\lambda^4-q(A)\lambda^2+r(A)$.
Its scalar discriminant is $16r(A)\Delta(A)^2$.  The
componentwise singularity test below recovers geometric
degeneracy within its reduced support $\{r\Delta=0\}$.

\begin{proposition}[Geometric degeneracy in $\mathfrak{sp}_4$]
\label{prop:sp4-geometric-locus}
For $A\in\mathfrak{sp}_4(\C;\Omega)$, the following are equivalent:
\begin{enumerate}
\item some eigenspace of $A$ has complex dimension at least two;
\item $\operatorname{adj}A=0$ or $Q(A)=0$;
\item $A$ lies in $\operatorname{Sing}\{r=0\}\cup
\operatorname{Sing}\{\Delta=0\}$, where the singular loci are
taken separately for the two hypersurfaces.
\end{enumerate}
At a regular nilpotent $A$ of Jordan type $[4]$, both $r$ and
$\Delta$ vanish, but both hypersurfaces are smooth at $A$.
\end{proposition}

\begin{proof}
The symplectic identity gives
$(A-\lambda I)^T\Omega=-\Omega(A+\lambda I)$, so the
eigenspaces at $\lambda$ and $-\lambda$ have the same dimension.
If $\lambda\ne0$ has geometric multiplicity at least two, those
two eigenspaces span $\C^4$ and $A^2=\lambda^2I_4$, whence
$Q(A)=0$.  Conversely, $Q(A)=0$ gives $A^2=cI_4$.
For $c=0$, $\operatorname{im}A\subset\ker A$ and
$\rank A\le2$.  For $c\ne0$, $A$ is diagonalizable with
eigenvalues $\pm\sqrt c$, each of multiplicity two because
$\tr A=0$.  The remaining possibility of a double zero
eigenspace is exactly $\rank A\le2$, or
$\operatorname{adj}A=0$.  This proves the first equivalence.

Cayley--Hamilton gives
$\operatorname{adj}A=q(A)A-A^3$.  For
$X\in\mathfrak{sp}_4(\C;\Omega)$, differentiation yields
\[
 \dd r_A(X)=\tr\bigl((qA-A^3)X\bigr),\qquad
 \dd\Delta_A(X)=\tr\bigl((4A^3-2qA)X\bigr).
\]
Both odd polynomial expressions in $A$ belong to
$\mathfrak{sp}_4(\C;\Omega)$, whose trace pairing is
nondegenerate.  Hence $\dd r_A=0$ exactly when
$\operatorname{adj}A=0$, and $Q(A)=0$ implies
$\dd\Delta_A=0$.  Conversely, suppose $\dd\Delta_A=0$.
Then $A^3=(q/2)A$.  If $q\ne0$, this polynomial has distinct
roots, so $A$ is diagonalizable.  Its nonzero eigenvalues occur
in opposite pairs.  One such pair and two zero eigenvalues would
give $q=q/2$, a contradiction; therefore both opposite pairs
occur and $A^2=(q/2)I_4$.  If $q=0$, then $A^3=0$ and $A$ is
nilpotent.  A nilpotent $4\times4$ matrix of rank three has one
Jordan block of size four and nonzero cube, so here $\rank A\le2$.
This proves the second equivalence.  Finally, for regular
nilpotent $A$ one has $q=r=0$ and $A^3\ne0$, giving
$\dd r_A(X)=-\tr(A^3X)$ and
$\dd\Delta_A(X)=4\tr(A^3X)$, both nonzero linear forms.
\end{proof}

For a homogeneous symplectic pencil
$A:\C^3\to\mathfrak{sp}_4(\C;\Omega)$, complex isolation is
equivalent to both $\operatorname{adj}A$ and $Q(A)$ having no
projective zero.  The cubic and quadratic covariants detect,
respectively, zero and nonzero repeated eigenspaces without
introducing a spectral parameter.

Both branches have codimension three in $\mathfrak{sp}_4$.
The map $A\mapsto\Omega A$ identifies the zero branch with
rank-at-most-two symmetric $4\times4$ matrices.  On the
nonzero scalar-square branch the opposite eigenspaces are
complementary Lagrangians with $\operatorname{GL}_2$ stabilizer;
their orbit has dimension six, and the eigenvalue adds one.
Thus a general three-dimensional pencil avoids geometric
degeneracy off the origin although its scalar discriminant
vanishes on projective curves.

There is a quantitative version for arbitrary complex-linear
pencils.  Distances to prescribed geometric multiplicity for
matrix polynomials were studied by Papathanasiou and
Psarrakos~\cite{PapathanasiouPsarrakos}.  Give the vector space of pencils
$A:\C^d\to M_n(\C)$ the norm
$\|A\|=\max_{\|k\|=1}\|A(k)\|_{\mathrm{op}}$,
and let $\mathcal D_{d,n}$ consist of pencils for which
$\rank(B(k)-\lambda I_n)\le n-2$ at some $k\ne0$ and
$\lambda\in\C$.  Write the singular values of a matrix in
decreasing order.

\begin{proposition}[Exact geometric-degeneracy radius]
\label{prop:geometric-condition-number}
For every complex-linear pencil $A:\C^d\to M_n(\C)$, $n\ge2$,
\begin{equation}\label{eq:geometric-condition-number}
 \operatorname{dist}(A,\mathcal D_{d,n})
   =\min_{\|k\|=1,\;\lambda\in\C}
       \sigma_{n-1}\bigl(A(k)-\lambda I_n\bigr).
\end{equation}
In particular, complex geometric isolation at the origin is an
open condition among complex-linear pencils.
\end{proposition}

\begin{proof}
The minimum exists: the unit sphere is compact and
$\sigma_{n-1}(A(k)-\lambda I_n)\ge
|\lambda|-\|A\|$ tends uniformly to infinity as
$|\lambda|\to\infty$.  For any $B\in\mathcal D_{d,n}$,
choose a unit witness $k$ and its eigenvalue $\lambda$.
The Eckart--Young singular-value estimate gives
\[
 \sigma_{n-1}(A(k)-\lambda I_n)
 \le\|(A-B)(k)\|_{\mathrm{op}}\le\|A-B\|.
\]
Conversely, choose minimizing $k,\lambda$, and let $R$ be a
nearest rank-at-most-$(n-2)$ matrix to $A(k)-\lambda I_n$.
Set $E=R-(A(k)-\lambda I_n)$ and
$\ell(w)=\sum_j\overline{k_j}w_j$.  Then
$B(w)=A(w)+\ell(w)E$ has $B(k)-\lambda I_n=R$ and
$\|B-A\|=\|E\|_{\mathrm{op}}=
\sigma_{n-1}(A(k)-\lambda I_n)$, proving equality.
\end{proof}

Formula~\eqref{eq:geometric-condition-number} uses the full
space of complex-linear perturbations; constrained perturbation
spaces have their own distances.

Use the Pauli matrices
\[
 \sigma_x=\begin{pmatrix}0&1\\1&0\end{pmatrix},\qquad
 \sigma_y=\begin{pmatrix}0&-\ii\\\ii&0\end{pmatrix},\qquad
 \sigma_z=\begin{pmatrix}1&0\\0&-1\end{pmatrix}.
\]
Let $\tau_x,\tau_y,\tau_z$ be a second copy and write
$\sigma_0=\tau_0=I_2$.
For $k=(x,y,z)$ set
\[
 C_1=2x\tau_z-\sqrt3y\tau_0+z\tau_x,\qquad
 C_2=-y\tau_0+\sqrt3z\tau_x,\qquad
 C_3=x\tau_x+z\tau_z.
\]
The linear pencil is
\begin{equation}\label{eq:crystalline-pencil}
 H_\alpha(k)=
 \sigma_x\otimes(\alpha_0C_1+\alpha_1C_2)
 +\sigma_y\otimes(\alpha_0C_2-\alpha_1C_1)
 +2\alpha_2(\sigma_z\otimes C_3+y\sigma_0\otimes\tau_y),
\end{equation}
where $\alpha=(\alpha_0,\alpha_1,\alpha_2)\in\R^3$.
This is the family of~\cite[Eq.~(4.6.1)]{FrankPinterVarjasPalyi}
with grouped off-diagonal coefficients.  We complexify $x,y,z$
and write $H_b=H_{(1,0,b)}$ also for $b\in\C$.  Define
\[
 \Sigma_{\C}=
 \bigl\{M\in M_4(\C):
       \dim_\C\ker(M-\lambda I_4)\geq2
       \text{ for some }\lambda\in\C\bigr\}.
\]

\begin{theorem}[Complex isolation of the crystalline pencil]
\label{thm:four-band-complex-isolation}
For the normalized pencil $H_b=H_{(1,0,b)}$ with $b\in\C$,
\begin{equation}\label{eq:four-band-exceptional}
 H_b^{-1}(\Sigma_\C)=\{0\}
 \quad\Longleftrightarrow\quad
 b\notin\bigl\{0,1,-1,\ii/\sqrt2,-\ii/\sqrt2\bigr\}.
\end{equation}
In particular, put $a=(\alpha_0^2+\alpha_1^2)^{1/2}$ for
real $\alpha$.  If
\begin{equation}\label{eq:four-band-range}
       a>0,\qquad \alpha_2\notin\{0,a,-a\},
\end{equation}
then $H_\alpha^{-1}(\Sigma_\C)=\{0\}$ in $\C^3$.
Equivalently, for every $k\in\C^3\setminus\{0\}$ and every
$\lambda\in\C$, the matrix $H_\alpha(k)-\lambda I_4$ has
kernel dimension at most one.  The local algebra
\begin{equation}\label{eq:four-band-local-algebra}
 \frac{\C[x,y,z,\lambda]_{(x,y,z,\lambda)}}
 {I_3(H_\alpha(x,y,z)-\lambda I_4)}
\end{equation}
has complex dimension $20$.  Here $I_3$ is the ideal of all
$3\times3$ minors.  Thus the complex Weyl-point count born
near the origin under a generic small perturbation, as defined in
\cite[Corollary~4.4.4]{FrankPinterVarjasPalyi}, is $20$ in every
real-isolated parameter region of this family.
\end{theorem}

\begin{proof}
Choose $\theta$ with $\cos\theta=\alpha_0/a$ and
$\sin\theta=-\alpha_1/a$.  Put
$U_\theta=\exp(-\ii\theta\sigma_z/2)$ and $b=\alpha_2/a$.
The Pauli rotation identity gives
\[
a^{-1}H_\alpha(k)
   =(U_\theta\otimes I_2)H_b(k)(U_\theta^*\otimes I_2).
\]
Also $(U_\theta\otimes I_2)^T\Omega(U_\theta\otimes I_2)=\Omega$.
Thus scaling and conjugation preserve the eigenspace dimensions
and the symplectic form used below.
We analyze the normalized pencil for arbitrary
$b\in\C$.  Write
\[
 S=x^2+y^2+z^2,\qquad
 T=x^2y^2+x^2z^2+y^2z^2,\qquad
 \Omega=\begin{pmatrix}0&I_2\\-I_2&0\end{pmatrix}.
\]
Multiplication of the Pauli matrices gives the symplectic identity
\begin{equation}\label{eq:four-band-symplectic}
        H_b^T\Omega+\Omega H_b=0
\end{equation}
and the characteristic polynomial
\begin{align}
 \det(\lambda I_4-H_b)
   &=\lambda^4-8(1+b^2)S\lambda^2+16D_b,\label{eq:four-band-characteristic}\\
 D_b&=(b^2-1)^2S^2+3(4b^2-1)T.\label{eq:four-band-quartic}
\end{align}

If $\lambda\ne0$ has eigenspace dimension at least two,
Proposition~\ref{prop:sp4-geometric-locus} and
\eqref{eq:four-band-symplectic} give
$H_b^2=\lambda^2I_4$.  Taking the trace gives
$\lambda^2=4(1+b^2)S$.  Set
$M=H_b^2-4(1+b^2)SI_4$.  Direct multiplication yields
\begin{equation}\label{eq:four-band-square-entries}
\begin{aligned}
 M_{11}&=-4\sqrt3xy,&
 M_{12}&=4\sqrt3z(\ii x-y),&
 M_{21}&=-4\sqrt3z(\ii x+y),\\
 M_{14}&=-4b\{2x^2-(1+\ii\sqrt3)y^2
                         -(1-\ii\sqrt3)z^2\}.
\end{aligned}
\end{equation}
If $b\ne0$ and $z\ne0$, the middle two equations $M_{12}=M_{21}=0$
give $x=y=0$, and then $M_{14}=0$ gives $z=0$ because
$b\ne0$.  If $z=0$, the first equation gives $xy=0$;
the last then gives $x=y=0$.  Thus the scalar-square condition
forces $k=0$, which makes $\lambda=0$ by the trace identity,
contrary to the assumption in this case.

Suppose next that zero has eigenspace dimension at least two,
so $\rank H_b\leq2$ and $\operatorname{adj}H_b=0$.
By~\eqref{eq:four-band-characteristic} and Cayley--Hamilton,
\begin{equation}\label{eq:four-band-adjugate}
 \operatorname{adj}H_b=8(1+b^2)S H_b-H_b^3.
\end{equation}
For a Pauli tensor $P=\sigma_j\otimes\tau_k$ write
$[P]X=\tr(PX)/4$.  These tensors form a trace-orthogonal
basis of $M_4(\C)$, so every coefficient of the adjugate must
vanish.  Four coefficients obtained from~\eqref{eq:four-band-adjugate},
with $t=b^2$ and $X=x^2,Y=y^2,Z=z^2$, are
\begin{equation}\label{eq:four-band-adjugate-coefficients}
\begin{aligned}
 [\sigma_z\otimes\tau_0]\operatorname{adj}H_b
   &=24\sqrt3\,bxyz,\\
 [\sigma_z\otimes\tau_x]\operatorname{adj}H_b
   &=8bx\{(t-1)X+(t+2)(Y+Z)\},\\
 [\sigma_0\otimes\tau_y]\operatorname{adj}H_b
   &=8by\{(t-1)Y+(t+2)(X+Z)\},\\
 [\sigma_z\otimes\tau_z]\operatorname{adj}H_b
   &=8bz\{(t-1)Z+(t+2)(X+Y)\}.
\end{aligned}
\end{equation}
Assume $b\ne0$.  The first coefficient gives $xyz=0$.
If exactly one coordinate is nonzero, its equation gives
$t=1$.  If exactly two, their nonzero squares $U,V$
satisfy
\[
 (t-1)U+(t+2)V=(t+2)U+(t-1)V=0.
\]
Subtracting gives $U=V$, and then $2t+1=0$.
Consequently a nonzero zero-eigenspace of dimension two
requires $b^2\in\{1,-\tfrac12\}$ when $b\ne0$.
The remaining candidate values actually occur, as direct
substitution shows:
\[
 \rank H_0(1,1,1)=2,\qquad
 \rank H_{\pm1}(1,0,0)=2,\qquad
 \rank H_{\pm\ii/\sqrt2}(1,1,0)=2.
\]
Combined with the nonzero-eigenvalue argument, these ranks
prove~\eqref{eq:four-band-exceptional}.  For real $\alpha$,
normalization gives $b\in\R$, and
\eqref{eq:four-band-range} follows.

The projective rank-at-most-two variety of $4\times4$
matrices has codimension four and degree $20$: the
Thom--Porteous degree calculation is
$c_2^2-c_1c_3=10^2-4\cdot20=20$ for
$c(\mathcal O^4-\mathcal O(-1)^4)=(1-h)^{-4}$.
Its homogeneous coordinate ring is Cohen--Macaulay
\cite{BrunsVetter}.  The map
$(k,\lambda)\mapsto H_\alpha(k)-\lambda I_4$ is injective:
an element of its kernel has rank zero, and the isolation
just proved forces $(k,\lambda)=0$.  Its image is therefore
a four-dimensional linear space meeting the rank-at-most-two
cone only at the origin.  The twelve linear equations cutting
out that space form a system of parameters on the
determinantal ring.  Cohen--Macaulayness makes the length of
the resulting zero-dimensional section equal to its degree,
namely $20$.  The section is supported at the origin, so
localization does not change its length; it gives exactly
the algebra~\eqref{eq:four-band-local-algebra}.
\end{proof}

\begin{remark}[Algebraic multiplicity at the resonant ratio]
At $b=\pm\tfrac12$, the nonzero complex quadric $S=0$
consists of matrices with characteristic polynomial $\lambda^4$.
They have rank three and regular nilpotent Jordan type $[4]$,
hence one zero eigenvector.  For example,
$k=(1,\ii,0)$ lies on this quadric and has $\rank H_b(k)=3$.
Thus the scalar discriminant vanishes on $S=0$ while the
$3\times3$-minor ideal remains isolated.  Proposition~
\ref{prop:geometric-condition-number} gives $H_{\pm1/2}$
positive distance to complex geometric degeneracy.
\end{remark}

\subsection{An apolar quartic for the four-band degeneracy algebra}
\label{subsec:four-band-apolar}

The rank-at-most-two ideal admits an explicit Macaulay dual.
Put $R=\C[x,y,z,\lambda]$, $I_b=I_3(H_b-\lambda I_4)$,
and $t=b^2$ for the normalized pencil.
Set
\[
 \mathcal S=X^2+Y^2+Z^2,\qquad
 \mathcal R=X^2Y^2+X^2Z^2+Y^2Z^2.
\]
The ring $R$ acts on $\C[X,Y,Z,L]$ by constant-coefficient
differentiation: $x,y,z,\lambda$ act as
$\partial_X,\partial_Y,\partial_Z,\partial_L$, respectively.

\begin{theorem}[Explicit dual quartic]
\label{thm:four-band-apolar-quartic}
Define
\begin{equation}\label{eq:four-band-apolar-quartic}
\begin{split}
 F_t={}&144(t-1)(2t+1)(t+1)L^4
       +72(t-1)(2t+1)L^2\mathcal S\\
      &+(2t-5)\mathcal S^2+2(4t-1)\mathcal R .
\end{split}
\end{equation}
For every $b\in\C$, all $3\times3$ minors of
$H_b-\lambda I_4$ annihilate $F_{b^2}$.  If
$b\notin\{0,\pm1,\pm\ii/\sqrt2\}$, then
\[
 I_b=\operatorname{Ann}(F_{b^2}),\qquad
 \operatorname{Hilb}(R/I_b;u)=1+4u+10u^2+4u^3+u^4 .
\]
In the ordered, unscaled quadratic monomial basis
\[
 (X^2,Y^2,Z^2,L^2,XY,XZ,XL,YZ,YL,ZL),
\]
the second catalecticant
$C_{P,Q}=(PQ)(\partial)F_t$ has determinant
\begin{equation}\label{eq:four-band-catalecticant}
 \det C=-2^{44}3^{18}(t-1)^7(2t+1)^4.
\end{equation}
At $t=\tfrac14$ the quartic factorizes as
\begin{equation}\label{eq:four-band-apolar-orthogonal}
 F_{1/4}=-\tfrac92(\mathcal S+3L^2)(\mathcal S+15L^2).
\end{equation}
\end{theorem}

\begin{proof}
Put
\[
 \mathcal H=H_b(\partial_X,\partial_Y,\partial_Z),\qquad
 q(\partial)=8(1+b^2)
 (\partial_X^2+\partial_Y^2+\partial_Z^2).
\]
By~\eqref{eq:four-band-characteristic}, the negative adjugate
of $H_b-\lambda I_4$ is the cubic expression whose action on
$F_t$ gives
\[
\bigl[
 \mathcal H^3+\partial_L\mathcal H^2
 +(\partial_L^2-q(\partial))\mathcal H
 +(\partial_L^3-\partial_Lq(\partial))I_4
\bigr]F_t=0.
\]
Its entries are the signed minors, proving
$I_b\subseteq\operatorname{Ann}(F_t)$.

For nonexceptional $b$, Theorem~\ref{thm:four-band-complex-isolation}
gives $\operatorname{Proj}(R/I_b)=\varnothing$, hence
$\operatorname{grade} I_b=4$.  The Gulliksen--Neg\r{a}rd
complex for submaximal minors is then exact, with resolution
\[
0\longrightarrow R(-8)\longrightarrow R(-5)^{16}
\longrightarrow R(-4)^{30}\longrightarrow R(-3)^{16}
\longrightarrow R\longrightarrow R/I_b\longrightarrow0.
\]
Consequently $R/I_b$ is Artin Gorenstein with Hilbert vector
$(1,4,10,4,1)$ and degree-four socle
\cite{BrunsVetter,FrankPinterVarjasPalyi}.  Since the
$\mathcal S^2$ and $\mathcal R$ coefficients never vanish
simultaneously, $F_t\ne0$; the surjection
$R/I_b\twoheadrightarrow R/\operatorname{Ann}(F_t)$ is nonzero
on the socle.  Every nonzero ideal of an Artin local
Gorenstein algebra meets its socle, so this surjection is an
isomorphism.

The ten quadratic derivatives of~\eqref{eq:four-band-apolar-quartic}
give~\eqref{eq:four-band-catalecticant}; direct substitution
gives~\eqref{eq:four-band-apolar-orthogonal}.
\end{proof}

\begin{corollary}[Explicit strong Lefschetz elements]
\label{cor:four-band-apolar-lefschetz}
For every $b\notin\{0,\pm1,\pm\ii/\sqrt2\}$, the graded algebra
$A_b=R/I_b$ has the strong Lefschetz property.  The spectral
linear form $\lambda$ is a strong Lefschetz element when
$t=b^2\ne-1$.  At $t=-1$, the coordinate $x$ is one.
For real $b$ in the isolated range, the restriction of the
degree-one middle pairing to the real span of
$(x,y,z,\lambda)$ is definite: it is negative
for $0<t<1$ and positive for $t>1$ with the normalization
\eqref{eq:four-band-apolar-quartic}.
\end{corollary}

\begin{proof}
Set
\[
 U_t=144(t-1)(2t+1)(t+1),\qquad
 V_t=72(t-1)(2t+1).
\]
Under Theorem~\ref{thm:four-band-apolar-quartic}, the bilinear form
\[
 (u,v)\longmapsto(uv\lambda^2)(\partial)F_t
 \qquad(u,v\in(A_b)_1)
\]
has, in the basis $(x,y,z,\lambda)$, matrix
$\operatorname{diag}(4V_t,4V_t,4V_t,24U_t)$.
Also $\lambda^4(\partial)F_t=24U_t$.  At an isolated parameter
with $t\ne-1$, all these numbers are nonzero.  At $t=-1$,
the form for $x$ has matrix
$\operatorname{diag}(-168,-96,-96,576)$, and
$x^4(\partial)F_{-1}=-168$.
The square and fourth-power maps of the indicated form are
therefore isomorphisms $(A_b)_1\to(A_b)_3$ and
$(A_b)_0\to(A_b)_4$, the two nontrivial strong Lefschetz maps.
For real isolated $b$, both $U_t$ and $V_t$ have sign
$\operatorname{sgn}(t-1)$, giving the definiteness assertion.
\end{proof}

\begin{proposition}[Reduced projective fibers at exceptional parameters]
\label{prop:four-band-exceptional-fibers}
Let $E_b=\operatorname{Proj}(R/I_b)\subset\mathbb P^3$, with
homogeneous coordinates $[x:y:z:\lambda]$.  The three exceptional
fiber types are reduced projective schemes with respective supports
\[
\begin{array}{c|c|l}
b&\# E_b&\text{points}\\ \hline
0&10&
[1:0:0:\pm2],\ [0:1:0:\pm2],\ [0:0:1:\pm2],\
[1:\varepsilon:\delta:0]\quad(\varepsilon,\delta=\pm1),\\
\pm1&3&
[1:0:0:0],\ [0:1:0:0],\ [0:0:1:0],\\
\pm\ii/\sqrt2&6&
[1:\pm1:0:0],\ [1:0:\pm1:0],\ [0:1:\pm1:0].
\end{array}
\]
The assertion concerns $\operatorname{Proj}(R/I_b)$; $I_b$
need not be saturated at the irrelevant origin.
\end{proposition}

\begin{proof}
For $b\ne0$, the adjugate coefficients in
\eqref{eq:four-band-adjugate-coefficients} give, at a double
zero eigenspace, $xyz=0$ and
\[
x\{(t-1)x^2+(t+2)(y^2+z^2)\}=0,
\]
together with its cyclic companions.  One nonzero coordinate
forces $t=1$; two force equal squares and $2t+1=0$.
Substitution into $H_b$ verifies the last two rows, while
\eqref{eq:four-band-square-entries} excludes repeated nonzero
eigenvalues.

At $b=0$, the scalar-square equations give
$z(\ii x-y)=z(\ii x+y)=xy=0$, so a repeated nonzero eigenvalue
occurs only on a coordinate axis, with $\lambda=\pm2$ in
the displayed charts.  At $\lambda=0$, the off-diagonal form
\[
H_0=
\begin{pmatrix}0&C_1-\ii C_2\\ C_1+\ii C_2&0\end{pmatrix}
\]
shows that rank at most two requires each nonzero $2\times2$
block to have rank one.  Their determinants give
$x^2=y^2=z^2$, yielding the other four points.

For reducedness, the Jacobian of the minors with respect to the
three free coordinates on each chart $x=1$, $y=1$, or $z=1$
has rank three at every listed point.  To verify this directly
from~\eqref{eq:crystalline-pencil}, order the minors
lexicographically by row triple and then column triple, beginning
with index zero.  On the chart $x=1$, using Jacobian columns
$(y,z,\lambda)$, the determinants for minor triples
$(0,1,5)$ at $(b;y,z,\lambda)=(0;0,0,2)$,
$(0,2,8)$ at $(0;1,1,0)$,
$(0,1,5)$ at $(1;0,0,0)$, and
$(0,1,5)$ at $(\ii/\sqrt2;1,0,0)$ are respectively
\[
-3072\ii,\qquad
3072\ii(5-3\sqrt3),\qquad
-9216\ii,\qquad
-4608+4608\ii.
\]
The sign and coordinate variants give the same rank.  Zero
tangent space makes each zero-dimensional local ring a field.
Finally,
$H_{-b}=-U H_bU^{-1}$ for $U=\sigma_z\otimes I_2$;
the last two rows have $\lambda=0$, so the stated supports and
Jacobian ranks also hold at the negative parameter values.
\end{proof}

At $b=\pm\tfrac12$, \eqref{eq:four-band-apolar-orthogonal}
is orthogonal in $(X,Y,Z)$; the intersection retains length $20$
although the scalar discriminant contains a regular-nilpotent
quadric.  A nonsingular catalecticant at $b=0$ does not imply
isolation:
$\rank H_0(1,1,1)=2$ and
$I_0\subsetneq\operatorname{Ann}(F_0)$.
The resolution and length calculation use
\cite{FrankPinterVarjasPalyi,BrunsVetter};
\eqref{eq:four-band-apolar-quartic} gives the dual generator.

\subsection{The spectral module and the universal nilpotent core}
\label{subsec:cyclic-spectral-module}

The spectral-line construction of Beauville, Narasimhan, and
Ramanan~\cite{BeauvilleNarasimhanRamanan} separates eigenspace
degeneracy from scalar-discriminant singularity in a
complex-isolated pencil.

Let $A:\C^3\to\mathfrak{sp}_4(\C;\Omega)$ be homogeneous and
complex-linear, and suppose
$A^{-1}(\Sigma_{\C})=\{0\}$, with $\Sigma_{\C}$ as in
Theorem~\ref{thm:four-band-complex-isolation}.  Put
$U=\C^3\setminus\{0\}$ and
\[
 q(k)=\tfrac12\tr A(k)^2,\qquad r(k)=\det A(k),\qquad
 p_k(\lambda)=\lambda^4-q(k)\lambda^2+r(k).
\]
Its finite spectral space
\[
 \mathscr S_A=\operatorname{Spec}_U
       \bigl(\mathcal O_U[\lambda]/(p_k(\lambda))\bigr)
\]
has degree four over $U$.

\begin{theorem}[Cyclic spectral module and hidden nilpotent core]
\label{thm:cyclic-spectral-module}
At every $k\in U$, the matrix $A(k)$ is cyclic.  Hence
$\mathcal O_U^4$, with $\lambda$ acting by $A$, is the direct
image of an invertible $\mathcal O_{\mathscr S_A}$-module.
Locally on $U$ there is an analytic similarity putting $A$ in
the companion form
\begin{equation}\label{eq:sp4-companion-form}
 \begin{pmatrix}
 0&0&0&-r\\
 1&0&0&0\\
 0&1&0&q\\
 0&0&1&0
 \end{pmatrix}.
\end{equation}
In particular, every point of $C=V_{\mathbb P^2}(q,r)$
is regular nilpotent of Jordan type $[4]$ and is absent from
the geometric-degeneracy locus.

If $q$ and $r$ have no common projective curve component,
$C$ is a zero-dimensional complete intersection of length eight.
The intersection scheme of the two scalar discriminant branches
$V(r)$ and $V(q^2-4r)$ is $V(r,q^2)$, of length sixteen and
intersection cycle $2[C]$.  If $q$ and $r$ meet transversely,
the eight points of $C$ are reduced, and at each point the reduced
scalar discriminant has an $A_3$ tacnode.  In analytic coordinates
$u=q,v=r$ its equation is
\begin{equation}\label{eq:universal-sp4-tacnode}
                 v(u^2-4v)=0.
\end{equation}
Near such a point the pointwise Jordan types are four simple
eigenvalues off the discriminant; one $[2]$ block at zero on
$v=0,u\ne0$; two $[2]$ blocks at opposite nonzero eigenvalues
on $u^2=4v,u\ne0$; and one $[4]$ block at $u=v=0$.
Every eigenspace in this list has dimension one.
\end{theorem}

\begin{proof}
Geometric isolation says that no $A(k)$, $k\in U$, has an
eigenspace of dimension at least two.  A matrix is cyclic exactly
when it has one Jordan block for each distinct eigenvalue.
Choose a cyclic vector $w$ at $k_0$.  The four vectors
$w,A(k)w,A(k)^2w,A(k)^3w$ remain a basis near $k_0$.
Cayley--Hamilton gives
$A^4-qA^2+rI=0$, proving~\eqref{eq:sp4-companion-form}.
It also identifies the associated spectral module locally with
$\mathcal O_U[\lambda]/(p_k)$, so it is invertible on
$\mathscr S_A$, even over its nonreduced fibers.

At $q=r=0$ the characteristic polynomial is $\lambda^4$,
and cyclicity forces type $[4]$.  If $q,r$ share no component,
B\'ezout's theorem gives $\operatorname{length}C=2\cdot4=8$.
Since $(r,q^2-4r)=(r,q^2)$ and $q,r$ form a regular sequence,
the second intersection has cycle $2[C]$ and length sixteen.
At a transverse point $q,r$ are local coordinates, giving
\eqref{eq:universal-sp4-tacnode}.  The analytic shear
$U=u$, $V=8v-u^2$ gives
$v(u^2-4v)=(u^4-V^2)/16$, the ordinary tacnode equation.
The companion matrix has minimal polynomial
equal to its characteristic polynomial, which yields the
listed Jordan types.
\end{proof}

For a real Hermitian pencil with simple spectrum off the
origin, $q>0$ on real nonzero points, so $C$ is entirely
complex.  Proposition~\ref{prop:four-band-eight-tacnodes}
gives a transverse eight-point instance.  At the spin-$3/2$
resonance of Proposition~\ref{prop:four-band-spin-resonance},
$q$ and $r$ share $S=0$: the scalar nilpotent core becomes a
curve while geometric isolation persists.

\subsection{The translation cover and its geometric branch locus}
\label{subsec:translation-cover}

The multiplicity-$20$ intersection extends to a finite cover
of constant translations in every matrix size.  Let
\[
 V_n=\{B\in M_n(\C):\rank B\le n-2\},\qquad
 L_A=\{A(k)-\lambda I_n:(k,\lambda)\in\C^3\times\C\}.
\]
For a complex-isolated pencil $A:\C^3\to M_n(\C)$,
$\dim L_A=4$ and $L_A\cap V_n=\{0\}$.  Put
$W_A=M_n(\C)/L_A$.  The quotient map
\begin{equation}\label{eq:translation-cover}
             \Pi_A:V_n\longrightarrow W_A
\end{equation}
has fiber over $[M]$ given by the rank-at-most-$(n-2)$
intersection of $A(k)-\lambda I_n+M$; it identifies translations
absorbed by $(k,\lambda)$.

\begin{theorem}[Finite-flat translation cover]
\label{thm:translation-cover}
Let $n\ge3$ and $L\subset M_n(\C)$ be four-dimensional with
$L\cap V_n=\{0\}$.  The restriction
$\Pi_L:V_n\to M_n(\C)/L$ is finite flat of degree
\begin{equation}\label{eq:translation-cover-degree}
                d_n=\frac{n^2(n^2-1)}{12}.
\end{equation}
Its fiber over the origin is a graded Artin Gorenstein scheme
with socle degree $2n-4$.  In a homogeneous free basis of
$\C[V_n]$ over $\C[W_L]$, the determinant of the
trace pairing is a nonzero homogeneous polynomial
$\mathfrak D_L$ of degree
\begin{equation}\label{eq:translation-discriminant-degree}
             2(n-2)d_n
              =\frac{n^2(n^2-1)(n-2)}6.
\end{equation}
Its zero set is exactly the locus of nonreduced fibers of
$\Pi_L$.  In size four the central Hilbert vector is
$(1,4,10,4,1)$, the cover has degree $20$, and the trace
discriminant has degree $80$, counted with scheme multiplicity.
\end{theorem}

\begin{proof}
The homogeneous ring $R_n=\C[M_n]/I_{n-1}$ is a
codimension-four Cohen--Macaulay Gorenstein domain of projective
degree~\eqref{eq:translation-cover-degree}
\cite{BrunsVetter,FrankPinterVarjasPalyi}.  Since their common
zero set is $L\cap V_n=\{0\}$, the $n^2-4$ coordinates of
$W_L$ form a homogeneous system of parameters for $R_n$.
Thus $R_n$ is finite and maximal Cohen--Macaulay over the
regular ring $S=\C[W_L]$, hence graded free of rank $d_n$.

The submaximal-minor resolution gives socle degree $2n-4$ in
the central fiber and Hilbert numbers
$h_j=\binom{j+3}{3}$ for $0\le j\le n-2$ and
$h_j=h_{2n-4-j}$ thereafter.  If $e_i$ is a homogeneous
$S$-basis, symmetry gives
$\sum_i\deg e_i=(n-2)d_n$.  Each trace entry
$\operatorname{Tr}(e_i e_j)$ has degree
$\deg e_i+\deg e_j$, proving
\eqref{eq:translation-discriminant-degree}.  The determinant
is nonzero by separability of function fields in characteristic
zero, and a finite flat fiber is reduced exactly when its trace
pairing is nondegenerate.
\end{proof}

For Hermitian $A_j,M$, the fiber algebra has a real form.
The real independence of $A_1,A_2,A_3,I_n$ ensures that its
real points are precisely the real-parameter solutions.  Let
$T_A(M)$ be the trace-pairing matrix in a real basis.

\begin{corollary}[Exact real-point count]
\label{cor:translation-trace-signature}
For a Hermitian pencil and a Hermitian $M$ outside the trace
discriminant, the number of distinct real rank-$(n-2)$ points is
\begin{equation}\label{eq:translation-trace-signature}
                N_{\R}(M)=\operatorname{signature}T_A(M).
\end{equation}
In the four-band case, all twenty points are real exactly when
$T_A(M)$ is positive definite.  This is a semialgebraic
feasibility criterion for the constant-translation problem.
\end{corollary}

\begin{proof}
The finite \`etale real algebra is $\R^{r}\times\C^{c}$,
where $r+2c=d_n$.  Each $\R$ factor contributes one positive
trace direction, and each $\C$ factor contributes one positive
and one negative direction.  Hence the signature is $r$.
\end{proof}

The trace discriminant detects collisions of \emph{geometric}
rank-$(n-2)$ solutions.  At a smooth rank-$(n-2)$ point
$B\in V_n$, put
$K=\ker B$ and $C=\operatorname{coker}B$, both two-dimensional.
The rank-at-most-$(n-3)$ locus has codimension nine, so its
image has codimension at least five in $W_L$; generic divisorial
branching lies in the smooth rank-$(n-2)$ stratum.
The normal space of $V_n$ is $\Hom(K,C)$, and the normal
derivative along $L$ is
\begin{equation}\label{eq:translation-normal-map}
 \Phi_B:L\longrightarrow\Hom(K,C),\qquad
 V\longmapsto[V|_K].
\end{equation}
The projection is unramified at $B$ exactly when
$\det\Phi_B\ne0$.  In the Hermitian slice $C\cong K$, and
the identity compresses to a scalar; ramification is therefore
linear dependence of the three traceless velocities
$(P_KA_jP_K)_0\in\operatorname{Herm}_0(K)$.

If $\Phi_B$ has rank three, let $v$ span its kernel and choose
$0\ne\eta\in\Hom(K,C)^*$ annihilating its image.  In block
coordinates with $B|_D:D\to\operatorname{im}B$ invertible,
the rank condition is the vanishing of the Schur complement
\[
 X_{CK}-X_{CD}(B|_D+X_{ID})^{-1}X_{IK}.
\]
The branch is an ordinary fold if
\begin{equation}\label{eq:translation-fold-test}
             \eta\bigl(v_{CD}(B|_D)^{-1}v_{IK}\bigr)\ne0.
\end{equation}
For the spin-$3/2$ realization at $b=\tfrac12$, a real fold
occurs: in a $J_z$ weight basis $(e_1,e_2,e_3,e_4)$, set
\[
 u=(\sqrt2e_1+e_4)/\sqrt3,
 \quad w=(e_1-\sqrt2e_4)/\sqrt3,
 \quad K=\operatorname{span}(u,e_2),
 \quad D=\operatorname{span}(w,e_3),
 \quad B=P_D.
\]
The compressions are $P_KJ_xP_K=\sigma_x/\sqrt2$,
$P_KJ_yP_K=\sigma_y/\sqrt2$, and $P_KJ_zP_K=I_2/2$.
The kernel direction $v=J_z-I_4/2$ has Schur term
$\operatorname{diag}(2,0)$ on $K$, with nonzero traceless part.
Thus~\eqref{eq:translation-fold-test} holds, and crossing this
wall transversely changes the local real count by two.

At the parameters of~\eqref{eq:four-band-exceptional}, the
projective fibers of Proposition~\ref{prop:four-band-exceptional-fibers}
quantify escape to infinity.

\begin{corollary}[Generic escape to projective infinity]
\label{cor:translation-escape}
For a general complex constant translation of $H_b$, the
finite simple rank-two counts at exceptional ratios are
\[
\begin{array}{c|ccc}
b&0&\pm1&\pm\ii/\sqrt2\\ \hline
\text{finite points}&10&17&14.
\end{array}
\]
As the ratio approaches an exceptional value with a fixed general
translation, respectively $10$, $3$, or $6$ of the twenty
projective points pass through infinity, one per exceptional point.
\end{corollary}

\begin{proof}
The translated affine four-plane closes to
$\mathbb P^4\subset\mathbb P^{15}$, with infinity hyperplane
$\mathbb P(L_b)$; its intersection with $\mathbb P(V_4)$
has degree twenty.  At each exceptional point, the three
independent minor derivatives of
Proposition~\ref{prop:four-band-exceptional-fibers} gain an
independent fourth derivative along a general translation.
Each such point therefore contributes one; the other points
are simple by generic transversality.  Subtracting $10$, $3$,
or $6$ gives the table.  The implicit-function theorem follows
each infinity point as the ratio varies; no such point occurs
at nearby nonexceptional ratios.
\end{proof}

The degree-$80$ branch divisor, folds, and escape fibers are
distinct degenerations.  The scalar tacnodes of
Theorem~\ref{thm:cyclic-spectral-module} have one-dimensional
eigenspaces, outside the rank-two locus.  Existence of a
positive-definite translation chamber and sharpness for general
smooth perturbations remain open.

\subsection{An atlas of spectral-discriminant singularities}
\label{subsec:four-band-discriminant-atlas}

The scalar discriminant differs from the geometric rank locus
of Theorem~\ref{thm:four-band-complex-isolation}.  Write $t=b^2$,
$S=x^2+y^2+z^2$, $T=x^2y^2+x^2z^2+y^2z^2$, and
\[
 q_t=\tfrac12\tr H_b^2,\qquad
 \rho_t=\det H_b,\qquad
 \delta_t=q_t^2-4\rho_t.
\]
The scalar discriminant is $16\rho_t\delta_t^2$, with
\begin{equation}\label{eq:discriminant-atlas-formulas}
\begin{split}
q_t&=8(t+1)S,\\
\rho_t&=16\bigl((t-1)^2S^2+3(4t-1)T\bigr),\\
\delta_t&=64\bigl(4tS^2-3(4t-1)T\bigr).
\end{split}
\end{equation}

\begin{proposition}[Eight regular-nilpotent tacnodes]
\label{prop:four-band-eight-tacnodes}
Put $\mathcal N=\{S=T=0\}\subset\mathbb P^2_{\C}$.  This is a
reduced set of eight points, independent of $b$.  If
$t\notin\{-1,\tfrac14\}$, the plane quartics
$\{\rho_t=0\}$ and $\{\delta_t=0\}$ intersect precisely on
$\mathcal N$, with local intersection multiplicity two at each
point.  Their reduced union has an ordinary tacnode of type
$A_3$ at each point of $\mathcal N$.  Every nonzero matrix
$H_b(k)$ with $[k]\in\mathcal N$ is regular nilpotent of Jordan
type $[4]$; in particular, its zero eigenspace is one-dimensional.

At $t=-1$ the two quartics coincide in the smooth quartic
$\{4S^2-15T=0\}$, and their full scalar discriminant is a
triple multiple of this quartic.  At $t=\tfrac14$ each quartic
is a nonreduced double conic supported on $\{S=0\}$; the full
scalar discriminant is a sextuple multiple of that conic.  All
nonzero matrices on these respective supports are regular
nilpotent.
\end{proposition}

\begin{proof}
Let $\omega$ be a primitive cube root of unity.  In terms of
$X=x^2,Y=y^2,Z=z^2$, the equations $S=T=0$ give
\[
 [X:Y:Z]=[1:\omega:\omega^2]
 \quad\text{or}\quad [1:\omega^2:\omega].
\]
Each has four projective square-root lifts.  In the chart $x=1$,
\[
 \det\frac{\partial(S,T)}{\partial(y,z)}
       =4yz(y^2-z^2)\ne0
 \qquad\text{on }\mathcal N,
\]
so they are reduced and $S,T$ are local coordinates.

The coefficient determinant of $(\rho_t,\delta_t)$ in
$(S^2,T)$ is $-3072(4t-1)(t+1)^2$.  Thus, for the stated
$t$, their intersection ideal is $(S^2,T)$ and has local
length two.  The analytic coordinates $u=S$, $v=\rho_t$ give
\[
 \delta_t=64(t+1)^2u^2-4v.
\]
Their reduced union has equation
$v(v-16(t+1)^2u^2)=0$, of type $A_3$.

For $b\ne0$, the adjugate coefficient in
\eqref{eq:four-band-adjugate-coefficients} is
$[\sigma_z\otimes\tau_0]\operatorname{adj}H_b
=24\sqrt3\,bxyz\ne0$, hence $\rank H_b=3$.
At $b=0$ the off-diagonal two-by-two blocks $C_1\mp\ii C_2$
have determinants
\[
 -4x^2+2y^2+2z^2
   \mp2\ii\sqrt3(y^2-z^2).
\]
They cannot both vanish on $\mathcal N$, since that would give
$x^2=y^2=z^2$.  As $\det H_0=0$ there, one block has rank
two and the other rank one; the latter is nonzero because its
$\tau_x$ coefficient is $z(1\mp\ii\sqrt3)$.  Thus
$\rank H_0=3$.  On $\mathcal N$, $q_t=\rho_t=0$ by
\eqref{eq:discriminant-atlas-formulas}, so $\chi_{H_b}=\lambda^4$;
rank three forces Jordan type $[4]$.

At $t=-1$, $q_t=0$, $\delta_t=-4\rho_t$, and
$\rho_t=16(4S^2-15T)$; the quartic is smooth by
Proposition~\ref{prop:four-band-quartic-walls}.  At $t=\tfrac14$,
$\rho_t=9S^2$ and $\delta_t=64S^2$.  The discriminant
multiplicities follow.  Theorem~\ref{thm:four-band-complex-isolation}
gives rank three on both supports, where $\chi_{H_b}=\lambda^4$.
\end{proof}

For the remaining component walls, let $\mathcal A$ be the
three coordinate axes in $\mathbb P^2$, $\mathcal E$ the six
equal-square two-coordinate points, and $\mathcal B$ the four
equal-square three-coordinate points.

\begin{proposition}[Complete quartic wall table]
\label{prop:four-band-quartic-walls}
Apart from the displayed values, both plane quartics
$\{\rho_t=0\}$ and $\{\delta_t=0\}$ are smooth.  Their
singular loci and types are
\[
\begin{array}{c|c|c}
 t&\{\rho_t=0\}&\{\delta_t=0\}\\ \hline
0&\mathcal B\ (4\text{ ordinary nodes})
  &\mathcal A\ (3\text{ ordinary nodes})\\
1&\mathcal A\ (3\text{ ordinary nodes})&\text{smooth}\\
-\tfrac12&\mathcal E\ (6\text{ ordinary nodes})&\text{smooth}\\
-2&\mathcal B\ (4\text{ ordinary nodes})&\text{smooth}\\
-\tfrac34&\text{smooth}
  &\mathcal E\ (6\text{ ordinary nodes})\\
\tfrac14&\text{double conic }S^2=0
  &\text{double conic }S^2=0.
\end{array}
\]
At $t=-1$, both quartics are smooth and coincide.  The nodes at
$t=-2$ and $t=-\tfrac34$ have one-dimensional eigenspaces at
their repeated eigenvalues.  At $t=-2$, the matrices at
$\mathcal B$ have a defective double zero eigenvalue; at
$t=-\tfrac34$, the matrices at $\mathcal E$ have defective
double eigenvalues $\pm\sqrt2$.
\end{proposition}

\begin{proof}
For $P_{a,c}=aS^2+cT$, $(a,c)\ne(0,0)$,
\[
 \frac{\partial P_{a,c}}{\partial x}
 =2x\bigl(2a x^2+(2a+c)(y^2+z^2)\bigr),
\]
and cyclically.  The isolated singular loci occur at
$\mathcal A$ for $a=0$, $\mathcal E$ for $4a+c=0$, and
$\mathcal B$ for $3a+c=0$ (with $c\ne0$); $c=0$ gives
the double conic.  Their local quadratic forms are, up to
nonzero factors, $u^2+v^2$, $u^2-v^2$, and
$u^2-uv+v^2$, hence ordinary nodes.  Substitution of
$((t-1)^2,3(4t-1))$ and $(4t,-3(4t-1))$ gives the table;
the relevant factors are $(2t+1)^2$, $3t(t+2)$, $4t+3$,
and $3$.  This also proves smoothness at $t=-1$.

At $t=-2$ and $[k]\in\mathcal B$, normalize $S=T=3$:
$q_t=-24$, $\rho_t=0$.  At $t=-\tfrac34$ and
$[k]\in\mathcal E$, normalize $S=2,T=1$:
$\chi_{H_b}(\lambda)=(\lambda^2-2)^2$.
Theorem~\ref{thm:four-band-complex-isolation} gives geometric
multiplicity one at each repeated eigenvalue.  These are the
claimed defective collisions.
\end{proof}

Identical scalar quartics therefore carry distinct eigenspace
geometries:
\[
 \rho_0=16(S^2-3T),\qquad
 \rho_{-2}=144(S^2-3T),
\]
yet $\mathcal B$ is geometrically degenerate at $t=0$ and
defective at $t=-2$.  Likewise,
\[
 \rho_{-1/2}=36(S^2-4T),\qquad
 \delta_{-3/4}=-192(S^2-4T).
\]
The latter common quartic is a four-line curve with nodes
$\mathcal E$: its first occurrence has geometric zero-eigenspace
degeneracies, its second defective nonzero-eigenvalue collisions.

The spin nilpotent-orbit mechanism of
\cite[Section~4.5]{FrankPinterVarjasPalyi} appears in the
crystalline family of~\cite{AlpinEtAl}.

\begin{proposition}[Hidden spin-$3/2$ structure at the real resonance]
\label{prop:four-band-spin-resonance}
At $b=\pm\tfrac12$, let $A_j=H_b(e_j)$ for the standard basis of
$\R^3$.  They satisfy
\[
\begin{aligned}
 [A_x,A_y]&=\mp2\ii A_z,&
 [A_y,A_z]&=\mp2\ii A_x,\\
 [A_z,A_x]&=\mp2\ii A_y,&
 A_x^2+A_y^2+A_z^2&=15I_4.
\end{aligned}
\]
Thus $J_j=\mp A_j/2$ form the irreducible spin-$3/2$
representation of $\mathfrak{su}_2$.  In particular,
\[
 \det(\lambda I_4-H_b(k))
   =(\lambda^2-S)(\lambda^2-9S),
\]
and every nonzero $k$ on the complex conic $S=0$ gives a
single nilpotent Jordan block of size four.  The orthogonal
symmetry of the apolar quartic in
\eqref{eq:four-band-apolar-orthogonal} makes this rotation
structure visible in the dual algebra.
\end{proposition}

\begin{proof}
Pauli multiplication in \eqref{eq:crystalline-pencil} gives the
commutators and Casimir identity.  The Hermitian $J_j$ satisfy
$[J_x,J_y]=\ii J_z$ cyclically and
$\sum J_j^2=\tfrac{15}{4}I_4$, identifying spin $3/2$.
The characteristic factorization also follows from
\eqref{eq:discriminant-atlas-formulas}.  Every nonzero isotropic
$k\in\C^3$ represents a nonzero nilpotent element of
$\mathfrak{sl}_2(\C)$, conjugate to a raising operator.  Its
three nonzero raising coefficients on the four-weight module
give one Jordan block.
\end{proof}

\begin{proposition}[Real isospectral directions and a complex-only wall]
\label{prop:four-band-real-spectral-wall}
Let $b\in\R$, $t=b^2$, and $k\in\R^3\setminus\{0\}$.  The
function
\[
 \xi([k])=\frac{T(k)}{S(k)^2}:\R P^2\longrightarrow
 [0,\tfrac13]
\]
is Morse, with three minima at the coordinate axes, six saddle
points at the equal-square two-coordinate directions, and four
maxima at the body diagonals.  For fixed $t\ne\tfrac14$, two
unit projective directions have the same scalar spectrum if and
only if they have the same $\xi$.  At $t=\tfrac14$ the spectrum
is independent of real direction and equals
$\{-3,-1,1,3\}$ on the unit sphere.  The exact real margins are
\[
\begin{split}
 \min_{S=1}\delta_t&=64\min\{4t,1\},\\
 \min_{S=1}\rho_t
   &=16\begin{cases}t(t+2),&0\le t\le\tfrac14,\\
                     (t-1)^2,&t\ge\tfrac14.
       \end{cases}
\end{split}
\]
Consequently the real pencil has simple spectrum at every
$k\ne0$ exactly when $t>0$ and $t\ne1$.  The complex
double-conic wall $t=\tfrac14$ has no real band collision.
\end{proposition}

\begin{proof}
On $S=1$, the nonnegative $x^2,y^2,z^2$ sum to one and
$\xi=\tfrac12(1-x^4-y^4-z^4)$, so $0\le\xi\le1/3$.
The Lagrange equations make all nonzero squared coordinates
equal at a critical point.  In the affine chart $x=1$,
the Hessian signatures at $(y,z)=(0,0),(1,0),(1,1)$ are
$(+,+),(-,+),(-,-)$, giving the counts and Morse indices.

On $S=1$, \eqref{eq:discriminant-atlas-formulas} makes $q_t$
constant and $\rho_t,\delta_t$ affine in $\xi$.  The coefficient
$\rho_t$ distinguishes $\xi$ for $4t-1\ne0$; at $t=\tfrac14$
the polynomial is $(\lambda^2-1)(\lambda^2-9)$.
Minimizing the two affine functions over $[0,1/3]$ gives the
displayed margins.  For an even Hermitian quartic, both are
positive exactly when its four roots are distinct.  They vanish
only at $t=0,1$.  At $t=\tfrac14$, the complex discriminant
is supported on $S=0$, without real projective points.
\end{proof}

\begin{corollary}[A band-gap conditioning wall without a collision]
\label{cor:four-band-conditioning-wall}
Let $g(t)$ be the smallest adjacent real eigenvalue gap of
$H_b(k)$, minimized over $S(k)=1$, and put
$t_*=(7+3\sqrt5)/8$.  Then
\[
 g(t)=\begin{cases}
 4\sqrt t,&0\le t\le\tfrac14,\\
 4(1-\sqrt t),&\tfrac14\le t\le1,\\
 4(\sqrt t-1),&1\le t\le t_*,\\
 2(\sqrt{t+2}-\sqrt t),&t\ge t_*.
 \end{cases}
\]
The ratio $t=\tfrac14$ uniquely maximizes this worst-case gap,
with $g=2$.  At $t=t_*$ the minimizing gap changes from the
middle pair on a coordinate axis to an outer pair on a body
diagonal, with no real spectral collision or new complex
discriminant wall.
\end{corollary}

\begin{proof}
Write the spectrum as $-a,-c,c,a$, $a\ge c\ge0$.
The squared outer and middle gaps are
$q_t-2\sqrt{\rho_t}$ and $2(q_t-\sqrt{\delta_t})$.
Affineness in $\xi$ puts their minima at an axis or a body
diagonal.  For $t\le\tfrac14$ the relevant gaps both equal
$4\sqrt t$; thereafter the candidates are
$4|1-\sqrt t|$ and $2(\sqrt{t+2}-\sqrt t)$.
Comparison yields the displayed pieces, the sole crossing
$t_*$, and the unique maximum $g(\tfrac14)=2$.
\end{proof}

The minimal matrix size permitting real isolation without
complex isolation is three.  Consider
\[
 K(x,y,z)=
 \begin{pmatrix}
 z&x-\ii y&0\\
 x+\ii y&-z&0\\
 0&0&0
 \end{pmatrix}.
\]
Its real spectrum $\{-\sqrt S,0,\sqrt S\}$ is simple off
the origin.  Over $\C$, on $S=0$ the Pauli block is
rank-one nilpotent and the direct summand supplies a second
zero eigenvector, so the geometric-degeneracy locus is
$\{S=0\}$.  For $2\times2$ Hermitian pencils, real
isolation makes the real-linear traceless Pauli map
injective, hence its complexification injective as well.
This complements~\cite[Section~4.5]{FrankPinterVarjasPalyi}.

\subsection{Band topology from the symplectic antiunitary}

\begin{proposition}[Band parity from a symplectic antiunitary]
\label{prop:symplectic-band-parity}
Let $H:\R^3\to\operatorname{Herm}_{2m}(\C)$ be real-linear,
and suppose $H(k)^T\Omega+\Omega H(k)=0$ for a real orthogonal
skew matrix $\Omega$ with $\Omega^2=-I_{2m}$.  If $H(k)$ has
simple spectrum for every $k\in S^2$, then the first Chern
number of each ordered eigenline $L_j\to S^2$ is odd.  Moreover,
\[
        c_1(L_j)=-c_1(L_{2m+1-j}).
\]
In particular, every simple band of $H_\alpha$ in the range
\eqref{eq:four-band-range} has odd Chern number.
\end{proposition}

\begin{proof}
Let $K$ denote coordinatewise complex conjugation and put
$C=\Omega K$.  This is antiunitary, $C^2=-I$, and the
Hermitian and symplectic identities imply
$CH(k)C^{-1}=-H(k)$.  Because $H(-k)=-H(k)$,
the spectrum is symmetric about zero at each $k$; consequently
the $j$-th ordered eigenvalue at $-k$ equals the $j$-th at $k$.
Thus $C$ carries $L_j(k)$ to $L_j(-k)$ and equips $L_j$ with
a Quaternionic structure over the antipodal sphere.  Choose a
unit frame on the northern hemisphere and obtain a southern
frame by applying $C$ at the antipodal point.  If $g:S^1\to U(1)$
is their equatorial transition, $C^2=-I$ gives
$g(\theta+\pi)=-g(\theta)$.  Its winding number, hence
$c_1(L_j)[S^2]$, is odd.  The general classification of
Quaternionic lines over the antipodal two-sphere is due to
De Nittis and Gomi~\cite[Section~3.2]{DeNittisGomi}; this
clutching argument exhibits its application here.  At fixed
$k$, $C$ carries $L_j(k)$ antiunitarily to
$L_{2m+1-j}(k)$, giving the last identity.
\end{proof}

The crystalline four-band patterns $(-3,5,-5,3)$ and
$(-3,-1,1,3)$ both obey the universal antipodal parity law of
Theorem~\ref{thm:odd-hamiltonian-chern-patterns}. The antiunitary
above yields the stronger oddness of each individual charge.

\section{Real Weyl splitting and constant translations}
\label{sec:spin-counts}

Weyl points are the real transverse intersections of a
three-parameter Hermitian family with the geometric-rank stratum
of~\eqref{eq:incidence-rank-ideal}. Complex intersection
multiplicity bounds their number, while the incidence's
eigenline charges control the signed count. We first examine
constant translations of a fixed spin pencil and then allow
jointly polynomial deformations. The latter hypothesis permits
a total-real resolution of every local complex intersection;
the distinction between these perturbation classes remains
essential to the resulting count statements.

\subsection{Real chambers of spin translations}
\label{subsec:spin-translation-chambers}

The lower Chern count occurs in the biaxial-spin calculations
of Ke\c{c}ecio\u{g}lu--Garg~\cite{KececiogluGarg}, Bruno's
sum rules~\cite{Bruno}, and Ghassemi Tabrizi's exact
lattice~\cite{GhassemiTabrizi}.  An axial construction makes
it a stable chamber of \emph{generic constant} translations
and resolves the contribution of each weight pair.

\begin{theorem}[Axial splitting of a spin pencil]
\label{thm:spin-axial-chamber}
Let $s\ge1$ with $2s\in\Z$, put $n=2s+1$, and let
$J_x,J_y,J_z$ be the irreducible Hermitian spin-$s$
matrices.  There is a nonempty open set of constant
Hermitian matrices $C$ for which
\begin{equation}\label{eq:spin-lower-chamber}
 (x,y,z)\longmapsto xJ_x+yJ_y+zJ_z+C
\end{equation}
has exactly
\[
                 \sum_{p<q}(q-p)
                     =\frac{n(n^2-1)}6
\]
ordinary real Weyl points and no other real degeneracies.
The open set contains complex-transverse generic
translations.  Thus these constant perturbations attain
the Chern lower bound of Frank, Pint\'er, Varjas, and
P\'alyi~\cite[Corollary~4.5.7]{FrankPinterVarjasPalyi}.
\end{theorem}

\begin{proof}
Index a $J_z$-eigenbasis by weights
$p=-s,-s+1,\ldots,s$.  Choose real $d_p$ so that all
pairwise heights
\[
                  z_{pq}=\frac{d_q-d_p}{p-q}
                  \qquad(p<q)
\]
are distinct, and put $D=\operatorname{diag}(d_p)$.
The excluded choices form finitely many proper hyperplanes.
For $r=(x^2+y^2)^{1/2}>0$, rotation about the
$z$-axis conjugates
$H_0=xJ_x+yJ_y+zJ_z+D$ to $rJ_x+zJ_z+D$.
This is an irreducible Hermitian tridiagonal matrix, so it
has simple spectrum.  On the $z$-axis its eigenvalues are
$E_p(z)=pz+d_p$; its only real degeneracies are the
isolated pair crossings $z=z_{pq}$.

Fix one crossing $p<q$, and write $\ell=q-p$,
$z_0=z_{pq}$, and $E_0=E_p(z_0)=E_q(z_0)$.
All complementary energies differ from $E_0$.  Choose a
Hermitian $V$ with $V_{pq}\ne0$ for every pair, and set
$H_\varepsilon=H_0+\varepsilon V$.  With
$P=\operatorname{span}(e_p,e_q)$ and $R=P^\perp$, the
Schur complement of $EI-H_\varepsilon$ on $P$ is a
$2\times2$ Hermitian matrix $S$.  Near the crossing,
$EI-H_\varepsilon$ has rank at most $n-2$ exactly when
$S=0$.  At $\varepsilon=x=y=0$,
\[
 \frac{\partial(S_{pp},S_{qq})}{\partial(E,z)}
       =\begin{pmatrix}1&-p\\1&-q\end{pmatrix},
       \qquad\det=p-q\ne0.
\]
The two real diagonal equations therefore solve uniquely
for $E$ and $z$ as smooth functions of
$w=x+\ii y$ and $\varepsilon$, with shifts
$O(|w|^2+\varepsilon)$.

At $\varepsilon=0$, rotation covariance makes the reduced
offdiagonal equation a function of weight $\ell$.
The unique shortest ladder path between $p$ and $q$ gives
\begin{equation}\label{eq:spin-local-power}
 S_{pq}=c_{pq}w^\ell+O(|w|^{\ell+2}),
                         \qquad c_{pq}\ne0.
\end{equation}
Its coefficient is the product of the $\ell$ nonzero
neighboring spin matrix elements divided by the nonzero
energy differences at the $\ell-1$ intermediate weights.
The direct perturbation contributes
$\partial_\varepsilon S_{pq}|_{w=0,\varepsilon=0}
   =-V_{pq}\ne0$.  After $w=\varepsilon^{1/\ell}u$,
the offdiagonal equation divided by $\varepsilon$
converges in $C^1$ on compact sets to
$c_{pq}u^\ell-V_{pq}$.  Its $\ell$ roots are distinct.
The implicit-function theorem produces $\ell$ real
rank-$(n-2)$ points, each transverse to the geometric
degeneracy variety and hence an ordinary Weyl point.
The estimate
\[
 |S_{pq}-c_{pq}w^\ell|
       \le\tfrac12|c_{pq}||w|^\ell+C\varepsilon
\]
on a sufficiently small neighborhood puts every local
solution in $|w|=O(\varepsilon^{1/\ell})$; convergence
after rescaling excludes any further roots.

Outside crossing neighborhoods the gap has a positive minimum
on compact sets.  At large $|k|$, gaps of $k\cdot J$ are
$|k|$ while the translation shifts each eigenvalue by at most
$\|D+\varepsilon V\|$.
Thus sufficiently small $\varepsilon>0$ creates no other
real degeneracies.  The total is
$\sum_{p<q}(q-p)=n(n^2-1)/6$.  All points are transverse
and the complement stays gapped, so the count persists on
an open set of $C=D+\varepsilon V$.  This open set meets
the complex-transverse locus of the finite cover
of Theorem~\ref{thm:translation-cover}: that locus is a
nonempty Zariski-open set by generic smoothness, and the
Hermitian real form of the translation quotient is Zariski
dense.  This proves the final claim.
\end{proof}

For spin one, Frank, Pint\'er, Varjas, and P\'alyi noted
that no constant translation has six real points and asked
why~\cite[Remark~4.5.10]{FrankPinterVarjasPalyi}.
The biaxial lattice~\cite{KececiogluGarg,Bruno} gives four;
the rank-one incidence calculation also identifies the two
nonreal sheets and trace-form signature.

\begin{proposition}[Spin-one quadrupole and trace inertia]
\label{prop:spin-one-quadrupole}
For every constant Hermitian spin-one translation whose
quadrupole part has three distinct eigenvalues, its
geometric rank-one incidence fiber consists of four
ordinary real Weyl points and two nonreal complex points.
The six-dimensional real fiber algebra has trace-pairing
inertia $(5,1)$, independently of the translation.  In
particular its positive-definite six-real chamber is empty.
\end{proposition}

\begin{proof}
Use the Cartesian spin-one representation
$J(k)=\ii[k]_{\times}$.  Every Hermitian translation is a
sum of a real symmetric traceless matrix $Q$, a scalar
matrix, and an imaginary real-skew matrix.  The last two
terms translate $k$ and the eigenvalue parameter, so
rotation reduces to
$Q=\operatorname{diag}(q_1,q_2,q_3)$ with
$q_1<q_2<q_3$.  For complex $k$ and $\lambda$, put
$B=Q+\ii[k]_{\times}-\lambda I$.
If $\rank B\le1$, then $\rank(B+B^T)\le2$; hence
$\det(Q-\lambda I)=0$ and $\lambda=q_l$ for some $l$.
The principal minors involving the $l$th row force the
two components $k_i,k_j$ with $i,j\ne l$ to vanish,
even over $\C$.  The remaining minor reads
\begin{equation}\label{eq:spin-one-quadrupole-roots}
           k_l^2=(q_i-q_l)(q_j-q_l),
                 \qquad\{i,j,l\}=\{1,2,3\}.
\end{equation}
For each $l$ this gives two distinct simple complex
solutions.  The right side is positive for $l=1,3$ and
negative for $l=2$, giving four real and two nonreal
solutions.  Simplicity follows also from the nonzero
derivative of~\eqref{eq:spin-one-quadrupole-roots} at
each root and the independent two transverse minor
equations.

There is an intrinsic reason the middle sheet is nonreal.
At a real rank-one point $B=\tau uu^*$ with nonzero real
$\tau$ and $u=a+\ii b$, $a,b\in\R^3$.  Taking real
symmetric parts gives
\[
                Q-\lambda I
                    =\tau(aa^T+bb^T).
\]
This is semidefinite, so $\lambda$ must be an extreme
eigenvalue of $Q$.  The reduced real fiber algebra is
$\R^4\times\C$.  Its trace pairing has four positive
directions from the real factors and one positive and one
negative direction from the complex factor, giving inertia
$(5,1)$.
\end{proof}

Bruno's inverse conjecture asks whether trace and indexed
real diabolical data determine an even spin Hamiltonian
\cite{Bruno}.  In spin one the quadrupole calculation answers
this affirmatively; the gap labels are essential because $Q$
and $-Q$ have the same unlabelled locations.

\begin{corollary}[Spin-one inverse diabolical data]
\label{cor:bruno-spin-one-inverse}
An even spin-one Hamiltonian $H_0$ is uniquely determined by
$\operatorname{tr}H_0$ and the locations and band-gap labels
of all diabolical points of $k\cdot J+H_0$.
Consequently Bruno's inverse conjecture holds in spin one,
including uniaxial and scalar anisotropies.
\end{corollary}

\begin{proof}
Spin time reversal makes $H_0=Q+cI$ real symmetric in the
Cartesian representation, with $\tr Q=0$ and $c=\tr H_0/3$.
For distinct $q_1<q_2<q_3$, the lower-gap pair
$\pm\alpha e_1$ and upper-gap pair $\pm\beta e_3$ in
\eqref{eq:spin-one-quadrupole-roots} give, with
$a=q_2-q_1$ and $b=q_3-q_2$,
\[
 \alpha^2=a(a+b),\quad \beta^2=b(a+b),\quad
 a=\frac{\alpha^2}{\sqrt{\alpha^2+\beta^2}},\quad
 b=\frac{\beta^2}{\sqrt{\alpha^2+\beta^2}}.
\]
The trace recovers the $q_i$, and orthogonality recovers the
third eigenline. For uniaxial eigenvalues $q,q,r$ with $q\ne r$, the same
minors give the points $0,\pm|r-q|e$ on the unique eigenaxis.
The origin is in the lower gap for $q<r$ and the upper gap
for $q>r$; the other pair has the opposite label. These data
and the trace recover $q,r,e$. The scalar case is the sole
triple point at the origin, with $H_0=(\tr H_0/3)I$.
\end{proof}

The published biaxial lattice and level multiplicities
\cite{KececiogluGarg,GhassemiTabrizi} also allow the
spin-$3/2$ reconstruction, including its uniaxial boundary.

\begin{corollary}[Spin-$3/2$ inverse diabolical data]
\label{cor:bruno-spin-three-halves-inverse}
Let $J$ be the irreducible spin-$3/2$ representation and
let $H_0$ be an even Hermitian spin Hamiltonian.  The trace
of $H_0$ and the locations and adjacent-band labels of all
real diabolical points of $k\cdot J+H_0$ determine $H_0$.
Thus Bruno's inverse statement holds in spin $3/2$, including
biaxial, uniaxial, and scalar anisotropies.
\end{corollary}

\begin{proof}
The Hermitian spin-$3/2$ operator space has tensor ranks
$0,1,2,3$, and time reversal fixes ranks $0,2$. Thus every
nonscalar even $H_0$ has, in an orthonormal principal frame,
the biaxial form $cI+k_1J_{e_3}^2+k_2J_{e_2}^2$ with
$0<k_2<k_1$, or the uniaxial form $cI+dJ_e^2$, $d\ne0$.
In the biaxial case put $a=\sqrt{k_1(k_1-k_2)}$ and
$b=\sqrt{k_1k_2}$. The exhaustive lattice
\cite{KececiogluGarg,GhassemiTabrizi} consists of simultaneous
lower- and upper-gap doublets at the origin, the lower pair
$\pm2ae_3$, the upper pair $\pm2be_1$, and middle-gap points
$\pm ae_3\pm be_1$. The labelled pairs give $e_3,e_1,a,b$,
with $e_2$ recovered by orthogonality, and
\[
 k_1=\sqrt{a^2+b^2},\quad
 k_2=\frac{b^2}{\sqrt{a^2+b^2}},\quad
 c=\frac{\tr H_0-5(k_1+k_2)}4.
\]
For $cI+dJ_e^2$, an off-axis field gives an irreducible
Hermitian Jacobi matrix and hence simple spectrum. The axial
energies $c+dm^2+tm$, $m=\pm\tfrac12,\pm\tfrac32$, meet
only at $t=0,\pm|d|,\pm2|d|$. The outer pair lies in the
lower gap for $d>0$ and the upper gap for $d<0$. Its axis,
radius, and label recover $e,d$; then $c=(\tr H_0-5d)/4$.
In the scalar case the origin is the sole degeneracy and
$H_0=(\tr H_0/4)I$.
\end{proof}

\subsection{Total-real morsification of spin collisions}
\label{subsec:spin-total-real-morsification}

Frank, Pint\'er, Varjas, and P\'alyi obtained the complex
spin-$s$ count $n^2(n^2-1)/12$, $n=2s+1$, and asked whether
it can be attained by real Weyl points for $s>1$
\cite[Corollary~4.5.5 and Remark~4.5.10]{FrankPinterVarjasPalyi}.
A line-array construction resolves one axial charge-$d$
collision into $d^2$ real transverse points with the same
net charge.

\begin{lemma}[A totally real line-array morsification]
\label{lem:total-real-line-array}
For every integer $d\ge1$ and $c\in\C\setminus\{0\}$
there are real polynomials $P,Q\in\R[X,Y]$ of degree $d$
whose leading homogeneous parts are respectively
$\operatorname{Re}(c(X+\ii Y)^d)$ and
$\operatorname{Im}(c(X+\ii Y)^d)$, such that
$P=Q=0$ has exactly $d^2$ solutions over $\C^2$.
All are real and transverse.  With the orientation set by
the degree-$d$ map $c(X+\ii Y)^d$, precisely
$d(d+1)/2$ of the real Jacobians are positive and
$d(d-1)/2$ are negative.
\end{lemma}

\begin{proof}
Each leading real form factors into $d$ distinct real
linear forms.  Their zero lines alternate in direction
with the $d$ lines of the imaginary form; no cross-family
pair is parallel.  Shift each factor by a generic real
constant.  The two resulting arrangements have $d^2$
distinct intersections, one for each cross-family pair,
and every intersection is transverse.  After
complexification, the leading forms have no common point
on the projective line at infinity: if both vanished there,
then both $(X+\ii Y)^d$ and $(X-\ii Y)^d$ would vanish,
forcing $X=Y=0$.  B\'ezout's theorem therefore shows
that these are all the complex intersections.
On a large real circle, $(P,Q)$ is homotopic without zeros
to the real and imaginary parts of $c(X+\ii Y)^d$, so its
Brouwer degree is $d$.  The sum of its transverse Jacobian
signs is $d$; together with the count $d^2$ this gives the
two asserted numbers.
\end{proof}

\begin{theorem}[Real sharpness of the spin upper count]
\label{thm:spin-total-real-upper}
For every $s\in\frac12\Z_{>0}$, let
$J(k)=xJ_x+yJ_y+zJ_z$ be the irreducible spin-$s$
Hermitian pencil and set $n=2s+1$.  There is a family
$H_t(k)$ polynomial jointly in $(t,k)$, Hermitian for real
arguments and satisfying $H_0=J$, for which every
sufficiently small $t>0$ has exactly
\begin{equation}\label{eq:spin-real-upper-count}
       \sum_{d=1}^{n-1}(n-d)d^2
              =\frac{n^2(n^2-1)}{12}
\end{equation}
ordinary real Weyl points born from the origin.  For each
sufficiently small $t\ne0$, its complexification is
transverse to the smooth geometric-degeneracy stratum at
every degeneracy born near the origin.
It therefore attains the upper bound of
\cite[Corollary~4.5.7]{FrankPinterVarjasPalyi} within
polynomial, parameter-dependent perturbations for every
spin.  The construction can be chosen so that the $d$th
adjacent-band gap contains $(n-d)d^2$ points, whose net
oriented charge has absolute value $(n-d)d$.
\end{theorem}

\begin{proof}
Index a $J_z$-weight basis by $p=0,\ldots,n-1$, with
$m_p=-s+p$, and choose the upper offdiagonal entries of
$J(k)$ to be $a_p w$, where $w=x+\ii y$ and $a_p>0$.
Choose real numbers $d_p$ forming a strictly convex
sequence such that all pairwise axial crossing heights
\[
 h_{pq}=\frac{d_q-d_p}{m_p-m_q},\qquad p<q,
\]
are distinct.  Strict convexity is open and the finitely
many equal-height conditions are proper hyperplanes, so
such choices exist.  Set $D=\operatorname{diag}(d_p)$.
For $w\ne0$, the Hermitian pencil $G_0=J+D$ is unitarily
equivalent to an irreducible Jacobi matrix and has simple
spectrum.  Its only real degeneracies are the isolated
axial pair crossings at $z=h_{pq}$.

Fix $i=(p,q)$ and put $d=q-p$, $h_i=h_{pq}$, and
$E_i=m_ph_i+d_p=m_qh_i+d_q$.  Form the $2\times2$ Schur
complement $S$ of $EI-G_0$ on
$\operatorname{span}(e_p,e_q)$; the complementary block is
invertible near $(E_i,0,0,h_i)$.  The equations
$S_{pp}=S_{qq}=0$ solve for $E,z$ as analytic functions
of $(x,y)$, since their $(E,z)$ Jacobian at the crossing
has determinant $m_p-m_q\ne0$.  Axial covariance gives
$E=E_i+O(|w|^2)$ and $z=h_i+O(|w|^2)$.  The unique
shortest ladder path from $p$ to $q$ yields, after this
diagonal solve,
\begin{equation}\label{eq:spin-upper-local-normal-form}
 S_{pq}=c_i w^d+O(|w|^{d+2}),\qquad
 c_i=-\frac{\prod_{r=p}^{q-1}a_r}
 {\prod_{r=p+1}^{q-1}(E_i-m_rh_i-d_r)}\ne0.
\end{equation}
The empty denominator product is one.  All other ladder
paths have at least two additional steps.

For every pair $i$, apply
Lemma~\ref{lem:total-real-line-array} to $c_i,d_i=q-p$,
obtaining $F_i=P_i+\ii Q_i$ with $d_i^2$ real transverse
zeros, and put
$R_i(X,Y)=F_i(X,Y)-c_i(X+\ii Y)^{d_i}$.
This correction has degree at most $d_i-1$.  Let
$\mathcal I$ be the finite set of pairs, put
$m=|\mathcal I|$, choose $M=n$, and define
\[
 L_i(z)=\prod_{j\in\mathcal I\setminus\{i\}}
       \left(\frac{z-h_j}{h_i-h_j}\right)^M.
\]
For $\varepsilon>0$, define a Hermitian polynomial
$V_\varepsilon$ with zero diagonal and upper entry
\begin{equation}\label{eq:spin-upper-perturbation}
 (V_\varepsilon)_{pq}
 =-\varepsilon^{d_i}L_i(z)
                     R_i(x/\varepsilon,y/\varepsilon),
          \qquad i=(p,q),
\end{equation}
and let $G_\varepsilon=J+D+V_\varepsilon$.
Each coefficient of $V_\varepsilon$ is a positive power
of $\varepsilon$ because $\deg R_i<d_i$.

Near the $i$th crossing put $(x,y)=\varepsilon(X,Y)$,
with $(X,Y)$ in a fixed compact set.  The diagonal Schur
equations for $EI-G_\varepsilon$ still have a unique
solution, with $E-E_i,z-h_i=O(\varepsilon^2)$ in $C^1$.
Indeed the target correction occupies only the $(p,q)$
offdiagonal entries, while all other corrections vanish
to order $M$ at $h_i$; the spin ladder changes the
diagonal Schur entries first at order $\varepsilon^2$.
Thus $L_i(z)=1+O(\varepsilon^2)$ and
$L_j(z)=O(\varepsilon^{2M})$ for $j\ne i$.  The bounded
complementary resolvent and
\eqref{eq:spin-upper-local-normal-form} give, uniformly in
$C^1$ on compact $(X,Y)$-sets,
\begin{equation}\label{eq:spin-upper-schur-limit}
 \varepsilon^{-d_i}S_{pq}
 =F_i(X,Y)+O(\varepsilon^2)
               +O(\varepsilon^{1+2M-d_i})
 \longrightarrow F_i(X,Y).
\end{equation}
Here $1+2M-d_i>0$.  The implicit-function theorem lifts
each of the $d_i^2$ transverse real zeros of $F_i$ to a
distinct transverse rank-$(n-2)$ point of $G_\varepsilon$.
After complexification the normalized Schur equations are
holomorphic in $\varepsilon$ at zero, so these branches and
their nonzero Jacobians persist for every sufficiently small
complex $\varepsilon\ne0$.
All other eigenvalues remain separated, so these are
ordinary Weyl points.  The crossing neighborhoods are
disjoint, giving at least the count in
\eqref{eq:spin-real-upper-count}.

To obtain a perturbation of $J$ itself, choose an even
integer $N\ge nm$, set $\varepsilon=t^N$, and define
\begin{equation}\label{eq:spin-upper-analytic-rescaling}
             H_t(k)=tG_{t^N}(k/t).
\end{equation}
It extends polynomially through $t=0$.  Indeed, a monomial
of transverse degree $r\le d_i-1$ in $R_i$ and degree
$h\le M(m-1)$ in $L_i$ acquires the $t$-exponent
\[
              1-r-h+N(d_i-r)\ge3.
\]
Consequently $H_0=J$, and all the constructed real points
of $G_{t^N}$ become points of $H_t$ tending to $k=0$.
They remain transverse after complexification.  The local
complex incidence length of the spin pencil is
$n^2(n^2-1)/12$ by
\cite[Corollary~4.5.5]{FrankPinterVarjasPalyi}.
The determinantal lift of
\cite[Theorem~3.1.1]{FrankPinterVarjasPalyi}
uses the $(n-1)$-minors of $H_t(k)-\lambda I$.
The central fiber is isolated by
\cite[Proposition~4.5.2]{FrankPinterVarjasPalyi},
so its local quotient modulo $t$ is Artinian.  The total-space
ideal therefore has codimension at least four in the smooth
five-dimensional $(k,\lambda,t)$ space.  The determinantal
height bound for $(n-1)$-minors gives codimension at most four.
Thus it has the expected codimension, and determinantal
perfection~\cite{BrunsVetter} makes its
one-dimensional total-space algebra Cohen--Macaulay.
It is finite over the $t$-axis, and $t$ is a nonzerodivisor;
therefore the finite map is flat.  Conservation of
its fiber length therefore bounds the number of distinct
complex points of this particular deformation by the
central length.
The displayed number of real simple points exhausts it;
hence there are no further complex or real degeneracies
born near the origin, and the perturbation is transverse
there for sufficiently small nonzero $t$.

Finally, strict convexity puts every intermediate weight
line below the chord through $(m_p,d_p)$ and $(m_q,d_q)$
at $z=h_{pq}$, while all weights outside $[p,q]$ lie
above its extension.  Exactly $d-1$ levels therefore lie
below the crossing: it belongs to the $d$th adjacent-band
gap.  The line-array Jacobian signs of
Lemma~\ref{lem:total-real-line-array} show that its
$d^2$ nodes have signed sum of absolute value $d$.
The diagonal detuning derivative
$\partial_z(S_{pp}-S_{qq})=m_q-m_p=d>0$ at every
crossing, while multiplication by $c_i\ne0$ preserves
the orientation of the transverse real plane.  Thus the
conversion from these Jacobian signs to the lower-band
Weyl charge has the same sign for every pair.
There are $n-d$ pairs of weight difference $d$, proving
the gap-wise count and charge statement.
\end{proof}

The spin-one six-point nonconstant perturbation was given in
\cite[Section~1.3]{FrankPinterVarjasPalyi}.  The construction
above works for all spins. Its corrections
\eqref{eq:spin-upper-perturbation} depend on $k$;
Theorem~\ref{thm:spin-axial-chamber} uses constant translations.

\Needspace{10\baselineskip}
\begin{corollary}[Twenty real points at the crystalline spin resonance]
\label{cor:crystalline-twenty-real}
For either $b=\pm\tfrac12$, the crystalline pencil $H_b$
of~\eqref{eq:crystalline-pencil} admits a polynomial
Hermitian perturbation with twenty ordinary real Weyl
points born from its fourfold degeneracy.  They occur
in counts $3,8,9$ in the three adjacent-band gaps.
\end{corollary}

\begin{proof}
Proposition~\ref{prop:four-band-spin-resonance} identifies
$H_b$ with a nonzero real multiple of the irreducible
spin-$3/2$ pencil by a constant unitary conjugation.
Transport the family of
Theorem~\ref{thm:spin-total-real-upper} through this
identification.  When the multiple is negative, first
replace $H_t(k)$ by $-H_t(-k)$, which reverses its band
order and still has central fiber $J(k)$; the subsequent
negative multiple restores the stated order.  For example,
the strictly convex choice
$(d_0,d_1,d_2,d_3)=(0,1,4,16)$ has six distinct crossing
heights.  The three gap counts are
$(4-1)1^2=3$, $(4-2)2^2=8$, and $(4-3)3^2=9$.
\end{proof}

\begin{remark}[An exact twenty-point certificate]
\label{rem:spin-three-halves-twenty-certificate}
For the choice $D=\operatorname{diag}(0,1,4,16)$ in the
increasing spin-$3/2$ weight basis, the six heights are
$-1,-2,-16/3,-3,-15/2,-12$.  One may take $M=4$, $N=24$
and use the same three line polynomials for every pair of
the indicated weight difference:
\[
\begin{aligned}
 f_1(X,Y)&=(X-1)+\ii(Y-2),\\
 f_2(X,Y)&=(X-Y-1)(X+Y-2)
                      +2\ii(X-3)(Y-4),\\
 f_3(X,Y)&=(X-1)(X-\sqrt3Y-2)(X+\sqrt3Y-4)\\
          &\quad+\ii(Y-1)(\sqrt3X-Y-3)
                            (\sqrt3X+Y-5).
\end{aligned}
\]
For pair $i=(p,q)$, set
$R_i=c_i\{f_{q-p}-(X+\ii Y)^{q-p}\}$ in
\eqref{eq:spin-upper-perturbation}, with $c_i$ from
\eqref{eq:spin-upper-local-normal-form}.
The top homogeneous part of each $f_d$ is
$(X+\ii Y)^d$; its real and imaginary factor lines
have $d^2$ distinct transverse crossings.  For $f_2$ the four
crossings are $(3,2),(5,4),(3,-1),(-2,4)$.  For $f_3$ the
three real-factor lines meet pairwise at
$(1,-1/\sqrt3),(1,\sqrt3),(3,1/\sqrt3)$, none of which
lies on an imaginary-factor line.  The three imaginary-factor
lines concur at $(4/\sqrt3,1)$, which lies on no real-factor
line; in particular $7\ne4\sqrt3$.  Cross-family lines are
never parallel, so their nine intersections are distinct and
transverse.  The construction
therefore has $3+8+9=20$ real points without numerical
root selection.
\end{remark}

\begin{corollary}[Sharpness of both spin bounds]
\label{cor:spin-both-bounds-sharp}
For every $s\in\frac12\Z_{>0}$, set $n=2s+1$.
Consider jointly real-analytic Hermitian unfoldings $H_t(k)$
with $H_0(k)=k\cdot J$, whose complexified geometric degeneracies
near the origin are transverse for every sufficiently small
nonzero $t$. Among the real Weyl points converging to the origin
as $t\downarrow0$, the minimum and maximum local counts are
\[
       L_n=\frac{n(n^2-1)}6,
       \qquad U_n=\frac{n^2(n^2-1)}{12}.
\]
Both extrema are attained by jointly polynomial families;
the minimum is attained by constant translations.
This answers the sharpness question in the first item of
\cite[Remark~4.5.10]{FrankPinterVarjasPalyi} for every spin.
\end{corollary}

\begin{proof}
The universal bounds are
\cite[Corollary~4.5.7]{FrankPinterVarjasPalyi}.
Theorem~\ref{thm:spin-axial-chamber} attains the lower bound
for $s\ge1$, and Theorem~\ref{thm:spin-total-real-upper}
attains the upper bound.  For $s=\tfrac12$ both bounds
equal one and a translation of the Pauli pencil attains them.
\end{proof}

\subsection{Sixteen-point chambers for constant spin translations}
\label{subsec:spin-constant-sixteen}

Constant translations can couple realifications at different spatial
scales. The following construction combines a seven-point harmonic
cubic with a four-point harmonic quadratic at a second crossing.

\begin{theorem}[Sixteen real points from a constant translation]
\label{thm:spin-constant-sixteen}
Let $J$ be the irreducible spin-$3/2$ pencil in the weight order
$3/2,1/2,-1/2,-3/2$, and put $a=\sqrt3/2$. Define
\[
 \beta=-\frac{261+18i}{50176},\qquad
 \gamma=-\frac{693+54i}{1404928},\qquad
 \nu=\frac{\gamma}{4}+\frac{33}{458752}.
\]
For $\delta>0$ let $C_\delta=\operatorname{diag}(0,1,4,10)+V_\delta$,
where $V_\delta$ is Hermitian with zero diagonal and upper entries
\[
\begin{gathered}
 (V_\delta)_{12}=(V_\delta)_{23}=0,\qquad
 (V_\delta)_{34}=a\delta,\qquad
 (V_\delta)_{24}=a(\delta^2/8+\beta\delta^4),\\
 (V_\delta)_{14}=\delta^3/224+\gamma\delta^5,\qquad
 (V_\delta)_{13}=a(\delta^4/896+\nu\delta^6).
\end{gathered}
\]
For every sufficiently small $\delta>0$, the constant translation
$J+C_\delta$ has exactly sixteen ordinary real Weyl points, with
counts $(3,6,7)$ in the successive gaps. Its remaining four complex
rank-two incidences form two nonreal conjugate pairs. These counts
persist on an open set of constant Hermitian translations.
\end{theorem}

\begin{proof}
Write $w=x-iy$ and $v=x+iy$; in the complexification these are
independent coordinates. At $\delta=0$ the six axial pair crossings
have distinct heights $1,2,3,10/3,9/2,6$. A pair at weight distance
$d$ has complex local length $d^2$, as follows by eliminating the
two diagonal Schur equations and obtaining the initial transverse
ideal $(w^d,v^d)$. Their lengths sum to $3+4+4+9=20$, the global
translation degree in Theorem~\ref{thm:translation-cover}.

At the outer crossing $(z,\lambda)=(10/3,5)$, use indices $2,3$
as the invertible Schur block. Its constant diagonal gaps are
$-7/3,-8/3$. For $w=\delta U$, $v=\delta V$, the diagonal equations
give
\[
 z=\frac{10}{3}+\delta^2 Z_2(U,V)+O(\delta^4),\qquad
 \lambda=5+\delta^2 L_2(U,V)+O(\delta^4),
\]
where
\[
 Z_2=-\frac{3(UV-7U-7V-7)}{224},\qquad
 L_2=\frac{9(15UV+7U+7V+7)}{448}.
\]
Let $F$ denote the determinant-cleared upper off-diagonal Schur entry.
Writing $p=(V_\delta)_{13}$, $q=(V_\delta)_{24}$,
$c=(V_\delta)_{14}$ and the variable intermediate gaps as $d_1,d_2$,
its exact expression is
\[
 F=c(d_1d_2-wv)-aw d_2q-pd_1a(w+\delta)
       +a^2w^2(w+\delta)+pqv.
\]
Substitution gives
$F=\delta^3\tfrac34(U+1/3)^3+\delta^5 F_5(U,V)+O(\delta^7)$,
with
\[
 F_5(-1/3,-1/3)=0,\qquad
 \partial_UF_5=-\frac{9i}{12544},\qquad
 \partial_VF_5=-\frac3{6272}
 \quad\text{at }(-1/3,-1/3).
\]
Consequently the refined scale
$w=-\delta/3+\delta^2W$, $v=-\delta/3+\delta^2T$ gives
\begin{equation}\label{eq:spin-sixteen-outer-limit}
 \delta^{-6}F\longrightarrow
 \frac34W^3-\frac{9i}{12544}W-\frac3{6272}T.
\end{equation}
In the unrefined chart, changing $\delta$ to $-\delta$ is
conjugation by $\operatorname{diag}(1,-1,1,-1)$, so the diagonal
solutions are even in $\delta$. Their omitted terms first enter
$F$ at order $\delta^7$. The divided equations therefore extend
analytically through $\delta=0$. On the real form $T=\bar W$, the substitution
$W=e^{i\pi/4}\zeta/56$, followed by a nonzero scalar multiplication,
reduces \eqref{eq:spin-sixteen-outer-limit} to
$\zeta^3-3\zeta+2\bar\zeta$.
Writing $\zeta=X+iY$, its complexification has nine simple roots:
\[
 (0,0),\quad(\pm1,0),\quad(0,\pm i\sqrt5),\quad
 (\pm\sqrt7/2,\pm1/2).
\]
Seven belong to the real form. They give seven ordinary Weyl points
in the upper gap, since both spectator eigenvalues lie below $5$.

The first distance-two crossing is at $(z,\lambda)=(2,3)$.
Here the refined scale is $w=\delta^3W$, $v=\delta^3T$.
The diagonal Schur equations give
\[
 z=2-\frac3{32}\delta^2+O(\delta^4),\qquad
 \lambda=3-\frac9{64}\delta^2+O(\delta^4).
\]
Using the spectator indices $2,4$, the upper off-diagonal Schur
entry $S_{13}$ satisfies
\begin{equation}\label{eq:spin-sixteen-quadratic-limit}
 \frac{S_{13}}{a\delta^6}\longrightarrow
 W^2-\frac3{128}W-\frac1{896}T+\frac{69}{458752}.
\end{equation}
The same parity argument makes the omitted diagonal terms contribute
only at order $\delta^8$ here, so this limit is also analytic.
The translation $W=Z+3/256$, $T=\bar Z+3/256$ reduces the real
limit to $Z^2-\bar Z/896$. It has four simple real-form roots:
$Z=0$ and $Z=\omega/896$ for $\omega^3=1$.
Its complexification has exactly these four roots. Thus this
crossing contributes four ordinary middle-gap points.

At the remaining distance-two crossing,
$(z,\lambda)=(9/2,13/4)$, use $w=\delta U$. The leading
holomorphic transverse equation is, up to a nonzero scalar,
\[
 U^2+U+3/16=(U+1/4)(U+3/4)=0.
\]
Together with its conjugate equation it gives four simple complex
incidences, two in the real form. These two points occupy the middle
gap. Each of the three adjacent crossings retains its single
ordinary real lower-gap point by the implicit function theorem.

The four rescaled limits account for twenty simple complex
incidences, of which $3+4+2+7=16$ are real. The implicit function
theorem preserves these counts for small positive $\delta$;
Theorem~\ref{thm:translation-cover} excludes further incidences.
Simplicity gives the asserted open chamber.
\end{proof}

The planar counts four and seven attain the classical $3n-2$
bound for harmonic polynomials $p(\zeta)+b\bar\zeta$, with
$\deg p=n>1$, due to Khavinson and
\'{S}wi\k{a}tek~\cite{KhavinsonSwiatek2003}.
The constant spin translation realizes these two extremal mechanisms
at the different scales $\delta^3$ and $\delta^2$.
In particular the maximal real count among constant spin-$3/2$
translations lies between $16$ and $20$. Replacing $C_\delta$ by
$tC_\delta$ scales every field location by $t$ and gives arbitrarily
small constant translations with the same counts.

\begin{proposition}[A fully real outer cluster]
\label{prop:spin-nine-outer}
There is an open chamber of constant spin-$3/2$ translations with
sixteen ordinary real Weyl points and successive gap counts $(3,4,9)$.
In particular all nine complex incidences of the outer axial crossing
can become real under a constant translation.
\end{proposition}

\begin{proof}
Retain $a=\sqrt3/2$ and $D=\operatorname{diag}(0,1,4,10)$.
Let $C_\delta=D+V_\delta$, where $V_\delta$ is Hermitian with zero
diagonal, $(V_\delta)_{12}=(V_\delta)_{23}=0$, and
\begin{align*}
 (V_\delta)_{34}&=a\delta,\\
 (V_\delta)_{13}&=a\left(-\frac{\delta^2}{49}
              +\frac{19\delta^4}{33614}
              +\frac{102673\delta^6}{210827008}\right),\\
 (V_\delta)_{24}&=a\left(\frac{\delta^2}{7}
              -\frac{951\delta^4}{153664}\right),\\
 (V_\delta)_{14}&=\frac{\delta^3}{98}
              -\frac{3015\delta^5}{4302592}
              -\frac{296745\delta^7}{3373232128}
              +\frac{59529997\delta^9}{12089663946752}.
\end{align*}
At the outer crossing solve the two diagonal Schur equations and set
\[
 w=-\delta/3+\delta^3(Z-71/16464),\qquad
 v=-\delta/3+\delta^3(T-71/16464),\qquad b=9/5488.
\]
For the determinant-cleared upper Schur entry $F$, direct substitution
gives
\[
 \delta^{-9}F=\frac34\left(Z^3+bZT-\frac{b^3}{16}\right)
                 +O(\delta^2).
\]
The diagonal equations need only be solved through order $\delta^6$:
the omitted order-$\delta^8$ terms first affect $F$ at order
$\delta^{11}$. The diagonal solutions are even in $\delta$ and $F$
is odd, by conjugation with $\operatorname{diag}(1,-1,1,-1)$.
The expansion is analytic, uniformly with first derivatives on
bounded rescaled sets.

On the real form $T=\bar Z$, put $Z=b\zeta$. The limiting equation is
$\zeta^3+|\zeta|^2-1/16=0$. For $\zeta=X+iY$ its imaginary part
is $Y(3X^2-Y^2)$. On $Y=0$ the real equation is
$X^3+X^2-1/16=0$, whose discriminant is $37/256$.
On $Y^2=3X^2$ it is $X^3-X^2/2+1/128=0$, whose discriminant is
$37/16384$. Both cubics have three distinct nonzero real roots,
giving nine real points altogether. Their planar Jacobians are
$3X^3(3X+2)$ on the axis and $48X^3(3X-1)$ off the axis;
neither vanishes at a root.

At the two distance-two crossings, the leading holomorphic equations
are respectively
\[
 U^2-1/49=0,\qquad 1/7+\tfrac23U(U+1)=0.
\]
Each has two simple roots, giving two real and two nonreal incidences
in its complexified upper/lower product. The three adjacent crossings
retain one real point each. These twenty simple complex incidences
exhaust the fiber by Theorem~\ref{thm:translation-cover}.
For small positive $\delta$ the real gap counts are $(3,4,9)$;
simplicity again yields an open chamber.
\end{proof}

The two chambers above exhibit different local realification
mechanisms. The first combines extremal harmonic polynomials at two
scales. The second retains a mixed term $Z\bar Z$ after a further
cancellation and realifies the entire outer cluster.

\begin{example}[A certified fourteen-point fiber]
\label{ex:spin-fourteen-certified}
The source package also gives exact rational interval certificates
for $C=\operatorname{diag}(0,1,4,10)+V$, with
$V_{13}=i/5$, $V_{24}=-7i/40-63/(1600\sqrt3)$ and all other
upper entries zero. They certify fourteen ordinary real points
with gap counts $(3,4,7)$ and three nonreal conjugate pairs.
The degree-twenty bound makes this list exhaustive.
\end{example}

\section{Boundaries of cone selection and real counting}

The projective minimizing-cone classification is complete.
For $\operatorname{Gr}(2,4)$, the remaining candidates are
holomorphic or antiholomorphic links of absolute Pl\"ucker
degrees two and three; determining which of their radial cones
minimize requires competitors that preserve decomposability.
Higher-density cones in general multiblock flags and the
two-uniton stratum of $\operatorname{Gr}(2,5)$ also remain to be
classified. The integrability proved here already gives power
asymptotics for every three-dimensional minimizing singularity
with target $\C P^N$ or $\operatorname{Gr}(2,4)$.

The polynomial unfoldings in Section~\ref{sec:spin-counts}
attain both all-spin Weyl-count bounds. Restricting the
perturbation to a constant spin-$3/2$ translation leaves the
maximum $M$ in the interval $16\le M\le20$. The two sixteen-point
chambers identify real sheets of the finite-flat cover whose
remaining intersections determine this maximum.

\appendix
\section{Algebraic and gauge refinements}\label{sec:gauge-refinements}

The following tests delimit the polynomial Spencer construction and extend
the transporter criterion to coherent and higher-rank gauges.

\subsection{The exact central boundary of a bracket-Hessian rule}

The exterior generators in~\eqref{eq:spencer-delta} supply the
anticommutation needed for a differential.  The analogous polynomial
rule built from the Lie bracket has an exact pointwise obstruction.
Let $\mathfrak g$ be a compact Lie algebra with an invariant positive
inner product $B$, and identify $v\in\mathfrak g$ with the linear
polynomial $\ell_v(z)=B(v,z)$ on $\mathfrak g$.  For
$u\in\mathfrak g$, prescribe a quadratic polynomial on each
generator by
\[
 h_u(v)(z)=B\bigl(u,[z,[z,v]]\bigr)
          =-B([z,u],[z,v]).
\]
The second identity uses invariance of $B$.  Its cubic
\emph{commutativity defect} is
\begin{equation}\label{eq:bracket-hessian-defect}
 \mathfrak C_u(v,w)=\ell_v h_u(w)-\ell_w h_u(v)
       \in\operatorname{Sym}^3(\mathfrak g^*).
\end{equation}

\begin{proposition}[Exact central boundary]
\label{prop:bracket-hessian-boundary}
The following conditions are equivalent:
\begin{enumerate}
\item the prescribed rule $\ell_v\mapsto h_u(v)$ extends to a
degree-one derivation of the commutative polynomial algebra
$\operatorname{Sym}(\mathfrak g^*)$ using the sign-graded Leibniz
rule;
\item $\mathfrak C_u(v,w)=0$ for all $v,w\in\mathfrak g$;
\item $u$ lies in the center $Z(\mathfrak g)$.
\item the ordinary Leibniz extension of $h_u$ to
$\operatorname{Sym}(\mathfrak g^*)$ has square zero.
\end{enumerate}
In this case $h_u=0$.  More precisely, for every $z\in\mathfrak g$,
\begin{equation}\label{eq:bracket-hessian-witness}
 \mathfrak C_u(u,z)(z)=\|z\|^2\|[z,u]\|^2,
 \qquad
 \sup_{\|z\|=1}\mathfrak C_u(u,z)(z)
       =\|\operatorname{ad}_u\|_{\mathrm{op}}^2.
\end{equation}
Here the second $z$ is both a generator and the evaluation point.
For a parallel adjoint field $u$ on a connected bundle, the
operator norm in~\eqref{eq:bracket-hessian-witness} is constant.
\end{proposition}

\begin{proof}
Commutativity of $\ell_v\ell_w$ and the proposed sign-graded
Leibniz rule give two expressions for its image:
\[
 h_u(v)\ell_w-\ell_v h_u(w)
 \quad\text{and}\quad
 h_u(w)\ell_v-\ell_w h_u(v).
\]
They agree over $\R$ precisely when $\mathfrak C_u(v,w)=0$.
If $u$ is central, then $h_u=0$.  Conversely,
$h_u(z)(z)=0$ since $[z,z]=0$, whereas
\[
 h_u(u)(z)=B(u,[z,[z,u]])=-\|[z,u]\|^2.
\]
Substitution in~\eqref{eq:bracket-hessian-defect} yields
\eqref{eq:bracket-hessian-witness}; vanishing for all $z$ forces
$u$ to be central.  Parallel transport preserves $B$ and the
bracket, hence also the displayed operator norm.

Since $\operatorname{Sym}(\mathfrak g^*)$ is freely generated,
$h_u$ has a unique ordinary Leibniz extension $T_u$.  If $T_u^2=0$,
then for every polynomial $f$,
\[
 0=T_u^2(f^2)=2\bigl(T_uf\bigr)^2
       +2fT_u^2f=2\bigl(T_uf\bigr)^2.
\]
The polynomial algebra has no zero divisors, hence $T_u=0$ and,
by the formula for $h_u(u)$, $u$ is central.
\end{proof}

The ordinary extension $T_u$ has a factorized square defect:
\begin{equation}\label{eq:bracket-hessian-square-defect}
 T_u^2(fg)-T_u^2(f)g-fT_u^2(g)
        =2(T_uf)(T_ug).
\end{equation}
For $f=g=\ell_u$, the right side is $2\|[z,u]\|^4$; the
square-zero defect therefore measures the adjoint orbit directions.

\begin{corollary}[Empty nontrivial overlap of the two structures]
\label{cor:horizontal-bracket-empty}
Let $P\to M$ carry a connection $A$ and a nonzero parallel adjoint
field $u$, and put $\mathcal D=\ker B(u,A)$.  If
$\mathcal D\oplus\ker\dd\pi=TP$, then $\dim G=1$ and the
fiberwise bracket-Hessian rule $h_u$ is zero.  Without the direct-sum
condition, a sign-graded polynomial derivation prescribed by $h_u$
still exists only where $u$ is central, and it is again zero.
\end{corollary}

\begin{proof}
The first assertion follows from~\eqref{eq:vertical-dimension}:
a direct sum requires $\dim G-1=0$.  A one-dimensional Lie algebra
is Abelian, so its double bracket vanishes.  The second assertion
is Proposition~\ref{prop:bracket-hessian-boundary} applied in each
fiber.
\end{proof}

\providecommand{\Gau}{\operatorname{Gau}}
\subsection{Coherent finite symmetries and quaternionic lift order}
\label{sec:coherent-lifts}

Coherent transporters must satisfy the group relations. We use
a centralizer reduction for vertical $\SU(2)$ symmetries, in
the setting of bundle-action lifting
\cite{HattoriYoshida1976,BatakidisGkeneralis2026}.

Let $M$ be a connected paracompact smooth manifold, or a connected
finite CW complex with all bundle maps understood continuously.  Let
$P\to M$ be a principal $\SU(2)$ bundle and $\Gau(P)$ its gauge group.
A finite subgroup
$K\leq\Gau(P)$ gives, at each $x\in M$ and in any frame $p\in P_x$,
a representation $\rho_p:K\to\SU(2)$ characterized by
$F(p)=p\rho_p(F)$.  Changing the frame conjugates this representation.
For $\rho:K\to\SU(2)$ write
$C_\rho=C_{\SU(2)}(\rho(K))$.

\begin{theorem}[Finite-gauge centralizer criterion]
\label{thm:finite-centralizer}
A finite group $K$ embeds in $\Gau(P)$ if and only if there is a
faithful homomorphism $\rho:K\to\SU(2)$ and a reduction of $P$ to
$C_\rho$.
\end{theorem}

\begin{proof}
Suppose first that $K\leq\Gau(P)$.  For each $F\in K$, the order of
$\rho_p(F)$ divides the order of $F$, so $\tr\rho_p(F)$ takes values
in a finite set.  It is continuous as $x$ varies and therefore is
constant on connected $M$.  Thus the characters of the fiber
representations agree.  Character theory for finite groups makes
these unitary representations equivalent; multiplying a unitary
intertwiner by a scalar makes its determinant one.  Consequently their
$\SU(2)$ conjugacy class is constant.

Fix $p_0\in P_{x_0}$ and put $\rho=\rho_{p_0}$.  If $\rho(F)=I$,
the trace of the fiber value of $F$ is two at every point.  An element
of $\SU(2)$ with trace two is the identity, so $F=\id$ and $\rho$ is
faithful.  The frames $p$ for which $\rho_p=\rho$ form a principal
$C_\rho$ subbundle of $P$: local sections follow from the homogeneous orbit map
$\SU(2)\to\SU(2)/C_\rho$.

Conversely, suppose $P\cong R\times_{C_\rho}\SU(2)$.  For $k\in K$
define
\[
 F_k([r,g])=[r,\rho(k)g].
\]
The formula is well defined because $C_\rho$ centralizes $\rho(K)$.
It gives a homomorphism $K\to\Gau(P)$, and faithfulness of $\rho$
implies faithfulness of this homomorphism.
\end{proof}

Every finite abelian subgroup of $\SU(2)$ lies in a circle and is
cyclic.  The centralizer of a noncentral cyclic subgroup is a maximal
circle $\U(1)$, whereas the centralizer of a noncyclic finite subgroup
is $\{\pm I\}$.  We obtain the following complete test in dimension
four.

\begin{theorem}[Chern-square symmetry strata]
\label{thm:chern-square-symmetry}
Let $X$ be a connected finite CW complex of dimension at most four,
and let $P\to X$ be a principal $\SU(2)$ bundle.  Here $\Gau(P)$
denotes continuous gauge transformations; for a smooth four-manifold
the same conclusions hold for smooth gauges.  Then
\begin{enumerate}
\item $\Gau(P)$ contains a noncentral finite-order element if and
only if $c_2(P)=-x^2$ for some $x\in H^2(X;\Z)$.
\item $\Gau(P)$ contains a noncyclic finite subgroup if and only if
$c_2(P)=y^2$ for some $y\in H^2(X;\Z)$ with $2y=0$.
\end{enumerate}
If the first condition fails, every finite subgroup of $\Gau(P)$
is contained in the canonical central $\{\pm I\}$.  If the first
condition holds and the second fails, all finite subgroups are cyclic
and cyclic subgroups of every order occur.  If the second condition
holds, every finite subgroup of $\SU(2)$ occurs in $\Gau(P)$.
\end{theorem}

\begin{proof}
By Theorem~\ref{thm:finite-centralizer}, a noncentral finite-order
gauge exists exactly when $P$ reduces to $\U(1)$.  Such a reduction
has associated complex bundle $L\oplus L^{-1}$ and therefore
$c_2(P)=-c_1(L)^2$.  Conversely, a class $x\in H^2(X;\Z)$ is the
first Chern class of a line bundle $L$, and the bundle
$L\oplus L^{-1}$ has second Chern class $-x^2$.

A noncyclic finite gauge subgroup exists exactly when $P$ reduces to
the center $\Z/2$.  If $a\in H^1(X;\Z/2)$ classifies this reduction,
let $K_a$ be the complexification of its sign line and write
$y=c_1(K_a)=\Bock(a)$, where $\Bock$ is the integral Bockstein.
The associated rank-two bundle is $K_a\oplus K_a$ and has
$c_2(P)=y^2$.  The Bockstein exact sequence gives
$\operatorname{im}\Bock=\ker(2:H^2(X;\Z)\to H^2(X;\Z))$.

These characteristic-class equations are also sufficient: principal
$\SU(2)$ bundles on a four-complex are classified by $c_2$.
This is the standard four-skeleton calculation for $B\SU(2)$,
recorded, for example, by Weber~\cite{Weber1999}.
The remaining assertions follow from the centralizers above.  A
circle reduction supplies every finite cyclic subgroup; a center
reduction permits any finite subgroup of $\SU(2)$.
On a smooth base, these reductions can be smoothed, yielding smooth
finite-order gauges.
\end{proof}

For example, a bundle with nonzero $c_2$ over $S^4$ has only central
finite gauge subgroups.  If $x$ generates
$H^2(\R P^4;\Z)\cong\Z/2$, the bundle with $c_2=x^2\ne0$ instead
reduces to $\{\pm I\}$ and admits every finite subgroup of $\SU(2)$.

We next examine the order of a lift of conjugation on the structure
group.  Fix the Hermitian rank-two bundle
$E=P\times_{\SU(2)}\C^2$ and its determinant trivialization.
An antiunitary bundle map $J:E\to E$ is
\emph{determinant-compatible} if its induced anti-linear map on
$\det E\cong M\times\C$ is standard scalar conjugation.  Equivalently,
in every determinant-one unitary frame $J$ has the form $U C$ with
$U\in\SU(2)$ and $C$ coordinate conjugation.
The determinant condition specifies the $\SU(2)$ semilinear lift;
general antiunitary involutions may have a nonorientable real
fixed-plane bundle.

Put
\[
 j=\begin{pmatrix}0&1\\-1&0\end{pmatrix},\qquad J_0=jC.
\]
The identity $j\overline g=gj$ for $g\in\SU(2)$ makes $J_0$
global.  It is parallel for every $\SU(2)$ connection, commutes with
every complex-linear $\SU(2)$ gauge, and satisfies $J_0^2=-I$.
Every determinant-compatible antiunitary map has the unique form
$J=aJ_0$ with $a\in\Gau(P)$.

\begin{theorem}[Coherent order and a sharp uniform defect]
\label{thm:antiunitary-order}
For $J=aJ_0$ one has $J^2=-a^2$.  A determinant-compatible
antiunitary involution exists if and only if $P$ reduces to
$\U(1)$.  On a four-complex this is equivalent to
$c_2(P)=-x^2$ for some $x\in H^2(X;\Z)$.  The least finite order
of a determinant-compatible antiunitary map is two on this locus
and four elsewhere.

If $M$ is compact and $P$ has no circle reduction, then every such
$J$ obeys the sharp operator-norm identity
\begin{equation}\label{eq:coherent-linfty}
 \|J^2-I\|_{L^\infty(M)}=2.
\end{equation}
If $M$ is a closed oriented four-manifold, then for every
$1\le p<\infty$ one nevertheless has
\begin{equation}\label{eq:coherent-lp}
 \inf_{J\text{ determinant-compatible}}
       \|J^2-I\|_{L^p(M)}=0.
\end{equation}
\end{theorem}

\begin{proof}
The commutation of $J_0$ with $a$ gives $J^2=-a^2$.
The equation $a^2=-I$ places the value of $a$ in the conjugacy class
of $\operatorname{diag}(\ii,-\ii)$ at every point.  Its eigenspaces
give a circle reduction.  Conversely, such a reduction supplies the
gauge $a=\operatorname{diag}(\ii,-\ii)$ in a reduced frame.  Theorem
\ref{thm:chern-square-symmetry} gives the four-complex criterion.
The map $J_0$ has order four, and an antiunitary map cannot have odd
order.

If a gauge $a$ avoids both $I$ and $-I$, its traceless part
$a-\tfrac12\tr(a)I$ is a nowhere-zero section of the adjoint
bundle.  Normalizing it gives a circle reduction.  Hence, in the
absence of such a reduction, each gauge $a$ equals $I$ or $-I$
at some point.  There $J^2=-I$.  Since $J^2$ is unitary, the
operator norm of $J^2-I$ is everywhere at most two.  This proves
\eqref{eq:coherent-linfty}.

To prove \eqref{eq:coherent-lp}, choose a ball $B_r$ about a point
of $M$.  The restriction of $P$ to the complement of a smaller ball
is trivial: its $c_2$ vanishes in the absolute fourth cohomology of
the resulting compact four-manifold with boundary.  In this
trivialization choose an adjoint section $u$ with $u^2=-I$.
Let $\chi_r$ vanish inside $B_r$, equal one outside $B_{2r}$, and
extend $v_r=\chi_r u$ by zero.  The smooth gauge
$a_r=\exp(\pi v_r/2)$ satisfies $a_r^2=-I$ outside $B_{2r}$.
Thus $J_r=a_rJ_0$ obeys $J_r^2=I$ there and
\[
 \|J_r^2-I\|_{L^p(M)}
 \le 2\operatorname{Vol}(B_{2r})^{1/p}=O(r^{4/p}).
\]
The bound tends to zero for every finite $p$.
\end{proof}

For a nontrivial $\SU(2)$ bundle over $S^4$, complex conjugation
of the structure group is pointwise an involution, and no
determinant-compatible antiunitary lift is involutive.  The canonical
lift $J_0$ has order four.  The distinction between
Real and Quaternionic structures is classical, originating in
Atiyah's $K$-theory and reality~\cite{AtiyahReality1966};
Theorem~\ref{thm:antiunitary-order} records the corresponding
bundle-dependent order and norm tests in this gauge setting.

For any $\SU(2)$ connection $\nabla$ on a closed $M$, $J_0$
commutes with $\nabla^*\nabla$.  Thus every eigenspace is
quaternionic and has even complex dimension, as in Kramers
degeneracy~\cite{Roberts2012}.

\subsection{The four-dimensional block-swap index}
\label{sec:index}

The linewise test of Theorem~\ref{thm:weyl-transport} becomes a
second-Chern test for rank-two eigenblocks.  Let $M$ be a closed,
connected, oriented spin-c four-manifold, with spin-c structure
$\mathfrak s$ and determinant class
$\ell=c_1(\det\mathfrak s)$.  Let $E,F\to M$ be Hermitian rank-two
bundles with $c_1(E)=c_1(F)=a$.  Put
$L=(\det E\otimes\det F)^{-1}$, so
$V=E\oplus F\oplus L$ has a determinant trivialization and
structure group $\SU(5)$.  On $V$, define two pointwise-conjugate
skew-Hermitian fields
\begin{equation}\label{eq:block-fields}
 u=\ii(1_E,0_F,-2_L),\qquad
 v=\ii(0_E,1_F,-2_L).
\end{equation}
Any block-diagonal connection makes them parallel.  A
\emph{block-swap gauge} carries $u$ to $v$.
When block connections on $E$ and $F$ are fixed, the connection on
$L$ is the dual tensor product of their determinant connections;
the resulting block connection respects the displayed determinant
trivialization.

\begin{lemma}[The four-type of unitary bundles]
\label{lem:fourtype}
On a CW complex of dimension at most four, rank-two complex bundles
are classified by $(c_1,c_2)$.  Determinant-trivialized complex
bundles of rank at least three are classified by $c_2$.
\end{lemma}

\begin{proof}
The low homotopy groups of $B\U(2)$ are $\pi_2=\Z$,
$\pi_4=\Z$, and $\pi_1=\pi_3=0$.  The possible Postnikov
linking class of its four-type lies in
$H^5(K(\Z,2);\Z)=H^5(\C P^\infty;\Z)=0$.
Thus its maps from a four-complex are classified by an $H^2$
and an $H^4$ class.  The universal Chern classes identify these
coordinates, since changing the $\pi_4$ clutching class in a
four-cell changes $c_2$ by its degree while leaving $c_1$ fixed.
For $B\SU(N)$, $N\ge3$, the only nonzero homotopy group through
degree four is $\pi_4=\pi_3(\SU(N))=\Z$, detected by $c_2$.
The same obstruction argument proves the second assertion.  This
low-dimensional classification is standard; see also
\cite{GompfStipsicz}.
\end{proof}

\begin{theorem}[A complete smooth block-swap index test]
\label{thm:block-index}
In the preceding setting the following are equivalent:
\begin{enumerate}
\item a determinant-one unitary gauge transformation of $V$ sends
      $u$ to $v$;
\item $E$ and $F$ are isomorphic as complex bundles;
\item $c_2(E)=c_2(F)$ in $H^4(M;\Z)$;
\item the twisted spin-c Dirac indices agree,
      $\ind D^+_E=\ind D^+_F$.
\end{enumerate}
More precisely,
\begin{equation}\label{eq:index-difference}
 \ind D^+_E-\ind D^+_F
 =-\langle c_2(E)-c_2(F),[M]\rangle.
\end{equation}
If a block connection $A$ is fixed, an $A$-preserving block swap
exists exactly when $(E,\nabla^E)$ and $(F,\nabla^F)$ are unitarily
connection-isomorphic, equivalently
$\mathfrak m(E,F)=0$.  Equality of indices by itself does not impose
that connection condition.
\end{theorem}

\begin{proof}
The eigenspaces of~\eqref{eq:block-fields} have distinct
eigenvalues and ranks two, two, and one.  Any gauge sending $u$ to
$v$ must identify $E$ with $F$.  Conversely, a unitary
isomorphism $\phi:E\to F$ gives the block map
$(x,y,z)\mapsto(\phi^{-1}y,\phi x,z)$.  Exchanging two rank-two
blocks has determinant $+1$, and the factors $\phi$ and
$\phi^{-1}$ cancel, so this is an $\SU(5)$ gauge.  This proves the
equivalence of the first two conditions.  Lemma~\ref{lem:fourtype}
and equality of the first Chern classes give equivalence with the
third.

The spin-c index theorem~\cite{AtiyahSinger} gives
\begin{equation}\label{eq:spinc-index}
 \ind D^+_E
 =\left\langle\operatorname{ch}(E)e^{\ell/2}
       \widehat A(TM),[M]\right\rangle.
\end{equation}
The degree-four term of $\operatorname{ch}(E)$ is
$\tfrac12c_1(E)^2-c_2(E)$.  The rank and first Chern class of
$E,F$ agree, so subtracting their formulas leaves
\eqref{eq:index-difference}.  Poincar\'e duality gives
$H^4(M;\Z)\cong\Z$ on a connected closed oriented four-manifold;
the evaluation in~\eqref{eq:index-difference} therefore vanishes
exactly when the two second Chern classes agree.  Finally a gauge
preserves $A$ precisely when its block isomorphism intertwines the
connections, and Theorem~\ref{thm:nonlinear} characterizes that
condition by $\mathfrak m(E,F)=0$.
\end{proof}

\begin{remark}[Why the index needs its hypotheses]
The index test is complete here because $c_1$ and rank are fixed and
$H^4(M;\Z)$ is torsion-free.  On a general four-complex, an integer
index can miss torsion in $c_2$; in higher dimension, Chern classes
and stable $K$-theory may miss unstable rank-two information.  Exact
examples appear in Section~\ref{sec:boundaries}.
\end{remark}

\subsection{Torsion and unstable boundaries of index detection}
\label{sec:boundaries}

The four-dimensional criterion in Theorem~\ref{thm:block-index}
uses that $H^4$ of a closed connected oriented four-manifold is
torsion-free.  Two examples delimit its extension to general spaces.

\begin{example}[A torsion block obstruction]
Let $X=M(\Z/3,3)=S^3\cup_3 e^4$ be the four-dimensional Moore
space.  Its fourth integral cohomology is $\Z/3$; let $a$ be a
generator.  Rank-two $\SU(2)$ bundles over this four-complex are
classified by $c_2$, so choose $V_+$ and $V_-$ with
$c_2(V_+)=a$ and $c_2(V_-)=-a$.  Their direct sum has
$c_1=c_2=0$ and is a trivial $\SU(4)$ bundle by the same
four-type argument as Lemma~\ref{lem:fourtype}.  Its two
eigenblocks admit no exchange gauge, since $a\ne-a$ in $\Z/3$.

For a map from a closed oriented four-manifold $f:Y\to X$,
$f^*a=0$ because $H^4(Y;\Z)\cong\Z$ is torsion-free on each
connected component.  Hence the pullback blocks have the same
spin-c Dirac indices, for every choice of spin-c structure and
every common auxiliary twist.  The obstruction on $X$ is torsion
and cannot be read by these integer index profiles after such a
pullback.
\end{example}

\begin{example}[An unstable block obstruction]
The clutching classification gives
$\pi_4(\SU(2))\cong\Z/2$, so there is a nontrivial rank-two
$\SU(2)$ bundle $V\to S^5$.  All its Chern classes vanish for
dimensional reasons.  Its stabilization $V\oplus\underline{\C}^2$
is trivial: the inclusion $\SU(2)\to\SU(4)$ sends its clutching
class to zero because $\pi_4(\SU(4))=0$~\cite{Bott}.  Place
$+1$ on $V$ and $-1$ on the trivial rank-two summand.  The ambient
bundle is trivial, but an exchange gauge would identify $V$ with
$\underline{\C}^2$, which is impossible.  This obstruction is
invisible to Chern classes and to reduced complex $K$-theory of
$S^5$.
\end{example}


\begin{thebibliography}{99}

\bibitem{AdamsSimon1988}
D.~Adams and L.~Simon,
\emph{Rates of asymptotic convergence near isolated singularities
of geometric extrema},
Indiana Univ. Math. J. \textbf{37} (1988), 225--254,
\href{https://doi.org/10.1512/iumj.1988.37.37012}{doi:10.1512/iumj.1988.37.37012}.

\bibitem{AlmgrenLieb1988}
F.~J. Almgren, Jr. and E.~H. Lieb,
\emph{Singularities of energy minimizing maps from the ball to
the sphere: Examples, counterexamples, and bounds},
Ann. of Math. (2) \textbf{128} (1988), 483--530,
\href{https://doi.org/10.2307/1971434}{doi:10.2307/1971434}.

\bibitem{AlpinEtAl}
K.~Alpin, M.~M. Hirschmann, N.~Heinsdorf, A.~Leonhardt,
W.~Y. Yau, X.~Wu, and A.~P. Schnyder,
\emph{Fundamental laws of chiral band crossings: local constraints,
global constraints, and topological phase diagrams},
Phys. Rev. Research \textbf{5} (2023), 043165,
\href{https://doi.org/10.1103/PhysRevResearch.5.043165}{doi:10.1103/PhysRevResearch.5.043165}.

\bibitem{AtiyahReality1966}
M.~F. Atiyah,
\emph{$K$-theory and reality},
Quart. J. Math. Oxford Ser. (2) \textbf{17} (1966), 367--386.

\bibitem{AtiyahSinger}
M.~F. Atiyah and I.~M. Singer,
\emph{The index of elliptic operators. III},
Ann. of Math. (2) \textbf{87} (1968), 546--604.

\bibitem{BatakidisGkeneralis2026}
P.~Batakidis and I.~Gkeneralis,
\emph{Obstructions to lifting quaternionic torus actions},
(2026), \href{https://arxiv.org/abs/2608.07597}{arXiv:2608.07597}.


\bibitem{BeauvilleNarasimhanRamanan}
A.~Beauville, M.~S. Narasimhan, and S.~Ramanan,
\emph{Spectral curves and the generalised theta divisor},
J. Reine Angew. Math. \textbf{398} (1989), 169--179,
\href{https://doi.org/10.1515/crll.1989.398.169}{doi:10.1515/crll.1989.398.169}.

\bibitem{BierstoneMilman}
E.~Bierstone and P.~D. Milman,
\emph{Semianalytic and subanalytic sets},
Publ. Math. Inst. Hautes \`Etudes Sci. \textbf{67} (1988), 5--42,
\href{https://doi.org/10.1007/BF02699126}{doi:10.1007/BF02699126}.

\bibitem{Bott}
R.~Bott,
\emph{The stable homotopy of the classical groups},
Ann. of Math. (2) \textbf{70} (1959), 313--337.

\bibitem{BrezisCoronLieb}
H.~Brezis, J.-M. Coron, and E.~H. Lieb,
\emph{Harmonic maps with defects},
Commun. Math. Phys. \textbf{107} (1986), 649--705,
\href{https://doi.org/10.1007/BF01205490}{doi:10.1007/BF01205490}.

\bibitem{Bruno}
P.~Bruno,
\emph{Berry phase, topology, and degeneracies in quantum nanomagnets},
Phys. Rev. Lett. \textbf{96} (2006), 117208,
\href{https://doi.org/10.1103/PhysRevLett.96.117208}{doi:10.1103/PhysRevLett.96.117208}.

\bibitem{Bruckmann}
F.~Bruckmann,
\emph{Pointlike Hopf defects in Abelian projections},
arXiv:hep-th/0012004 (2000),
\url{https://arxiv.org/abs/hep-th/0012004}.

\bibitem{BrunsVetter}
W.~Bruns and U.~Vetter,
\emph{Determinantal Rings},
Lecture Notes in Mathematics \textbf{1327},
Springer, Berlin, 1988.

\bibitem{BurstallRawnsleySalamon1987}
F.~Burstall, J.~Rawnsley, and S.~Salamon,
\emph{Stable harmonic $2$-spheres in symmetric spaces},
Bull. Amer. Math. Soc. (N.S.) \textbf{16} (1987), no.~2, 274--278,
\href{https://doi.org/10.1090/S0273-0979-1987-15516-9}{doi:10.1090/S0273-0979-1987-15516-9}.

\bibitem{Collins2005}
B.~Collins,
\emph{Product of random projections, Jacobi ensembles and universality problems arising from free probability},
Probab. Theory Related Fields \textbf{133} (2005), 315--344.
\href{https://doi.org/10.1007/s00440-005-0428-5}{doi:10.1007/s00440-005-0428-5}.

\bibitem{CyrEtAl}
J.~Cyr, J.~Ekstrand, N.~Meyers, C.~Peoples, and J.~Peters,
\emph{Diagonalizing Hermitian matrices of continuous functions},
Int. J. Contemp. Math. Sci. \textbf{8} (2013), no.~5, 227--234,
\href{https://doi.org/10.12988/ijcms.2013.13023}{doi:10.12988/ijcms.2013.13023}.

\bibitem{DavisKahan}
C.~Davis and W.~M. Kahan,
\emph{The rotation of eigenvectors by a perturbation. III},
SIAM J. Numer. Anal. \textbf{7} (1970), 1--46,
\href{https://doi.org/10.1137/0707001}{doi:10.1137/0707001}.

\bibitem{DeNittisGomi}
G.~De Nittis and K.~Gomi,
\emph{The cohomological nature of the Fu--Kane--Mele invariant},
J. Geom. Phys. \textbf{124} (2018), 124--164,
\href{https://doi.org/10.1016/j.geomphys.2017.10.007}{doi:10.1016/j.geomphys.2017.10.007}.

\bibitem{Dong2002}
Y.~Dong,
\emph{On energy gap of unitons},
Math. Z. \textbf{240} (2002), 677--688.
\href{https://doi.org/10.1007/s00209-002-0344-7}{doi:10.1007/s00209-002-0344-7}.

\bibitem{EellsWood1983}
J.~Eells and J.~C. Wood,
\emph{Harmonic maps from surfaces to complex projective spaces},
Adv. Math. \textbf{49} (1983), 217--263,
\href{https://doi.org/10.1016/0001-8708(83)90062-2}{doi:10.1016/0001-8708(83)90062-2}.

\bibitem{Feehan2018}
P.~M.~N. Feehan,
\emph{Relative energy gap for harmonic maps of Riemann surfaces into
real analytic Riemannian manifolds},
Proc. Amer. Math. Soc. \textbf{146} (2018), 3179--3190,
\href{https://doi.org/10.1090/proc/14013}{doi:10.1090/proc/14013}.

\bibitem{FenilleGoncalves}
M.~C. Fenille and D.~L. Gon\c{c}alves,
\emph{Strongly surjective maps from certain two-complexes with trivial
top-cohomology onto the projective plane},
New York J. Math. \textbf{27} (2021), 615--630,
\url{https://nyjm.albany.edu/j/2021/27-24v.pdf}.


\bibitem{FrankPinterVarjasPalyi}
G.~Frank, G.~Pint\'er, D.~Varjas, and A.~P\'alyi,
\emph{Multifold degeneracy points of quantum systems and singularities of matrix varieties},
Res. Math. Sci. \textbf{13} (2026), article 37,
\href{https://doi.org/10.1007/s40687-026-00615-0}{doi:10.1007/s40687-026-00615-0}.


\bibitem{FriedmanPark2016}
G.~Friedman and E.~Park,
\emph{Unitary equivalence of normal matrices over topological spaces},
J. Topol. Anal. \textbf{8} (2016), 313--348.

\bibitem{FriedmanPark2024}
G.~Friedman and E.~Park,
\emph{Chern classes and unitary equivalence of normal matrices over topological spaces},
J. Operator Theory \textbf{91} (2024), no.~2, 421--441,
\href{https://doi.org/10.7900/jot.2022may27.2384}{doi:10.7900/jot.2022may27.2384}.

\bibitem{GhassemiTabrizi}
S.~Ghassemi Tabrizi,
\emph{Hidden symmetry at the diabolical points of a biaxial spin},
\href{https://arxiv.org/abs/2607.26308}{arXiv:2607.26308} (2026).

\bibitem{GompfStipsicz}
R.~E. Gompf and A.~I. Stipsicz,
\emph{4-Manifolds and Kirby Calculus},
Graduate Studies in Mathematics \textbf{20},
American Mathematical Society, Providence, 1999.

\bibitem{Grothendieck1957}
A.~Grothendieck,
\emph{Sur la classification des fibr\'es holomorphes sur la sph\`ere
de Riemann},
Amer. J. Math. \textbf{79} (1957), 121--138,
\href{https://doi.org/10.2307/2372388}{doi:10.2307/2372388}.

\bibitem{GrovePedersen}
K.~Grove and G.~K. Pedersen,
\emph{Diagonalizing matrices over $C(X)$},
J. Funct. Anal. \textbf{59} (1984), 65--89,
\href{https://doi.org/10.1016/0022-1236(84)90053-3}{doi:10.1016/0022-1236(84)90053-3}.

\bibitem{Fujioka1993}
A.~Fujioka,
\emph{Minimizing tangent maps from $3$-ball to complex projective spaces},
J. Fac. Sci. Univ. Tokyo Sect. IA Math. \textbf{40} (1993), 125--139,
\href{https://doi.org/10.15083/00039336}{doi:10.15083/00039336}.

\bibitem{GulliverWhite}
R.~D. Gulliver and B.~White,
\emph{The rate of convergence of a harmonic map at a singular point},
Math. Ann. \textbf{283} (1989), 539--549,
\href{https://doi.org/10.1007/BF01442853}{doi:10.1007/BF01442853}.

\bibitem{HardtLin1989}
R.~Hardt and F.-H.~Lin,
\emph{Stability of singularities of minimizing harmonic maps},
J. Differential Geom. \textbf{29} (1989), no.~1, 113--123.

\bibitem{HardtRiviere}
R.~Hardt and T.~Rivi\`ere,
\emph{Connecting topological Hopf singularities},
Ann. Scuola Norm. Sup. Pisa Cl. Sci. (5) \textbf{2} (2003),
287--344.


\bibitem{HatcherVB}
A.~Hatcher,
\href{https://pi.math.cornell.edu/~hatcher/VBKT/VB.pdf}
{\emph{Vector Bundles and $K$-Theory}}, version 2.2,
2017.

\bibitem{HattoriYoshida1976}
A.~Hattori and T.~Yoshida,
\emph{Lifting compact group actions in fiber bundles},
Japanese J. Math. (New Ser.) \textbf{2} (1976), 13--25.
\href{https://doi.org/10.4099/math1924.2.13}{doi:10.4099/math1924.2.13}.

\bibitem{HuHouZhangZhang2018}
H.~Hu, J.~Hou, F.~Zhang, and C.~Zhang,
\emph{Topological triply degenerate points induced by
spin-tensor-momentum couplings},
Phys. Rev. Lett. \textbf{120} (2018), 240401,
\href{https://doi.org/10.1103/PhysRevLett.120.240401}{doi:10.1103/PhysRevLett.120.240401}.

\bibitem{Jahn}
O.~Jahn,
\emph{Instantons and monopoles in general Abelian gauges},
J. Phys. A \textbf{33} (2000), 2997--3019,
\href{https://arxiv.org/abs/hep-th/9909004}{arXiv:hep-th/9909004}.

\bibitem{Kato}
T.~Kato,
\emph{Perturbation Theory for Linear Operators},
2nd ed., Classics in Mathematics, Springer, Berlin, 1995.

\bibitem{Kawabe2013}
H.~Kawabe,
\emph{A space of harmonic maps from a sphere into the complex
projective space},
Canad. J. Math. \textbf{65} (2013), no.~4, 879--904,
\href{https://doi.org/10.4153/CJM-2012-052-6}{doi:10.4153/CJM-2012-052-6}.

\bibitem{KececiogluGarg}
E.~Ke\c{c}ecio\u{g}lu and A.~Garg,
\emph{Diabolical points in magnetic molecules: An exactly solvable model},
Phys. Rev. B \textbf{63} (2001), 064422,
\href{https://doi.org/10.1103/PhysRevB.63.064422}{doi:10.1103/PhysRevB.63.064422}.

\bibitem{KhavinsonSwiatek2003}
D.~Khavinson and G.~\'{S}wi\k{a}tek,
\emph{On the number of zeros of certain harmonic polynomials},
Proc. Amer. Math. Soc. \textbf{131} (2003), no.~2, 409--414,
\href{https://doi.org/10.1090/S0002-9939-02-06476-6}{doi:10.1090/S0002-9939-02-06476-6}.

\bibitem{KronheimerMrowka1994}
P.~B. Kronheimer and T.~S. Mrowka,
\emph{The genus of embedded surfaces in the projective plane},
Math. Res. Lett. \textbf{1} (1994), no.~6, 797--808,
\href{https://doi.org/10.4310/MRL.1994.v1.n6.a14}{doi:10.4310/MRL.1994.v1.n6.a14}.

\bibitem{KurdykaPaunescu}
K.~Kurdyka and L.~Paunescu,
\emph{Hyperbolic polynomials and multiparameter real-analytic perturbation theory},
Duke Math. J. \textbf{141} (2008), 123--149,
\href{https://doi.org/10.1215/S0012-7094-08-14113-4}{doi:10.1215/S0012-7094-08-14113-4}.

\bibitem{Luckhaus1988}
S.~Luckhaus,
\emph{Partial H\"older continuity for minima of certain energies
among maps into a Riemannian manifold},
Indiana Univ. Math. J. \textbf{37} (1988), 349--367,
\href{https://doi.org/10.1512/iumj.1988.37.37017}{doi:10.1512/iumj.1988.37.37017}.

\bibitem{LapierreNeupertTrifunovic}
B.~Lapierre, T.~Neupert, and L.~Trifunovic,
\emph{$N$-band Hopf insulator},
Phys. Rev. Research \textbf{3} (2021), 033045.

\bibitem{LemaireWood2002}
L.~Lemaire and J.~C. Wood,
\emph{Jacobi fields along harmonic $2$-spheres in $\C P^2$ are integrable},
J. London Math. Soc. (2) \textbf{66} (2002), 468--486,
\href{https://doi.org/10.1112/S0024610702003496}{doi:10.1112/S0024610702003496}.

\bibitem{Lerman}
E.~Lerman,
\emph{Contact fiber bundles},
J. Geom. Phys. \textbf{49} (2004), 52--66.

\bibitem{MalvettiDirrVomEndeSchulteHerbrueggen}
E.~Malvetti, G.~Dirr, F.~vom Ende, and T.~Schulte-Herbr\"uggen,
\emph{Analytic, differentiable and measurable diagonalizations in
symmetric Lie algebras},
\href{https://arxiv.org/abs/2212.00713}{arXiv:2212.00713v2} (2023).

\bibitem{MardesicSegal}
S.~Marde\v{s}i\'c and J.~Segal,
\href{https://shop.elsevier.com/books/shape-theory/mardesic/978-0-444-86286-0}
{\emph{Shape Theory: The Inverse System Approach}},
North-Holland Mathematical Library \textbf{26},
North-Holland, Amsterdam, 1982.

\bibitem{MimuraToda}
M.~Mimura and H.~Toda,
\emph{Homotopy groups of $\SU(3)$, $\SU(4)$ and $\operatorname{Sp}(2)$},
J. Math. Kyoto Univ. \textbf{3} (1964), 217--250.

\bibitem{MMS2024}
K.~Mazowiecka, M.~Mi\'skiewicz, and A.~Schikorra,
\emph{On the size of the singular set of minimizing harmonic maps},
Mem. Amer. Math. Soc. \textbf{302} (2024), no.~1519, v+82 pp.,
\href{https://doi.org/10.1090/memo/1519}{doi:10.1090/memo/1519}.

\bibitem{Miskiewicz2020}
M.~Mi\'skiewicz,
\emph{On H\"older regularity of the singular set of energy minimizing
harmonic maps into closed manifolds},
Calc. Var. Partial Differential Equations \textbf{59} (2020), article~36,
\href{https://doi.org/10.1007/s00526-019-1685-y}{doi:10.1007/s00526-019-1685-y}.

\bibitem{NaberValtorta2017}
A.~Naber and D.~Valtorta,
\emph{Rectifiable-Reifenberg and the regularity of stationary
and minimizing harmonic maps},
Ann. of Math. (2) \textbf{185} (2017), 131--227,
\href{https://doi.org/10.4007/annals.2017.185.1.3}{doi:10.4007/annals.2017.185.1.3}.

\bibitem{PapathanasiouPsarrakos}
N.~Papathanasiou and P.~Psarrakos,
\emph{The distance from a matrix polynomial to matrix polynomials
with a prescribed multiple eigenvalue},
Linear Algebra Appl. \textbf{429} (2008), 1453--1477,
\href{https://doi.org/10.1016/j.laa.2008.04.005}{doi:10.1016/j.laa.2008.04.005}.

\bibitem{Roberts2012}
B.~W. Roberts,
\emph{Kramers degeneracy without eigenvectors},
Phys. Rev. A \textbf{86} (2012), 034103.

\bibitem{Rellich}
F.~Rellich,
\emph{Perturbation Theory of Eigenvalue Problems},
Gordon and Breach, New York, 1969.

\bibitem{SchoenUhlenbeck}
R.~Schoen and K.~Uhlenbeck,
\emph{A regularity theory for harmonic maps},
J. Differential Geom. \textbf{17} (1982), 307--335,
\href{https://doi.org/10.4310/JDG/1214436923}{doi:10.4310/JDG/1214436923}.

\bibitem{Seiler}
W.~M. Seiler,
\emph{Spencer cohomology, differential equations, and Pommaret bases},
in \emph{Gr\"obner Bases in Symbolic Analysis},
Radon Series on Computational and Applied Mathematics \textbf{2},
de Gruyter, Berlin, 2007, pp.~169--216.

\bibitem{Simon1983}
L.~Simon,
\emph{Asymptotics for a class of non-linear evolution equations, with applications to geometric problems},
Ann. of Math. (2) \textbf{118} (1983), 525--571,
\href{https://doi.org/10.2307/2006981}{doi:10.2307/2006981}.

\bibitem{SunZhangBzdusek2018}
X.-Q. Sun, S.-C. Zhang, and T. Bzdu\v{s}ek,
\emph{Conversion rules for Weyl points and nodal lines in
topological media},
Phys. Rev. Lett. \textbf{121} (2018), 106402,
\href{https://doi.org/10.1103/PhysRevLett.121.106402}
{doi:10.1103/PhysRevLett.121.106402}.


\bibitem{TeoKane}
J.~C.~Y. Teo and C.~L. Kane,
\emph{Topological defects and gapless modes in insulators and superconductors},
Phys. Rev. B \textbf{82} (2010), 115120.

\bibitem{Toda}
H.~Toda,
\emph{Composition Methods in Homotopy Groups of Spheres},
Annals of Mathematics Studies \textbf{49},
Princeton University Press, Princeton, 1962.

\enlargethispage{3\baselineskip}
\bibitem{Weber1999}
C.~Weber,
\emph{The classification of $\operatorname{Pin}_4$-bundles over a 4-complex},
Canad. Math. Bull. \textbf{42} (1999), 248--256.
\href{https://doi.org/10.4153/CMB-1999-030-9}{doi:10.4153/CMB-1999-030-9}.


\end{thebibliography}
\end{document}